\documentclass[11pt, a4paper, reqno]{amsart}
\usepackage[T1]{fontenc}
\usepackage[latin1]{inputenc}
\usepackage[english]{babel}
\usepackage[reqno]{amsmath}
\usepackage{amsthm}
\usepackage{latexsym}
\usepackage{amssymb}
\usepackage{mathrsfs}
\usepackage{mathtools}
\usepackage{mathabx}
\usepackage{lipsum}
\usepackage{titlesec}

\usepackage[totalwidth=15cm,totalheight=20cm, hmarginratio=1:1]{geometry}
\usepackage{dsfont} 
\usepackage{mathtools}
\usepackage{color}
\usepackage{physics}
\usepackage{faktor}
\usepackage[shortlabels]{enumitem}
\usepackage{tikz} 
\usepackage{hyperref}
\usepackage{graphicx}
\usepackage{tabularx}
\usepackage{cases} 
\newcolumntype{C}{>{\centering\arraybackslash}X} 
\usepackage{csquotes}
\usepackage{url}

\usepackage{accents}

\makeatletter
\newtheorem*{rep@theorem}{\rep@title}
\newcommand{\newreptheorem}[2]{%
\newenvironment{rep#1}[1]{%
 \def\rep@title{#2 \ref{##1}}%
 \begin{rep@theorem}}%
 {\end{rep@theorem}}}
\makeatother

\makeatletter
\newtheorem*{rep@cor}{\rep@title}
\newcommand{\newrepcor}[2]{%
\newenvironment{rep#1}[1]{%
 \def\rep@title{#2 \ref{##1}}%
 \begin{rep@cor}}%
 {\end{rep@cor}}}
\makeatother

\makeatletter
\newtheorem*{rep@prop}{\rep@title}
\newcommand{\newrepprop}[2]{%
\newenvironment{rep#1}[1]{%
 \def\rep@title{#2 \ref{##1}}%
 \begin{rep@prop}}%
 {\end{rep@prop}}}
\makeatother

\newtheorem{theorem}{Theorem}[section]
\newreptheorem{theorem}{Theorem}
\numberwithin{theorem}{section}

\newtheorem{theoremx}{Theorem}

\newtheorem{lemma}[theorem]{Lemma}
\newtheorem{corollary}[theorem]{Corollary}

\newrepcor{corollary}{Corollary}
\newtheorem{corollaryx}[theoremx]{Corollary}

\newtheorem{proposition}[theorem]{Proposition}

\newrepprop{prop}{Proposition}
\newtheorem{propx}[theoremx]{Proposition}

\theoremstyle{definition}
\newtheorem*{definition*}{Definition}
\newtheorem{definition}[theorem]{Definition}

\theoremstyle{remark}
\newtheorem{remark}[theorem]{Remark}

\makeatletter
\def\paragraph{\@startsection{paragraph}{4}%
  \z@\z@{-\fontdimen2\font}%
  {\normalfont\bfseries}}
\makeatother

\numberwithin{equation}{section}
\allowdisplaybreaks

\patchcmd{\subsection}{-.5em}{.5em}{}{}

\makeatletter

\renewcommand\section{\@startsection{section}{1}%
  \z@{.7\linespacing\@plus\linespacing}{.5\linespacing}%
  {\normalfont\scshape\centering}}

\renewcommand\subsection{\@startsection{subsection}{2}%
  \z@{-.5\linespacing\@plus-.7\linespacing}{.5\linespacing}%
  {\bfseries}}
  
\renewcommand\subsubsection{\@startsection{subsubsection}{3}%
  \z@{-.5\linespacing\@plus-.7\linespacing}{.5\linespacing}%
  {\itshape}}
  
\def\l@paragraph{\@tocline{4}{0pt}{1pc}{7pc}{}}

\makeatother

\newcommand{\defin}{\vcentcolon =}
\newcommand{\C}{\mathbb{C}}
\newcommand{\R}{\mathbb{R}}

\newcommand{\Z}{\mathbb{Z}}

\newcommand{\Hyp}{\mathbb{H}}

\newcommand{\Teich}{\mathcal{T}}
\newcommand{\WP}{\textit{WP}}

\newcommand{\MS}{\mathcal{MS}}

\newcommand{\I}{I}
\newcommand{\II}{I \! \! I}

\newcommand{\Lsl}{\mathfrak{sl}}
\newcommand{\Dlie}{\mathcal{L}}

\newcommand{\Lsymp}{\mathfrak{S}}
\newcommand{\Lham}{\mathfrak{H}}

\newcommand{\hyp}{\mathfrak{h}}
\newcommand{\id}{\textit{id}}
\newcommand{\1}{\mathds{1}}
\newcommand{\vertical}{\textit{v}}

\newcommand{\conf}{\mathfrak{c}}
\newcommand{\rep}{\mathfrak{rep}}

\newcommand{\Ker}{\mathrm{Ker}}

\newcommand{\dPSL}{\mathbb{P}\mathrm{SL}(2,\mathbb{R})\times \mathbb{P}\mathrm{SL}(2,\mathbb{R})}

\newcommand{\PSL}{\mathbb{P}\mathrm{SL}}
\newcommand{\T}{\mathcal{T}}

\newcommand{\Isom}{\mathrm{Isom}}

\newcommand{\SO}{\mathrm{SO}}

\newcommand{\B}{\mathbb{B}}

\newcommand{\tw}{\mathrm{tw}}

\renewcommand{\i}{\mathbf{I}}
\renewcommand{\j}{\mathbf{J}}
\renewcommand{\k}{\mathbf{K}}
\newcommand{\p}{\mathbf{P}}
\newcommand{\g}{\mathbf{g}}
\newcommand{\Ll}{\mathbf{L}}
\newcommand{\mappa}[3]{#1 \colon #2 \rightarrow #3}
\newcommand{\hsk}{\hskip0pt}
\newcommand{\chidf}{\chi_0(\Sigma, \PSL(2,\B))}

\DeclarePairedDelimiterX{\scal}[2]{\langle}{\rangle}{#1 \mid #2}
\DeclarePairedDelimiterX{\scall}[2]{\langle}{\rangle}{#1, #2}
\DeclarePairedDelimiter{\set}{\{}{\}}
\DeclareMathOperator{\Imm}{Im}
\DeclareMathOperator{\Span}{Span}

\DeclareMathOperator{\End}{End}

\DeclareMathOperator{\SL}{\mathrm{SL}}

\DeclareMathOperator{\arccosh}{arccosh}

\DeclareMathOperator{\divr}{div}

\DeclareMathOperator{\grd}{grad}
\DeclareMathOperator{\Area}{Area}

\DeclareMathOperator{\Symp}{Symp}

\DeclareMathOperator{\Ham}{Ham}

\DeclareMathOperator{\Ad}{Ad}
\DeclareMathOperator{\Lie}{Lie}

\DeclareMathOperator{\hol}{hol}

\begin{document}

\setcounter{secnumdepth}{3}
\setcounter{tocdepth}{2}

\title[Para-hyperK{\"a}hler geometry of the space of AdS structures]{Para-hyperK{\"a}hler geometry of the deformation space of maximal globally hyperbolic anti-de Sitter three-manifolds}

\author[Filippo Mazzoli]{Filippo Mazzoli}
\address{FM: Department of Mathematics, University of Virginia, Charlottesville (VA), USA.} \email{filippomazzoli@me.com} 

\author[Andrea Seppi]{Andrea Seppi}
\address{AS: Institut Fourier, UMR 5582, Laboratoire de Math\'ematiques,
Universit\'e Grenoble Alpes, CS 40700, 38058 Grenoble cedex 9, France. } \email{andrea.seppi@univ-grenoble-alpes.fr}

\author[Andrea Tamburelli]{Andrea Tamburelli}
\address{AT: Department of Mathematics, Rice University, Houston (TX), USA.} \email{andrea.tamburelli@libero.it}

\date{\today}

\begin{abstract} In this paper we study the para-hyperK\"ahler geometry of the deformation space of MGHC anti-de Sitter structures on $\Sigma\times\mathbb R$, for $\Sigma$ a closed oriented surface. We show that a neutral pseudo-Riemannian metric and three symplectic structures coexist with an integrable complex structure and two para-complex structures, satisfying the relations of para-quaternionic numbers. We show that these structures are directly related to the geometry of MGHC manifolds, via the Mess homeomorphism, the parameterization of Krasnov-Schlenker by the induced metric on $K$-surfaces, the identification with the cotangent bundle $T^*\mathcal T(\Sigma)$, and the circle action that arises from this identification. Finally, we study the relation to the natural para-complex geometry that the space inherits from being a component of the $\PSL(2,\B)$-character variety, where $\B$ is the algebra of para-complex numbers, and the symplectic geometry deriving from Goldman symplectic form. 
\end{abstract}

\maketitle

\tableofcontents



\section{Introduction}

The main purpose of this paper is to study the geometry of the deformation space of maximal globally hyperbolic Cauchy-compact three-dimensional Anti-de Sitter three-manifolds. In short, our results show that these deformation spaces are endowed with a mapping-class group invariant para-hyperK\"ahler metric, and then provide geometric interpretations to each element that constitutes the para-hyperK\"ahler structure. 

\subsection{Motivation and state-of-the-art}

Since the pioneering work of Mess of 1990 \cite{mess2007lorentz}, maximal globally hyperbolic Anti-de Sitter manifolds in dimension three have been largely studied, motivated on the one hand by the striking analogies with quasi-Fuchsian manifolds, and on the other by the deep relations with Teichm\"uller theory. See also, among others, \cite{notes,benedetti2009canonical_wick,barbotadsqf,barbotkleinian,bonsante2020antide}. 

In particular, the deformation space of maximal globally hyperbolic Anti-de Sitter manifolds, which in this paper we will denote by $\mathcal{MGH}(\Sigma)$, where $\Sigma$ is a closed surface of genus $\geq 2$, is intimately related with the Teichm\"uller space $\mathcal T(\Sigma)$. This has been first observed by Mess, who provided a parametrization of $\mathcal{MGH}(\Sigma)$ by the product $\mathcal T(\Sigma)\times\mathcal T(\Sigma)$; other parametrizations again by $\mathcal T(\Sigma)\times\mathcal T(\Sigma)$ or by $T^*\mathcal T(\Sigma)$ were introduced in \cite{krasnov_schlenker_minimal}. The latter relies on existence and uniqueness results for maximal surfaces, as shown in \cite{bbz,bonsante2010maximal}. A further understanding of the geometry of Anti-de Sitter manifolds and their deformation space has been obtained by the study of geometric invariants such as the \emph{convex core} (\cite{bsduke,JMscarinci}), its \emph{width} (\cite{bonsante2010maximal,seppijems}), its \emph{volume} (\cite{bst}), by means of \emph{surfaces with curvature conditions} (\cite{bbz,abbz,bonsepK,tambH,chentam}), and by the \emph{symplectic geometric} approach  (\cite{bonsante2013a_cyclic,bonsante2015a_cyclic,bonsepequiv}).

Stepping back to the parallel with quasi-Fuchsian hyperbolic  manifolds, recently Donaldson highlighted the existence of a natural \emph{hyperK\"ahler} structure on a neighborhood of the Fuchsian locus in the deformation space of almost-Fuchsian manifolds, seen as a neighborhood of the zero section in the cotangent bundle $T^*\mathcal T(\Sigma)$. See \cite{donaldson2003moment,hodgethesis,trautwein_thesis,trautwein2019hyperkahler}. The purpose of this paper is to develop a similar approach for maximal globally hyperbolic Anti-de Sitter manifolds, and to demonstrate that the natural structure that appears in this setting is a \emph{para-hyperK\"ahler structure}. For more details on para-K\"ahler and para-hyperK\"ahler geometry, see \cite{parahkl2,pahahkl3,parahkh1,parahkl4}. We will see that this structure recovers many of the geometric constructions that have been introduced before, so as to elucidate the global picture and the relations between different approaches.
We now give the fundamental definitions and state our main results. 

\subsection{Deformation space of MGHC AdS manifolds}

We give here the standard definition of \emph{maximal globally hyperbolic Cauchy compact Anti-de Sitter manifolds} (in short, MGHC AdS). A \emph{Cauchy surface} in a Lorentzian manifold is an embedded hypersurface that intersects every inextensible causal curve exactly in one point; a Lorentzian manifold admitting a Cauchy surface is called \emph{globally hyperbolic}. It is moreover \emph{maximal} if every isometric embedding in another globally hyperbolic manifold sending a Cauchy surface to a Cauchy surface is surjective. 
Finally, a MGHC AdS manifold  is a maximal globally hyperbolic Lorentzian manifold of constant sectional curvature $-1$ admitting a closed Cauchy surface. A simple example of MGHC AdS manifolds are \emph{Fuchsian manifolds}, whose metric $G$ can be written globally as a warped product 
\begin{equation}\label{eq:fuchsian_metric}
G=-dt^2+\cos^2(t)h~,
\end{equation} 
for $t\in(-\pi/2,\pi/2)$ and $h$ a hyperbolic metric on a closed manifold. In this case the Cauchy surface $t=0$ is totally geodesic.

A classical fact in Lorentzian geometry (see \cite{geroch,beem,bernal}) is that globally hyperbolic Lorentzian manifolds are diffeomorphic to $\Sigma\times\R$, where $\Sigma$ is a Cauchy surface, and any two Cauchy surfaces are diffeomorphic. In this work we will only consider three-dimensional AdS manifolds whose Cauchy surfaces are closed. Hence from now on we will fix a closed oriented surface $\Sigma$. We then define the \emph{deformation space} of MGHC AdS manifolds as follows:
$$\mathcal{MGH}(\Sigma):=\{G\,|\,G\text{ is a MGHC AdS metric on }\Sigma\times\R\}/\mathrm{Diff}_0(\Sigma\times\R)~,$$
where the group $\mathrm{Diff}_0(\Sigma\times\R)$ of diffeomorphisms isotopic to the identity acts by pull-back of $G$. It turns out that $\mathcal{MGH}(S^2)$ is empty, $\mathcal{MGH}(T^2)$ is a four-dimensional manifold, while if $\Sigma$ is a surface of genus $\geq 2$, then $\mathcal{MGH}(\Sigma)$ has dimension $6|\chi(\Sigma)|$. Observe moreover that there is a natural action of the  mapping class-group $MCG(\Sigma)=\mathrm{Diff}_+(\Sigma)/\mathrm{Diff}_0(\Sigma)$ on $\mathcal{MGH}(\Sigma)$, again by pull-back. When $\Sigma$ has genus $\geq 2$, the deformation space $\mathcal{MGH}(\Sigma)$ contains the \emph{Fuchsian locus} $\mathcal F(\Sigma)$, namely those manifolds whose metric is of the form \eqref{eq:fuchsian_metric}, which is $MCG(\Sigma)$-invariant and naturally identified to the Teichm\"uller space $\mathcal{T}(\Sigma)$.

\subsection{Para-hyperK\"ahler structures}
We now introduce the notion of para-hyperK\"ahler structure and state our first result. Recall that a \emph{pseudo-K\"ahler structure} on a manifold $M$ consists of a pair $(\g,\i)$ where $\g$ is a pseudo-Riemannian metric and $\i$ is an integrable almost complex structure (i.e. $\i^2=-\1$) such that $\g(\i v,w)=-\g(v,\i w)$ and the 2-form $\omega_\i(\cdot,\cdot):=\g(\cdot,\i\cdot)$ is closed (hence a symplectic form). Similarly, a \emph{para-K\"ahler structure} consists of an integrable almost para-complex structure $\p$, which means that 
\begin{itemize}
\item $\p^2=\1$;
\item the $\p$-eigenspaces of $1$ and $-1$ have the same dimension;
\item the distributions on $M$ given by the $1$ and $-1$ eigenspaces of $\p$ are integrable;
\end{itemize}
and $\p$ is such that $\g(\p v,w)=-\g(v,\p w)$ and the 2-form $\omega_\p(\cdot,\cdot):=\g(\cdot,\p\cdot)$ is closed. 

Observe that a direct consequence of the existence of a para-K\"ahler structure is that $\g(\p\cdot,\p\cdot)=-\g(\cdot,\cdot)$, hence $\g$ is necessarily of neutral signature. Moreover the condition that $d\omega_\i=0$ (resp. $d\omega_\p=0$) is known to be equivalent to $\nabla\i=0$ (resp. $\nabla\p=0$), for $\nabla$ the Levi-Civita connection of $\g$. We finally give the definition of para-hyperK\"ahler structure:

\begin{definition*}
 A para-hyperK\"ahler structure on a manifold $M$ is the data $(\g,\i,\j,\k)$, where $(\g,\i)$ is a pseudo-K\"ahler structure,  $(\g,\j)$ and $(\g,\k)$ are para-K\"ahler structures, and $(\i,\j,\k)$ satisfy the para-quaternionic relations.
\end{definition*}

By para-quaternionic relations we mean the identities $\i^2=-\1$, $\j^2=\k^2=\1$ --- which are implicitly assumed by the condition that $\i$ (resp. $\j$, $\k$) is a complex (resp. para-complex) structure --- and moreover $\i\j=-\j\i=\k$.

We remark that, given a para-hyperK\"ahler structure $(\g,\i,\j,\k)$, a \emph{complex symplectic form} is defined by:
$$\omega_\i^\C:=\omega_\j+i\omega_\k~.$$
It is complex in the sense that it is a $\C$-valued symplectic form and satisfies
 $\omega^\C_\i(\i v,w)=\omega^\C_\i(v,\i w)=i\omega^\C_\i(v,w)$.
 Similarly, one has two \emph{para-complex symplectic forms} defined by
 $$\omega_\j^{\mathbb B}:=\omega_\i+\tau\omega_\k\qquad\text{ and} \qquad \omega_\k^{\mathbb B}:=\omega_\i-\tau\omega_\j~,$$
where we denote by $\mathbb B=\R\oplus\tau\R$ the algebra of para-complex numbers, i.e. $\tau^2=1$. Again,  these are para-complex in the sense that
$\omega^{\mathbb B}_\j(\j v,w)=\omega^{\mathbb B}_\j(v,\j w)=\tau\omega^{\mathbb B}_\j(v,w)$ and $\omega^{\mathbb B}_\k(\k v,w)=\omega^{\mathbb B}_\k(v,\k w)=\tau\omega^{\mathbb B}_\k(v,w)$.

Only manifolds of dimension $4n$ can support a para-hyperK\"ahler structure. Our first result is that $\mathcal{MGH}(\Sigma)$, whose dimension is four if $\Sigma$ has genus one and $6|\chi(\Sigma)|$ otherwise, does support a very natural one. 

\begin{theoremx}\label{thm:parahyper_structure}
Let $\Sigma$ be a closed oriented surface of genus $\geq 1$. Then $\mathcal{MGH}(\Sigma)$ admits a $MCG(\Sigma)$-invariant para-hyperK\"ahler structure $(\g,\i,\j,\k)$. When $\Sigma$ has genus $\geq 2$, the Fuchsian locus $\mathcal F(\Sigma)$ is totally geodesic and $(\g,\i)$ restricts to (a multiple of) the Weil-Petersson K\"ahler structure of Teichm\"uller space.   
\end{theoremx}

The para-hyperK\"ahler structure of $\mathcal{MGH}(\Sigma)$ is extremely natural from the point of view of AdS geometry, in the sense that all the elements that constitute the para-hyperK\"ahler structure have (at least one) interpretation in terms of the geometry of MGHC AdS manifolds. We now state and explain all these interpretations. 

\subsection{Parameterizations of \texorpdfstring{$\mathcal{MGH}(\Sigma)$}{MGH(S)}}

The first interpretation is in terms of the cotangent bundle of Teichm\"uller space. There is a natural map
$$\mathcal F:\mathcal{MGH}(\Sigma)\to T^*\mathcal{T}(\Sigma)~,$$
which associates to a MGHC AdS manifold $(\Sigma\times\R,G)$ the pair $(J,q)$, where $J$ is the (almost-)complex structure of the first fundamental form of the unique maximal Cauchy surface in $(M,G)$, and $q$ is the holomorphic quadratic differential whose real part is the second fundamental form. The map $\mathcal F$ is a ($MCG(\Sigma)$-equivariant) diffeomorphism if $\Sigma$ has genus $\geq 2$; for genus one it is a diffeomorphism onto the complement of the zero section. The cotangent bundle $T^*\mathcal{T}(\Sigma)$ is naturally a complex symplectic manifold; our first geometric interpretation is the fact that the map $\mathcal F$ is anti-holomorphic and preserves the complex symplectic forms up to conjugation. 

\begin{theoremx}\label{thm:cotangent}
Let $\Sigma$ be a closed oriented surface of genus $\geq 1$. Then 
$$\mathcal F^*(\mathcal I_{T^*\mathcal{T}(\Sigma)},\Omega^\C_{T^*\mathcal{T}(\Sigma)})=\left(-\i,-\frac{i}{2}\overline\omega_\i^\C\right)~,$$
where $\mathcal I_{T^*\mathcal{T}(\Sigma)}$ denotes the complex structure of $T^*\mathcal{T}(\Sigma)$ and $\Omega^\C_{T^*\mathcal{T}(\Sigma)}$ its complex symplectic form.
\end{theoremx}

Let us assume (until the end of this section) that $\Sigma$ has genus $\geq 2$. In \cite{mess2007lorentz}, Mess proved that $\mathcal{MGH}(\Sigma)$ is parameterized by the product of two copies of the Teichm\"uller space of $\Sigma$, by a map
$$\mathcal M:\mathcal{MGH}(\Sigma)\to \mathcal{T}(\Sigma)\times \mathcal{T}(\Sigma)~,$$
that essentially gives (under the isomorphism between the isometry group of AdS space and $\PSL(2,\R)\times\PSL(2,\R)$), the left and right components of the holonomy map of a MGHC AdS manifold $(M,G)$. The manifold $\mathcal{T}(\Sigma)\times \mathcal{T}(\Sigma)$ is easily a para-complex manifold, where the para-complex structure $\mathcal P_{\mathcal{T}(\Sigma)\times \mathcal{T}(\Sigma)}$ is the endomorphism of the tangent bundle for which the integral submanifolds of the distribution of $1$-eigenspaces are the slices $\mathcal T(\Sigma)\times\{*\}$, and those for the $(-1)$-eigenspaces are the slices $\{*\}\times\mathcal T(\Sigma)$. It has moreover a para-complex symplectic form compatible with $\mathcal P_{\mathcal{T}(\Sigma)\times \mathcal{T}(\Sigma)}$:
$$\Omega_{\mathcal{T}(\Sigma)\times \mathcal{T}(\Sigma)}:=\frac{1}{2}(\pi_l^*\Omega_{WP}+\pi_r^*\Omega_{WP})+\frac{\tau}{2} (\pi_l^*\Omega_{WP}-\pi_r^*\Omega_{WP})$$
where $\Omega_{WP}$ is the Weil-Petersson symplectic form and $\pi_l,\pi_r$ denote the projections on the left and right factor. 

\begin{theoremx}\label{thm:mappaM}
Let $\Sigma$ be a closed oriented surface of genus $\geq 2$. Then 
$$\mathcal M^*(\mathcal P_{\mathcal{T}(\Sigma)\times \mathcal{T}(\Sigma)},4\Omega^{\mathbb B}_{\mathcal{T}(\Sigma)\times \mathcal{T}(\Sigma)})=(\j,\omega_\j^{\mathbb B})~,$$
where $\mathcal P_{\mathcal{T}(\Sigma)\times \mathcal{T}(\Sigma)}$ denotes the para-complex structure of ${\mathcal{T}(\Sigma)\times \mathcal{T}(\Sigma)}$ and $\Omega^{\mathbb B}_{\mathcal{T}(\Sigma)\times \mathcal{T}(\Sigma)}$ its para-complex symplectic form.
\end{theoremx}

Combining Theorems \ref{thm:cotangent} and \ref{thm:mappaM} in a particular case, we see that $(1/2)\omega_\k$ equals on the one hand the pull-back by $\mathcal M$ of the symplectic form $\pi_l^*\Omega_{WP}-\pi_r^*\Omega_{WP}$, and on the other hand the pull-back by $\mathcal F$ of minus the real part of $\Omega^\C_{T^*\mathcal{T}(\Sigma)}$ (i.e. the natural real symplectic form of the cotangent bundle). This identity has been proved in \cite[Theorem 1.14]{JMscarinci}, by completely different methods.

There is another parameterization of $\mathcal{MGH}(\Sigma)$ by the product of two copies of the Teichm\"uller space of $\Sigma$, which has been introduced in \cite{krasnov_schlenker_minimal}. It is given by the map
$$\mathcal C:\mathcal{MGH}(\Sigma)\to \mathcal{T}(\Sigma)\times \mathcal{T}(\Sigma)~,$$
which associates to $(M,G)$ the first fundamental forms of the two Cauchy surfaces (one future-convex, one past-convex) of constant intrinsic curvature $-2$. These two Cauchy surfaces of constant curvature are unique (\cite{barbot2011prescribing,bonsepK}), and we rescale their first fundamental forms by a factor so as to consider them as hyperbolic metrics. We show:

\begin{theoremx}\label{thm:mappaC}
Let $\Sigma$ be a closed oriented surface of genus $\geq 2$. Then 
$$\mathcal C^*(\mathcal P_{\mathcal{T}(\Sigma)\times \mathcal{T}(\Sigma)},4\Omega^{\mathbb B}_{\mathcal{T}(\Sigma)\times \mathcal{T}(\Sigma)})=(\k,\omega_\k^{\mathbb{B}})~,$$
where $\mathcal P_{\mathcal{T}(\Sigma)\times \mathcal{T}(\Sigma)}$ denotes the para-complex structure of ${\mathcal{T}(\Sigma)\times \mathcal{T}(\Sigma)}$ and $\Omega^{\mathbb B}_{\mathcal{T}(\Sigma)\times \mathcal{T}(\Sigma)}$ its para-complex symplectic form.
\end{theoremx}

We remark that there are formal analogues of Theorem \ref{thm:mappaM} and Theorem \ref{thm:mappaC} in genus one (see Section \ref{subsec:Mess_homeo} and \ref{sec:constant_curv}), but the corresponding maps $\mathcal M,\mathcal C:\mathcal{MGH}(T^2)\to \mathcal T(T^2)\times  \mathcal T(T^2)$ do not have the same geometric interpretation (namely, the holonomy map or the constant curvature surfaces) as in the higher genus case, which is why we restricted to genus $\geq 2$ when stating these results here.

\subsection{The circle action}

We now move on to studying a \emph{circle action} on $\mathcal{MGH}(\Sigma)$. Using the diffeomorphism $\mathcal F:\mathcal{MGH}(\Sigma)\to T^*\mathcal{T}(\Sigma)$, the circle action on $T^*\mathcal T(\Sigma)$ given by $e^{i\theta}\cdot (J,q)=(J,e^{i\theta}q)$ (where $J$ is an almost-complex structure on $\Sigma$ and $q$ a holomorphic quadratic differential) induces an action of $S^1$ on $\mathcal{MGH}(\Sigma)$. Let us denote by $R_\theta:\mathcal{MGH}(\Sigma)\to \mathcal{MGH}(\Sigma)$ the corresponding self-diffeomorphism.  For genus $\geq 2$, this action of $S^1$ induces an action on $\mathcal{T}(\Sigma)\times \mathcal{T}(\Sigma)$ by means of the map $\mathcal M$. The so obtained $S^1$-action on $\mathcal{T}(\Sigma)\times \mathcal{T}(\Sigma)$ has been studied in \cite{bonsante2013a_cyclic,bonsante2015a_cyclic} under the name of \emph{landslide flow}. 

It will be relevant to introduce the function 
$$\mathcal A:\mathcal{MGH}(\Sigma)\to\R$$ which associates to a MGHC AdS manifold the area of its unique maximal Cauchy surface. It is easy to see that $\mathcal A$ is constant on the orbits of the circle action. We show:

\begin{theoremx}\label{thm:circle}
Let $\Sigma$ be a closed oriented surface of genus $\geq 1$. The circle action on $\mathcal{MGH}(\Sigma)$ is Hamiltonian with respect to $\omega_\i$, and satisfies 
$$R_\theta^*\g=\g\qquad R_\theta^*\omega_{\i}=\omega_{\i}\qquad R_\theta^*\omega_{\i}^\C=e^{-i\theta}\omega_{\i}^\C~.$$
When $\Sigma$ has genus $\geq 2$, the function $\mathcal A$ is a Hamiltonian function.
\end{theoremx}

We remark that, in terms of the (para-)complex structures $\i,\j,\k$, the pull-back relations of Theorem \ref{thm:circle} read:
\begin{equation}\label{eq:pullbackijk}
R_\theta^*\i=\i\qquad R_\theta^*\j=\cos(\theta)\j+\sin(\theta)\k \qquad R_\theta^*\k=-\sin(\theta)\j+\cos(\theta)\k~.
\end{equation}

In \cite{bonsante2015a_cyclic}, Bonsante, Mondello and Schlenker showed that the landslide flow is  Hamiltonian with respect to the symplectic form $\pi_{l}^{*}\Omega_{\WP}+ \pi_{r}^{*}\Omega_{\WP}$. As a consequence of Theorem 
\ref{thm:mappaM} and the first part of Theorem \ref{thm:circle}, we thus recover (by independent methods) their results   and include it in a more general context.

The map $\mathcal A:\mathcal{MGH}(\Sigma)\to\R$ that encodes the area of the maximal Cauchy surface is also applied in  the following context. Given a para-K\"ahler structure $(\g,\p)$ on a manifold $M$, a \emph{para-K\"ahler potential} is a smooth function $\rho:M\to\R$ such that $\omega_\p=(\tau/2)\overline\partial_\p\partial_\p\rho$. We then prove:

\begin{theoremx}\label{thm:potentialintro}
Let $\Sigma$ be a closed oriented surface of genus $\geq 1$. Then   the para-K\"ahler structures $(\g,\j)$ and $(\g,\k)$ admit a para-K\"ahler potential, which coincides up to a constant with  a Hamiltonian function for the circle action.
\end{theoremx}

One could alternatively have used the map $\mathcal C$ to induce a circle action on $\mathcal{T}(\Sigma)\times \mathcal{T}(\Sigma)$.  However, the obtained action is the same as when using $\mathcal M$ (i.e. the landslide flow), as a consequence of the observation that $\mathcal M=\mathcal C\circ R_{-\pi/2}$. By this relation, Theorem \ref{thm:mappaC} immediately follows from Theorem \ref{thm:mappaM} and Theorem \ref{thm:circle}.

In fact we can define a one-parameter family of maps 
$$\mathcal C_\theta:\mathcal{MGH}(\Sigma)\to \mathcal{T}(\Sigma)\times \mathcal{T}(\Sigma)~,$$
simply defined by $\mathcal C_\theta=\mathcal C\circ R_\theta$. 
An immediate consequence of our previous Theorems \ref{thm:mappaC} and \ref{thm:circle} is the following identity:

\begin{equation}\label{eq:scemenza}
    \mathcal C_\theta^*(\mathcal  P_{\mathcal{T}(\Sigma)\times \mathcal{T}(\Sigma)},4\Omega^{\mathbb B}_{\mathcal{T}(\Sigma)\times \mathcal{T}(\Sigma)})=(\cos\theta\k-\sin\theta\j,\omega_\i-\tau(\cos(\theta)\omega_\j+\sin(\theta)\omega_\k))~.
\end{equation}

The maps $\mathcal C_\theta$ have the following
 interpretation purely in terms of harmonic maps and Teichm\"uller theory. From the theory of harmonic maps between hyperbolic surfaces (\cite{sampson1978some_properties,wolf1989the_teichmuller,wolf1991a,wolf1991b,minski1992}), Teichm\"uller space admits a parameterization by the vector space of holomorphic quadratic differentials $H^0(\Sigma,\mathcal K_J^2)$ with respect to a fixed complex structure $J$ on $\Sigma$. The construction goes as follows. To a holomorphic quadratic differential $q$, we associate the hyperbolic metric $h_{(J,q)}$ on $\Sigma$ (unique up to isotopy) such that the (unique) harmonic map $(\Sigma,J)\to(\Sigma,h)$ isotopic to the identity has Hopf differential $q$. We now let $J$ vary over $\mathcal T(\Sigma)$. Then the map $$\mathcal H_\theta:=\mathcal C_\theta\circ\mathcal F^{-1}:T^*\mathcal T(\Sigma)\to \mathcal T(\Sigma)\times \mathcal T(\Sigma)$$ can be interpreted as follows:
$$\mathcal H_\theta(J,q)=(h_{(J,-e^{i\theta}q)},h_{(J,e^{i\theta}q)})~.$$
There is a completely analogous construction in genus one, by replacing hyperbolic surfaces by flat tori. As a consequence of Equation \eqref{eq:scemenza}, we obtain:

\begin{theoremx}\label{thm:Htheta}
Let $\Sigma$ be a closed oriented surface of genus $\geq 1$. Then
$$\Im \mathcal H_\theta^*(2\Omega^{\mathbb B}_{\mathcal{T}(\Sigma)\times \mathcal{T}(\Sigma)})=-\Re (ie^{i\theta}\Omega^\C_{T^*\mathcal T(\Sigma)})~.$$
\end{theoremx}

We remark that the statement above is expressed purely in terms of Teichm\"uller theory, and is independent of Anti-de Sitter geometry.

\subsection{The character variety}

Let us now consider the character variety of the fundamental group $\pi_1(\Sigma)$ in the isometry group of AdS space. We have already observed that the isometry group is isomorphic to $\PSL(2,\R)\times\PSL(2,\R)$; using the model of Hermitian matrices (\cite{dancigerGT,dancigerJT}), it can be described as the Lie group $\PSL(2,\mathbb B)$, where as usual $\mathbb B$ denotes the algebra of para-complex numbers. Using the (para-complex) Killing form, the character variety
$\chi(\pi_1(S),\PSL(2,\mathbb B))$ is endowed
with a para-complex symplectic form $\Omega_{Gol}^\B$, which is defined by adapting the work of Goldman (\cite{Goldman_symplecticnature}) to this context. It is para-complex with respect to the para-complex structure $\mathcal T$ induced by multiplication by $\tau$. It can be checked that, under the isomorphism $\PSL(2,\R)\times\PSL(2,\R)\cong\PSL(2,\mathbb B)$, the para-complex structure $\mathcal T$ corresponds to the para-complex structure $\mathcal P$ for which the integral distributions of the $1$ and $-1$ eigenspaces are the horizontal and vertical slices. 

Hence if we denote by
$$\mathcal{H}ol:\mathcal{MGH}(\Sigma)\to \chi(\pi_1(S),\PSL(2,\mathbb B))$$
the map that associates to a MGHC AdS manifold its holonomy representation, we obtain the following corollary of Theorem \ref{thm:mappaM}:

\begin{corollaryx}\label{cor:holonomy}
Let $\Sigma$ be a closed oriented surface of genus $\geq 2$. Then
$$\mathcal{H}ol^*(\mathcal T,4\Omega_{Gol}^\B)=(\j,\omega_\j^{\mathbb B})~.$$
\end{corollaryx}

We conclude the overview of our results by a concrete description of the para-complex symplectic structure $\omega_\j^{\mathbb B}$. In \cite{Tambu_FN}, the third author introduced $\B$-valued Fenchel-Nielsen coordinates. Roughly speaking, these are defined as follows. Let $\rho=(\rho_+,\rho_-):\pi_1(\Sigma)\to \PSL(2,\R)\times\PSL(2,\R)$ be the holonomy representation of a MGHC AdS manifold. Since both $\rho_-$ and $\rho_+$ are Fuchsian representations, $\rho_\pm(\alpha)$ are loxodromic elements for any non-trivial $\alpha\in \pi_1(\Sigma)$. As a consequence, we can associate to $\alpha$ a \emph{principal axis} $\tilde\alpha$ in AdS space, which is the spacelike geodesic with endpoints in $\R\mathbb P^1\times \R\mathbb P^1$ given by the pair of attracting and the pair of repelling fixed points of $\rho_\pm(\alpha)$. Then the Fenchel-Nielsen coordinates of $\rho$ are $(\ell^{\B,j}_{\rho},\tw^{\B,j}_{\rho})$ (for $\gamma_j$ a pants decomposition of $\Sigma$), where $\ell^{\B,j}_{\rho}$ are para-complex numbers whose real part corresponds to the translation length and imaginary part to the bending angle of $\rho(\gamma_j)$ on the principal axis $\tilde\gamma_j$; a similar interpretation can be given for the (para-complex) twist  coordinates $\tw^{\B,j}_{\rho}$. 

These coordinates are an analogue of the complex Fenchel-Nielsen coordinates on the space of quasi-Fuchsian manifolds, which are Darboux coordinates (\cite{wolpert1983on_the_symplectic},\cite{Platis_complexFN},\cite{Parker_Platis}). Here we show that an analogous result holds for $\omega_\j^{\mathbb B}$, which we recall corresponds (up to a multiplicative constant) both to the para-complex sympletic form on $\mathcal T(\Sigma)\times \mathcal T(\Sigma)$ and to the Goldman form $\Omega_{Gol}^{\B}$. 

\begin{theoremx}\label{thm:fenchel}
The $\mathbb B$-valued Fenchel-Nielsen coordinates are para-holomorphic for $\j$, and are Darboux coordinates with respect to the para-complex symplectic form $\omega_\j^{\mathbb B}$.
\end{theoremx}

In other words, we can express the symplectic form $\omega_\j^{\mathbb B}$, which coincides up to a multiplicative constant with the para-complex Goldman form $\Omega_{Gol}^{\B}$, as 
$$\omega_\j^{\mathbb B}=\frac{1}{4}\sum_{j=1}^{n} d\ell^{\B,j}_{\rho}\wedge d\tw^{\B,j}_{\rho}$$
where $\ell^{\B,j}_{\rho}$ and $\tw^{\B,j}_{\rho}$ are the $\B$-valued length and twist parameters on the curve $\gamma_j$ in a pants decomposition of $\Sigma$ (where $n=(3/2)|\chi(\Sigma)|$ is the number of such curves). 

Finally, we give a formula for the value of the symplectic form $\omega_\j^{\mathbb B}$ along two twist deformations, generalizing Wolpert's cosine formula. For this purpose, given $\alpha,\beta\in \pi_{1}(\Sigma)$ two intersecting closed curves, the principal axes of $\rho(\alpha)$ and $\rho(\beta)$ admit a common orthogonal geodesic of timelike type. Then we define the $\B$-\emph{valued} angle as the para-complex number whose imaginary part is the signed timelike distance between the two axes, and the real part is the angle between one principal axis and the parallel transport of the other, along the common orthogonal geodesic.

\begin{theoremx}\label{thm:goldmancosine} Let $\rho=(\rho_+,\rho_-):\pi_1(\Sigma)\to\PSL(2,\R)\times\PSL(2,\R)$ be the holonomy of a MGHC AdS manifold, and let $\alpha, \beta$ be non-trivial simple closed curves. Then 
\[
 \omega_{\j}^\B\left( \frac{\partial}{\partial \tw^{\B, \alpha}_{\rho}}, \frac{\partial}{\partial \tw^{\B, \beta}_{\rho}}\right)=\frac{1}{4}\sum_{p \in \alpha \cap \beta} \cos(d^{\B}(\tilde{\alpha}_{\rho}, \tilde{\beta}_{\rho})) \ ,
\]
where $\tilde\alpha_\rho$ and $\tilde\beta_\rho$ are the principal axes of $\rho(\alpha)$ and $\rho(\beta)$ on AdS space, and $d^{\B}(\tilde{\alpha}_{\rho}, \tilde{\beta}_{\rho})$ is their $\B$-valued angle.
\end{theoremx}

Here the cosine of a para-complex number is formally defined by the power series expansion.

\subsection{Outline of techniques and proofs}
We now provide an overview of the techniques that we apply in the proofs.

\subsubsection*{The ``toy model'': genus one}
After reviewing the necessary preliminaries on the deformation space of Anti-de Sitter three-manifolds in Section \ref{sec:preliminaries}, we complete in Section \ref{sec:toy model} the proof of our results from Theorem \ref{thm:parahyper_structure} to Theorem \ref{thm:Htheta} (or their formal analogues, for Theorems \ref{thm:cotangent} and \ref{thm:mappaM}) in genus one. The reason  is that the methods used to prove these results in genus one also provide the fundamental step on which the proofs in higher genus rely. Let us give a quick outline.

Denote by $\mathcal J(\R^2)$ the space of linear almost-complex structures on $\R^2$ compatible with the standard orientation, namely those endomorphisms $J$ of $\R^2$ such that $J^2=-\1$ and that $\{v,Jv\}$ is a positive basis. This space is naturally identified to $\Hyp^2$, in such a way that the natural transitive $\SL(2,\R)$  action by conjugation on $\mathcal J(\R^2)$ corresponds to the $\SL(2,\R)$ action on the upper half-space model. We construct an explicit $\SL(2,\R)$-invariant para-hyperK\"ahler structure on the cotangent bundle $T^*\mathcal J(\R^2)$, which is in turn identified to the space of pairs $(J,\sigma)$ where $\sigma$ is a symmetric bilinear form satisfying $\sigma(J\cdot,J\cdot)=-\sigma$, and give geometric interpretations to each object of the para-hyperK\"ahler structure. For instance, the complex symplectic form $\omega_\i^\C$ coincides up to conjugation and a multiplicative factor  with the complex symplectic form of $T^*\mathcal J(\R^2)$, and the almost-complex structure $\i$ is compatible with its complex structure. Also, the restriction to the zero section is identified, up to a factor, with the area form of $\Hyp^2$. All the theorems until Theorem \ref{thm:Htheta} have an analogue that we prove in this context. 

The construction of the para-hyperK\"ahler structure on $\mathcal{MGH}(T^2)$ (Theorem \ref{thm:parahyper_structure}), as well as all the other statements \emph{in genus one}, follow immediately by identifying (using the first and second fundamental form of the unique maximal surface)  the deformation space $\mathcal{MGH}(T^2)$ with the complement of the zero section in $T^*\mathcal J(\R^2)$, in such a way that the mapping class group action corresponds to the action of $\SL(2,\Z)<\SL(2,\R)$ (Lemma \ref{lemma torus identifications}.)

\subsubsection*{Setup in higher genus}
Let us now move on from genus one to genus $\geq 2$.  The proof of Theorem \ref{thm:parahyper_structure}, and of the geometric interpretations from Theorem \ref{thm:cotangent} to Theorem \ref{thm:Htheta}, rely on the following construction, which is developed in Section \ref{sec:para_hyperkahler_str_on_MS}. Fix an area form $\rho$ on $\Sigma$. Then we consider the space of all pairs $(J,\sigma)$ where (similarly to $T^*\mathcal J(\R^2)$ above) $J$ is an almost-complex structure on $\Sigma$ and $\sigma$ a smooth symmetric bilinear form satisfying $\sigma(J\cdot,J\cdot)=-\sigma$. This infinite-dimensional space, that is denoted by $T^*\mathcal J(\Sigma)$, can be endowed formally with three symplectic structures $\omega_\i$, $\omega_\j$, $\omega_\k$. Indeed the tangent space $T_p\Sigma$ at every point has a para-hyperK\"ahler structure induced by choosing an area-preserving linear isomorphism between $T_p\Sigma$ and $\R^2$, which induces an identification between $T^*\mathcal J(\R^2)$ and $T^*\mathcal J(T_p\Sigma)$. Since the para-hyperK\"ahler structure on $T^*\mathcal J(\R^2)$ is $\SL(2,\R)$-invariant, the induced structures on $T^*\mathcal J(T_p\Sigma)$ do not depend on the chosen area-preserving linear isomorphism. Then one can formally integrate each symplectic form over $\Sigma$, evaluated on first-order deformations $(\dot J,\dot\sigma)$. 

To make this construction more formal, $T^*\mathcal J(\Sigma)$ is identified with the space of smooth sections of the bundle 
\[
P(T^{*}\mathcal{J}(\R^{2}))  = \faktor{P \times T^{*}\mathcal{J}(\R^{2})}{\SL(2,\R)} ,
\]
where $P$ is the $\SL(2,\R)$-principal bundle over $\Sigma$ whose fiber over $p\in\Sigma$ is the space of linear isomorphisms between $\R^2$ and $T_p\Sigma$ that pull-back the area form $\rho$ on $T_p\Sigma$ to the standard area form of $\R^2$. Hence one can introduce the formal symplectic forms 
\begin{align}
    (\omega_\mathbf{X})_{(J,\sigma)} ((\dot{J}, \dot{\sigma}), (\dot{J}', \dot{\sigma}')) \defin \int_\Sigma \hat{\omega}_\mathbf{X}((\dot{J}, \dot{\sigma}), (\dot{J}', \dot{\sigma}')) \, \rho ,
\end{align}
for $\mathbf{X} = \mathbf{I}, \mathbf{J}, \mathbf{K}$, where $\hat\omega_{\mathbf X}$ is the symplectic form induced on the vertical space at every point, and analogously for the formal pseudo-Riemannian metric 
\begin{align}
    \g_{(J,\sigma)} ((\dot{J}, \dot{\sigma}), (\dot{J}', \dot{\sigma}')) \defin \int_\Sigma \hat{\g}((\dot{J}, \dot{\sigma}), (\dot{J}', \dot{\sigma}')) \, \rho \ ,
\end{align}
on the infinite-dimensional space of sections. Similarly, one can define the endomorphisms $\i$, $\j$ and $\k$ by applying pointwise those that have been defined on $T^*\mathcal J(\R^2)$ under the identification as above.
(See the discussion on symplectic reduction below for more details.)

Now, recalling that every MGHC AdS manifold contains a unique maximal Cauchy surface, the essential point is to determine the space of solutions to the Gauss-Codazzi equations for maximal surfaces in AdS space, as a subspace of sections $(J,\sigma)$. Given an almost-complex structure $J$, one can build a Riemannian metric $g_J=\rho(\cdot,J\cdot)$. However, in general this metric cannot be realised as the first fundamental form of a maximal surface in a MGHC AdS manifold, because the area form of $g_J$ coincides with $\rho$, hence all metrics $g_J$ have the same area, which is not the case for maximal surfaces. Instead, inspired by a similar ``change of variables'' used in the construction of the hyperK\"ahler structure on a neighborhood of the Fuchsian locus in the space of quasi-Fuchsian manifolds (\cite{donaldson2003moment,hodgethesis,trautwein2019hyperkahler}), we define the metric $h=(1+f(\|\sigma\|_{g_J}))g_J$, where $\|\sigma\|_{g_J}$ is the norm of the 2-tensor $\sigma$ with respect to the metric $g_J$, and $f(t)=\sqrt{1+t^2}$. Then imposing the Gauss-Codazzi equations on $\I=h$ and $\II=h^{-1}\sigma$, we determine a $\Symp_0(\Sigma,\rho)$-invariant subspace of the space of smooth section $(J,\sigma)$, which we denote by $\widetilde{\mathcal{MS}}_0(\Sigma,\rho)$, whose quotient $\mathcal{MS}_0(\Sigma,\rho)$ by the action of $\Symp_0(\Sigma,\rho)$ is identified to $\mathcal{MGH}(\Sigma)$ thanks to the existence and uniqueness result for maximal surfaces, together with a standard application of Moser's trick. We remark that, unlike the construction in the quasi-Fuchsian setting where the correct function to apply the change of variables is $f(t)=\sqrt{1-t^2}$, hence only defined for $t=\|\sigma\|_{g_J}<1$, in the AdS setting this change of variables permits to recover \emph{all} the maximal surfaces in MGHC AdS manifolds, and not only a neighborhood of the Fuchsian locus. 

\subsubsection*{A distinguished complement to the orbit}

In order to induce a para-hyperK\"ahler structure on $\mathcal{MS}_0(\Sigma,\rho)$, by restriction of the symplectic forms $\omega_\i$, $\omega_\j$ and $\omega_\k$ and of the metric $\g$, we need to construct a distribution of subspaces of the tangent space of $\widetilde{\mathcal{MS}}_0(\Sigma,\rho)$ which are isomorphic at every point $(J,\sigma)$ to the tangent space to the quotient $\mathcal{MGH}(\Sigma)$, and invariant for the action of $\Symp_0(\Sigma,\rho)$. The construction of such invariant distribution, which is denoted by $V_{(J,\sigma)}$ and defined as the space of solutions of a system of partial differential equations, constitutes the main technical difficulty of Section \ref{sec:para_hyperkahler_str_on_MS}. The defining equations can be formulated in several equivalent ways, as in the following technical statement, whose proof is done in Section \ref{subsec:proof_of_prop} by overcoming a number of technical difficulties.

\begin{propx}\label{prop:equivalent_def_subspace_V}
	Given $(J, \sigma) \in \widetilde{\mathcal{MS}_{0}}(\Sigma,\rho)$, and $(\dot{J}, \dot{\sigma}) \in T_{(J,\sigma)} T^* \mathcal{J}(\Sigma)$, the following conditions are equivalent:
	\begin{enumerate}[i)]
		\item the pair $(\dot{J}, \dot{\sigma})$ satisfies
		\begin{equation} \label{eq:description_V1}\tag{V1}
			\begin{cases}
				\divr_g (f^{-1}g^{-1} \dot{\sigma}_0) = - f^{-1} \scal{\nabla^g_{J \bullet} \sigma}{\dot{J}} , \\
				\divr_g \dot{J} = - f^{-2} \scall{\nabla^g_{J \bullet} \sigma}{\dot{\sigma}_0} .
			\end{cases}
		\end{equation}
		\item the endomorphisms $Q^\pm = Q^\pm(\dot{J}, \dot{\sigma}) \defin f^{-1} g^{-1} \dot{\sigma}_0 \pm \dot{J}$ satisfy
		\begin{equation} \label{eq:description_V2}\tag{V2}
			\begin{cases}
            \divr_g (Q^+ J J_l) = - \scal{\nabla^g_{J \bullet} \sigma}{Q^+} , \\
            \divr_g (Q^- J J_r) = \scal{\nabla^g_{J \bullet} \sigma}{Q^-} ,
        \end{cases}
		\end{equation}
		where $J_l$ and $J_r$ denote the almost complex structures of the components of the Mess map $\mathcal{M}$;
		\item the endomorphisms $Q^\pm$ satisfy
		\begin{equation} \label{eq:description_V3}\tag{V3}
			\begin{cases}
				\divr_g Q^+ = - f^{-1} \, \scal{\nabla^g_{J \bullet} \sigma}{Q^+} , \\ 
				\divr_g Q^- = + f^{-1} \, \scal{\nabla^g_{J \bullet} \sigma}{Q^-} .
			\end{cases}
		\end{equation}
	\end{enumerate}
    
\end{propx}

Defining $V_{(J,\sigma)}$ as the space of solutions to some (hence all) the conditions \ref{eq:description_V1}-\ref{eq:description_V3}, the fundamental result is the following, whose proof again contains a number of technical difficulties and is done in Section \ref{sec:proofThm410}.

\begin{theoremx}\label{thm:identification_tangent&V}
    For every pair $(J,\sigma)$ lying in $\widetilde{\mathcal{MS}_{0}}(\Sigma, \rho)$, the vector space $V_{(J,\sigma)}$ is contained inside $T_{(J,\sigma)} \widetilde{\mathcal{MS}_{0}}(\Sigma, \rho)$, it is invariant by the action of $\mathbf{I}$, $\mathbf{J}$ and $\mathbf{K}$, and it defines a $\Symp(\Sigma,\rho)$-\hsk invariant distribution $\mathcal{V} = \set{V_{(J,\sigma)}}_{(J,\sigma)}$ on $\widetilde{\mathcal{MS}_{0}}(\Sigma, \rho)$. Moreover, the natural projection $\mappa{\pi}{\widetilde{\mathcal{MS}_{0}}(\Sigma,\rho)}{\mathcal{MS}_0(\Sigma, \rho)}$ induces a linear isomorphism $\dd \pi_{(J,\sigma)}: V_{(J,\sigma)} \rightarrow T_{[J,\sigma]}\mathcal{MS}_0(\Sigma, \rho)$.
\end{theoremx}

Once these steps are achieved, the proofs of all the results from Theorem \ref{thm:parahyper_structure} to Theorem \ref{thm:Htheta} mostly follow from their analogues that are showed in Section \ref{sec:toy model}, by applying the same computations pointwise, and recognizing the geometric interpretations in terms of the geometry of MGHC AdS manifolds either in a similar fashion to what has been done in the case of the torus (for instance for Theorems \ref{thm:cotangent}, \ref{thm:circle} and Theorem \ref{thm:potentialintro}), or by recalling the known constructions of Mess homeomorphism and constant curvature surfaces explained in Section \ref{sec:preliminaries} (Theorems \ref{thm:mappaM} and \ref{thm:mappaC}).
These proofs are provided in the first part of Section \ref{sec:geom_inter}. The second part then contains the proofs of Theorem \ref{thm:fenchel} and Theorem \ref{thm:goldmancosine}, which mainly rely on the isomorphism $\PSL(2,\B)\cong\PSL(2,\R)\times \PSL(2,\R)$ for the isometry group of Anti-de Sitter space, and some geometric constructions in the $\PSL(2,\R)$ model.

\subsubsection*{Relation with symplectic reduction}
Although the proofs of the results stated above are essentially self-contained and completed entirely in Sections \ref{sec:toy model}, \ref{sec:para_hyperkahler_str_on_MS} and \ref{sec:geom_inter}, in Section \ref{sec:inf_dim_reduction} we include a discussion with the aims, on the one hand, of highlighting the relations with symplectic reduction, which is also the motivation that led us to the definition of the distribution $V_{(J,\sigma)}$ as in Proposition \ref{prop:equivalent_def_subspace_V} and Theorem \ref{thm:identification_tangent&V}; on the other hand, of explaining the additional difficulties that do not permit to apply directly the strategy of symplectic reduction to obtain our results.

The starting point is a general theorem of Donaldson (\cite[Theorem 9]{donaldson2003moment}, see Theorem \ref{thm:donaldson_maps} for the statement with minimal hypothesis that we apply in our setting). In short, it turns out that the $\SL(2,\R)$-action on $\T^*\mathcal J(\R^2)$ is Hamiltonian with respect to any of the three symplectic forms constituting the para-hyperK\"ahler structure on $\T^*\mathcal J(\R^2)$, with moment maps that we denote by $\eta_\i$, $\eta_\j$ and $\eta_\k$. Donaldson's theorem gives a formula to compute a corresponding map (that we call $\mu_\i$, $\mu_\j$ and $\mu_\k$) for the action on $T^{*}\mathcal{J}(\Sigma)$ of the group $\Ham(\Sigma,\rho)$ of Hamiltonian symplectomorphisms of $(\Sigma,\rho)$. It turns out that $\mu_{\j}$ and $\mu_{\k}$ are moment maps for the action of $\Ham(\Sigma, \rho)$, whereas $\mu_{\i}$ needs to be modified by an adding a scalar multiple of $\rho$. We denote this new map by $\tilde{\mu}_{\i}$. Moreover, although $\Ham(\Sigma,\rho)$ is a proper normal subgroup of $\Symp_0(\Sigma,\rho)$, $\mu_\j$ and $\mu_\k$ can actually be promoted to moment maps $\tilde\mu_\j$ and $\tilde\mu_\k$ for the action of $\Symp_0(\Sigma,\rho)$; $\tilde{\mu}_\i$ cannot, but it still satisfies some additional properties that make it ``almost'' a moment map for $\Symp_0(\Sigma,\rho)$. We compute explicit formulas for these three maps, and show that the kernel of $\tilde\mu_\j+i\tilde\mu_\k$ consists precisely of the pairs $(J,\sigma)$ such that $\sigma$ is the real part of a holomorphic quadratic differential on $(\Sigma,J)$. The intersection with the kernel of the third map $\tilde\mu_\i$ is then precisely the space $\widetilde{\mathcal{MS}}_0(\Sigma,\rho)$. Hence one is tempted to apply the symplectic reduction in order to induce a pseudo-Riemannian metric and three symplectic stuctures, in the quotient $\tilde\mu_\i^{-1}(0)\cap \tilde\mu_\j^{-1}(0)\cap \tilde\mu_\k^{-1}(0)/\Symp_0(\Sigma,\rho)$. 

However, at this point the usual construction by which the quotient inherits a hyperK\"ahler structure fails because of the fact that our metric on each fiber is not positive definite. Hence we do not have a natural Hilbert space structure on the tangent space to the space of sections $T^*\mathcal{J}(\Sigma)$. Concretely, this means that we cannot   take the orthogonal complement to the tangent space to the $\Symp_0(\Sigma,\rho)$-orbit as a distribution which is invariant by the actions of $\i$, $\j$ and $\k$ and isomorphic to the tangent space to the quotient. Nevertheless, inspired by the properties satisfied in the hyperK\"ahler setting, we prove the following characterization of $V_{(J,\sigma)}$, which is the main result of Section \ref{sec:inf_dim_reduction}:

\begin{theoremx}\label{thm:char_VJSigma}
For every $(J,\sigma)\in \widetilde{\mathcal{MS}}_0(\Sigma,\rho)$, 
$V_{(J,\sigma)}$ is the largest subspace of $T_{(J,\sigma)}\widetilde{\mathcal{MS}}_0(\Sigma,\rho)$ that is:
\begin{itemize}
    \item invariant under $\i$, $\j$ and $\k$;
    \item $\g$-orthogonal to $T_{(J,\sigma)}(\Symp_{0}(\Sigma, \rho)\cdot (J,\sigma))$
\end{itemize} 
\end{theoremx}

As said before, Theorem \ref{thm:char_VJSigma} is not applied in the proof of any of the previous results; nevertheless, it serves as a motivation for Proposition \ref{prop:equivalent_def_subspace_V} and Theorem \ref{thm:identification_tangent&V}, namely the two technical tools which play an essential role in passing from the para-hyperK\"ahler structure in genus one (or more precisely, on $T^*\mathcal J(\R^2)$) to that for higher genus surfaces. 

\subsection*{Acknowledgments}

We would like to thank the participants of the reading group on hyperK\"ahler structures on quasi-Fuchsian space held at the University of Luxembourg in the Fall 2018, for many valuable discussions. The second author is grateful to Andy Sanders for interesting discussions.

The third author acknowledges support from the U.S. National Science Foundation under grant NSF-DMS:2005501.

\section{Preliminaries on Anti-de Sitter geometry}\label{sec:preliminaries}

In this preliminary section, we introduce the necessary notions concerning maximal globally hyperbolic Cauchy compact Anti-de Sitter three-manifolds (in short, MGHC AdS manifold)  and their deformation space.  

\subsection{Maximal surfaces} \label{subsec:maximal}

The starting point of our construction comes from the role of maximal surfaces in MGHC AdS manifolds. Recall that a \emph{maximal surface} in a Lorentzian three-manifold is a spacelike surface (i.e. its first fundamental form is a Riemannian metric) whose mean curvature vanishes identically. Then we have the following existence and uniqueness result:

\begin{theorem}[{\cite{bbz,bonsante2010maximal}}]
Any MGHC AdS three-manifold admits a unique maximal Cauchy surface.
\end{theorem}
 
This result permits to obtain a parameterization of $\mathcal{MGH}(\Sigma)$ by means of embedding data of maximal surfaces. Recall that the embedding data of a spacelike surface in a Lorentzian manifold consists of the pair $(h,B)$, where $h$ is the first fundamental form and $B$ the shape operator, and these satisfy the Gauss-Codazzi equations
\begin{equation} \label{eq:gauss_codazzi}\tag{GC}
	\begin{cases}
		K_h = - 1 - \det B , \\
		\dd^{\nabla^h} B = 0 ,
	\end{cases}
\end{equation}
where the exterior derivative $\dd^{\nabla^h} B$ is the $T \Sigma$-\hsk valued $2$-\hsk form $$(\dd^{\nabla^h} B)(X,Y) = (\nabla^h_X B)Y - (\nabla^h_Y B)X~,$$ for $X$, $Y$ tangent vector fields on $\Sigma$, $\nabla^h$ being the Levi-Civita connection of $h$ and $K_h$ its curvature. By definition, the surface is maximal if and only if $B$ is traceless. Conversely, every pair $(h,B)$ satisfying \eqref{eq:gauss_codazzi}, with $h$ a Riemannian metric and $B$ a traceless $h$-self-adjoint tensor, represents the embedding data of a maximal surface in a MGHC AdS manifold diffeomorphic to $\Sigma\times\R$, whose metric has the following explicit expression \emph{in a tubular neighbourhood} of the surface (namely for $t\in(-\epsilon,\epsilon)$):
$$-dt^2+h((\cos(t)\1+\sin(t)B)\cdot,(\cos(t)\1+\sin(t)B)\cdot) \ . $$ 
See \cite[Lemma 6.2.2, Proposition 6.2.3]{bonsante2020antide}. Moreover the above correspondence is natural by the actions of $\mathrm{Diff}(\Sigma)$. 
Motivated by these observations, let us introduce the space of embedding data of maximal surfaces:

$$\mathcal{MS}(\Sigma):=\left\{(h,B)\left| \begin{aligned}&\quad\;\; h\text{ is a Riemannian metric}\\  &\,B\text{ is traceless and }h\text{-self-adjoint} \\ &\quad \text{equations }\eqref{eq:gauss_codazzi}\text{ are satisfied} \end{aligned}\right. \right\}/\mathrm{Diff}_0(\Sigma)~,$$
where $\mathrm{Diff}_0(\Sigma)$ denotes the group of diffeomorphisms isotopic to the identity.

In summary we have:

\begin{proposition}\label{prop:MGHandMS}
Let $\Sigma$ be a closed surface of genus $g\geq 1$. There is a $MCG(\Sigma)$-invariant homeomorphism between $\mathcal{MGH}(\Sigma)$ and $\mathcal{MS}(\Sigma)$, given by the embedding data of the unique maximal Cauchy surface.
\end{proposition}

In the remainder of the paper, we will often implicitly identify $\mathcal{MGH}(\Sigma)$ with $\mathcal{MS}(\Sigma)$.



\subsection{Cotangent bundle of Teichm\"uller space}\label{subsec:cotangent}

Maximal surfaces also permit to obtain a parameterization of $\mathcal{MGH}(\Sigma)$ by means of the cotangent bundle of Teichm\"uller space. Recall that the Teichm\"uller space of the surface $\Sigma$ is defined as:
$$\mathcal{T}^{\conf}(\Sigma):=\{J\in\Gamma(\mathrm{End}(T\Sigma))\,|\,J^2=-\1,(v,J(v))\text{ is an oriented frame} \}$$

We use the notation $\mathcal{T}^{\conf}$ to highlight that this is the Teichm\"uller space defined in terms of (almost-)complex structures $J$, to distinguish with the other incarnations of Teichm\"uller space (see Section \ref{subsec:parameterizations} below). 

To explain this parameterization, let us first provide several equivalent descriptions of the Codazzi equation $\dd^{\nabla^h} B = 0$, which we summarize in the following statement:

\begin{lemma} \label{lem:hol_quadr_diff_equivalence}
    Let $h$ be a Riemannian metric on $\Sigma$, and denote by $J$ the almost-complex structure induced by $h$. Assume a $(1,1)$ tensor $B$ and a $(2,0)$ tensor $\sigma$ are related by $\sigma=hB$ (i.e. $\sigma(\cdot,\cdot)=h(\cdot,B\cdot)$). Then $B$ is $h$-self-adjoint and traceless if and only if $\sigma$ is the real part of a quadratic differential, which can be expressed as $q=\sigma-i\sigma(\cdot,J\cdot)$. If this holds, then the following conditions are equivalent:
    \begin{enumerate}[i)]
            \item $B$ is $h$-\hsk Codazzi, i.e. $\dd^{\nabla^h} B = 0$;
        \item $\sigma$ is the real part of a holomorphic quadratic differential on $(\Sigma, J)$;
        \item $\sigma$ is $h$-\hsk divergence-\hsk free;

        \item for every tangent vector field $X$ on $\Sigma$ we have $\nabla^h_{J X} \sigma = (\nabla^h_X \sigma)(\cdot, J \cdot)$.
    \end{enumerate}
\end{lemma}
A standard reference for these equivalences is \cite{tromba2012teichmuller}; see also \cite{hopf,taubes,krasnov_schlenker_minimal}. 

\begin{remark}\label{rmk:independent_conformal_metric}
Since the condition $ii)$ in Lemma \ref{lem:hol_quadr_diff_equivalence} only depends on $J$, and not on the metric $h$, it follows that if conditions $iii)$ or $iv)$ hold for \emph{some} metric $h$ compatible with $J$, then they hold for \emph{any} other metric conformal to $h$.
\end{remark}

\begin{remark}\label{rmk:rotateB}
We will repeatedly use the following fact.  Suppose $q$ is a holomorphic quadratic differential and $B=h^{-1}\Re(q)$ is the corresponding traceless, $h$-self-adjoint, $h$-Codazzi tensor. Writing
$q=\sigma-i\sigma(\cdot,J\cdot)$ and multiplying by $e^{i\theta}$, one checks immediately that
$$\Re(e^{i\theta}q) = \cos\theta \ \sigma + \sin\theta \ \sigma(\cdot,J\cdot)~.$$
Hence the traceless, $h$-self-adjoint, $h$-Codazzi tensor $h^{-1}\Re(e^{i\theta}q)$ corresponding to $e^{i\theta}q$ is precisely $(\cos(\theta)\1-\sin(\theta)J)B$. 
\end{remark}

Based on Lemma \ref{lem:hol_quadr_diff_equivalence}, Krasnov and Schlenker established the following result:

\begin{theorem}[{\cite[Lemma 3.6, Theorem 3.8]{krasnov_schlenker_minimal}}]\label{thm:emd_data_from_hqd}
Let $\Sigma$ be a closed oriented surface of genus $\geq 2$. Given a complex structure $J$ on $\Sigma$ and a holomorphic quadratic differential $q$ on $(\Sigma,J)$, there exists a unique Riemannian metric $h$ compatible with $J$ such that the pair $(h,B=h^{-1}\sigma)$ satisfies \eqref{eq:gauss_codazzi}, where $\sigma=\Re(q)$.
\end{theorem}

Again, the map that associates to a pair $(h,B)$ satisfying the Gauss-Codazzi equations the pair $(J,q)$, where $J$ is the complex structure of $h$ and $\sigma=hB=\Re(q)$, is natural with respect to the action of $\mathrm{Diff}(\Sigma)$.  

We remark that Theorem \ref{thm:emd_data_from_hqd} is proved in \cite{krasnov_schlenker_minimal} for a closed surface of genus $\geq 2$. However, the case of genus one holds true, and can be proved directly, provided $\sigma\neq 0$. 

\begin{proposition}\label{prop:emd_data_from_hqd_torus}
Given a complex structure $J$ on $T^2$ and a \emph{non-zero} holomorphic quadratic differential $q$ on $(T^2,J)$, there exists a unique Riemannian metric $h$ compatible with $J$ such that the pair $(h,B=h^{-1}\sigma)$ satisfies \eqref{eq:gauss_codazzi}, where $\sigma=\Re(q)$.
\end{proposition}
\begin{proof}
It turns out that, given any MGHC AdS manifold diffeomorphic to $T^2\times\R$, the (unique) maximal Cauchy surface is intrinsically flat, and the MGHC metric can be written (now \emph{globally}) from the maximal surface as 
\begin{equation}\label{eq:torus metric}
(1+\sin(2t))dx^2+(1-\sin(2t))dy^2-dt^2~,
\end{equation}
where:
\begin{itemize}
\item $t\in (-\pi/4,\pi/4)$ represents the timelike distance from the maximal surface;
\item $x,y$ are global flat coordinates for the first fundamental form $h = dx^2 + dy^2$ of the universal cover of the maximal surface.
\end{itemize}
See Section 5 of \cite{bonsante2020antide}, and in particular Lemma 5.2.4, after a simple change of coordinates (translate the vertical coordinate by $\pi/4$, and then perform a change of variables on $x,y$). 
Moreover, since the second fundamental form of the maximal surface $\{t=0\}$ equals one half the derivative at $t=0$ of the metric on the level sets of $t$, we find that $\sigma$ is expressed as $dx^2-dy^2$, hence (in the complex coordinate $z=x+iy$) $q=dz^2$ is the holomorphic quadratic differential whose real part equals $\sigma$, as in Lemma \ref{lem:hol_quadr_diff_equivalence}. 

Inspired by these observations, let us now reconstruct a MGHC AdS manifold from a pair $(J,\sigma)$ on the torus. Let us realize the complex structure $J$ as that induced by a flat metric on $T^2$, obtained as the quotient of the complex plane by a lattice $\Lambda$. In other words, we find a biholomorphism between $(T^2,J)$ and $\C/\Lambda$ endowed with the complex structure induced by the standard complex structure of $\C$. We stress that for the moment we do not assume any normalization on the lattice $\Lambda$, which is therefore determined up to automorphisms of $\C$. 

 Let us call $x,y$ (where $z=x+iy$) the flat coordinates on $\C$, which induce flat coordinates on the quotient $\C/\Lambda$. Then the holomorphic quadratic differential $q$ has the expression $cdz^2$, for $c\neq 0$ a complex number. Multiplying the coordinate $z$ by a square root of $1/c$, we can also assume that $q=dz^2$. This has the effect of rescaling and rotating the lattice $\Lambda$, thus obtaining a new lattice $\Lambda'$ inducing the same $J$ on $T^2$. Hence the expression \eqref{eq:torus metric} gives a MGHC AdS metric on $(\C/\Lambda')\times (-\pi/4,\pi/4)$, for which $\{t=0\}$ is a maximal surface whose corresponding complex structure is $J$ and quadratic differential is $q$.  
\end{proof}

Recalling that the cotangent bundle of Teichm\"uller space is identified with the bundle of holomorphic quadratic differentials, as consequence of Theorem \ref{thm:emd_data_from_hqd} (for higher genus) and the above discussion (for genus one), one obtains:

\begin{theorem}\label{thm:cotangent_parameterization}
Let $\Sigma$ be a closed oriented surface of genus $\geq 2$. The map sending a pair $(h,B)$ to the pair $(J,q)$, where $J$ is the complex structure induced by $h$ and $\Re(q)=hB$, induces a $MCG(\Sigma)$-invariant homeomorphism
\[
    \mappa{\mathcal F}{\mathcal{MGH}(\Sigma)\cong\MS(\Sigma)}{T^* \Teich^\conf(\Sigma)}~.
\]
If $\Sigma=T^2$, the same map gives a $MCG(T^2)$-invariant homeomorphism between $\mathcal{MGH}(T^2)$ and the complement of the zero section in $T^* \Teich^\conf(T^2)$.
\end{theorem}

\subsection{Mess' parameterization}\label{subsec:parameterizations}

Let us now focus on genus $\geq 2$ and move on to another parameterization of $\MS(\Sigma)$, which has been essentially introduced by Mess. For this purpose, let us introduce the hyperbolic model of Teichm\"uller space, namely:
$$\mathcal{T}^{\hyp}(\Sigma):=\{h\,|\,h\text{ is a hyperbolic metric on }\Sigma  \}/\mathrm{Diff}_0(\Sigma)~.$$

The ``hyperbolic'' Teichm\"uller space is $MCG(\Sigma)$-invariantly homeomorphic to $\mathcal{T}^{\conf}(\Sigma)$ by the uniformization theorem. Also, the holonomy representation provides a $MCG(\Sigma)$-invariant homeomorphism to the space of Fuchsian representations, which we denote by:
$$\mathcal{T}^{\rep}(\Sigma):=\{\rho:\pi_1(\Sigma)\to\Isom(\Hyp^2)\text{ discrete and faithful representations} \}/\Isom(\Hyp^2)~,$$
where $\Isom(\Hyp^2)$ acts by conjugation (\cite{goldmanthesis}). Now, given a pair $(h,B)$ satisfying \eqref{eq:gauss_codazzi}, one can construct two hyperbolic metrics by the formula
\begin{equation}\label{eq:left_right_metric}
(h,B)\mapsto (h_l:=h((\1-JB)\cdot,(\1-JB)\cdot),h_r:=h((\1+JB)\cdot,(\1+JB)\cdot))~.
\end{equation}
(Here $h_l$ and $h_r$ stand for ``left'' and ``right'', and in fact these metrics can be interpreted as the pull-backs via the so-called \emph{left and right Gauss maps}, see \cite[\S 3]{krasnov_schlenker_minimal}, \cite[\S 6]{bonsante2020antide}, \cite[\S 6]{barbotkleinian}.) It can be proved (see for instance the indicated references) that the metrics $h_l$ and $h_r$ are hyperbolic and that, interpreting the isometry group of Anti-de Sitter space as $\PSL(2,\R)\times \PSL(2,\R)$, the map \eqref{eq:left_right_metric} gives the left and right components of the holonomy map
$$\hol:\pi_1(\Sigma)\to\PSL(2,\R)\times \PSL(2,\R)$$
of the AdS MGHC manifold determined by the embedding data $(h,B)$ of the maximal Cauchy surface, under the isomorphism between $\mathcal{T}^{\hyp}(\Sigma)$ and $\mathcal{T}^{\rep}(\Sigma)$.

In summary, we state the following theorem:

\begin{theorem}[{\cite{mess2007lorentz},\cite[Theorem 3.17]{krasnov_schlenker_minimal},\cite[Theorem 5.5.4]{bonsante2020antide}}] \label{thm:mess_homeo}
Let $\Sigma$ be a closed oriented surface of genus $\geq 2$. The map sending a pair $(h,B)$ to the pair of hyperbolic metrics $(h_l,h_r)$ in \eqref{eq:left_right_metric} induces a $MCG(\Sigma)$-invariant homeomorphism
\[
    \mappa{\mathcal M^\hyp}{\MS(\Sigma)}{\Teich^{\hyp}(\Sigma)\times\Teich^{\hyp}(\Sigma)}~.
\]
Under the natural homeomorphism  $\Teich^{\hyp}(\Sigma)\cong \Teich^{\rep}(\Sigma)$, such map coincides with the map 
\[
    \mappa{\mathcal M^\rep}{\mathcal{MGH}(\Sigma)}{\Teich^{\rep}(\Sigma)\times\Teich^{\rep}(\Sigma)}~.
\]
sending a MGHC AdS manifold to the conjugacy class of its holonomy representation.
\end{theorem}

It will be useful for future computations to express this  map in the conformal model of Teichm\"uller space. This follows easily by observing that the complex structure of a metric of the form $h(A\cdot,A\cdot)$ equals the $A$-conjugate of the complex structure of $h$.

\begin{lemma}\label{lem:mess_complex_str}
Let $\Sigma$ be a closed oriented surface of genus $\geq 2$. The $MCG(\Sigma)$-invariant homeomorphism 
\[
    \mappa{\mathcal M^\conf}{\MS(\Sigma)}{\Teich^{\conf}(\Sigma)\times\Teich^{\conf}(\Sigma)}~.
\]
is induced by the map
$$(h,B) \mapsto \left(J_l:= (\1 - J B)^{-1} J (\1 - J B) ,J_r:= (\1 + J B)^{-1} J (\1 + J B) \right)~,$$
where $J$ is the complex structure defined by the metric $h$.
\end{lemma}

The map $\mathcal M$ has also an interpretation in terms of harmonic maps. Indeed, it turns out that the identity map $\id:(\Sigma,J)\to (\Sigma,h_l)$ is harmonic and its Hopf differential is the holomorphic quadratic differential $iq$, where $\Re(q)=hB$. Indeed, from the expression \eqref{eq:left_right_metric}, we see that
$$h_l=\left(1-\det(B)\right)h-2h(\cdot,JB\cdot)$$
where we used that $JB$ is $h$-self-adjoint and traceless, hence by the Cayley-Hamilton theorem $(JB)^2=-\det(JB)\1=-\det(B)\1$. Hence the $(2,0)$-part of $h_l$ with respect to the complex structure $J$ is  the quadratic differential $iq$, where $\sigma=\Re(q)=hB$ (Remark \ref{rmk:rotateB}). Moreover it is holomorphic by Lemma \ref{lem:hol_quadr_diff_equivalence}. This implies that $\id:(\Sigma,J)\to (\Sigma,h_l)$ is harmonic, since it is a diffeomorphism and its Hopf differential is holomorphic (see \cite[\S 9]{sampson1978some_properties}). Analogously $\id:(\Sigma,J)\to (\Sigma,h_r)$ is harmonic and has Hopf differential $-iq$. We summarize this in the following lemma.

\begin{lemma}
Let $\Sigma$ be a closed oriented surface of genus $\geq 2$. Then
$$\mathcal M^\hyp\circ\mathcal F^{-1}(J,q)=(h_{(J,iq)},h_{(J,-{i}q)})~,$$
where $h(J,q)$ denotes the unique hyperbolic metric on $\Sigma$ such that $\id:(\Sigma,J)\to(\Sigma,h_{(J,q)})$ is harmonic. 
\end{lemma}

\subsection{Constant curvature surfaces and circle action}\label{subsec:cc_and_circle}

Another parameterization of $\MS(\Sigma)$ by the product of two copies of Teichm\"uller space is constructed as follows. Given a maximal surface in a MGHC AdS manifold of genus $\geq 2$, a standard computation shows that the two surfaces at distance $\pi/4$ from the maximal surface have intrinsic curvature $-2$ (see \cite{chentam}, \cite[Theorem 7.1.4]{bonsante2020antide}). Hence, multiplying the first fundamental forms of these surfaces by a factor $2$, so that they become of intrinsic curvature $-1$, one finds two hyperbolic metrics on $\Sigma$, which are expressed by:
\begin{equation}\label{eq:minus_plus_metric}
(h,B)\mapsto (h_-:=h((\1-B)\cdot,(\1-B)\cdot),h_+:=h((\1+B)\cdot,(\1+B)\cdot))~.
\end{equation}

By arguments similar to those leading to Theorem \ref{thm:mess_homeo}, one can prove that this produces again a natural homeomorphism, namely:

\begin{theorem}[{\cite[Theorem 3.21]{krasnov_schlenker_minimal}}] \label{thm:cc_homeo}
Let $\Sigma$ be a closed oriented surface of genus $\geq 2$. The map sending a pair $(h,B)$ to the pair of hyperbolic metrics $(h_-,h_+)$ in \eqref{eq:minus_plus_metric} induces a $MCG(\Sigma)$-invariant homeomorphism
\[
    \mappa{\mathcal C^\hyp}{\MS(\Sigma)}{\Teich^{\hyp}(\Sigma)\times\Teich^{\hyp}(\Sigma)}~.
\]
The hyperbolic metrics $(h_-,h_+)$ are obtained as the first fundamental forms of the $\pi/4$-equidistant surfaces from the maximal surface, after rescaling by a suitable constant.
\end{theorem}

As in Lemma \ref{lem:mess_complex_str}, we can express this map in terms of the conformal model of Teichm\"uller space.

\begin{lemma}\label{lem:cc_complex_str}
Let $\Sigma$ be a closed oriented surface of genus $g\geq 2$. The $MCG(\Sigma)$-invariant homeomorphism 
\[
    \mappa{\mathcal C^\conf}{\MS(\Sigma)}{\Teich^{\conf}(\Sigma)\times\Teich^{\conf}(\Sigma)}~.
\]
is induced by the map
$$(h,B) \mapsto \left(J_l:= (\1 -  B)^{-1} J (\1 -  B) ,J_r:= (\1 +  B)^{-1} J (\1 +  B) \right)~,$$
where $J$ is the complex structure defined by the metric $h$.
\end{lemma}

From Theorem \ref{thm:cotangent_parameterization}, we see that $\MS(\Sigma)$ is endowed with a circle action, which acts on a pair $(J,q)$, where $J$ is a complex structure and $q$ a holomorphic quadratic differential, by multiplying the holomorphic quadratic differential by $e^{i\theta}$. By Remark \ref{rmk:rotateB},  this $S^1$ action on $\MS(\Sigma)$ has the following expression in terms of the pairs $(h,B)$,
$$e^{i\theta}\cdot (h,B)=(h,((\cos\theta)\1-(\sin\theta) J)B)~,$$
and it can be checked directly that the new pair obtained in this way still gives the embedding data of a maximal surface in a MGHC AdS manifold. Denoting by $R_\theta:\MS(\Sigma)\to \MS(\Sigma)$ the $S^1$ action, it follows immediately from \eqref{eq:left_right_metric} and \eqref{eq:minus_plus_metric} that:
\begin{equation}\label{eq:CMrotate}
\mathcal M^\hyp=\mathcal C^\hyp\circ R_{-\pi/2}~.
\end{equation}
Observe moreover that $R_{\pi}$ has the effect of switching the left and right components under the maps $\mathcal C^\hyp$ and $\mathcal M^\hyp$.

By conjugating the circle action by the map $\mathcal M^\hyp$ (or $\mathcal C^\hyp$), one gets an induced circle action on $\Teich^{\hyp}(\Sigma)\times\Teich^{\hyp}(\Sigma)$, which has been defined in \cite{bonsante2013a_cyclic} as the \emph{landslide flow}. 
Motivated by this construction, we will consider the map 
$$\mathcal C_\theta^\hyp=\mathcal C^\hyp\circ R_\theta:{\MS(\Sigma)}\to {\Teich^{\hyp}(\Sigma)\times\Teich^{\hyp}(\Sigma)}~.$$

Clearly $\mathcal C_0^\hyp=\mathcal C^\hyp$ and $\mathcal C_{- \pi/2}^\hyp=\mathcal M^\hyp$.
Let us introduce $\mathcal H_\theta=\mathcal C_\theta^\hyp\circ\mathcal F^{-1}$, which is the composed map from $T^*\mathcal T(\Sigma)\to \mathcal T(\Sigma)\times \mathcal T(\Sigma)$ used in Theorem \ref{thm:Htheta}. Then we immediately obtain:

\begin{lemma}\label{lemma:mapH}
Let $\Sigma$ be a closed oriented surface of genus $\geq 2$. Then
$$\mathcal H_\theta(J,q)=(h_{(J,-e^{i\theta}q)},h_{(J,e^{i\theta}q)})~,$$
where $h(J,q)$ denotes the unique hyperbolic metric on $\Sigma$ such that $\id:(\Sigma,J)\to(\Sigma,h_{(J,q)})$ is harmonic. 
\end{lemma}

\subsection{An equivalent model for \texorpdfstring{$\MS(\Sigma)$}{MS(S)}} \label{subsec:change_coord}

We introduce here a fundamental ``change of variables'', which permits us to adopt a simpler model to study $\MS(\Sigma)$. The basic idea is to replace the metric $h$ (which is the first fundamental form of the maximal surface) by a suitable conformal metric $g$, so that the area of $g$ is independent of the point of $\MS(\Sigma)$. By a standard argument in symplectic geometry, we will be allowed to assume that the the area form of $g$ is a fixed symplectic form $\rho$ on $\Sigma$. 

To make this concrete, let us introduce the function $f(t) = \sqrt{1 + t^2}$, which we will always apply to $t=\norm{\sigma}_g^2$, for $g$ a Riemannian metric conformal to $h$. Recall that in general, if $\sigma$ is a symmetric (2,0) tensor on $\Sigma$, $\norm{\sigma}_g^2$ denotes the squared norm of the operator $A=g^{-1}\sigma$, namely one half the trace of $A^TA$, where $A^T$ is the $g$-adjoint operator of $A$. When $\sigma$ is the real part of a quadratic differential, $A$ is $g$-self-adjoint and traceless by Lemma \ref{lem:hol_quadr_diff_equivalence}, hence $\norm{\sigma}_g^2=-\det A$.

Let $\Sigma$ be a closed orientable surface of genus $\geq 2$. We now introduce the following space:
$$\mathcal{MS}_0(\Sigma):=\left\{(g,\sigma)\left| \begin{aligned}&\qquad\quad g\text{ is a Riemannian metric on }\Sigma\\  &\,\,\sigma\text{ is the real part of a }g\text{-quadratic differential} \\ &\,(h=(1 + f(\norm{\sigma}_g)) \, g,B=h^{-1}\sigma) \text{ satisfy }\eqref{eq:gauss_codazzi} \end{aligned}\right. \right\}/\mathrm{Diff}_0(\Sigma)$$


The map sending $(g,\sigma)$ to $(h,B)$, where $h = (1 + f(\norm{\sigma}_g)) \, g$ and $B=h^{-1}\sigma$ tautologically induces a $MCG(\Sigma)$-equivariant map from $\mathcal{MS}_0(\Sigma)$ to $\mathcal{MS}(\Sigma)$. By \cite{krasnov_schlenker_minimal}, the principal curvatures of a maximal Cauchy surface of genus $\geq 2$ in a MGHC AdS manifold are strictly less than one in absolute value, which implies that $\det B\in (-1,0]$. (On the other hand, if $\Sigma=T^2$, then $\det B\equiv -1$ since the maximal surface is flat.) Hence the following lemma shows that the map induced by
$$(h,B)\mapsto\left(\frac{1+\det B}{2}h,hB\right)$$
is an inverse, and $\mathcal{MS}(\Sigma)$ and $\mathcal{MS}_0(\Sigma)$ are homeomorphic, under the assumption that $\Sigma$ has genus $\geq 2$.

\begin{lemma}\label{lem:det_B}
Given a metric $g$ and a $(2,0)$-tensor $\sigma$ on $\Sigma$, let $h = (1 + f(\norm{\sigma}_g)) \, g$ and $B=h^{-1}\sigma$. Then
\begin{equation} \label{eq:det_B}
\det B = - \frac{\norm{\sigma}_g^2}{(1 + f(\norm{\sigma}_g))^2} \qquad\text{ and }\qquad
1 + \det B = \frac{2}{1 + f(\norm{\sigma}_g)} ~.
\end{equation}
\end{lemma}
\begin{proof}
The first identity comes from observing that $\norm{\sigma}_h^2=-\det B$ as remarked above, and that if $h=e^u g$ then $\norm{\sigma}_h^2=e^{-2u}\norm{\sigma}_g^2$. The second identity is an easy algebraic manipulation using the definition of $f$.
\end{proof}

An immediate consequence is the following:

\begin{lemma}\label{lem:same_area}
Let $(g,\sigma)\in\mathcal{MS}_0(\Sigma)$. Then the area of $g$ equals $-\pi\chi(\Sigma)$.
\end{lemma}
\begin{proof}
Let $h= (1 + f(\norm{\sigma}_g)) \, g$ as usual. Using \eqref{eq:det_B}, the area forms of $g$ and $h$ satisfy the identity:
$$dA_g=\frac{1}{1 + f(\norm{\sigma}_g)}dA_h=\frac{1+\det B}{2}dA_h~.$$
Since the pair $(h,B)$ satisfy \eqref{eq:gauss_codazzi} by hypothesis, $1+\det B=-K_h$, hence 
$$\int_\Sigma dA_g=-\frac{1}{2}\int_\Sigma K_hdA_h=-\pi\chi(\Sigma)$$
by Gauss-Bonnet.
\end{proof}

From now on, we will fix an area form $\rho$ of total area $-\pi\chi(\Sigma)$. Given an almost-complex structure $J$ on $\Sigma$, we define the Riemannian metric
$$g_J(\cdot,\cdot):=\rho(\cdot,J\cdot)~.$$
Clearly $dA_{g_J}=\rho$. 
We introduce the space $\mathcal{MS}_0(\Sigma,\rho)$ of pairs $(J,\sigma)$ such that $(g_J,\sigma)$ satisfy the conditions in the definition of $\mathcal{MS}_0(\Sigma)$, namely the pair $(h=(1 + f(\norm{\sigma}_{g_J})) \, g_J,B=h^{-1}\sigma)$ satisfy \eqref{eq:gauss_codazzi}, and we quotient by the action of $\mathrm{Symp}_0(\Sigma,\rho)$ (i.e. the identity component in the group of diffeomorphisms of $\Sigma$ that preserve $\rho$).\\

It remains to show that the map $(J,\sigma)\mapsto (h,B)$ induces a homeomorphism between $\mathcal{MS}_0(\Sigma,\rho)$ and $\mathcal{MS}(\Sigma)$. This follows from standard arguments relying on the Moser trick, and we only give a sketch here, see for instance \cite[\S 3.2.3]{hodgethesis} for more details. Moser's stability theorem asserts that given a smooth family $\omega_t$ of cohomologous symplectic forms on a closed manifold, there exists a family of diffeomorphisms $\phi_t$ such that $\phi_t^*\omega_t=\omega_0$. On a surface $\Sigma$, given two area forms $\rho$ and $\rho'$ of the same total area, one can apply Moser's theorem to the family $\rho_t=(1-t)\rho+t\rho'$ and deduce that there exists $\phi\in\mathrm{Diff}_0(\Sigma)$ such that $\phi^*\rho'=\rho$.  This implies that any $(g,\sigma)$ as in the definition of $\mathcal{MS}(\Sigma)$ has a representative in its $\mathrm{Diff}_0(\Sigma)$-orbit whose area form is $\rho$, i.e. a representative of the form $(g_J,\sigma)$. Moreover, if $\psi_t$ is a family of diffeomorphisms such that $\psi_0=\id$ and $\psi_1^*\rho=\rho$, by applying again Moser's theorem to the family $\rho_t=\psi_t^*\rho_1$ one can deform $\psi_t$ to a family of symplectomorphisms $\phi_t$ such that $\phi_0=\id$ and $\phi_1=\psi_1$. This shows that $\mathrm{Diff}_0(\Sigma)\cap \mathrm{Symp}(\Sigma,\rho)=\mathrm{Symp}_0(\Sigma,\rho)$. In conclusion, we have:

\begin{proposition}
Let $\Sigma$ be a closed oriented surface of genus $\geq 2$. The map
$$(J,\sigma)\mapsto (h=(1 + f(\norm{\sigma}_{J})) \, g_J,B=h^{-1}\sigma)$$
induces a
 $MCG(\Sigma)$-invariant homeomorphism between $\mathcal{MS}_0(\Sigma,\rho)$ and $\mathcal{MS}(\Sigma)$.  
    \end{proposition}

\section{The toy model: genus 1}\label{sec:toy model}

The purpose of this section is to provide a para-hyperK\"ahler structure on the cotangent bundle of the space $\mathcal J(\R^2)$ of linear complex structures on $\R^2$. Interpreting the complement of the zero section in  $T^*\mathcal J(\R^2)$ as the space $\mathcal{MGH}(T^2)$, we will deduce the case $\Sigma=T^2$ in all the results stated in the introduction.

\label{sec:toy_model_via_complex_structures}

\subsection{Space of linear almost-complex structures} \label{subsec:linear_almost_complex}

We begin by defining the space $\mathcal J(\R^2)$. In this section, $\rho$ denotes the standard volume form $\rho = \dd{x} \wedge \dd{y}$ on $\R^2$.

\begin{definition}
 We denote by $\mathcal{J}(\R^2)$ the set of endomorphisms $J$ of $\R^2$ such that $J^2 = - \1$, and satisfying $\rho (v, J v) > 0$ for some (and consequently for every) non-\hsk zero vector $v \in \R^2$. 
 \end{definition}

In other words, $\mathcal{J}(\R^2)$ is the collection of all (linear) complex structures on $\R^2$ that are compatible with its standard orientation. It turns out that $\mathcal{J}(\R^2)$ is a two-dimensional manifold.\\

It is simple to see that, given any $J \in \mathcal{J}(\R^2)$, the tensor $g_J \defin \rho(\cdot , J \cdot )$ is a positive definite scalar product on $\R^2$, with respect to which $J$ is an orthogonal transformation. By differentiating the identity $J^2 = - \1$, we see that the tangent space of $\mathcal{J}(\R^2)$ can be described as
\[
T_J \mathcal{J}(\R^2) = \set{ \dot{J} \in \End(\R^2) \mid \dot{J} J + J \dot{J} = 0 } .
\]
Equivalently, $T_J \mathcal{J}(\R^2)$ is the space of endomorphisms $\dot{J}$ that are traceless and $g_J$-\hsk self-\hsk adjoint. The tangent space $T_J \mathcal{J}(\R^2)$ is endowed with a natural (almost) complex structure, given by ${\mathcal {I}} (\dot{J}) \defin - J \dot{J}$. \\

We will represent the cotangent space of $\mathcal{J}(\R^2)$ as follows:
\[
T^*_J \mathcal{J}(\R^2) = \set{ \sigma \in S_2(\R^2) \mid J^* \sigma = \sigma (J \cdot , J \cdot ) = - \sigma } ,
\]
where $S_2(\R^2)$ stands for the space of symmetric bilinear forms of $\R^2$. An equivalent way to describe the cotangent space at $J$ is as the set of bilinear forms on $\R^2$ that can be written as $\sigma = \Re \phi$, where $\phi$ is a symmetric $\C$-\hsk valued bilinear form that is complex-\hsk linear with respect to $J$. In other words, $\phi$ satisfies $\phi(J \cdot, \cdot) = \phi(\cdot, J \cdot) = i \phi(\cdot, \cdot)$. When this is the case, then $\phi$ can be expressed as follows:
\begin{equation} \label{eq:quadratic_diff}
\phi = \sigma - i \, \sigma(\cdot, J \cdot) .
\end{equation}
Observe also that $\sigma$ belongs to $T_J^* \mathcal{J}(\R^2)$ if and only if $g_J^{-1} \sigma$ belongs to $T_J \mathcal{J}(\R^2)$ (here $g_J^{-1} \sigma$ represents the $g_J$-\hsk self-\hsk adjoint operator associated to $\sigma$, that is, $\sigma(\cdot,\cdot)=g_J((g_J^{-1} \sigma)\cdot,\cdot)=g_J(\cdot,(g_J^{-1} \sigma)\cdot)$). The natural pairing between the tangent and the cotangent space is the following:
\begin{equation}\label{eq:pairing}
\scal{\sigma}{\dot{J}}_J = \frac{1}{2} \tr_{g_J} (\sigma(\cdot, \dot{J} \cdot)) = \frac{1}{2} \mathrm{tr}(g_J^{-1} \sigma \dot{J}) .
\end{equation}
We also define the following positive definite scalar products:
\[
\scall{\sigma}{\sigma'}_J \defin \frac{1}{2} \tr(g_J^{-1} \sigma g_J^{-1} \sigma') , \qquad \scall{\dot{J}}{\dot{J}'}_J \defin \frac{1}{2} \mathrm{tr}(\dot{J} \dot{J}')
\]
for every $\dot{J}, \dot{J}' \in T_J \mathcal{J}(\R^2)$ and $\sigma, \sigma' \in T_J^* \mathcal{J}(\R^2)$. It is immediate to check that the almost complex structure ${\mathcal {I}} $ preserves both scalar products. We also denote $\norm{\sigma}_J^2=\langle \sigma,\sigma\rangle_J$. If $\phi$ is a quadratic differential whose real part is equal to $\sigma$, we set $\norm{\phi}_J \defin \norm{\sigma}_J$.

\subsection{The tangent space of \texorpdfstring{$T^* \mathcal{J}(\R^2)$}{T*J(R2)}}\label{subsec:tangent space toy}

We now provide a characterization of the tangent space to $T^* \mathcal{J}(\R^2)$. 

\begin{lemma} \label{lem:characterization_tangent_space}
Let $(J, \sigma) \in T^* \mathcal{J}(\R^2)$. Then $(\dot{J}, \dot{\sigma}) \in \End(\R^2) \times S_2(\R^2)$ belongs to $T_{(J, \sigma)} T^* \mathcal{J}(\R^2)$ if and only if
\[
\dot{J} \in T_J \mathcal{J}(\R^2), \ \dot{\sigma}_0 \in T_J^* \mathcal{J}(\R^2) \text{ and } \tr_{g_J} \dot{\sigma} = - 2\scal{\sigma}{J \dot{J}}_J ,
\]
where $\dot{\sigma}_0$ denotes the $g_J$-\hsk traceless part of $\dot{\sigma}$. 
\end{lemma}
\begin{proof} We first compute
\begin{equation} \label{eq:derivative_scal_prod}
\dot{g}_J = \rho(\cdot , \dot{J} \cdot ) = - \rho (\cdot , J^2 \dot{J} \cdot ) = - g_J(\cdot , J \dot{J}) \ .
\end{equation}
Then, because $\sigma$ is $g_{J}$-traceless, we have
\[
0 = (\tr_{g_J} \sigma)' = - \tr(g_J^{-1} \, \dot{g}_J \, g_J^{-1} \sigma) + \tr(g_J^{-1} \dot{\sigma}) 
\]
hence
\[
\tr_{g_J} \dot{\sigma} = - \tr(J \dot{J} g_J^{-1} \sigma) = -2 \scal{\sigma}{J \dot{J}}_J~.
\]
This concludes the proof.
\end{proof}

The group $\SL(2,\R)$ naturally acts on $\mathcal{J}(\R^2)$ by conjugation and, more generally on its tangent and cotangent space as follows:
\begin{align*}
(J, \dot{J}) \in T \mathcal{J}(\R^2), \quad A \cdot (J, \dot{J}) & \defin (A J A^{-1}, A \dot{J} A^{-1}) \\
(J, \sigma) \in T^* \mathcal{J}(\R^2), \quad A \cdot (J, \sigma) & \defin (A J A^{-1}, (A^{-1})^* \sigma)
\end{align*}
for any $A \in \SL(2,\R)$. The action of $\SL(2,\R)$ induces a faithful action of $\PSL(2,\R)=\SL(2,\R)/\{\pm\1\}$.

\begin{lemma} \label{lem:scalar_prod_invariance}
For every $A \in \SL(2,\R)$ and $J \in \mathcal{J}(\R^2)$, we have:
\begin{align*}
\scal{A \cdot \sigma}{A \cdot \dot{J}}_{A \cdot J} & = \scal{\sigma}{\dot{J}}_J , \\
\scall{A \cdot \dot{J}}{A \cdot \dot{J}'}_{A \cdot J} & = \scall{\dot{J}}{\dot{J}'}_J , \\
\scall{A \cdot \sigma}{A \cdot \sigma'}_{A \cdot J} & = \scall{\sigma}{\sigma'}_J ,
\end{align*}
where $\dot{J}, \dot{J}' \in T_J \mathcal{J}(\R^2)$ and $\sigma, \sigma' \in T_J^* \mathcal{J}(\R^2)$.
\end{lemma}
\begin{proof}
The proof is immediate, once one checks that $g_{AJA^{-1}}^{-1}(A^{-1})^*\sigma=A(g_J^{-1}\sigma)A^{-1}$.
\end{proof}

By differentiating the  $\SL(2,\R)$-action on $T_J^* \mathcal{J}(\R^2)$, we obtain a linear isomorphism from $T_{(J,\sigma)} T^* \mathcal{J}(\R^2)$ to $T_{A\cdot(J,\sigma)} T^* \mathcal{J}(\R^2)$. By a little abuse of notation, we still denote this isomorphism by $A$. It is explicitly given by:
$$(\dot J, \dot \sigma) \in T_{(J,\sigma)}T^* \mathcal{J}(\R^2), \quad A \cdot (\dot J,\dot \sigma)  \defin (A \dot J A^{-1}, (A^{-1})^* \dot \sigma)~.$$
We remark that, using Lemma \ref{lem:scalar_prod_invariance}, one could verify by hands that the conditions of Lemma \ref{lem:characterization_tangent_space} are preserved by this expression.\\

It is also useful to provide a natural linear isomorphism between the tangent space of $T^*_{(J,\sigma)}\mathcal J(\R^2)$ and the product of two copies of $T_J\mathcal J(\R^2)$.

\begin{proposition}\label{prop:characterization_tangent_space2}
The map
\[
\begin{matrix}
\Psi_{(J,\sigma)}:&T_{(J,\sigma)} T^* \mathcal{J}(\R^2) & \longrightarrow & (T_J \mathcal{J}(\R^2))^2 \\
&(\dot{J}, \dot{\sigma}) & \longmapsto & (\dot{J}, g_J^{-1} \dot{\sigma}_0)
\end{matrix}
\]
is a linear isomorphism, with inverse
\[
\begin{matrix}
\Xi_{(J,\sigma)}:&(T_J \mathcal{J}(\R^2))^2 & \longrightarrow & T_{(J,\sigma)} T^* \mathcal{J}(\R^2) \\
&(\dot{J}, \dot{K}) & \longmapsto & (\dot{J}, g_J(\cdot, \dot{K} \cdot)- \scal{\sigma}{J \dot{J}}_J \ g_J ) .
\end{matrix}
\]
Moreover, $\Psi_{(J,\sigma)}$ commutes with the $\SL(2,\R)$ actions, meaning that 
$$(A,A)\circ \Psi_{(J,\sigma)}=\Psi_{A\cdot(J,\sigma)}\circ A~.$$
\end{proposition}
\begin{proof}
The proof follows immediately from Lemma \ref{lem:characterization_tangent_space} and Lemma \ref{lem:scalar_prod_invariance}.
\end{proof}

\subsection{A para-\hsk hyperK\"ahler structure on \texorpdfstring{$T^*\mathcal{J}(\R^2)$}{T*J(R2)}} \label{subsec:parahyperkahler_toy}

Throughout the paper,  $f$ will denote the function $f(t) \defin \sqrt{1 + t^2}$ for $t \in \R$ which we introduced already in Section \ref{subsec:change_coord}. Unless otherwise stated, we will denote by $\scall{\cdot}{\cdot}$, $\scal{\cdot}{\cdot}$ the scalar products $\scall{\cdot}{\cdot}_J$ on $T_J \mathcal{J}(\R^2)$ and $T_J^* \mathcal{J}(\R^2)$ (it will be clear from the context which one of these we will refer to) and the pairing $\scal{\cdot}{\cdot}_J$ between $T_J \mathcal{J}(\R^2)$ and its dual. Similarly $\norm{\cdot} = \norm{\cdot}_J$. 

\begin{definition}
Let us define the symmetric bi-linear form $\mathbf{g}$ on $T_{(J,\sigma)}T^{*}\mathcal{J}(\R^{2})$:
\begin{equation} \label{eq:definition_g_toy}
\mathbf{g}_{(J,\sigma)}((\dot{J}, \dot{\sigma}), (\dot{J}', \dot{\sigma}')) = f(\norm{\sigma}) \, \scall{\dot{J}}{\dot{J}'} - \frac{1}{f(\norm{\sigma})} \, \scall{\dot{\sigma}_0}{\dot{\sigma}_0'}
\end{equation}
and the endomorphisms $\mathbf{I},\mathbf{J},\mathbf{K}$ of  $T_{(J,\sigma)}T^{*}\mathcal{J}(\R^{2})$:
\begin{align}
\mathbf{I}_{(J,\sigma)}(\dot{J}, \dot{\sigma}) & = \left( - J \dot{J}, \ - \dot{\sigma}_0(\cdot, J \cdot) - \scal{\sigma}{\dot{J}} \, g_J \right) \label{eq:definition_I_toy} \\
\mathbf{J}_{(J,\sigma)}(\dot{J}, \dot{\sigma}) & = \left( \frac{1}{f(\norm{\sigma})} \, g_J^{-1} \dot{\sigma}_0 , \ f(\norm{\sigma}) \, g_J(\cdot , \dot{J} \cdot) + \frac{\scall{\sigma}{\dot{\sigma}_0(\cdot, J \cdot)}}{f(\norm{\sigma})} \, g_J \right) \label{eq:definition_J_toy} \\
\mathbf{K}_{(J,\sigma)}(\dot{J}, \dot{\sigma}) & = \left( - \frac{1}{f(\norm{\sigma})} \, J g_J^{-1} \dot{\sigma}_0 , \ - f(\norm{\sigma}) \, g_J(\cdot , \dot{J} J \cdot) - \frac{\scall{\sigma}{\dot{\sigma}_0}}{f(\norm{\sigma})} \, g_J \right) \label{eq:definition_K_toy}
\end{align}
where $f(t)=\sqrt{1+t^2}$.
\end{definition}

For future reference, we also record the expressions of the forms:
\[
\omega_\mathbf{I} = \mathbf{g}(\cdot, \mathbf{I} \cdot), \quad \omega_\mathbf{J} = \mathbf{g}(\cdot, \mathbf{J} \cdot), \quad \omega_\mathbf{K} = \mathbf{g}(\cdot, \mathbf{K} \cdot).
\]
These are given by:
\begin{align}
(\omega_\mathbf{I})_{(J,\sigma)} ((\dot{J}, \dot{\sigma}), (\dot{J}', \dot{\sigma}')) & = - f(\norm{\sigma}) \scall{\dot{J}}{J \dot{J}'} + \frac{1}{f(\norm{\sigma})} \scall{\dot{\sigma}_0}{\dot{\sigma}_0'(\cdot , J \cdot)} , \label{eq:definition_omegaI_toy} \\
(\omega_\mathbf{J})_{(J,\sigma)} ((\dot{J}, \dot{\sigma}), (\dot{J}', \dot{\sigma}')) & = \scal{\dot{\sigma}_0'}{\dot{J}} - \scal{\dot{\sigma}_0}{\dot{J}'} , \label{eq:definition_omegaJ_toy} \\
(\omega_\mathbf{K})_{(J,\sigma)} ((\dot{J}, \dot{\sigma}), (\dot{J}', \dot{\sigma}')) & = \scal{\dot{\sigma}_0'}{J \dot{J}} - \scal{\dot{\sigma}_0}{J \dot{J}'} . \label{eq:definition_omegaK_toy}
\end{align}

\begin{theorem}\label{thm:parahyper_structure_toy}
The quadruple $(\mathbf g,\mathbf{I},\mathbf{J},\mathbf{K})$ is an $\SL(2,\R)$-invariant para-hyperK\"ahler structure on $T^*\mathcal J(\R^2)$.
\end{theorem}
\begin{proof}
Observe that, through the linear isomorphism of Lemma \ref{lem:characterization_tangent_space}, the endomorphisms $\mathbf{I}, \mathbf{J}, \mathbf{K}$ can be represented as:
\begin{align*}
(\Xi_{(J,\sigma)}^*\mathbf{I}_{(J,\sigma)}) (\dot{J}, \dot{K}) & = (- J \dot{J}, \, J \dot{K}) , \\
(\Xi_{(J,\sigma)}^*\mathbf{J}_{(J,\sigma)}) (\dot{J}, \dot{K}) & = (f(\norm{\sigma})^{-1} \dot{K} , \, f(\norm{\sigma}) \dot{J} ) , \\
(\Xi_{(J,\sigma)}^*\mathbf{K}_{(J,\sigma)}) (\dot{J}, \dot{K}) & = (- f(\norm{\sigma})^{-1} J \dot{K}, \, f(\norm{\sigma}) J \dot{J}) .
\end{align*}
From here, it is clear that the following relations are satisfied:
\[
\mathbf{I}^2 = - \1, \quad \mathbf{J}^2 = \mathbf{K}^2 = \1, \quad \mathbf{K} = \mathbf{I} \mathbf{J} .
\]

We will see in Lemma \ref{lemma:Cvaluedform} that the $2$-\hsk forms $\omega_\mathbf{J}$ and $\omega_\mathbf{K}$ are respectively the real and imaginary part of the complex symplectic form $\omega^\C$ on $T^* \mathcal{J}(\R^2)$, which is defined as the differential of $- \lambda^\C$. Hence the forms $\omega_\mathbf{J}$ and $\omega_\mathbf{K}$ are obviously exact. We will show in Corollary \ref{cor:mess_paraholo} that $\omega_{\i}$ is closed, as well. One can now apply a general argument to conclude that that $\i,\j,\k$ are integrable, see Lemma \ref{lm:integrability} in Appendix \ref{appendixB}. However, we will also deduce this directly using some of the interpretations of $\i,\j,\k$ that we show below. More concretely, integrability of $\i$ is proved in Corollary \ref{cor:Iintegrable_torus}; integrability of $\j$ and $\k$ follows from the pull-back results of Corollary \ref{cor:mess_paraholo} and Corollary \ref{cor:cc_paraholo}. This shows that the quadruple $(\g,\i,\j,\k)$ defines a para-hyperK\"ahler structure on $T^{*}\mathcal{J}(\R^{2})$. The $\SL(2,\R)$-invariance is checked in each expression by applying Lemma \ref{lem:scalar_prod_invariance}. 
\end{proof}

\begin{remark}\label{rmk:compareH2}
Let us observe that $\i$ preserves the tangent space to the $0$-\hsk section of $T^* \mathcal{J}(\R^2)$, since $\i_{(J,0)}(\dot J,0)=(-J\dot J,0)$. We will denote by $G$ and $\Omega$ the restrictions of $\g$ and $\omega_\mathbf{I}$ to the $0$-\hsk section of $T^* \mathcal{J}(\R^2)$, which is identified to $\mathcal J(\R^2)$. In particular, for every $J \in \mathcal{J}(\R^2)$
\[
G_J(\dot{J}, \dot{J}') \defin\scall{\dot{J}}{ \dot{J}'}_J\qquad  \Omega_J(\dot{J}, \dot{J}') \defin - \scall{\dot{J}}{J \dot{J}'}_J ,
\]
where $\dot{J}, \dot{J}' \in T_J \mathcal{J}(\R^2)$. Hence $(G,\Omega)$ is a K\"ahler structure on $\mathcal J(\R^2)$.

Now, it turns out that  $\mathcal{J}(\R^2)$ is diffeomorphic to the hyperbolic plane $\Hyp^2$. To see this, one can define a map from $\mathcal{J}(\R^2)$ to the upper half-space model of $\Hyp^2$ by declaring that the standard linear complex structure 
$$J_0=\begin{pmatrix}0 & -1 \\ 1 & 0\end{pmatrix}$$
is mapped to $i\in\Hyp^2$. Observe that $\SO(2)<\SL(2,\R)$ stabilizes both $J_0$ (for the $\SL(2,\R)$-action on $\mathcal{J}(\R^2)$ introduced in Section \ref{subsec:tangent space toy}) and $i$ (for the classical action in the upper half-plane model). Since both actions are transitive, one can uniquely extend this assignment  to an $\SL(2,\R)$-equivariant diffeomorphism. It is not hard to check that this is a K\"ahler isometry, that is, the metric $G$ corresponds to the hyperbolic metric of $\Hyp^2$, the complex structure ${\mathcal {I}}  $ to the standard complex structure of $\Hyp^2$, and therefore the symplectic form $\Omega$ to the area form of $\Hyp^2$. See \cite[Lemma 4.3.2]{trautwein_thesis}, \cite[\S 3.1]{trautwein2019hyperkahler} or \cite[\S 2.2.2]{hodgethesis} for more details.
\end{remark}

\subsection{Liouville form on \texorpdfstring{$T^* \mathcal{J}(\R^2)$}{T*J(R2)}} \label{subsec:liouville_form}

Recall that the Liouville form of any manifold $M$ is the 1-form $\lambda$ on $T^*M$ defined by $\lambda_{(p,\alpha)}(v)=\alpha(\pi_*v)$, for $\alpha\in T^*_pM$. If $M$ is a complex manifold, one has an induced complex structure $\i$ on $T^*M$, and $\lambda$ is the real part of the complex-valued 1-form $\lambda^\C=\lambda-i\lambda\circ{\i} $. The term $\lambda\circ{\i}$ can also be written as $\alpha(\pi_*(\mathcal I(v)))$, where $\mathcal I$ is the complex structure of $M$, since $\pi$ is holomorphic. In our setting, we therefore have the following expression for 
the (complex-\hsk valued) Liouville form on $T^* \mathcal{J}(\R^2)$:
\[
\lambda^\C_{(J, \sigma)}(\dot{J}, \dot{\sigma}) = \scal{\sigma}{\dot{J}}_J + i \, \scal{\sigma}{J \dot{J}}_J .
\]
We define the complex-\hsk valued \emph{cotangent symplectic structure} of $T^* \mathcal{J}(\R^2)$ by setting 
$$\omega^\C \defin - \dd{\lambda^\C}~.$$
In this section we show the following:

\begin{lemma}\label{lemma:Cvaluedform}
The $\C$-\hsk valued symplectic form $\omega^\C$ on $T^*\mathcal J(\R^2)$ equals $\omega_\j+i\omega_\k$.
\end{lemma}

In other words, we have to show that $\omega^\C$ has the following expression:
\begin{align*}
\omega^\C ((\dot{J}, \dot{\sigma}), (\dot{J}', \dot{\sigma}')) = (\scal{\dot{\sigma}_0'}{\dot{J}}_J - \scal{\dot{\sigma}_0}{\dot{J}'}_J) + i \, (\scal{\dot{\sigma}_0'}{J \dot{J}}_J - \scal{\dot{\sigma}_0}{J \dot{J}'}_J) .
\end{align*}
Before providing the proof, we give a useful lemma.

\begin{lemma} \label{lem:product_in_TJ}
    For every $\dot{J}, \dot{J}' \in T_J \mathcal{J}(\R^2)$ we have
    \[
    \dot{J} \dot{J}' = \scall{\dot{J}}{\dot{J}'}_J \, \1 - \scall{J \dot{J}}{\dot{J}'}_J \, J .
    \]
    Moreover, we have:
    \begin{enumerate}[i)]
    \item $\dot{J}^2 = - (\det\dot{J}) \1 = \|\dot{J}\|_J^2 \1$ for every $\dot{J} \in T_J \mathcal{J}(\R^2)$;
    \item $\mathrm{tr}(\dot{J} \dot{J}' \dot{J}'') = 0$ for every $\dot{J}, \dot{J}', \dot{J}'' \in T_J \mathcal{J}(\R^2)$.
    \end{enumerate}
\end{lemma}

\begin{proof} Let us first observe that
\[
J \dot{J} \dot{J}' = - \dot{J} J \dot{J}' =  \dot{J} \dot{J}' J .
\]
Therefore, $\dot{J} \dot{J}'$ commutes with $J$. It is simple to check that $M \in \End(\R^2)$ commutes with $J$ if and only if $M$ belongs to $\Span(\1, J)$, hence 
\[
    \dot{J} \dot{J}' = \scall{\dot{J}}{\dot{J}'}_J \, \1 - \scall{J \dot{J}}{\dot{J}'}_J \, J .
\]
By Cayley-Hamilton theorem,
\[
0=\dot{J}^2 - (\mathrm{tr} \dot{J}) \dot{J} + (\det \dot{J}) \, \1 = \dot{J}^2 + (\det\dot{J}) \, \1 .
\]
Therefore, we have $2 \norm*{\dot{J}}^2_J = \mathrm{tr}(\dot{J}^2) = - 2 \det \dot{J}$. For the last assertion, we apply the first part of the statement:
\begin{align*}
\mathrm{tr}(\dot{J} \dot{J}' \dot{J}'') & = \scall{\dot{J}}{\dot{J}'}_J \, \mathrm{tr}(\dot{J}'') - \scall{J \dot{J}}{\dot{J}'}_J \, \mathrm{tr}(J \dot{J}'') .
\end{align*}
This expression vanishes because $\dot{J}''$ and $J \dot{J}''$ are both traceless, being elements of $T_J \mathcal{J}(\R^2)$.
\end{proof}

\begin{proof}[Proof of Lemma \ref{lemma:Cvaluedform}]
The set $T^* \mathcal{J}(\R^2)$ can be considered as a submanifold of the vector space $\End(\R^2) \times S_2(\R^2)$. In particular, any tangent vector $(\dot{J}, \dot{\sigma}) \in T_{(J, \sigma)} T^*\mathcal{J}(\R^2)$ can be extended to a vector field, that we continue to denote with abuse by $(\dot{J}, \dot{\sigma})$, on a neighborhood of $(J, \sigma)$ in $\End(\R^2) \times S_2(\R^2)$, with values in $\End(\R^2) \times S_2(\R^2)$. Moreover, we can require that the component $\dot{J}$ depends only on the variable $J \in \End(\R^2)$. In all the following computations, we will denote by $D$ the standard flat connection on the vector space $\End(\R^2) \times S_2(\R^2)$, and we consider $\scal{\cdot}{\cdot}$ as a pairing defined for elements of the entire vector spaces $\End(\R^2)$ and $S_2(\R^2)$, where the extension is given by the same expression \eqref{eq:pairing}. 

If $\lambda \defin \Re \lambda^\C$, then we have:
\begin{align*}
\dd{\lambda}((\dot{J}, \dot{\sigma}), (\dot{J}', \dot{\sigma}')) & = (\dot{J}, \dot{\sigma}) (\lambda(\dot{J}', \dot{\sigma}')) - (\dot{J}', \dot{\sigma}') (\lambda(\dot{J}, \dot{\sigma})) - \lambda([(\dot{J}, \dot{\sigma}), (\dot{J}', \dot{\sigma}')]) \\
& = (\dot{J}, \dot{\sigma}) (\scal{\sigma}{\dot{J}'}_J) - (\dot{J}', \dot{\sigma}') (\scal{\sigma}{\dot{J}}_J) - \scal{\sigma}{[\dot{J}, \dot{J}']}_J .
\end{align*}
Observe that $\dot{J}(g_J) = - g_J(\cdot, J \dot{J} \cdot)$. Then
\begin{align*}
(\dot{J}, \dot{\sigma}) ( \scal{\sigma}{\dot{J}'}_J ) & = \frac{1}{2}  (\dot{J}, \dot{\sigma}) ( \mathrm{tr}(g_J^{-1} \sigma \dot{J}'))  \\
& = \frac{1}{2} ( - \mathrm{tr}(g_J^{-1} \dot{J}(g_J) g_J^{-1} \sigma \dot{J}') + \mathrm{tr}(g_J^{-1} \dot{\sigma} \dot{J}') + \mathrm{tr}(g_J^{-1} \sigma D_{\dot{J}} \dot{J}') ) \\
& =  \frac{1}{2} \mathrm{tr}(J \dot{J} g_J^{-1} \sigma \dot{J}') + \frac{1}{2} \mathrm{tr}(g_J^{-1} \dot{\sigma}_0 \dot{J}') + \scal{\sigma}{D_{\dot{J}} \dot{J}'}_J \\
& =  \scal{\dot{\sigma}_0}{\dot{J}'}_J + \scal{\sigma}{D_{\dot{J}} \dot{J}'}_J .
\end{align*}
In the last step we applied point \textit{ii)} of Lemma \ref{lem:product_in_TJ}. Replacing this relation in the expression for $\dd{\lambda}$ we obtain
\begin{align*}
\dd{\lambda}((\dot{J}, \dot{\sigma}), (\dot{J}', \dot{\sigma}')) & = \scal{\dot{\sigma}_0}{\dot{J}'}_J - \scal{\dot{\sigma}_0'}{\dot{J}}_J + 
\scal{\sigma}{ D_{\dot{J}} \dot{J}' - D_{\dot{J}'} \dot{J} - [\dot{J}, \dot{J}']}_J \\
& = \scal{\dot{\sigma}_0}{\dot{J}'}_J - \scal{\dot{\sigma}_0'}{\dot{J}}_J . 
\end{align*}
The same type of argument applies to the imaginary part of $\lambda^\C$. More concretely, setting $\mu:=\Im\lambda^\C$, we arrive at the expression:
\begin{align*}
\dd{\mu}((\dot{J}, \dot{\sigma}), (\dot{J}', \dot{\sigma}')) & = \scal{\dot{\sigma}_0}{J\dot{J}'}_J - \scal{\dot{\sigma}_0'}{J\dot{J}}_J + 
\scal{\sigma}{ D_{\dot{J}} (J\dot{J}') - D_{\dot{J}'} (J\dot{J}) - J[\dot{J}, \dot{J}']}_J~.
\end{align*}
It remains to check that the last term vanishes. But this term equals $\scal{\sigma}{\dot J\dot J'-\dot J'\dot J}$. Using Lemma \ref{lem:product_in_TJ}, $\dot J\dot J'-\dot J'\dot J$ is proportional to $J$, hence this scalar product vanishes because $\scal{\sigma}{J}=(1/2)\mathrm{tr}(g_J^{-1}\sigma J)=0$. 
\end{proof}

\begin{corollary}\label{cor:Iintegrable_torus}
The almost-complex structure $\i$ equals the almost-complex structure induced by $\mathcal I$ on $T^*\mathcal J(\R^2)$. In particular, $\i$ is an integrable almost-complex structure.
\end{corollary}
\begin{proof}
We have showed in Lemma \ref{lemma:Cvaluedform} that $\omega_\j$ and $\omega_\k$ are the real and imaginary parts of the complex symplectic form $\omega^\C$ of $T^*\mathcal J(\R^2)$. Denote by $\hat\i$ the almost complex structure of $T^*\mathcal J(\R^2)$. Since $\mathcal J(\R^2)$ is a complex manifold, $\omega^\C$ is a complex symplectic form, which means $\omega^\C(\cdot,\hat\i\cdot)=i\omega^\C(\cdot,\cdot)$. Taking the imaginary parts in this equality we find $\omega_\j=\omega_\k(\cdot,\hat\i\cdot)$. But from the definitions of $\omega_\j$ and $\omega_\k$, we see that 
$$\omega_\j=\g(\cdot,\j\cdot)=\g(\cdot,\k\i\cdot)=\omega_\k(\cdot,\i\cdot)~.$$
Since $\omega_\j$ and $\omega_\k$ are nondegenerate, this implies $\hat\i=\i$.
\end{proof}


\subsection{Relation with \texorpdfstring{$\mathcal{MGH}(T^2)$}{MGH(T2)}}\label{subsec:identificationT2}
Theorem \ref{thm:cotangent_parameterization} furnishes a diffeomorphism $\mathcal{F}: \mathcal{MGH}(T^2) \rightarrow T^{*}\mathcal{T}^{\conf}(T^{2})$ between the deformation space of MGHC anti-de Sitter manifolds diffeomorphic to $T^{2}\times \R$ and the complement of the zero section of the cotangent space $T^{*}\mathcal{T}^{\conf}(T^{2})$ to the Teichm\"uller space of the torus. The latter can be identified with $T^{*}\mathcal{J}(\R^{2})$, as we show in the following lemma.

\begin{lemma}\label{lemma torus identifications}
There is a homeomorphism between $\mathcal J(\R^2)$ and $\mathcal T^\conf(T^2)$, which is equivariant with respect to the actions of $\SL(2,\Z)\cong{MCG}(T^2)$. 
\end{lemma}
\begin{proof}
The map from $\mathcal J(\R^2)$ to $\mathcal T^\conf(T^2)$ is defined by considering a linear almost-complex structure $J$  as a (constant) tensor on $\R^2$, which therefore induces an almost-complex structure on the torus $T^2\cong \R^2/\Z^2$. The map is a bijection because every element in $\mathcal T^\conf(T^2)$, namely an isotopy class of almost-complex structures on $T^2$, can be represented as the conformal structure induced by $J_0$ on $\R^2/\Lambda$, for $\Lambda\cong\Z^2$ a (marked) lattice. One can moreover assume (up to homothety of $\Lambda$) that $\R^2/\Lambda$ has area 1, and such representation is unique up to conjugation in $\SO(2)$. Conjugating $J_0$ by the unique element of $\SL(2,\R)$ that maps $\Lambda$ to $\Z^2$ (as marked lattices), one finds the unique $J\in \mathcal J(\R^2)$ that is mapped to the given class in $\mathcal T^\conf(T^2)$. Identifying $MCG(T^2)$ with $\SL(2,\Z)$, the homeomorphism is clearly equivariant. 
\end{proof}

\begin{remark} \label{rmk:compareOmegaWP}
We observe that the symplectic form $\Omega$ introduced in Remark \ref{rmk:compareH2} coincides with $4 \Omega_\WP$, where $\Omega_\WP$ denotes the Weil-Petersson symplectic form on the space $\mathcal{J}(\R^2) \cong \Teich^\conf(T^2)$ through the identification described in Lemma \ref{lemma torus identifications} (see also Lemma \ref{lem:relation_Omega_WP}).
\end{remark}

Hence we immediately obtain the proof of Theorem \ref{thm:parahyper_structure} in genus one.

\begin{reptheorem}{thm:parahyper_structure}[genus one]
The deformation space $\mathcal{MGH}(T^2)$ admits a $MCG(T^2)$-invariant para-hyperK\"ahler structure $(\g,\i,\j,\k)$. 
\end{reptheorem}
\begin{proof}
By Proposition \ref{prop:MGHandMS} and Theorem \ref{thm:cotangent_parameterization}, $\mathcal{MGH}(T^2)$ is identified to the complement of the zero section in $T^*\mathcal T^\conf(T^2)$, which is in turn naturally identified to the complement of the zero section in $T^*\mathcal J(\R^2)$ by Lemma \ref{lemma torus identifications}. Hence the existence of the para-hyperK\"ahler structure follows immediately from Proposition \ref{thm:parahyper_structure_toy}. Since all the identifications are equivariant with respect to the action of $MCG(T^2)\cong\SL(2,\Z)$, and the para-hyperK\"ahler structure of $T^*\mathcal J(\R^2)$ is $\PSL(2,\R)$-invariant by Proposition \ref{thm:parahyper_structure_toy}, it follows that the obtained para-hyperK\"ahler structure on $\mathcal{MGH}(T^2)$ is mapping-class group invariant.
\end{proof}

More concretely, we can  see the induced map from $T^*\mathcal J(\R^2)$ to $T^*\mathcal T^\conf(T^2)$ as the map sending the pair $(J,\sigma)$, for $\sigma$ a symmetric bilinear form on $\R^2$ satisfying  $\sigma(J\cdot, J\cdot)=-\sigma$, to the pair $([J],\phi)$ where $[J]$ is the isotopy class as in the proof of Lemma \ref{lemma torus identifications} and $\phi$ is the holomorphic quadratic differential whose real part is $\sigma$ as in Equation \eqref{eq:quadratic_diff}.

\begin{remark}\label{rmk:constant_hqd1}
 We remark that the complex-linear quadratic differential $\phi$ is holomorphic with respect to $J$ simply because, identifying $(T^2,J)$ with a biholomorphism to the quotient of $\R^2$ by a lattice $\Lambda$ as in the proof of Lemma \ref{lemma torus identifications}, the lift of $\phi$ to $\R^2$ is constant. As a matter of fact, every holomorphic quadratic differential on $(T^2,J)$  lifts to a constant holomorphic quadratic differential, namely of the form $adz^2$ for $a\in\C$, on $\R^2\cong \C$. 
\end{remark}

\begin{remark}\label{rmk:constant_hqd2}
It will be useful to interpret the map from $\mathcal J(\R^2)$ to $\mathcal T^\conf(T^2)$ in terms of Beltrami differentials. Let us consider a smooth path $J_t\in \mathcal J(\R^2)$, and denote $J_{t=0}=J$. As usual, we consider these as almost-complex structures on $T^2=\R^2/\Z^2$. Then the Beltrami differential of the identity map $\mappa{\id}{(T^{2}, J)}{(T^{2}, J_t)}$ coincides with
    \[
    \nu_t = (\1 - J_t J)^{-1} (\1 + J_t J) .
    \]
    Indeed $\nu_t$ coincides with $L^{-1} \circ A$, where $L$ and $A$ are the complex linear and complex anti-\hsk linear parts of $\mappa{\dd(\id)}{(T_\cdot T^{2}, J)}{(T_\cdot T^{2} J_t)}$, respectively. 
    Hence $\nu_t$ is constant over $T^2$, which means that it is a harmonic Beltrami differential (i.e. of the form $g^{-1}\overline\psi$ for $g$ the flat metric and $\psi$ a holomorphic quadratic differential, see Remark \ref{rmk:constant_hqd1} above). A simple computation shows that the derivative of $\nu_t$ at $t=0$, which represents an element of $T\mathcal T^\conf(T^2)$, is again harmonic and has the expression
   \begin{equation}\label{eq:dotnu}
   \dot{\nu} = \frac{1}{2} \dot{J} J
   \end{equation}
\end{remark}

 We are now ready to conclude the proof of Theorem \ref{thm:cotangent} in genus one. 
 Considering $\mathcal{MGH}(T^2)$ as the complement of the zero section in $T^*\mathcal J(\R^2)$, the map $\mathcal{F}:\mathcal{MGH}(T^2) \rightarrow T^{*}\mathcal{T}^{\conf}(T^{2})$ given by Theorem \ref{thm:cotangent_parameterization} is nothing but the restriction of the map that we used in Lemma \ref{lemma torus identifications}, which we still denote as  $\mathcal{F}:T^{*}\mathcal{J}(\R^{2}) \rightarrow T^{*}\mathcal{T}^{\conf}(T^{2})$ by a little abuse of notation.

\begin{reptheorem}{thm:cotangent}[genus one]
 We have $$\mathcal F^*(\mathcal I_{T^*\mathcal{T}(T^2)},\Omega^\C_{T^*\mathcal{T}(T^2)})=\left(-\i,-\frac{i}{2}\overline\omega_\i^\C\right)~,$$
where $\mathcal I_{T^*\mathcal{T}(T^2)}$ denotes the complex structure of $T^*\mathcal{T}(T^2)$ and $\Omega^\C_{T^*\mathcal{T}(T^2)}$ its complex symplectic form.
\end{reptheorem} 

Before the proof, we recall that (in any genus) the cotangent bundle of $\mathcal{T}^{\conf}(\Sigma)$ has a natural complex symplectic form $\Omega^{\C}_{T^{*}\mathcal{T}(\Sigma)}$ defined as $-d\lambda^{\C}_{T^{*}\mathcal{T}(\Sigma)}$, where $\lambda^{\C}_{T^{*}\mathcal{T}(\Sigma)}$ is the Liouville form. Given a point $([J],\phi)\in T^{*}\mathcal{T}^{\conf}(\Sigma)$, the pairing between holomorphic quadratic differentials and tangent vectors expressed as classes of Beltrami differentials $[\dot\nu]$ is the following:
 \begin{equation}\label{eq:pairing_Beltrami}
   \scal{\phi}{\dot{\nu}}_{[J]} = \int_\Sigma \phi \bullet \dot{\nu} 
    \end{equation}
    where 
     \begin{equation}\label{eq:pairing2}
 (\phi \bullet \dot{\nu})(v,w) \defin \frac{1}{2 i} (\phi(\dot{\nu}(v), w) - \phi(\dot{\nu}(w), v))
    \end{equation}
    (see e. g. \cite[Section 2.1]{bonsante2015a_cyclic}).
The complex Liouville form is then simply expressed as:
\[
    (\lambda^\C_{T^* \Teich^\conf})_{([J], \phi)}([\dot{\nu}], \dot{\phi}) = \scal{\phi}{\dot{\nu}}_{[J]} 
\]
We now prove Theorem \ref{thm:cotangent} in the torus case.

\begin{proof}[Proof of Theorem \ref{thm:cotangent}, genus $g=1$]
    Let $(J_t)_t$ be a smooth path in $\mathcal J(\R^2)$, with $J_{t=0} = J$. If $\pi:T^*\mathcal{T}^\conf(T^2)\to \mathcal{T}^\conf(T^2)$ is the projection, Equation \eqref{eq:dotnu} in Remark \ref{rmk:constant_hqd2} shows:
    \[
    \dd\pi\circ \dd{\mathcal F}_{(J,\sigma)}(\dot{J}, \dot{\sigma}) =  \frac{1}{2} \dot{J} J =:\dot \nu~.    \]
    Recall  that $\mathcal F(J,\sigma)=([J],\phi)$, where $\phi = \sigma - i \, \sigma(\cdot, J \cdot)$ is the holomorphic quadratic differential whose real part is $\sigma$. Let now $g_J$ be the bilinear form $\rho(\cdot, J \cdot)$, and let $\{e_1, e_2 = J e_1\}$ be a $g_J$-\hsk orthonormal basis of $\R^{2}$. Since all the quantities that we consider lift to constant tensors on $\R^2$, we can get rid of the integral over $T^2\cong\R^2/\Z^2$ in Equation \eqref{eq:pairing_Beltrami}, and we find:
    \begin{align*}
        \scal{\phi}{\dot{\nu}}_{[J]} &= \frac{1}{2 i} \, (\phi(\dot{\nu}(e_1), e_2) - \phi(\dot{\nu}(e_2), e_1)) \\
        & = \frac{1}{4 i} \, (\sigma(\dot{J} J e_1, e_2) - i \, \sigma(\dot{J} J e_1, J e_2) - \sigma(\dot{J} J e_2, e_1) + i \, \sigma(\dot{J} J e_2, J e_1)) \\
        & = \frac{1}{4 i} \, (\sigma(\dot{J} e_2, e_2) + i \, \sigma(\dot{J} J e_1, e_1) + \sigma(\dot{J} e_1, e_1) + i \, \sigma(\dot{J} J e_2, e_2)) \\
        & = \frac{1}{4 i} \left( \tr(g_J^{-1} \sigma \dot{J}) + i \tr(g_J^{-1} \sigma \dot{J} J) \right) \\
        & = - \frac{i}{2} \left( \scal{\sigma}{\dot{J}}_J - i \scal{\sigma}{J\dot{J}}_J \right) \\
        & = - \frac{i}{2} \, \overline{\lambda_{(J,\sigma)}^\C(\dot{J}, \dot{\sigma})} .
    \end{align*}
    Therefore we have shown
    \[
    (\mathcal F^* \lambda^\C_{T^* \Teich(T^2)})_{(J, \sigma)}(\dot{J}, \dot{\sigma}) = - \frac{i}{2} \overline{\lambda_{(J,\sigma)}^\C(\dot{J}, \dot{\sigma})}.
    \]
    Taking differentials , we obtain
    \[
    \mathcal F^* \Omega_{T^* \Teich(T^2)}^\C =  -\frac{i}{2} (\omega_\mathbf{J} - i \, \omega_\mathbf{K}) = - \frac{i}{2} \overline\omega_\mathbf{I}^\C~.
    \]
 To show that $\mathcal F$ pulls-back the almost-complex structure of $T^*\mathcal T(T^2)$ to $- \i$, one can argue exactly as in the proof of Corollary \ref{cor:Iintegrable_torus}, observing that $\Omega_{T^* \Teich(T^2)}^\C$ is a complex symplectic form with respect to the complex structure of $T^* \Teich(T^2)$, while $\overline\omega_\mathbf{I}^\C$ is a complex symplectic form with respect to the (almost)-complex structure $-\i$.
\end{proof}


\subsection{A formal Mess homeomorphism}\label{subsec:Mess_homeo}
In this section we interpret the para-complex structure $\j$ and the para-complex sympletic form $\omega_\j^{\mathbb B}$ as the pull-back of natural structures on $\mathcal T(T^2)\times \mathcal T(T^2)$ via a map
$$\mathcal M:T^*\mathcal{J}(\R^2)\to \mathcal{T}(T^2)\times \mathcal{T}(T^2)~,$$
which is formally defined essentially as the Mess homeomorphism (Section \ref{subsec:parameterizations}), although in the genus one case this map does not have the same geometric significance from the Anti-de Sitter viewpoint as for the higher genus case. Nevertheless, this ``toy model'' is essential for the higher genus case, since the map $\mathcal M$ studied here will then induce Mess homeomorphism for genus $\geq 2$. 

Inspired by Section \ref{subsec:change_coord}, given a pair $(J,\sigma)\in T^*\mathcal J(\R^2)$, we define $h = h(J,\sigma)$ the Riemannian metric
\[
h \defin (1 + f(\norm{\sigma})) g_J = \left( 1 + \sqrt{1 + \norm{\sigma}^2} \right) g_J ~,
\]
where we recall that $g_J=\rho(\cdot,J\cdot)$ and $\norm{\cdot}=\norm{\cdot}_J$. We also set $B \defin h^{-1} \sigma$.
Exactly as in Lemma \ref{lem:det_B}, we have the identities
\begin{equation} \label{eq:det_B2}
\det B = - \frac{\norm{\sigma}^2}{(1 + f(\norm{\sigma}))^2} \qquad\text{ and }\qquad
1 + \det B = \frac{2}{1 + f(\norm{\sigma})} ~.
\end{equation}

The second identity shows that $\1 \mp J B$ is invertible for every $(J,\sigma) \in T^* \mathcal{J}(\R^2)$. Indeed, $\det(\1\mp JB)=1+\det B$, since $J B$ is traceless. We can thus define
\begin{gather*}
\mathcal{M} \vcentcolon  T^* \mathcal{J}(\R^2)  \longrightarrow  \mathcal{J}(\R^2) \times \mathcal{J}(\R^2) \\
\mathcal{M}(J, \sigma) \defin \left( (\1 - J B)^{-1} J (\1 - J B) , (\1 + J B)^{-1} J (\1 + J B) \right) \ .
\end{gather*}

\begin{remark}
We remark that this is formally the analogue of the expression of Mess homeomorphism in terms of almost-complex structures, given in Lemma \ref{lem:mess_complex_str}. In particular, the left and right components of $\mathcal M$ are the linear complex structures associated to the metrics $h((\1\mp JB)\cdot,(\1\mp JB)\cdot)$ on $\R^2$. As a reminder of this fact, we will often denote them by $J_l$ and $J_r$, respectively, to be consistent with the notation that will be used in Section \ref{sec:para_hyperkahler_str_on_MS}, where the higher genus case is discussed. 
\end{remark}

Now, the space $\mathcal{J}(\R^2)\times \mathcal{J}(\R^2)$ is naturally endowed with the almost para-\hsk complex structure $\mathcal{P}$ coming from its product structure. This is defined at any $(J,J') \in \mathcal{J}(\R^2)\times \mathcal{J}(\R^2)$ as follows:
\[
\mathcal{P}(\dot{J}, \dot{J}') \defin (\dot{J}, - \dot{J}') .
\]
Moreover, denoting by $\Omega_\WP$ the Weil-Petersson symplectic form of $\mathcal J(\R^2)$ (see also Remark \ref{rmk:compareOmegaWP}), and by $\pi_l,\pi_r:\mathcal{J}(\R^2)\times \mathcal{J}(\R^2)\to\mathcal{J}(\R^2)$ the projections to the left and right factor, $\mathcal{J}(\R^2)\times \mathcal{J}(\R^2)$ has two symplectic forms given by $\pi_l^*\Omega_\WP\pm \pi_r^*\Omega_\WP$. Together, they are combined into a para-complex symplectic form
$$\Omega^{\mathbb B}:=\frac{1}{2}(\pi_l^*\Omega_\WP+\pi_r^*\Omega_\WP)+\frac{\tau}{2} (\pi_l^*\Omega_\WP-\pi_r^*\Omega_\WP)~.$$
It is easily checked that $\Omega^{\mathbb B}$ is para-complex with respect to $\mathcal P$ in the sense that $\Omega^{\mathbb B}(\mathcal P\cdot,\cdot)=\Omega^{\mathbb B}(\cdot,\mathcal P\cdot)=\tau \Omega^{\mathbb B}(\cdot,\cdot)$.

\begin{reptheorem}{thm:mappaM}[baby version]
We have 
$$\mathcal M^*(\mathcal P,4 \Omega^{\mathbb B})=(\j, \omega_\j^{\mathbb B})~.$$
where $\mathcal P$ denotes the para-complex structure of $\mathcal{J}(\R^2)\times \mathcal{J}(\R^2)$ and $\Omega^{\mathbb B}$ its para-complex symplectic form.
\end{reptheorem}

\begin{proof}
    Since $B$ and $J B$ are traceless and $\det J = 1$, we have $B^2 = (J B)^2 = - \det B \ \1$. We will make use of this relation all along the current proof. A simple consequence is the following:
    \begin{equation}\label{eq:inverse_E-JB}
    (\1 \mp J B)^{-1} = \frac{1}{1 + \det B} (\1 \pm J B) .
    \end{equation}
    Applying this relation, we can develop the left and right complex structures $J_l$ and $J_r$ (the left and right components of $\mathcal{M}$, respectively) as follows:
    \begin{align*}
    J_{l, r} & = (\1 \mp J B)^{-1} J (\1 \mp J B) \\
    & = \frac{1}{1 + \det B} (\1 \pm J B) J (\1 \mp J B) \\
    & = \frac{1}{1 + \det B} (J \pm J B J \mp J^2 B - J B J^2 B) \\
    & = \frac{1 - \det B}{1 + \det B} J \pm \frac{2}{1 + \det B} B .
    \end{align*}
    In these relations and the ones that will follow the upper sign in $\pm$ or $\mp$ always refers to $J_l$ and the lower sign to $J_r$. From the second identity in \eqref{eq:det_B2}, we get:
    \begin{equation}\label{eq:det_f}
    \frac{2}{1 + \det B} = 1 + f,  \qquad \frac{1 - \det B}{1 + \det B} = f~,
    \end{equation}
    where $f = f(\norm{\sigma})$. Combining these relations with the development above, we obtain:
    \begin{equation} \label{eq:expression_mess_map}
    J_l = f \, J + g_J^{-1} \sigma , \qquad J_r = f \, J - g_J^{-1} \sigma .
    \end{equation}
    We are interested in computing the first order variation of these expressions along a direction $(\dot{J}, \dot{\sigma}) \in T_{(J,\sigma)} T^* \mathcal{J}(\Sigma)$. First we compute the derivative of the function $f = f(\norm{\sigma})$:
    \begin{align*}
        f' & = \left( \sqrt{1 + \norm{\sigma}^2} \right)' \\
        & = \frac{(\norm{\sigma}^2)'}{2 \sqrt{1 + \norm{\sigma}^2}} \\
        & = \frac{1}{2 f} \tr(g_J^{-1} \sigma (g_J^{-1} \sigma)') \\
        & = \frac{1}{2 f} \tr(g_J^{-1} \sigma (- g_J^{-1} \dot{g}_J g_J^{-1} \sigma + g_J^{-1} \dot{\sigma} ))
    \end{align*}
    As already observed in relation \eqref{eq:derivative_scal_prod}, the first order variation of $g_J$ with respect to $\dot{J}$ is equal to $- g_J(\cdot, J \dot{J} \cdot)$. Therefore Lemma \ref{lem:product_in_TJ} part \textit{ii)} implies that the trace of the product $g_J^{-1} \sigma g_J^{-1} \dot{g}_J g_J^{-1} \sigma$ vanishes. Moreover, in the term $\tr(g_J^{-1} \sigma g_J^{-1} \dot{\sigma})$ there is no contribution from the trace part of $\dot{\sigma}$. In conclusion, we deduce that $f' = \frac{\scall{\sigma}{\dot{\sigma}_0}}{f}$. For convenience, we also set
    \[
    Q^\pm = Q^\pm(\dot{J}, \dot{\sigma}) \defin \frac{1}{f} g_J^{-1} \dot{\sigma}_0 \pm \dot{J} .
    \]
    Then we have
    \begin{align*}
    \dot{J}_{l,r} & = (f \, J \pm g_J^{-1} \sigma)' \\ 
    & = \pm f \, Q^\pm \pm \frac{1}{2} \tr_{g_J} \, \dot{\sigma} \, \1 \mp g_J^{-1} \dot{g}_J g_J^{-1} \sigma + f' \, J \\
    & = \pm f \, Q^\pm \mp \scal{\sigma}{J \dot{J}} \, \1 \pm J \dot{J} g_J^{-1} \sigma + \frac{\scall{\sigma}{\dot{\sigma}_0}}{f} \, J \tag{Lemma \ref{lem:characterization_tangent_space} and eq. \eqref{eq:derivative_scal_prod}} \\
    & = \pm f \, Q^\pm \pm \scal{\sigma}{\dot{J}} \, J + \frac{\scal{\sigma}{g_J^{-1} \dot{\sigma}_0}}{f} \, J \tag{Lemma \ref{lem:product_in_TJ}} \\
    & = \pm f \, Q^\pm + \scal{\sigma}{Q^\pm} \, J ,
    \end{align*}
    where $\scall{\cdot}{\cdot} = \scall{\cdot}{\cdot}_J$ and $\scal{\cdot}{\cdot} = \scal{\cdot}{\cdot}_J$. Then the differential of the map $\mathcal{M}$ can be expressed as follows:
    \[
    \dd{\mathcal{M}}_{(J,\sigma)}(\dot{J}, \dot{\sigma}) = (\scal{\sigma}{Q^+} \, J + f \, Q^+, \scal{\sigma}{Q^-} \, J - f \, Q^-) .
    \]
   Let us now determine the pull-\hsk back of the forms $\pi_l^*\Omega \pm \pi_r^*\Omega$ by the map $\mathcal{M}$, where $\Omega$ was defined in Remark \ref{rmk:compareH2}. To emphasize the dependence on $(J, \sigma)$, we will now write $J_l = (\pi_l \circ \mathcal{M})(J,\sigma)$ and $J_r = (\pi_r \circ \mathcal{M})(J,\sigma)$. Given $(\dot{J}, \dot{\sigma})$ and $(\dot{J}', \dot{\sigma}')$ two tangent vectors at $(J,\sigma)$, we also set $Q^\pm = Q^\pm(\dot{J}, \dot{\sigma})$ and $R^\pm = Q^\pm(\dot{J}', \dot{\sigma}')$, to simplify the notation. Then, making use of the fact that $\tr(Q^\pm) = \tr(R^\pm) = 0$ and of Lemma \ref{lem:product_in_TJ}, we deduce that
\begin{align*}
    (\pi_{l,r} \circ \mathcal{M})^* \Omega((\dot{J}, \dot{\sigma}), (\dot{J}', \dot{\sigma}')) & = - \scall{\dd(\pi_{l,r} \circ \mathcal{M})(\dot{J}, \dot{\sigma})}{J_{l, r} \cdot \dd(\pi_{l,r} \circ \mathcal{M})(\dot{J}', \dot{\sigma}')}_{J_{l,r}} \\
    & = - \frac{1}{2} \tr((\scal{\sigma}{Q^\pm} \, J \pm f \, Q^\pm)(f J \pm g_J^{-1} \sigma )(\scal{\sigma}{R^\pm} \, J \pm f \, R^\pm)) \\
    & = -\frac{f}{2}\tr(\scal{\sigma}{Q^{\pm}} \ Jg_{J}^{-1}\sigma R^\pm +\scal{\sigma}{R^\pm} \ Q^{\pm} g_{J}^{-1}\sigma J) + \\
    & \qquad \qquad \qquad \qquad \qquad \qquad \qquad \qquad \qquad \qquad - \frac{f^{3}}{2}\tr(Q^{\pm} J R^{\pm})  \\
    & = -f \left(\scall{g_{J}^{-1}\sigma}{Q^{\pm}} \scall{Jg_{J}^{-1}\sigma}{R^{\pm}}- \scall{g_{J}^{-1}\sigma}{R^{\pm}}\scall{Q^{\pm}}{Jg_{J}^{-1}\sigma}\right) + \\
    & \qquad \qquad \qquad \qquad \qquad \qquad \qquad \qquad \qquad \qquad - f^{3}\scall{Q^{\pm}}{JR^{\pm}} 
\end{align*}
Now, because $Q^\pm, R^\pm \in T_{J} \mathcal{J}(\R^{2})$, assuming $\sigma \neq 0$, we can write
\[
    Q^{\pm}=\frac{1}{\norm{\sigma}}\left(\scall{g_{J}^{-1}\sigma}{Q^{\pm}} \ g_{J}^{-1}\sigma+\scall{Jg_{J}^{-1}\sigma}{Q^{\pm}} \ Jg_{J}^{-1}\sigma\right) ,
\]
and similarly for $R^\pm$, so that the term
\begin{align}\label{eq:diff}
    &\scall{g_{J}^{-1}\sigma}{Q^{\pm}}\scall{Jg_{J}^{-1}\sigma}{R^{\pm}}-\scall{g_{J}^{-1}\sigma}{R^{\pm}}\scall{Q^{\pm}}{Jg_{J}^{-1}\sigma} 
\end{align}
coincides with $-\norm{\sigma}^{2}\scall{Q^{\pm}}{JR^{\pm}}$ by direct computation. Note that when $\sigma=0$, the expression \eqref{eq:diff} vanishes. In any case, recalling that $Q^\pm = Q^\pm(\dot{J}, \dot{\sigma})$ and $R^\pm = Q^\pm(\dot{J}', \dot{\sigma}')$, we can thus conclude that
\begin{align*}
    (\pi_{l,r} \circ \mathcal{M})^* \Omega((\dot{J}, \dot{\sigma}), (\dot{J}', \dot{\sigma}')) & = -f \scall{Q^{\pm}}{JR^{\pm}}(-\norm{\sigma}^{2}+f^{2}) \\
    & = - \frac{f}{2} \tr(Q^\pm \, J \, R^\pm) \tag{$f^2 = 1 + \norm{\sigma}^2$} \\
    & = - f \scall{\dot{J}}{J \dot{J}'}_J + \frac{1}{f} \scall{\dot{\sigma}_0}{\dot{\sigma}_0'(\cdot, J \cdot)}_J
    \mp (\scal{\dot{\sigma}_0}{J \dot{J}'}_J - \scal{\dot{\sigma}_0'}{J \dot{J}}_J) \\
    & = (\omega_\mathbf{I} \pm \omega_\mathbf{K})((\dot{J}, \dot{\sigma}), (\dot{J}', \dot{\sigma}')) \ . 
\end{align*}
    Therefore we have
    \[
    \mathcal{M}^*(\pi_l^*\Omega+\pi_r^*\Omega) = 2 \omega_\mathbf{I} , \qquad \mathcal{M}^*(\pi_l^*\Omega-\pi_r^*\Omega) = 2 \omega_\mathbf{K} .
    \]
By what observed in Remark \ref{rmk:compareOmegaWP}, this proves that $\mathcal M^*(4 \Omega^{\mathbb B})=\omega_\j^{\mathbb B}$.
    
     One can then check directly that $\dd{\mathcal{M}}_{(J,\sigma)} \circ \mathbf{J}=\mathcal{P} \circ \dd{\mathcal{M}}_{(J,\sigma)}$, using that  $Q^\pm \circ \mathbf{J}(\dot{J}, \dot{\sigma}) = \pm Q^\pm(\dot{J}, \dot{\sigma})$. However, the fact that $\mathcal M^*\mathcal P=\j$ also follows immediately by the same trick as in Corollary \ref{cor:Iintegrable_torus}. Indeed, since $\Omega^{\mathbb B}$ is para-complex with respect to $\mathcal P$, if we denote $\hat\j=\mathcal M^*\mathcal P$, then $\omega_\j^\mathbb B$ is para-complex with respect to $\hat\j$, which means that $\omega_\k(\cdot,\cdot)=\omega_\i(\cdot,\hat\j\cdot)$. But since the same holds for $\j$, we deduce that $\hat\j=\j$.
\end{proof}

As an immediate consequence, we  obtain:

\begin{corollary} \label{cor:mess_paraholo}
 The 2-forms $\omega_\i$ and $\omega_\k$ on $T^*\mathcal J(\R^2)$ are closed, and $\j$ is an integrable almost para-complex structure.
\end{corollary}

\subsection{The circle action}\label{subsec:circle_action_toy}
Let us now study the behavior of a natural circle action on $T^{*}\mathcal{J}(\R^{2})$ with respect to the para-hyperK\"ahler structure $(\g,\i,\j,\k)$. 
Recall that a cotangent vector $\sigma \in T^{*}_{J}\mathcal{J}(\R^{2})$ can be seen as the real part of a complex-valued, $J$-complex linear symmetric form $\phi$. By Remark \ref{rmk:rotateB}, multiplication by $e^{i\theta}$ on complex bi-linear forms induces a circle action on $T^{*}\mathcal{J}(\R^{2})$ given by
\[
    e^{i\theta}\cdot (J,\sigma)=(J,\cos(\theta)\sigma+\sin(\theta)\sigma(\cdot, J\cdot)) \ .
\]
We will denote by $R_{\theta}:T^{*}\mathcal{J}(\R^{2})\to T^{*}\mathcal{J}(\R^{2})$ the action of $e^{i\theta}$. Observe that clearly $R_\theta$ preserves the zero section, hence it induces a circle action on $\mathcal{MGH}(T^2)$, which is identified to the complement of the zero section in $T^{*}\mathcal{J}(\R^{2})$ as in the proof of the genus one version of  Theorem \ref{thm:parahyper_structure} in Section \ref{subsec:identificationT2}.

\begin{reptheorem}{thm:circle}[genus one]
The circle action on $T^{*}\mathcal{J}(\R^{2})$ is Hamiltonian with respect to $\omega_\i$, generated by the function $H(J,\sigma)=f(\|\sigma\|)$ and satisfies 
$$R_\theta^*\g=\g\qquad R_\theta^*\omega_{\i}=\omega_{\i}\qquad R_\theta^*\omega_{\i}^\C=e^{-i\theta}\omega_{\i}^\C~.$$
\end{reptheorem}

\begin{proof} The infinitesimal generator of the circle action is
\[
    V_{\theta}(J,\sigma)=\left.  \frac{d}{d\theta} (J,\cos(\theta)\sigma+\sin(\theta)\sigma(\cdot, J\cdot)) \right|_{\theta = 0}=(0, \sigma(\cdot, J\cdot)) \ .
\]
We now compute
\begin{align*}
    \iota_{V_{\theta}}\omega_{\i}(\dot{J}, \dot{\sigma})&=-\omega_{\i}((\dot{J}, \dot{\sigma}), (0, \sigma(\cdot, J\cdot)))=-\g((\dot{J}, \dot{\sigma}), \i(0,\sigma(\cdot, J\cdot)))\\
    &=\g((\dot{J}, \dot{\sigma}), (0, \sigma(\cdot, J^{2}\cdot)))=-\g((\dot{J},\dot{\sigma}),(0,\sigma))\\
    &=\frac{1}{f(\|\sigma\|)}\scall{\dot{\sigma}_0}{\sigma}=\frac{1}{f(\|\sigma\|)}\scall{\dot{\sigma}}{\sigma}=dH_{(J,\sigma)}(\dot{J},\dot{\sigma}) \ ,
\end{align*}
which proves the first statement. This immediately implies that $R_\theta^*\omega_\i=\omega_\i$, by Cartan's magic formula.

We now compute $R_{\theta}^{*}\omega^{\C}$. Let us first find an expression for the differential of $R_{\theta}$. By definition, we have
\begin{equation}\label{eq:dR1}
    (dR_{\theta})_{(J,\sigma)}(\dot{J},\cos(\theta)\dot{\sigma}) =(\dot{J}, \cos(\theta)\dot{\sigma}+\sin(\theta)\dot{\sigma}(\cdot, J\cdot)+\sin(\theta)\sigma(\cdot, \dot{J}\cdot)) \ .
\end{equation}
By Lemma \ref{lem:characterization_tangent_space}, we can write
\[
    \dot{\sigma}=\dot{\sigma}_{0}-\scal{\sigma}{J\dot{J}}_{J}g_{J}
\]
so that
\begin{equation}\label{eq:dR2}
    \dot{\sigma}(\cdot, J\cdot)=\dot{\sigma}_{0}(\cdot, J\cdot)+\scal{\sigma}{J\dot{J}}_{J}\rho \ . 
\end{equation}
Moreover, by Lemma \ref{lem:product_in_TJ}, we have
\begin{equation*}
    g_{J}^{-1}\sigma \dot{J}= \scall{g_{J}^{-1}\sigma}{\dot{J}}_{J}\1-\scall{Jg_{J}^{-1}\sigma}{\dot{J}}_{J}J
\end{equation*}
which implies that
\begin{equation}\label{eq:dR3}
    \sigma(\cdot, \dot{J}\cdot)= \scal{\sigma}{\dot{J}}_{J}g_{J}-\scal{\sigma}{J\dot{J}}_{J}\rho \ .
\end{equation}
Combining \eqref{eq:dR2} and \eqref{eq:dR3} with \eqref{eq:dR1}, we obtain that the differential of $R_{\theta}$ can be expressed as follows
\[
    (dR_{\theta})_{(J,\sigma)}(\dot{J},\dot{\sigma})=(\dot{J}, \cos(\theta)\dot{\sigma}+\sin(\theta)\dot{\sigma}_{0}(\cdot, J\cdot)+\sin(\theta)\scal{\sigma}{\dot{J}}g_{J}) \ .
\]
Hence, 
\begin{align*}
    (R_{\theta}^{*}\omega^\C)((\dot{J}, \dot{\sigma}), (\dot{J}', \dot{\sigma}')) &= \scal{\cos(\theta)\dot{\sigma}_0'+\sin(\theta)\dot{\sigma}_{0}'J}{\dot{J}} - \scal{\cos(\theta)\dot{\sigma}_0+\sin(\theta)\dot{\sigma}_{0}J}{\dot{J}'}\\
    & + i(\scal{\cos(\theta)\dot{\sigma}_0'+\sin(\theta)\dot{\sigma}_{0}'J}{J \dot{J}} - \scal{\cos(\theta)\dot{\sigma}_0+\sin(\theta)\dot{\sigma}_{0}J}{J \dot{J}'}) \\
    &=\cos(\theta)(\scal{\dot{\sigma}_0'}{\dot{J}}-\scal{\dot{\sigma}_0}{\dot{J}'})+\sin(\theta)(\scal{\dot{\sigma}_{0}'}{J\dot{J}}-\scal{\dot{\sigma}_{0}}{J\dot{J}'}\\
    & + i(\cos(\theta)(\scal{\dot{\sigma}_0'}{J\dot{J}}-\scal{\dot{\sigma}_0}{J\dot{J}'})-\sin(\theta)(\scal{\dot{\sigma}_{0}'}{ \dot{J}}-\scal{\dot{\sigma}_{0}}{\dot{J}'})\\
    &=(\cos(\theta)\omega_{\j}+\sin(\theta)\omega_{\k})+i(\cos(\theta)\omega_{\k}-\sin(\theta)\omega_{\j})\\
    &=e^{-i\theta}\omega^{\C}((\dot{J}, \dot{\sigma}), (\dot{J}', \dot{\sigma}')) \ . 
\end{align*}

Let us finally check that $R_\theta$ preserves the metric $\g$. We denote by 
$$(r_{\theta})_{(J,\sigma)}(\dot{J}, \dot{\sigma})=\cos(\theta)\dot{\sigma}+\sin(\theta)\dot{\sigma}_{0}(\cdot, J\cdot)+\sin(\theta)\scal{\sigma}{\dot{J}}g_{J}$$ the second component of the differential of $R_{\theta}$ at the point $(J, \sigma)$. Moreover, we remark that $\|\sigma\|_{J}=\| \cos(\theta)\sigma+\sin(\theta)\sigma(\cdot, J\cdot)\|_{J}$, in other words the circle action preserves the scalar product on $T^{*}_{J}\mathcal{J}(\R^{2})$. 
We can now compute
\begin{align*}
    (R_{\theta}^{*} \g)_{(J, \sigma)}&((\dot{J}, \dot{\sigma}), (\dot{J}', \dot{\sigma}'))= \g_{R_{\theta}(J,\sigma)}((\dot{J}, (r_{\theta})_{(J,\sigma)}(\dot{J},\dot{\sigma})), (\dot{J}', (r_{\theta})_{(J,\sigma)}(\dot{J}',\dot{\sigma}'))) \\
    &=f(\| \sigma \|_{J})\scall{\dot{J}}{\dot{J}'}_{J}-\frac{1}{f(\|\sigma\|_{J})}\scall{\cos(\theta)\dot{\sigma}_{0}+\sin(\theta)\dot{\sigma}_{0}(\cdot, J\cdot)}{\cos(\theta)\dot{\sigma}_{0}'+\sin(\theta)\dot{\sigma}_{0}'(\cdot, J\cdot)}_{J} \\
    &=f(\| \sigma \|_{J})\scall{\dot{J}}{\dot{J}'}_{J}-\frac{1}{f(\|\sigma\|_{J})}\scall{\dot{\sigma}_{0}}{\dot{\sigma}_{0}'}_{J}\\
    &=\g_{(J,\sigma)}((\dot{J}, \dot{\sigma}), (\dot{J}', \dot{\sigma}')) \ .
\end{align*}
This concludes the proof.
\end{proof}

We then immediately obtain the relations of \eqref{eq:pullbackijk} for the pull-back of $\i,\j,\k$, which we re-write here for completeness. 
\begin{equation}\label{eq:pullbackijk2}
R_\theta^*\i=\i\qquad R_\theta^*\j=\cos(\theta)\j+\sin(\theta)\k \qquad R_\theta^*\k=-\sin(\theta)\j+\cos(\theta)\k~.
\end{equation}

We now turn to the proof of the genus one case of Theorem \ref{thm:potentialintro}, concerning a para-K\"ahler potential for $\g$. We recall that a smooth function $f$ is a para-K\"ahler potential for a para-K\"ahler structure $(\g,\p)$ if $\omega_\p=(\tau/2)\overline\partial_\p\partial_\p f$. We will make use of the following identity, whose proof is similar to the K\"ahler case and is given in Appendix \ref{appendixB}, Lemma \ref{lemma:deldelbar}.
\begin{equation}\label{eq:deldelbar}
2\tau\bar{\partial}_{\p}\partial_{\p}f=d(df\circ \p)
\end{equation}

\begin{reptheorem}{thm:potentialintro}[genus one]
The function $-4H$ is a para-K\"ahler potential for  the para-K\"ahler structures $(\g,\j)$ and $(\g,\k)$ on $T^*\mathcal J(\R^2)$.
\end{reptheorem}

\begin{proof}  From the previous result we have:
\begin{align*}
    \mathcal{L}_{V_{\theta}}(\omega_{\j}+i\omega_{\k})=\left.\frac{d}{d\theta}R_{\theta}^{*}(\omega_{\j}+i\omega_{\k})\right|_{\theta=0}=\left.\frac{d}{d\theta}R_{\theta}^{*}\omega^{\C}_\i\right|_{\theta=0}=-i\omega_\i^{\C}=-i\omega_{\j}+\omega_{\k}
\end{align*}
which implies
\[
    \mathcal{L}_{V_{\theta}}\omega_{\j}=\omega_{\k} \ \ \ \text{and} \ \ \ \mathcal{L}_{V_{\theta}}\omega_{\k}=-\omega_{\j} \ .
\]
Therefore,
\begin{align*}
    2\tau \bar{\partial}_{\j}\partial_{\j}(-H)=-d(dH\circ \j)=-d(\iota_{V_{\theta}}\omega_{\i}(\j\cdot))=-d(\iota_{V_{\theta}}\omega_{\k})=-\mathcal{L}_{V_{\theta}}\omega_{\k}=\omega_{\j}
\end{align*}
and similarly,
\begin{align*}
    2\tau \bar{\partial}_{\k}\partial_{\k}(-H)=-d(dH\circ \k)=-d(\iota_{V_{\theta}}\omega_{\i}(\k\cdot))=d(\iota_{V_{\theta}}\omega_{\j})=\mathcal{L}_{V_{\theta}}\omega_{\j}=\omega_{\k} \ ,
\end{align*}
which shows that $(-4H)$ is a para-K\"ahler potential for $\g$ with respect to $\j$ and $\k$.
\end{proof}

\subsection{A one-parameter family of maps}\label{sec:constant_curv}

Using the circle action, we can define for every $\theta \in S^{1}$ the map 
\[
    \mathcal{C}_{\theta}=\mathcal{C} \circ R_{\theta}: T^{*}\mathcal{J}(\R^{2})\rightarrow \mathcal{J}(\R^2)\times \mathcal{J}(\R^2) \ .
\]

\begin{remark}
Using Remark \ref{rmk:rotateB}, we observe that for $\theta=-\pi/2$ the map $\mathcal C=\mathcal C_{-\pi/2}$ has the expression 
$$
\mathcal{C}(J, \sigma) \defin \left( (\1 -  B)^{-1} J (\1 -  B) , (\1 +  B)^{-1} J (\1 +  B) \right)
$$
and is therefore a formal analogue of the parameterization 
$$\mathcal C:\mathcal{MGH}(\Sigma)\to \mathcal{T}(\Sigma)\times \mathcal{T}(\Sigma)~,$$
of the deformation space $\mathcal{MGH}(\Sigma)$ by means of the induced metric on the surfaces of constant curvature $-2$. 
\end{remark}

\begin{remark}
Exactly as in Lemma \ref{lemma:mapH}, identifying $\mathcal J(\R^2)$ with $\mathcal T^\conf (T^2)$ (Lemma \ref{lemma torus identifications}), we obtain a map
$$T^*\mathcal J(\R^2)\to \mathcal T^\conf(T^2)\times \mathcal T^\conf(T^2)$$
which can be interpreted as the map
$$\mathcal H_\theta(J,q)=(h_{(J,-e^{i\theta}q)},h_{(J,e^{i\theta}q)})~,$$
where $h(J,q)$ denotes the unique complex structure on $T^2$ such that the identity is harmonic with respect to the flat metric $g_J$ on $T^2\cong \R^2/\Z^2$ on the source and to $h(J,q)$ on the target. 
\end{remark}

Now, using the genus one versions of Theorem \ref{thm:mappaM} and Theorem \ref{thm:circle} (see also \eqref{eq:pullbackijk2}), which have been proved above in this section, and the identity $\mathcal{C}_{\theta}=\mathcal{M}\circ R_{\theta+\frac{\pi}{2}}$, which follows from \eqref{eq:CMrotate}, we see that
$$\mathcal C_\theta^* (\mathcal P,4 \Omega^{\mathbb B})=(-\sin(\theta)\j+\cos(\theta)\k ,\omega_\i-\tau(\cos(\theta)\omega_\j+\sin(\theta)\omega_\k))~.$$
As an immediate consequence, we have the following ``baby versions'' of Theorems 
\ref{thm:mappaC} and \ref{thm:Htheta}.

\begin{reptheorem}{thm:mappaC}[baby version]
We have
$$\mathcal C^*(\mathcal P,4\Omega^{\B})=(\k,\omega_\k^{\mathbb{B}})~,$$
where $\mathcal P$ denotes the para-complex structure of $\mathcal{J}(\R^2)\times \mathcal{J}(\R^2)$ and $\Omega^{\mathbb B}$ its para-complex symplectic form.
\end{reptheorem}

\begin{reptheorem}{thm:Htheta}[baby version]
We have
$$\Im \mathcal H_\theta^*(2\Omega^{\mathbb B})=-\Re (ie^{i\theta}\Omega^\C_{T^*\mathcal J(\R^2)})~.$$
where  $\Omega^{\mathbb B}$ is the para-complex symplectic form of $\mathcal{J}(\R^2)\times \mathcal{J}(\R^2)$.
\end{reptheorem}

As a consequence, we obtain, among other things, a direct proof of the integrability of $\k$, thus completing Theorem \ref{thm:parahyper_structure_toy}.

\begin{corollary} \label{cor:cc_paraholo}
 The 2-forms $\omega_\i$ and $\omega_\j$ on $T^*\mathcal J(\R^2)$ are closed, and $\k$ is an integrable almost para-complex structure.
\end{corollary}

\section{The general case: genus \texorpdfstring{$\geq 2$}{g>=2}} \label{sec:para_hyperkahler_str_on_MS}

In the next two sections we give a proof of Theorem \ref{thm:parahyper_structure} in the general case of closed surfaces of genus $\geq 2$. In this section, we realize the deformation space of MGHC anti-de Sitter structures $\mathcal{MGH}(\Sigma)$ as the quotient by $\Symp_0(\Sigma, \rho)$ (the group of symplectomorphisms of $(\Sigma, \rho)$ isotopic to the identity, see Section \ref{subsec:change_coord}) of a set $\widetilde{\mathcal{MS}}_{0}(\Sigma,\rho)$ sitting inside an infinite dimensional space $T^{*}\mathcal{J}(\Sigma)$ that is formally endowed with a para-hyperK\"ahler structure $(\g, \i,\j,\k)$. We then give a distribution inside $T T^{*}\mathcal{J}(\Sigma)$ that is preserved by $\i$, $\j$, and $\k$ and maps isomorphically to the tangent space to $\mathcal{MS}_{0}(\Sigma,\rho)$, thus deducing that these structures descend to the quotient. To this aim, we characterize tangent vectors in several different ways (Proposition \ref{prop:equivalent_def_subspace_V}), that we prove are equivalent in Section \ref{subsec:proof_of_prop}. The proof of Theorem \ref{thm:parahyper_structure} is then completed in Section \ref{sec:geom_inter}, where we show that the induced symplectic forms are non-degenerate and closed, generalizing the constructions seen for the toy model in Section \ref{sec:toy_model_via_complex_structures}. 

\subsection{The group of (Hamiltonian) symplectomorphisms and its Lie algebra}
Let us fix a symplectic form $\rho$ on $\Sigma$. By Cartan's formula for every vector field $V$ on $\Sigma$ we have
\[
\Dlie_V \rho = \iota_V \dd{\rho} + \dd(\iota_V \rho) = \dd(\iota_V \rho) ,
\]
since $\dd{\rho} = 0$. Therefore, the flow of $V$ acts by symplectomorphisms on $(\Sigma, \rho)$ if and only if the $1$-\hsk form $\iota_V \rho$ is closed. Hence we can define the Lie algebra of the group $\Symp_0(\Sigma,\rho)$ of symplectomorphisms of $(\Sigma, \rho)$ isotopic to the identity as follows:
\[
\Lsymp(\Sigma,\rho) \defin \Lie(\Symp_0(\Sigma, \rho)) = \set{V \in \Gamma(T \Sigma) \mid \dd(\iota_V \rho) = 0 } \cong_\rho Z^1(\Sigma) ,
\]
where in the last step we are using the identification between $\Gamma(T \Sigma)$ and $\Lambda^1(\Sigma) = \Gamma(T^* \Sigma)$ induced by $\rho$, and $Z^{1}(\Sigma)$ denotes the space of smooth closed 1-forms.
A symplectomorphism $\psi$ is Hamiltonian if there is an isotopy $\psi_{\bullet}:[0,1] \rightarrow \Symp_{0}(\Sigma, \rho)$, with $\psi_{0}=\mathrm{id}$ and $\psi_{1}=\psi$, and a smooth family of functions $H_{t}:\Sigma\rightarrow \R$ such that $\iota_{V_{t}}\rho=dH_{t}$, where $V_{t}$ is the infinitesimal generator of the symplectomorphism $\psi_{t}$. We denote by $\Ham(\Sigma, \rho)$ the group of Hamiltonian symplectomorphisms of $(\Sigma, \rho)$. This is a normal subgroup of $\Symp(\Sigma, \rho)$ and its Lie algebra is defined as 
\[
\Lham(\Sigma, \rho) \defin \Lie(\Ham(\Sigma, \rho)) = \set{V \in \Gamma(T \Sigma) \mid \iota_V \rho \text{ exact}} \cong_\rho B^1(\Sigma) ,
\]
where $B^1(\Sigma)$ denote the space of smooth exact $1$-\hsk forms on $\Sigma$.

We have the following non-\hsk degenerate pairings:
\[
\begin{matrix}
\scal{\cdot}{\cdot}_\Lsymp \vcentcolon & \faktor{\Lambda^1(\Sigma)}{B^1(\Sigma)} \times Z^1(\Sigma) & \longrightarrow & \R \\
& ([\alpha], \beta) & \longmapsto & \int_\Sigma \alpha \wedge \beta ,
\end{matrix}
\]
\[
\begin{matrix}
\scal{\cdot}{\cdot}_\Lham \vcentcolon & \faktor{\Lambda^1(\Sigma)}{Z^1(\Sigma)} \times B^1(\Sigma) & \longrightarrow & \R \\
& ([\alpha], \beta) & \longmapsto & \int_\Sigma \alpha \wedge \beta .
\end{matrix}
\]
Since $Z^1(\Sigma) = \ker \dd$, the group $\Lambda^1(\Sigma) / Z^1(\Sigma)$ identifies with $B^2(\Sigma)$ through the differential map $\dd$. In particular we have
\begin{equation}\label{eq:pairing_lie}
\faktor{\Lambda^1(\Sigma)}{B^1(\Sigma)} \subset \mathfrak{S}(\Sigma,\rho)^* , \quad B^2(\Sigma) \cong_{\dd} \faktor{\Lambda^1(\Sigma)}{Z^1(\Sigma)} \subset \mathfrak{H}(\Sigma,\rho)^* .
\end{equation}

Observe that, for every tangent vector field $V$ and for every $1$-form $\alpha$, we have
\begin{equation}\label{eq:iota}
    \iota_V \alpha \, \rho = \alpha \wedge \iota_V \rho .
\end{equation}
In particular, if $V$ is a Hamiltonian vector field, with $\iota_V \rho = \dd{H}$, then
\begin{equation}\label{eq:Stokes}
\scal{[\alpha]}{\dd{H}}_\Lham = \int_\Sigma \alpha \wedge \dd{H} = \int_\Sigma H \dd{\alpha} =  \int_\Sigma \alpha(V) \, \rho  ,
\end{equation}
for every $H \in \mathscr{C}^\infty(\Sigma)$ and $[\alpha] \in \Lambda^1(\Sigma) / Z^1(\Sigma)$. \\

\subsection{The Teichm\"uller space as a symplectic quotient}\label{subsec:Teich_inf_dim_reduction}
Before treating the case of $\mathcal{MS}(\Sigma)$, we recall briefly how we can recover Teichm\"uller space as an infinite dimensional symplectic quotient. Most of the computations of this section can already be found in \cite{donaldson2003moment} and \cite{trautwein_thesis}: we report them here for reference purposes. \\

We denote by $P$ the $\SL(2,\R)$-principal bundle over $(\Sigma, \rho)$ whose fibers are linear maps $F:\R^{2}\rightarrow T_{p}\Sigma$ that identify the area form $\rho_{p}$ with the standard area form $\rho_{0}$ on $\R^{2}$ via pull-back. 
In other words, we require that $F^{*}\rho_{p}=dx\wedge dy$. The $\SL(2,\R)$-action is defined by $A\cdot (p,F)=(p,F\circ A^{-1})$.

Observe that any symplectomorphism $\phi$ of $(\Sigma, \rho)$ naturally lifts to a diffeomorphism $\hat{\varphi}$ of the total space $P$, by setting 
\[
\hat{\varphi}(p,F) \defin (\varphi(p), \dd{\varphi}_p \circ F) \in P ,
\]
for every $(p,F) \in P$. We now consider the bundle over $\Sigma$
\[
P(\mathcal{J}(\R^{2})) \defin P \times_{\SL(2,\R)} \mathcal{J}(\R^{2}) = \faktor{P \times \mathcal{J}(\R^{2})}{\SL(2,\R)} ,
\]
where $\SL(2,\R)$ acts diagonally on the two factors.  A section of $P(\mathcal{J}(\R^2)) \rightarrow \Sigma$ induces a complex structure $J$ on $\Sigma$ which is compatible with $\rho$, in the sense that $\rho(\cdot, J \cdot)$ is positive definite: this is defined on $T_p\Sigma$ by the endomorphism $F_p\circ J_p\circ F_p^{-1}$. Recalling that $\SL(2,\R)$ acts on $\mathcal{J}(\R^{2})$ by conjugation, one easily checks that this section $J$ is well-defined, that is, if two pairs $((p,F),J_p)$ and $((p,F'),J'_p)$ differ by the diagonal action of $\SL(2,\R)$, then they induce the same complex structure on $T_p\Sigma$. We will often confuse sections of $P(\mathcal{J}(\R^2)) \rightarrow \Sigma$ with complex structures $J$.

We will denote by $g_J$ the Riemannian metric $\rho(\cdot, J \cdot)$. By construction, the area form of $g_J$ is equal to $\rho$ for every complex structure $J$ as above. We set
\[
\mathcal{J}(\Sigma) \defin \Gamma(\Sigma, P(\mathcal{J}(\R^2))) .
\]
Given $J \in \mathcal{J}(\Sigma)$, a tangent vector $\dot{J} \in T_J \mathcal{J}(\Sigma)$ identifies with a section of the pull-back vector bundle $J^*(T^\vertical P(\mathcal{J}(\R^2))) \rightarrow \Sigma$, where $T^\vertical P(\mathcal{J}(\R^2))$ denotes the vertical subbundle of $T P(\mathcal{J}(\R^2))$ with respect to the projection over $\Sigma$. In other words, $\dot{J}$ is a section of $\End(T \Sigma)$ that satisfies $\dot{J} J + J \dot{J} = 0$. The space $\mathcal{J}(\Sigma)$ is formally an infinite dimensional symplectic manifold with symplectic form $\Omega_{J}$ given by
\[ 
    \Omega_J(\dot{J}, \dot{J}') = - \frac{1}{2} \int_{\Sigma} \tr(\dot{J} J \dot{J}') \rho \ .
\]

\begin{definition} \label{def:moment_map}
Let $(X, \omega)$ be a symplectic manifold, and assume that a Lie group $G$ acts on $(X, \omega)$ by symplectomorphisms. We say that the action is Hamiltonian if there exists a smooth function $\mappa{\mu}{X}{\mathfrak{g}^*}$ satisfying the following properties:
\begin{enumerate}[i)]
\item $\mu$ is $\Ad^*$-\hsk equivariant, i.e. for every $g \in G$ and $p \in X$ we have:
\[
\mu_{g \cdot p} = \Ad^*(g) (\mu_p) = \mu_p \circ \Ad(g^{-1}) \in \mathfrak{g}^* ;
\]
\item given $\xi \in \mathfrak{g}$, we denote by $V_\xi$ the vector field of $X$ generating the action of the $1$-\hsk parameter subgroup generated by $\xi$, i.e. $V_\xi (p) \defin \dv{t} \exp( t \xi) \cdot p |_{t = 0}$. Moreover, we set $\mu^\xi$ to be the function $p \mapsto \mu_p (\xi) \in \R$ on $X$. Then, for every $\xi \in \mathfrak{g}$ we have:
\[
\dd{\mu^\xi} = \iota_{V_\xi} \omega = \omega(V_\xi, \cdot) .
\]
\end{enumerate}
A map $\mu$ satisfying the properties above is called a \emph{moment map} for the action of $G$ on $(X, \omega)$.
\end{definition}

In the following, we denote by $K_J \in \mathscr{C}^\infty(\Sigma)$ the Gaussian curvature of the metric $g_J = \rho(\cdot, J \cdot)$. 

\begin{theorem}[\cite{donaldson2003moment}, \cite{trautwein_thesis}]\label{thm:momentmap_Teich}
Set $c \defin \frac{2 \pi \chi(\Sigma)}{\Area(\Sigma, \rho)}$. Then the function
\[
    \begin{matrix}
    \mu \vcentcolon & \mathcal{J}(\Sigma) & \longrightarrow & \Lham(\Sigma, \rho)^* \\
    & J & \longmapsto & -2 (K_J - c) \rho
    \end{matrix}
    \]
is a moment map for the action of $\Ham(\Sigma, \rho)$ on $(\mathcal{J}(\Sigma), \Omega)$.
\end{theorem}

Here, we are using the inclusion $B^{2}(\Sigma)\subset \Lham(\Sigma, \rho)^*$ introduced in \eqref{eq:pairing_lie}. By property \textrm{i)} in Definition \ref{def:moment_map}, the preimage $ \mu^{-1}(0)$ is invariant by the action of the Hamiltonian group $\Ham(\Sigma,\rho)$. Consequently, any variation $\dot{J}=\Dlie_{V}J$ induced by a Hamiltonian vector field $V$ lies in the kernel of $\dd{\mu}$. By property \textrm{ii)} in Definition \ref{def:moment_map}, for any $J \in \mu^{-1}(0)$ the space $\Ker (\dd{\mu}_J)$ coincides with the $\Omega_J$-\hsk orthogonal of the tangent space $T_{J}(\Ham(\Sigma, \rho)\cdot J)$ to the orbit of $J$ under the action of the Hamiltonian group. Therefore, there is a well-defined induced symplectic form on the quotient $\widetilde{\Teich}(\Sigma)=\mu^{-1}(0)/\Ham(\Sigma, \rho)$. The classical Teichm\"uller space $\mathcal{T}(\Sigma)$ can be identified (see \cite[\S 2.2]{donaldson2003moment}) with the further quotient of $\widetilde{\Teich}(\Sigma)$ by 
\[
    H:=\Symp_{0}(\Sigma,\rho)/\Ham(\Sigma,\rho) \ ,
\]
as briefly sketched at the end of Section \ref{subsec:change_coord}. Because the orbits of $H$ are symplectic sub-manifolds of $\widetilde{\mathcal{T}}(\Sigma)$ (see \cite[\S 2.2]{donaldson2003moment}, \cite[Lemma~4.4.8]{trautwein_thesis}), we can define a symplectic form on $\mathcal{T}(\Sigma)$ by setting
\begin{equation}\label{eq:fakeWP}
    \widehat{\Omega}_{[J]}([\dot{J}], [\dot{J}']):=\Omega_{J}(\dot{J}^{h}, \dot{J}'^{h})
\end{equation}
where $\dot{J}^{h} \in \Ker(\dd\mu)$ denotes a lift of $\dot{J}$ that is $\Omega_{J}$-orthogonal to the orbit of $\Symp_{0}(\Sigma, \rho)$. These lifts can be described by a differential geometric property. For this purpose, we introduce the notion of divergence of an endomorphism: given $A \in \End(T \Sigma)$ and given $G$ a Riemannian metric on $\Sigma$, we define $\divr_G A$ to be the $1$-\hsk form
\[
(\divr_G A)(X) \defin \sum_i G((\nabla^G_{e_i} A)X, e_i),
\]
where $(e_i)_i$ is a local $G$-\hsk orthonormal frame, $\nabla^G$ is the Levi-\hsk Civita connection of $G$ and $X$ is a vector field on $\Sigma$. We will also denote by $\divr_G V$ the usual divergence of a vector field $V$ with respect to the Riemannian metric $G$. 
Whenever we are dealing with a fixed almost complex structure $J$, we will omit the dependence of metric $g_J = \rho(\cdot, J \cdot)$ on $J$ and simply write $g$, in order to simplify the notation. In particular, if we write $\divr_g A$, it has to be interpreted as the divergence of the endomorphism $A$ with respect to $g_J$. Moreover, because $J$ is $\nabla^{g}$-parallel, i. e. $(\nabla^g_X J) Y = 0$ for every $X$, $Y$ tangent vector fields on $\Sigma$, we deduce
\begin{equation}\label{eq:div_JA}
	\divr_g(J A) = - \divr_g (A J) = - (\divr_g A) \circ J 
\end{equation}
for any $A \in \End(T\Sigma)$. Another immediate relation that we will use repeatedly is the following:
\begin{equation}\label{eq:div_contraction}
	\divr_g(X) = d(\iota_X\rho)(v,Jv) 
\end{equation}
for any unit vector $v$.

\begin{proposition}\label{prop:orthogonal_symp} Let $J$ be in $\mu^{-1}(0) \subset \mathcal{J}(\Sigma)$. An element $\dot{J} \in T_J \mathcal{J}(\Sigma)$ lies in the kernel of $\dd{\mu}$ if and only if $\divr_g \dot{J}$ is a closed $1$-\hsk form. Moreover, $\dot{J} \in \Ker(\dd \mu)$ is $\Omega_{J}$-orthogonal to $T_{J}(\Symp_{0}(\Sigma, \rho)\cdot J)$ inside $\Ker(\dd \mu)$ if and only if $\divr_g\dot{J}$ is an exact $1$-form.
\end{proposition}
\begin{proof} Let $V$ be a vector field on $\Sigma$. We observe that
\begin{equation} \label{eq:divergence_rel}
    \frac{1}{2} \tr(\dot{J} J \Dlie_V J) =  (\divr_g \dot{J})(V) - \divr_g(\dot{J} V) \ .
\end{equation}
To see this, first we notice that
\begin{align*}
   (\divr_g \dot{J})(V) & = \sum_i g((\nabla_{e_i} \dot{J})V, e_i) \\
   & = \sum_i g(\nabla_{e_i} (\dot{J} V) - \dot{J} \nabla_{e_i} V , e_i ) \\
   & = \divr_{g}(\dot{J}V) - \sum_i g(\dot{J}\nabla_{e_{i}}V, e_{i})\\
   & = \divr_{g}(\dot{JV})-\trace(\dot{J} A_V) ,
\end{align*}
where $\nabla$ denotes the Levi-Civita connection of $g$ and $A_V$ stands for the endomorphism $A_V (X) \defin \nabla_X V$. As shown in the proof of Lemma \ref{lem:I_of_lie_deriv} below, the endomorphism $\Dlie_V J$ can be expressed as $J A_V - A_V J$ (see relation \eqref{eq:lie_j}). In particular we have
\begin{align*}
    \trace(\dot{J} A_V) &= - \tr(\dot{J} J J A_V) \tag{$J^2 = - \1$} \\
    & = - \frac{1}{2} \left( \tr(\dot{J} J J A_V) - \tr(J \dot{J} J A_V) \right) \tag{$\dot{J} \in T_J \mathcal{J}(\Sigma)$} \\
    & = - \frac{1}{2} \left( \tr(\dot{J} J J A_V) - \tr(\dot{J} J A_V J) \right) \\
    &  = - \frac{1}{2} \tr(\dot{J} J \Dlie_V J) ,
\end{align*}
and so relation \eqref{eq:divergence_rel} follows. Now, applying such identity we find
\begin{align}\label{eq:Omega}
\begin{split}
    \Omega_{J}(\dot{J}, \Dlie_{V}J)&=
    - \frac{1}{2} \int_{\Sigma} \tr(\dot{J} J \Dlie_V J) \, \rho \\
    &= -\int_{\Sigma} \left( \frac{1}{2} \tr(\dot{J} J \Dlie_V J) + \divr_g(\dot{J} V) \right)\rho \\
    &=-\int_{\Sigma} (\divr_g \dot{J})(V) \, \rho \\
    &=-\int_{\Sigma}(\divr_g\dot{J}) \wedge \iota_{V}\rho \ .
\end{split}
\end{align}

Consider now $\dot{J}$ in kernel of the differential of the moment map $\dd \mu$. By property \textrm{ii)} in Definition \ref{def:moment_map}, we have $\Omega_J(\dot{J}, \Dlie_V J) = 0$ for every Hamiltonian vector field $V$. If  $H$ denotes the Hamiltonian function of $V$, then by relation \eqref{eq:Omega} we have
\[
0 = \int_{\Sigma} (\divr_g\dot{J}) \wedge \dd H = \int_{\Sigma} H \dd{(\divr_g\dot{J}) } ,
\]
where in the last step we applied Stokes' theorem on the $2$-\hsk form $\dd{(H (\divr_g\dot{J}))}$. Therefore, by letting the Hamiltonian function $H$ vary, we deduce that $\divr_g \dot{J}$ is a closed $1$-\hsk form.

Similarly, if $\divr_g\dot{J}=\dd f$ is exact and $V \in \Lsymp(\Sigma, \rho)$, then
\begin{align*}
    \dd f\wedge \iota_{V}\rho&=\dd f\wedge\iota_{V}\rho+ f\dd(\iota_{V}\rho) \tag{$\iota_{V}\rho$ is closed} =\dd(f\iota_{V}\rho)
\end{align*}
is exact and $\Omega_{J}(\dot{J}, \Dlie_{V}J) = 0$ for every $V \in \Lsymp(\Sigma, \rho)$. Vice versa, assume that $\dot{J}$ is in $\Ker \dd{\mu}$ and it satisfies $\Omega_{J}(\dot{J}, \Dlie_{V}J) = 0$ for every $V \in \Lsymp(\Sigma, \rho)$. Then, again by relation \eqref{eq:Omega}, 
\[
\int_\Sigma (\divr_g \dot{J}) \wedge \alpha = 0
\]
for every closed $1$-\hsk form $\alpha$. Since $\divr \dot{J}$ is closed, and since the pairing
\[
    (\alpha, \beta) \mapsto \int_{\Sigma} \alpha \wedge \beta 
\]
between closed $1$-forms is non-degenerate in $H^{1}(\Sigma) \times H^{1}(\Sigma)$, we deduce that $\divr_g \dot{J}$ represents the trivial class inside the first de Rham cohomology group or, in other words, that $\divr_g \dot{J}$ is exact.
\end{proof}

\begin{remark} \label{rmk:derivative_curvature}
The argument described in the proof of Proposition \ref{prop:orthogonal_symp} combined with Theorem \ref{thm:momentmap_Teich} provides us with a convenient way to express the first order variation of the curvature $K_J$ with respect to $\dot{J}$, that is
\[
\dd{K_J}(\dot{J}) \, \rho = \frac{1}{2} \dd(\divr_g \dot{J}) .
\]
In the following, we briefly see how to deduce this relation. On the one hand, by the explicit expression of the moment map $\mu$ from Theorem \ref{thm:momentmap_Teich}, we have that
\[
\scal{\dd{\mu}(\dot{J})}{V}_\Lham = - 2 \int_\Sigma  H \dd{K_J}(\dot{J}) \rho 
\]
for any Hamiltonian vector field $V$ with Hamiltonian function $H$. On the other hand, being $\mu$ a moment map for the action of the Hamiltonian group, it satisfies
\begin{align*}
    \scal{\dd{\mu}(\dot{J})}{V}_\Lham & = \Omega_J(\dot{J}, \Dlie_V J) \\
    & = -\int_{\Sigma}(\divr_g\dot{J}) \wedge \dd{H} \tag{relation \eqref{eq:Omega}} \\
    & = - \int_{\Sigma} H \dd(\divr_g\dot{J}) ,
\end{align*}
once again for any Hamiltonian vector field $V$ with Hamiltonian function $H$. Combining the relations above, we find that
\[
- 2 \int_\Sigma  H \dd{K_J}(\dot{J}) \rho = - \int_{\Sigma} H \dd(\divr_g\dot{J}) ,
\]
and by letting the Hamiltonian function vary, we deduce the desired relation.
\end{remark}

Classically, one further re-normalizes the lift $\dot{J}^{h}$ to be $L^{2}$-orthogonal to the tangent space to the orbit. This gives the additional condition that $\divr_g\dot{J}=0$ (\cite{trautwein_thesis}, \cite{tromba2012teichmuller}), which recovers the description of the tangent space to Teichm\"uller space via traceless Codazzi tensors and the formula of (a multiple of) the Weil-Petersson symplectic form, if we choose an area form $\rho$ on $\Sigma$ such that $\Area(\Sigma, \rho)=-2\pi \chi(\Sigma)$, which means that $c=-1$ in Theorem \ref{thm:momentmap_Teich}. 

\begin{lemma}[{\cite[Section 2.1]{bonsante2015a_cyclic}}] \label{lem:relation_Omega_WP} Let $\dot{J},\dot{J}' \in T_{[J]}\mathcal{T}^{\conf}(\Sigma)$ be traceless Codazzi tensors representing tangent vectors to Teichm\"uller space. Then, 
\[
    (\Omega_\WP)_{[J]}(\dot{J}, \dot{J}') = - \frac{1}{8} \int_\Sigma \tr(\dot{J} J \dot{J}') \dd{a} ,
\]
and the Weil-Petersson metric can be expressed as
\[
    (G_\WP)_{[J]}(\dot{J}, \dot{J}') =  \frac{1}{8} \int_\Sigma \tr(\dot{J} \dot{J}') \dd{a} 
\]
    where $\dd{a}$ is the volume form of the unique hyperbolic metric with conformal structure $J$.
\end{lemma}

\begin{remark}
    The change in the sign with respect to the relation appearing in \cite[\S 2.1]{bonsante2015a_cyclic} is due to the fact that here we are considering the complex structure $\dot{\nu} \mapsto - J \dot{\nu}$ on the space of Beltrami differentials, which is opposite to the one used by \cite{bonsante2015a_cyclic}.
\end{remark}

\begin{remark}\label{rmk:fakeWP} We remark, however, that any choice of a supplement $W$ of $T_{J}(\Symp_{0}(\Sigma, \rho)\cdot J)$ inside $\Ker(\dd\mu)$ that is $\Omega_{J}$-orthogonal to $T_{J}(\Symp(\Sigma, \rho)\cdot J)$ gives a well-defined model for the tangent space to $\mathcal{T}^{\conf}(\Sigma)$ with the property that $(W, \Omega_{J}|_W)$ is symplectomorphic to $(T_{J}\mathcal{T}^{\conf}(\Sigma), 4\Omega_{\WP})$. 
\end{remark}

\subsection{The construction of \texorpdfstring{$\mathcal{MS}_{0}(\Sigma,\rho)$}{MS0(S,r)}}\label{subsec:case_MS}
Let us now consider the bundle over $\Sigma$ defined by
\[
P(T^{*}\mathcal{J}(\R^{2})) \defin P \times_{\SL(2,\R)} T^{*}\mathcal{J}(\R^{2}) = \faktor{P \times T^{*}\mathcal{J}(\R^{2})}{\SL(2,\R)} ,
\]
where $P$ is the frame bundle introduced in Section \ref{subsec:Teich_inf_dim_reduction} and $\SL(2,\R)$ acts diagonally. The fiber of $P(T^{*}\mathcal{J}(\R^{2}))$ over the point $p\in \Sigma$ identifies with $T^* \mathcal{J}(T_p \Sigma)$, i. e. the space of pairs $(J_p, \sigma_p)$ where $J_p$ is an almost complex structure of $T_p \Sigma$ compatible with $\rho_p$, and $\sigma_p$ is a $g_{J_p}$-\hsk traceless and symmetric bilinear form on $T_p \Sigma$ that satisfies $\sigma_p(J_p, J_p) = - \sigma_p$. Since the para-\hsk hyperK\"ahler structure of $T^* \mathcal{J}(\R^2)$ is $\SL(2,\R)$-\hsk invariant (see Theorem \ref{thm:parahyper_structure_toy}), each fiber $T^* \mathcal{J}(T_p \Sigma)$ is naturally endowed with a para-\hsk hyperK\"ahler structure $(\hat{\g}_p, \hat{\i}_p, \hat{\j}_p, \hat{\k}_p)$, obtained by identifying $T_p \Sigma$ with $\R^2$ using an area-preserving isomorphism $\mappa{F_p}{T_p \Sigma}{\R^2}$.

The space of smooth sections 
\[
 T^{*}\mathcal{J}(\Sigma):=\Gamma(P(T^* \mathcal{J}(\R^2))
\]
can be identified with the set of pairs $(J,\sigma)$, where $J$ is a complex structure on $\Sigma$, and $\sigma$ is a symmetric and $g$-\hsk traceless $2$-\hsk tensor, where $g = g_J = \rho(\cdot, J \cdot)$. The element $\sigma$ can be equivalently characterized as the real part of a complex valued $J$-\hsk complex linear and symmetric $2$-\hsk tensor $\phi$. We identify the tangent space of $T^* \mathcal{J}(\Sigma)$ at $(J,\sigma)$ with the space of sections of the vector bundle $(J,\sigma)^* (T^\vertical P(T^* \mathcal{J}(\R^2))) \rightarrow \Sigma$, where $T^\vertical P(T^* \mathcal{J}(\R^2))$ stands for the vertical sub-\hsk bundle of $T P(T^* \mathcal{J}(\R^2))$ with respect to the projection map $P(T^* \mathcal{J}(\R^2)) \rightarrow \Sigma$. In particular, we can consider a tangent vector $(\dot{J}, \dot{\sigma})$ at $(J,\sigma)$ as the data of (see Lemma \ref{lem:characterization_tangent_space}):
\begin{itemize}
    \item a section $\dot{J}$ of $\End(T \Sigma)$ satisfying $\dot{J} J + J \dot{J} = 0$. In other words, $\dot{J}$ is a $g$-\hsk self-\hsk adjoint and traceless endomorphism of $T \Sigma$;
    \item a symmetric $2$-\hsk tensor $\dot{\sigma}$ satisfying 
    \[
    \dot{\sigma} = \dot{\sigma}_0 - \scal{\sigma}{J \dot{J}} \, g ,
    \]
    where $\dot{\sigma}_0$ is a symmetric and $g$-\hsk traceless $2$-\hsk tensor. Observe in particular that the $g$-\hsk full trace part of $\dot{\sigma}$ is uniquely determined by $\dot{J}$.
\end{itemize}

Formally, $T^{*}\mathcal{J}(\Sigma)$ is an infinite dimensional para-\hsk hyperK\"ahler manifold, where the symplectic forms are defined as 
\begin{align}
    (\omega_\mathbf{X})_{(J,\sigma)} ((\dot{J}, \dot{\sigma}), (\dot{J}', \dot{\sigma}')) \defin \int_\Sigma \hat{\omega}_\mathbf{X}((\dot{J}, \dot{\sigma}), (\dot{J}', \dot{\sigma}')) \, \rho ,
\end{align}
for $\mathbf{X} = \mathbf{I}, \mathbf{J}, \mathbf{K}$, and the pseudo-Riemannian metric is given by
\begin{align}
    \g_{(J,\sigma)} ((\dot{J}, \dot{\sigma}), (\dot{J}', \dot{\sigma}')) \defin \int_\Sigma \hat{\g}((\dot{J}, \dot{\sigma}), (\dot{J}', \dot{\sigma}')) \, \rho \ ,
\end{align}
where $\hat{\omega}$ and $\hat{\g}$ denote the symplectic form and the pseudo-Riemannian metric obtained by identifying the fibers of $P(T^{*}\mathcal{J}(\R^2)) \rightarrow \Sigma$ with the space of linear almost-complex structures on $T_\cdot \Sigma$ as described above. Similarly we have linear endomorphisms
\[
\i, \j, \k \vcentcolon T_{(J,\sigma)} T^* \mathcal{J}(\Sigma) \longrightarrow T_{(J,\sigma)} T^* \mathcal{J}(\Sigma) ,
\]
obtained by applying pointwisely the endomorphisms $\hat{\i}, \hat{\j}, \hat{\k}$ to a smooth section $(\dot{J}, \dot{\sigma})$. Their definition is formally identical to the ones in relations \eqref{eq:definition_I_toy}, \eqref{eq:definition_J_toy}, and \eqref{eq:definition_K_toy}, with the only difference that now $J$, $\sigma$, $g_J$, $\dot{J}$, $\dot{\sigma}$ are all tensors, and $f(\norm{\sigma})$ is a smooth function over $\Sigma$.

\begin{remark}\label{rmk:Mess_inf_dim} The expression of Mess homeomorphism introduced in Section \ref{subsec:Mess_homeo} can be formally applied to define a map
\begin{gather*}
\mathcal{M} \vcentcolon  T^* \mathcal{J}(\Sigma)  \longrightarrow  \mathcal{J}(\Sigma) \times \mathcal{J}(\Sigma) \\
\mathcal{M}(J, \sigma) \defin \left( (\1 - J B)^{-1} J (\1 - J B) , (\1 + J B)^{-1} J (\1 + J B) \right) \ ,
\end{gather*}
which takes as input a almost complex structure $J$ and a $g_J$-\hsk traceless symmetric $2$-\hsk tensor $\sigma$, and provides a pair of almost complex structures on $\Sigma$. In what follows, we will denote by $J_l$ and $J_r$ the left and right components of $\mathcal{M}$, and a tangent vector at
\[
T_{(J_l,J_r)} \mathcal{J}(\Sigma) \times \mathcal{J}(\Sigma) \cong T_{J_l} \mathcal{J}(\Sigma) \times T_{J_r} \mathcal{J}(\Sigma)
\]
will be given by a pair $(\dot{J}_l, \dot{J}_r)$. 
\end{remark}

\begin{remark}
The notation introduced is intentionally abusive, to emphasize the similarities between the toy model $T^* \mathcal{J}(\R^2)$ and the infinite-\hsk dimensional manifold $T^* \mathcal{J}(\Sigma)$. In what follows, we will often make use of the relations proved in Section \ref{sec:toy_model_via_complex_structures}, which concern $T^* \mathcal{J}(\R^2)$, in the context of $T^* \mathcal{J}(\Sigma)$. These arguments are legitimate because identities at the level of the toy model can be interpreted as pointwise identities at the level of smooth sections inside $T^* \mathcal{J}(\Sigma)$. 
\end{remark}

\subsection{The para-hyperK\"ahler structure of \texorpdfstring{$\mathcal{MS}_0(\Sigma,\rho)$}{MS0(S,rho)}}

We will now give an explicit description of the tangent bundle of the space $\mathcal{MS}_0(\Sigma,\rho)$ (introduced in Section \ref{subsec:change_coord}), which is well suited to present its para-\hsk hyperK\"ahler structure.

We recall from Section \ref{subsec:change_coord} that $\mathcal{MS}_0(\Sigma,\rho)$ is the quotient of the infinite-\hsk dimensional manifold
\[
\widetilde{\mathcal{MS}_{0}}(\Sigma, \rho) \defin \left\{(J,\sigma)\left| \begin{aligned}& \quad g = \rho(\cdot,J \cdot)\text{ is a Riemannian metric on }\Sigma , \\  &\,\,\sigma\text{ is the real part of a }J\text{-quadratic differential} , \\ &\,(h=(1 + f(\norm{\sigma}_g)) \, g,B=h^{-1}\sigma) \text{ satisfy }\eqref{eq:gauss_codazzi} \end{aligned}\right. \right\} 
\]
by the action of $\Symp_{0}(\Sigma, \rho)$, where as usual $f=f(\|\sigma\|_g)=\sqrt{1+\|\sigma\|_g^2}$. We will first introduce a very specific distribution $\mathcal{V} = \set{V_{(J,\sigma)}}_{(J, \sigma)}$ tangent to $\widetilde{\mathcal{MS}_{0}}(\Sigma, \rho)$, presenting several characterizations of it in Proposition \ref{prop:equivalent_def_subspace_V}. The proof of the equivalence of these equivalent characterizations requires a certain amount of computations, and it will be postponed to Section \ref{subsec:proof_of_prop}. 
Theorem \ref{thm:identification_tangent&V} will then describe the identification between $V_{(J,\sigma)}$ and the tangent space to $\mathcal{MS}_0(\Sigma, \rho)$ at the equivalence class of $(J,\sigma)$ by the action of $\Symp_{0}(\Sigma, \rho)$. Its proof is fragmented into several lemmas, which will constitute the main technical core of Section \ref{subsec:proof_of_prop}. 

\begin{remark}
The references to the process of infinite-\hsk dimensional symplectic reduction are limited to the case of the Teichm\"uller space, which has been described in Section \ref{subsec:Teich_inf_dim_reduction}. However, it is right and proper to acknowledge the reader that the definition itself of the vector space $V_{(J,\sigma)}$, together with the ideas behind the proofs of its properties, are all results of a deeper analysis developed in analogy to the hyperK\"ahler symplectic reduction process of Donaldson \cite{donaldson2003moment} in our context of interest (see Section \ref{sec:inf_dim_reduction}).
\end{remark}

\begin{repprop}{prop:equivalent_def_subspace_V}
	Given $(J, \sigma) \in \widetilde{\mathcal{MS}_{0}}(\Sigma,\rho)$, and $(\dot{J}, \dot{\sigma}) \in T_{(J,\sigma)} T^* \mathcal{J}(\Sigma)$, the following conditions are equivalent:
	\begin{enumerate}[i)]
		\item the pair $(\dot{J}, \dot{\sigma})$ satisfies
		\begin{equation} \tag{\ref{eq:description_V1}}
			\begin{cases}
				\divr_g (f^{-1}g^{-1} \dot{\sigma}_0) = - f^{-1} \scal{\nabla^g_{J \bullet} \sigma}{\dot{J}} , \\
				\divr_g \dot{J} = - f^{-2} \scall{\nabla^g_{J \bullet} \sigma}{\dot{\sigma}_0} .
			\end{cases}
		\end{equation}
		\item the endomorphisms $Q^\pm = Q^\pm(\dot{J}, \dot{\sigma}) \defin f^{-1} g^{-1} \dot{\sigma}_0 \pm \dot{J}$ satisfy
		\begin{equation} \tag{\ref{eq:description_V2}}
			\begin{cases}
            \divr_g (Q^+ J J_l) = - \scal{\nabla^g_{J \bullet} \sigma}{Q^+} , \\
            \divr_g (Q^- J J_r) = \scal{\nabla^g_{J \bullet} \sigma}{Q^-} ,
        \end{cases}
		\end{equation}
		where $J_l$ and $J_r$ denote the components of the Mess map $\mathcal{M}$;
		\item the endomorphisms $Q^\pm$ satisfy
		\begin{equation} \tag{\ref{eq:description_V3}}
			\begin{cases}
				\divr_g Q^+ = - f^{-1} \, \scal{\nabla^g_{J \bullet} \sigma}{Q^+} , \\ 
				\divr_g Q^- = + f^{-1} \, \scal{\nabla^g_{J \bullet} \sigma}{Q^-} .
			\end{cases}
		\end{equation}
	\end{enumerate}
	Moreover, the $1$-forms $\divr_{h_{l}}\dot{J}_{l}$ and $\divr_{h_{r}}\dot{J}_{r}$ are exact.
\end{repprop}

\begin{definition} \label{def:subspace_V}
	Given $(J,\sigma) \in \widetilde{\mathcal{MS}_{0}}(\Sigma, \rho)$, we define $V_{(J, \sigma)}$ to be the subspace of $T_{(J,\sigma)} T^* \mathcal{J}(\Sigma)$ of those elements $(\dot{J}, \dot{\sigma})$ that satisfy one of (and therefore all) the conditions in Proposition \ref{prop:equivalent_def_subspace_V}.
\end{definition}

\begin{reptheorem}{thm:identification_tangent&V}
    For every pair $(J,\sigma)$ lying in $\widetilde{\mathcal{MS}_{0}}(\Sigma, \rho)$, the vector space $V_{(J,\sigma)}$ is contained inside $T_{(J,\sigma)} \widetilde{\mathcal{MS}_{0}}(\Sigma, \rho)$, it is invariant by the action of $\mathbf{I}$, $\mathbf{J}$ and $\mathbf{K}$, and it defines a $\Symp(\Sigma,\rho)$-\hsk invariant distribution $\mathcal{V} = \set{V_{(J,\sigma)}}_{(J,\sigma)}$ on $\widetilde{\mathcal{MS}_{0}}(\Sigma, \rho)$. Moreover, the natural projection $\mappa{\pi}{\widetilde{\mathcal{MS}_{0}}(\Sigma,\rho)}{\mathcal{MS}_0(\Sigma, \rho)}$ induces a linear isomorphism $\dd \pi_{(J,\sigma)}: V_{(J,\sigma)} \rightarrow T_{[J,\sigma]}\mathcal{MS}_0(\Sigma, \rho)$.
\end{reptheorem}

    The proof of Theorem \ref{thm:identification_tangent&V} is postponed to Section \ref{sec:proofThm410}.
    Lemma \ref{lem:V_linearizes_Gauss_Codazzi} shows that $V_{(J,\sigma)}$ is tangent to the locus of those $(J,\sigma)$ that satisfy the Gauss-\hsk Codazzi equations, which is precisely the definition of the subset $\widetilde{\mathcal{MS}_{0}}(\Sigma,\rho)$ of $T^* \mathcal{J}(\Sigma)$. The invariance of $V_{(J,\sigma)}$ by the action of $\mathbf{I}$, $\mathbf{J}$ and $\mathbf{K}$ is proved in Lemma \ref{lem:V_IJK_invariant} and the invariance of the distribution $\mathcal{V}$ by $\Symp(\Sigma,\rho)$ follows from Lemma \ref{lem:invariant_distrib}. We will show in Lemma \ref{lem:V_transverse_orbit} that the differential of the projection map $\pi$ is injective. Finally, in Lemma \ref{lem:dimensionV} we show that the dimension of $V_{(J,\sigma)}$ is larger than or equal to the dimension of $\mathcal{MS}_0(\Sigma, \rho)$, which is $6 \abs{\chi(\Sigma)}$, and therefore we conclude that the differential of the projection $\pi$ induces a linear isomorphism between $V_{(J,\sigma)}$ and the tangent space to $\mathcal{MS}_0(\Sigma, \rho)$.  

We are now ready to summarize the proof of Theorem \ref{thm:parahyper_structure} in genus $g\geq 2$, although some of the steps of the proof will follow from the geometric interpretations that we provide in Section \ref{sec:geom_inter}.

\begin{reptheorem}{thm:parahyper_structure}[genus $\geq 2$]
Let $\Sigma$ be a closed oriented surface of genus $\geq 2$. Then $\mathcal{MGH}(\Sigma)$ admits a $MCG(\Sigma)$-invariant para-hyperK\"ahler structure $(\g,\i,\j,\k)$. Moreover the Fuchsian locus $\mathcal F(\Sigma)$ is totally geodesic and $(\g,\i)$ restricts to (a multiple of) the Weil-Petersson K\"ahler structure of Teichm\"uller space. 
\end{reptheorem}
\begin{proof}
Identifying the tangent space $T_{[J,\sigma]}\mathcal{MS}_{0}(\Sigma,\rho)$ with $V_{(J,\sigma)}$ (Theorem \ref{thm:identification_tangent&V}), we can define on $\mathcal{MS}_{0}(\Sigma,\rho)$ para-complex structures $\j$ and $\k$, a complex structure $\i$, and a pseudo-Riemannian metric $\g$ by restriction from the infinite dimensional space $T^* \mathcal{J}(\Sigma)$. The definition is well-posed, namely it does not depend on the representative in a given $\Symp_0(\Sigma,\rho)$-orbit, by the invariance statement in Theorem \ref{thm:identification_tangent&V} and the invariance of $\g$, $\i$, $\j$ and $\k$ (which is proved immediately, with the same tools as in the proof of Lemma \ref{lem:invariant_distrib}).

It is clear that $\i$, $\j$ and $\k$ are still compatible with $\g$ and satisfy the para-\hsk quaternionic relations. We also have corresponding $2$-forms $\omega_{\i}$, $\omega_{\j}$ and $\omega_{\k}$. A priori the metric $\g$, and consequently the $2$-\hsk forms $\omega_{\i}$, $\omega_{\j}$ and $\omega_{\k}$ may be degenerate when restricted to $V_{(J,\sigma)}$ and thus on $\mathcal{MS}_{0}(\Sigma,\rho)$. We rule out this possibility in Section \ref{sec:geom_inter} by identifying these forms with well-known symplectic forms (therefore closed and non-\hsk degenerate) on $T^{*}\mathcal{T}^{\conf}(\Sigma)$ and on $\mathcal{T}(\Sigma)\times \mathcal{T}(\Sigma)$. From the results of Section \ref{sec:geom_inter}, we also obtain that $\omega_{\i}$, $\omega_{\j}$ and $\omega_{\k}$ are integrable. See Corollaries \ref{cor:Iintegrable_higher_genus}, \ref{cor:omega_ik_nondeg}, \ref{cor:g_nondeg} and \ref{cor:omega_j_nondeg} for all these statements. We can then conclude that the quadruple $(\g, \i,\j,\k)$ endows $\mathcal{MS}_{0}(\Sigma,\rho)$ with a para-hyperK\"ahler structure.

The mapping class group invariance of the para-hyperK\"ahler structure follows from the $\Symp(\Sigma,\rho)$-invariance of Theorem \ref{thm:identification_tangent&V}, since we have a natural isomorphism between $MCG(\Sigma)$ and $\Symp(\Sigma,\rho)/\Symp_0(\Sigma,\rho)$, and again the $\Symp(\Sigma,\rho)$-invariance of $\g$, $\i$, $\j$ and $\k$.

Observe that the Fuchsian locus of $\mathcal{MGH}(\Sigma)$ corresponds to the pairs $(J,\sigma)$ with $\sigma=0$. By Proposition \ref{prop:equivalent_def_subspace_V}, its tangent space consists of the pairs $(\dot J,\dot\sigma)$ with $\dot\sigma=0$ and $\divr_g \dot J=0$, hence it corresponds precisely to the model of the tangent space of Teichm\"uller space that we described in Section \ref{subsec:Teich_inf_dim_reduction} (see Lemma \ref{lem:relation_Omega_WP} and the preceding discussion). By comparing the expression of the Weil-Petersson metric in Lemma \ref{lem:relation_Omega_WP} with the restriction of the metric $\g$ (see \eqref{eq:definition_g_toy}), we see immediately that $\g|_{\mathcal F(\Sigma)}$ coincides with $4 G_{WP}$. Finally, the Fuchsian locus is the set of fixed points of the circle action, that is isometric for the metric $\g$ by Theorem \ref{thm:circle}, which is proved in Section \ref{sec:geom_inter}. By a standard argument, this implies that the Fuchsian locus is totally geodesic. 
\end{proof}

\subsection{The proof of Theorem \ref{thm:identification_tangent&V}}\label{sec:proofThm410}
This subsection is dedicated to the proof of Lemmas \ref{lem:invariant_distrib}, \ref{lem:V_IJK_invariant}, \ref{lem:V_linearizes_Gauss_Codazzi}, \ref{lem:dimensionV} and \ref{lem:V_transverse_orbit}. Together, these results prove  Theorem \ref{thm:identification_tangent&V} to identify $V_{(J,\sigma)}$ with the tangent space to $\mathcal{MS}_{0}(\Sigma, \rho)$ at $[(J,\sigma)]$. The last of them is definitely the most challenging, and it requires some technical ingredients, which will be described along the way and which will be useful for the part concerning the $\infty$-\hsk dimensional symplectic reduction (Section \ref{sec:inf_dim_reduction}).

\subsubsection{Proof of Lemma \ref{lem:invariant_distrib}: invariance under symplectomorphisms.}
\begin{lemma} \label{lem:invariant_distrib}
	The distribution $\mathcal{V} = \set{V_{(J,\sigma)}}_{(J, \sigma) \in \widetilde{\mathcal{MS}_{0}}(\Sigma,\rho)}$ is invariant under the action of $\Symp(\Sigma, \rho)$. In other words, for every symplectomorphism $\varphi$ of $(\Sigma,\rho)$ and for every $(\dot{J}, \dot{\sigma}) \in V_{(J, \sigma)}$, we have $(\varphi^* \dot{J}, \varphi^* \dot{\sigma}) \in V_{(\varphi^* J, \varphi^* \sigma)}$.
\end{lemma}

\begin{proof}
	The main point where the condition of symplectomorphism is essential concerns the metric $g_J = \rho(\cdot, J \cdot)$. The relation $\varphi^* \rho = \rho$ implies that $g_{\varphi^* J}$, the metric with area form equal to $\rho$ and complex structure $\varphi^* J$, is equal to $\varphi^* g$, the pull-back of the Riemannian metric $g = g_J$ by $\varphi$. This is clearly equivalent to say that $\mappa{\varphi}{(\Sigma, g_{\varphi^* J})}{(\Sigma, g_J)}$ is an isometry. In particular, for every endomorphism of the tangent bundle $A$ we have that
	\[
	\varphi^* (\divr_g A) = \divr_{\varphi^* g} (\varphi^* A) . 
	\]
	Moreover, it is simple to check that the inner products and pairings of the tensors (see for instance \ref{eq:pairing}) are preserved by $\varphi$. The statement then follows from the naturality of the action and the expressions defining the subspace $V_{(J, \sigma)}$.
\end{proof}

\subsubsection{Proof of Lemma \ref{lem:V_IJK_invariant}: Invariance under $\i,\j$ and $\k$}
\begin{lemma} \label{lem:V_IJK_invariant}
	For every $(J,\sigma) \in \widetilde{\mathcal{MS}_{0}}(\Sigma,\rho)$, the subspace $V_{(J,\sigma)}$ is preserved by $\mathbf{I}$, $\mathbf{J}$ and $\mathbf{K}$.
\end{lemma}

\begin{proof}
	This is a simple consequence of the description of $V_{(J, \sigma)}$ provided by Proposition \ref{prop:equivalent_def_subspace_V}, part \textit{i)}. Since $\mathbf{K} = \mathbf{I} \mathbf{J}$, it is enough to check that $\mathbf{I} (\dot{J}, \dot{\sigma})$ and $\mathbf{J} (\dot{J}, \dot{\sigma})$ belong to $V_{(J, \sigma)}$, whenever $(\dot{J}, \dot{\sigma})$ is in $V_{(J,\sigma)}$. 
	
	The first component of $\mathbf{J} (\dot{J}, \dot{\sigma})$ is equal to $f^{-1} \, g^{-1} \dot{\sigma}_0$, while the $g$-\hsk traceless part of the second component is equal to $f \, g(\cdot, \dot{J} \cdot)$. If we replace $\dot{J}$ with $f^{-1} \, g^{-1} \dot{\sigma}_0$, and $\dot{\sigma}_0$ with $f \, g(\cdot, \dot{J} \cdot)$ in the equations \eqref{eq:description_V1}, we obtain the invariance of $V_{(J,\sigma)}$ under the action of $\mathbf{J}$. For the invariance under $\i$, we observe that, for every $(\dot{J}, \dot{\sigma})$ in $V_{(J, \sigma)}$ we have
	\begin{align*}
		\divr_g (- J \dot{J}) & = \divr_g (\dot{J} J) \\
		& = (\divr_g \dot{J}) \circ J \tag{rel. \eqref{eq:div_JA}} \\
		& = - f^{-2} \scall{\nabla^g_{J^2 \bullet} \sigma}{\dot{\sigma}_0} \\
		& = - f^{-2} \scal{\nabla^g_{J \bullet} \sigma}{ J g^{-1} \dot{\sigma}_0} \tag{Lemma \ref{lem:hol_quadr_diff_equivalence} part \textit{iv)}} \\
		& = - f^{-2} \scall{\nabla^g_{J \bullet} \sigma}{- \dot{\sigma}_0(\cdot, J \cdot)} .
	\end{align*}
	This shows that $\mathbf{I}(\dot{J}, \dot{\sigma})$ satisfies the first equation in \eqref{eq:description_V1}. Arguing simililarly for $\divr_g (f^{-1} g^{-1} \dot{\sigma}_0 J)$, we obtain the invariance of $V_{(J,\sigma)}$ by $\mathbf{I}$.
\end{proof}

\subsubsection{Proof of Lemma \ref{lem:V_linearizes_Gauss_Codazzi}: $V_{(J,\sigma)}$ lies in the kernel of the linearized GC-equations}
\begin{lemma} \label{lem:linearized_codazzi}
	Every element $(\dot{J}, \dot{\sigma})$ of $V_{(J,\sigma)}$ lies in the kernel of the linearized Codazzi equation $\dd^{\nabla^h} B = 0$.
\end{lemma}

As observed in Lemma \ref{lem:hol_quadr_diff_equivalence}, the Codazzi equation is equivalent to the requirement that $g^{-1} \sigma$ is a $g$-\hsk Codazzi tensor. In order to prove the statement above, we need to compute the first order variation of the expression $\nabla^g_X (g^{-1} \sigma Y) - \nabla^g_Y (g^{-1} \sigma X)- g^{-1}\sigma[X,Y]$, for any pair of tangent vector fields $X$, $Y$, as we vary $(J,\sigma)$ in the direction $(\dot{J}, \dot{\sigma})$. To simplify the notation, we will denote by $\nabla$ the Levi-\hsk Civita connection of $g = g_J$ and by $\dot{\nabla}$ its variation as we change the complex structure $J$.

\begin{lemma} \label{lem:variation_levicivita}
	Let $\dot{J} \in T_J \mathcal{J}(\Sigma)$ be an infinitesimal variation of complex structures on $\Sigma$. If $\dot{\nabla}$ denotes the first order variation of the Levi-\hsk Civita connection of $g \defin \rho(\cdot, J \cdot)$ along $\dot{J}$, then the following relation holds:
	\[
	\dot{\nabla}_X Y = - \frac{1}{2} ( (\divr \dot{J})(X) \ JY + J (\nabla_X \dot{J})Y ) ,
	\]
	for every pair of tangent vector fields $X$, $Y$ on $\Sigma$.
\end{lemma}

\begin{proof}
	The statement follows from Koszul's formula, which asserts that
	\[
	2 \, g(\nabla_X Y, Z) = X(g(Y,Z)) + Y(g(X,Z)) - Z(g(X,Y)) + g([X,Y], Z) - g([X,Z],Y) - g([Y,Z],X) ,
	\]
	for every tangent vector fields $X, Y, Z$. Its derivative leads to the equation
	\begin{align*}
	    2 \, \dot{g}(\nabla_X Y, Z) + 2 \, g(\dot{\nabla}_X Y, Z) & = X(\dot{g}(Y,Z)) + Y(\dot{g}(X,Z)) - Z(\dot{g}(X,Y)) + \\
	    & \qquad + \dot{g}([X,Y], Z) - \dot{g}([X,Z],Y) - \dot{g}([Y,Z],X) .
	\end{align*}
	As seen in relation \eqref{eq:derivative_scal_prod}, a first order variation of complex structures $\dot{J}$ determines a variation of the Riemannian metric $g = g_J$ of the form $\dot{g} = - g(\cdot, J \dot{J} \cdot)$. By developing the above expression in terms of the covariant derivatives of $\dot{g}$, and by the fact that $\nabla$ is a torsion-free connection, we can express the term $2 \, g(\dot{\nabla}_X Y, Z)$ as follows
	\begin{align*}
		2 \, g(\dot{\nabla}_X Y, Z) & = (\nabla_X \dot{g})(Y,Z) + (\nabla_Y \dot{g})(X,Z) - (\nabla_Z \dot{g})(X,Y) \\
		& = - g(Y,J(\nabla_X \dot{J})Z) - g(X,J(\nabla_Y \dot{J})Z) + g(X,J(\nabla_Z \dot{J})Y) \tag{$J$ $\nabla$-\hsk parallel} \\
		& = - g(J(\nabla_X \dot{J})Y + J(\nabla_Y \dot{J})X, Z) + g(J (\nabla_Z \dot{J})X, Y) ,
	\end{align*}
	where in the last step we used the fact that $J (\nabla_V \dot{J})$ is $g$-\hsk symmetric for every tangent vector field $V$. Let $M_X \in \End(T \Sigma)$ denote the endomorphism $M_X V \defin (\nabla_V \dot{J}) X$, and let $M_X^*$ be its $g$-\hsk adjoint. Then the relation above can be rephrased in the following terms:
	\begin{align*}
		2 \, g(\dot{\nabla}_X Y, Z) & = - g(J(\nabla_X \dot{J})Y + J M_X Y , Z) + g(J M_X Z, Y) \\
		& = - g(J(\nabla_X \dot{J})Y + J M_X Y , Z) - g(Z, M_X^* J Y) \tag{$J^* = - J$} \\
		& =- g((J(\nabla_X \dot{J}) + J M_X + M_X^* J)Y, Z) .
	\end{align*}
	Since this holds for every vector field $Z$, we deduce that
	\begin{equation} \label{eq:rel_nabla_dot}
		\dot{\nabla}_X Y = - \frac{1}{2} (J(\nabla_X \dot{J}) + J M_X + M_X^* J) Y .
	\end{equation}
	The endomorphism $J M_X + M_X^* J$ satisfies
	\[
	(J M_X + M_X^* J)^* = M_X^* J^* + J^* M_X  = - (J M_X + M_X^* J) ,
	\]
	since the adjoint of the complex structure coincides with $- J$.
	In other words, $J M_X + M_X^* J$ is a $g$-\hsk skew-\hsk symmetric endomorphism of $T \Sigma$, and consequently it is of the form $\lambda \, J$, for some smooth function $\lambda$ over $\Sigma$ (the space of $g_p$-\hsk skew-\hsk symmetric endomorphisms of $T_p \Sigma$ has dimension $1$). The function $\lambda$ can be determined in the following way:
	\begin{align*}
		\lambda & = - \frac{1}{2} \tr((J M_X + M_X^* J)J) \\
		& = \tr(M_X) \tag{$\tr(M_X) = \tr(M_X^*)$ and $J^2 = - \1$} \\
		& = \sum_i g((\nabla_{e_i} \dot{J})X, e_i) \\
		& = (\divr \dot{J})X ,
	\end{align*}
	where $(e_i)_i$ is a local $g$-\hsk orthonormal frame. Replacing the expression $J M_X + M_X^* J = \lambda J$ in relation \eqref{eq:rel_nabla_dot}, we deduce the desired assertion. 
\end{proof}

In the proof of Lemma \ref{lem:linearized_codazzi} we will also need the following technical lemma:

\begin{lemma} \label{lem:relation_codazzi&divr}
    If $A$ is a traceless endomorphism, then $$(\nabla_X A)Y - (\nabla_Y A)X = (\divr A)(Y) \, X - (\divr A)(X) \, Y.$$
\end{lemma}

\begin{proof}
    It is enough to check the identity for $X = e_1$, $Y = e_2$, where $e_1, e_2$ is a local orthonormal frame on $\Sigma$. To simplify the notation, we set $(\nabla_i A)_j^k \defin g((\nabla_{e_i} A) e_j, e_k)$. In particular we have that $(\nabla_i A)_1^1 = - (\nabla_i A)^2_2$ for every $i$. Since $A$ is traceless, the same is true for $\nabla_X A$. With this identity in mind, we proceed to check the desired relation:
    \begin{align*}
        (\nabla_{e_1} A) e_2 - (\nabla_{e_2} A) e_1 & = (\nabla_1 A)_2^1 \, e_1 + (\nabla_1 A)_2^2 \, e_2 - (\nabla_2 A)_1^1 \, e_1 - (\nabla_2 A)_1^2 \, e_2 \\
        & = (\nabla_1 A)_2^1 \, e_1 - (\nabla_1 A)_1^1 \, e_2 + (\nabla_2 A)_2^2 \, e_1 - (\nabla_2 A)_1^2 \, e_2 \\
        & = ((\nabla_1 A)_2^1 + (\nabla_2 A)_2^2) \, e_1 - ((\nabla_1 A)_1^1 + (\nabla_2 A)_1^2) \, e_2 \\
        & = (\divr A)(e_2) \, e_1 - (\divr A)(e_1) \, e_2 .
    \end{align*}
\end{proof}

We now have all the elements to prove the statement of Lemma \ref{lem:linearized_codazzi}:

\begin{proof}[Proof of Lemma \ref{lem:linearized_codazzi}]
	First we compute the first order variation of the tensor $g^{-1} \sigma$ along $(\dot{J}, \dot{\sigma}) \in T_{(J,\sigma)} \mathcal{J}(\Sigma)$. As seen in Lemma \ref{lem:characterization_tangent_space}, the variation of the metric $g$ is equal to $\dot{g} = - g(\cdot, J \dot{J} \cdot)$. In particular, we have
    \begin{align*}
        (g^{-1} \sigma)' & = - g^{-1} \dot{g} g^{-1} \sigma + g^{-1} \dot{\sigma} \\
        & = - g^{-1} (- g J \dot{J}) g^{-1} \sigma + g^{-1} \dot{\sigma}_0 - \scal{\sigma}{J \dot{J}} \, \1 \tag{Lemma \ref{lem:characterization_tangent_space}} \\
        & = J \dot{J} g^{-1} \sigma + g^{-1} \dot{\sigma}_0 - \scal{\sigma}{J \dot{J}} \, \1 .
    \end{align*}
    As aforementioned, we need to compute the derivative of the expression $(\dd^{\nabla} g^{-1} \sigma)(X,Y) = \nabla_X (g^{-1} \sigma Y) - \nabla_Y (g^{-1} \sigma X)- g^{-1}\sigma[X,Y]$ as we vary $(J, \sigma)$ along a direction $(\dot{J}, \dot{\sigma})$. The final goal will be to show that, if $(\dot{J}, \dot{\sigma})$ belongs to our preferred space $V_{(J,\sigma)}$, then such variation is equal to $0$. We observe that
    \begin{align*}
        ((\dd^{\nabla} g^{-1} \sigma)(X,Y))' & = \nabla_X ((g^{-1} \sigma)' Y) - \nabla_Y ((g^{-1} \sigma)' X) - (g^{-1}\sigma)'[X,Y] + \\
        & \qquad + \dot{\nabla}_X (g^{-1} \sigma Y) - \dot{\nabla}_Y (g^{-1} \sigma X) \\
        & = (\nabla_X (g^{-1} \sigma)')Y - (\nabla_Y (g^{-1} \sigma)')X + \dot{\nabla}_X (g^{-1} \sigma Y) - \dot{\nabla}_Y (g^{-1} \sigma X) \\
        & = \underbrace{(\divr (g^{-1} \sigma)')(Y) \, X - (\divr (g^{-1} \sigma)')(X) \, Y}_{\text{term $1$}} + \tag{Lemma \ref{lem:relation_codazzi&divr}} \\
        & \qquad + \underbrace{\dot{\nabla}_X (g^{-1} \sigma Y) - \dot{\nabla}_Y (g^{-1} \sigma X)}_{\text{term $2$}} .
    \end{align*}
    
    First we rephrase term $1$ in the expression above. Applying the expression for $(g^{-1}\sigma)'$ derived at the beginning of the proof, we have
    \begin{align*}
        (\divr (g^{-1} \sigma)')(Y) \, X - (\divr (g^{-1} \sigma)')(X) \, Y & = \divr(J \dot{J} g^{-1} \sigma + g^{-1} \dot{\sigma}_0 )(Y) \, X + \\
        & \qquad \quad - \divr(J \dot{J} g^{-1} \sigma + g^{-1} \dot{\sigma}_0)(X) \, Y + \\
        & \qquad - Y(\scal{\sigma}{J \dot{J}}) \, X + X(\scal{\sigma}{J \dot{J}}) \, Y .
    \end{align*}
    Assume now that $(\dot{J},\dot{\sigma})$ belongs to $V_{(J,\sigma)}$. Taking the sum of the equations in \eqref{eq:description_V2} (Proposition \ref{prop:equivalent_def_subspace_V} part \textit{ii)}) and using the expressions of $J_l$ and $J_r$ from \eqref{eq:expression_mess_map}, we obtain the relation
    $$\divr ( J \, \dot{J} \, g^{-1} \sigma + g^{-1} \dot{\sigma}_0) = \scal{\nabla^g_{J \bullet} \sigma}{\dot{J}}.$$
    Together with Lemma \ref{lem:hol_quadr_diff_equivalence} part \textit{iv)}, we see that
    \begin{align*}
        (\divr (g^{-1} \sigma)')(Y) \, X - (\divr (g^{-1} \sigma)')(X) \, Y & = \scal{\nabla_{J Y} \sigma}{\dot{J}} \, X - \scal{\nabla_{J X} \sigma}{\dot{J}} \, Y + \\
        & - Y(\scal{\sigma}{J \dot{J}}) \, X + X(\scal{\sigma}{J \dot{J}}) \, Y \\
        & = \scal{\sigma}{J (\nabla_X \dot{J})} \, Y - \scal{\sigma}{J (\nabla_Y \dot{J})} \, X .
    \end{align*}
    This is the expression for term $1$ that we will use in the final computation. Now we focus our attention on term $2$. Applying the variation formula of the Levi-\hsk Civita connection given by Lemma \ref{lem:variation_levicivita}, we obtain that
	\begin{align*}
		\dot{\nabla}_X (g^{-1} \sigma Y) - \dot{\nabla}_Y (g^{-1} \sigma X) & =  \frac{1}{2} \left( (\divr \dot{J})(Y) \, J g^{-1} \sigma X + J (\nabla_{Y} \dot{J}) g^{-1} \sigma X  + \right. \\
		& \left. - (\divr \dot{J})(X) \, J g^{-1} \sigma Y - J (\nabla_{X} \dot{J})\, g^{-1} \sigma Y \right) \\
		& = \frac{1}{2}  J g^{-1} \sigma \left( (\divr \dot{J})(Y) \, X - (\divr \dot{J})(X) \, Y \right) + \\
		& + \frac{1}{2} J \left( (\nabla_{Y} \dot{J}) g^{-1} \sigma X - (\nabla_{X} \dot{J})\, g^{-1} \sigma Y \right) \\
		& = \frac{1}{2}  J g^{-1} \sigma \left( (\nabla_{X} \dot{J})Y - (\nabla_{Y} \dot{J})X \right) + \\
		& + \frac{1}{2} J \left( (\nabla_{Y} \dot{J}) g^{-1} \sigma X - (\nabla_{X} \dot{J})\, g^{-1} \sigma Y \right)
	\end{align*}
    where, in the last step, we applied Lemma \ref{lem:relation_codazzi&divr} to $A = \dot{J}$. As a last remark, we apply Lemma \ref{lem:product_in_TJ} to $(\nabla_{\bullet} \dot{J})\, g^{-1} \sigma$ and $g^{-1} \sigma \, (\nabla_{\bullet} \dot{J})$, deriving:
	\begin{equation} \label{eq:commutator}
		g^{-1} \sigma \, (\nabla_{\bullet} \dot{J}) - (\nabla_{\bullet} \dot{J})\, g^{-1} \sigma = 2 \scal{\sigma}{J(\nabla_{\bullet} \dot{J})} \, J .
	\end{equation}
 	
 	We finally combine the expressions found above for terms $1$ and $2$ with this last relation. If $(J, \sigma)$ is a point of $\widetilde{\mathcal{MS}_{0}}(\Sigma,\rho)$ and $(\dot{J}, \dot{\sigma})$ is an element of $V_{(J, \sigma)}$, then the variation of the differential $\dd^{\nabla} g^{-1} \sigma$ can be expressed as follows:
	\begin{align*}
		((\dd^{\nabla} g^{-1} \sigma)(X,Y))' & = \scal{\sigma}{J (\nabla_X \dot{J})} \, Y - \scal{\sigma}{J (\nabla_Y \dot{J})} \, X + \\
		& \qquad \qquad \qquad \qquad \qquad \qquad + \frac{1}{2}  J g^{-1} \sigma \left( (\nabla_{X} \dot{J})Y - (\nabla_{Y} \dot{J})X \right) + \\
		& \qquad \qquad \qquad \qquad \qquad \quad + \frac{1}{2} J \left( (\nabla_{Y} \dot{J}) g^{-1} \sigma X - (\nabla_{X} \dot{J})\, g^{-1} \sigma Y \right) \\
		& = \frac{1}{2} \left( 2 \scal{\sigma}{J (\nabla_X \dot{J})} \, Y + J g^{-1} \sigma (\nabla_{X} \dot{J}) Y - J (\nabla_X \dot{J}) g^{-1} \sigma Y \right) + \\
		& \qquad - \frac{1}{2} \left( 2 \scal{\sigma}{J (\nabla_Y \dot{J})} \, X + J g^{-1} \sigma (\nabla_{Y} \dot{J}) X - J (\nabla_Y \dot{J}) g^{-1} \sigma X \right) \\
		& = \frac{1}{2} \left( 2 \scal{\sigma}{J (\nabla_X \dot{J})} \, Y + 2 \scal{\sigma}{J (\nabla_X \dot{J})} \, J^2 Y \right) + \\
		& \qquad \qquad \qquad \qquad - \frac{1}{2} \left( 2 \scal{\sigma}{J (\nabla_Y \dot{J})} \, X + 2 \scal{\sigma}{J (\nabla_Y \dot{J})} \, J^2 X \right) \tag{rel. \eqref{eq:commutator} for $\bullet = X , Y$} \\
		& = 0 .
	\end{align*}
	Therefore $(\dot{J}, \dot{\sigma})$ lies in the kernel of the linearized Codazzi equation.
\end{proof}

\begin{lemma} \label{lem:V_linearizes_Gauss_Codazzi}
	Every element $(\dot{J}, \dot{\sigma}) \in V_{(J,\sigma)}$ lies in the kernel of the linearized Gauss-Codazzi equations \eqref{eq:gauss_codazzi}.
\end{lemma}

\begin{proof}
    We start by relating the Riemannian metrics $h_l$ and $h_r$ associated to the complex structures $J_l$ and $J_r$ with $h$ and the endomorphism $B$:
    \begin{align*}
        h_{l, r} & = \rho(\cdot, (\1 \mp J B)^{-1} J (\1 \mp J B) \cdot) \\
        & = \rho((\1 \mp J B)^{-1} (\1 \mp J B) \cdot, (\1 \mp J B)^{-1} J (\1 \mp J B) \cdot) \\
        & = \det(\1 \mp J B)^{-1} \rho((\1 \mp J B) \cdot, J (\1 \mp J B) \cdot) \\
        & = \frac{f + 1}{2} g((\1 \mp J B) \cdot, (\1 \mp J B) \cdot) \tag{rel. \eqref{eq:det_f}} \\
        & = \frac{1}{2} h((\1 \mp J B) \cdot, (\1 \mp J B) \cdot) .
    \end{align*}
    We remark that the metrics $h_{l,r}$ differ by the left and right hyperbolic metrics given by Mess' homeomorphism by a factor $2$. Indeed, since the tensor $\1 \mp J B$ is $h$-\hsk Codazzi, the curvature of the metric $h_{l,r}$ can be computed as follows
    \[
    K_{h_{l,r}} = \frac{2 K_h}{\det(\1 \mp J B)} = \frac{2 K_h}{1 + \det B} .
    \]
    A proof of this fact can be found in \cite{krasnov_schlenker_minimal}. This shows that the Gauss-Codazzi equations have the following equivalent descriptions:
    \[
    \begin{cases}
        \dd^{\nabla^h} B = 0 , \\
        K_h = -1 - \det B ,
    \end{cases} \Leftrightarrow
    \begin{cases}
        \dd^{\nabla^g} g^{-1} \sigma = 0 , \\
        K_{h_l} = - 2 ,
    \end{cases} \Leftrightarrow
    \begin{cases}
        \dd^{\nabla^g} g^{-1} \sigma = 0 , \\
        K_{h_r} = - 2 .
    \end{cases}
    \]
    Assume now that $(\dot{J}, \dot{\sigma})$ is an element of $V_{(J,\sigma)}$. By Lemma \ref{lem:linearized_codazzi}, $(\dot{J}, \dot{\sigma})$ lies in the kernel of the linearized Codazzi equation. Therefore, by what previously observed, it is enough to show that the first order variation of the curvature $2$-\hsk form $K_{h_l} \, \dd{a}_{h_l} = K_{h_l} \, \rho$ (or $K_{h_r} \, \rho$) along the direction $\dot{J}_{l} = \dd{\mathcal{M}}(\dot{J}, \dot{\sigma})_1$ (or $\dot{J}_{r} = \dd{\mathcal{M}}(\dot{J}, \dot{\sigma})_2$, respectively) is equal to $0$.
    
    The first order variation of the $2$-\hsk form $K_{h_l} \, \rho$ coincides with $\frac{1}{2} \dd{}(\divr_{h_l} \dot{J}_l)$ (see \cite{donaldson2003moment}), and by the last statement of Proposition \ref{prop:equivalent_def_subspace_V}, the divergence $\divr_{h_l} \dot{J}_l$ is an exact $1$-\hsk form. Since $\dd^2 = 0$, we deduce that the derivative of $K_{h_l} \, \dd{a}_{h_l}$ along $\dot{J}_l$ is $0$, and therefore $(\dot{J}, \dot{\sigma})$ lies in the kernel of the linearized Gauss-\hsk Codazzi equations.
\end{proof}

\subsubsection{Proof of Lemma \ref{lem:dimensionV}: the dimension of $V_{(J,\sigma)}$ is $\geq 6 \abs{\chi(\Sigma)}$}
\begin{lemma}\label{lem:dimensionV}
For every $(J,\sigma) \in \widetilde{\mathcal{MS}_{0}}(\Sigma,\rho)$, $$\dim V_{(J,\sigma)}\geq 6 \abs{\chi(\Sigma)}.$$
\end{lemma}
\begin{proof}
    
    We make use of the description provided by Proposition \ref{prop:equivalent_def_subspace_V} part \textit{iii)}. 
    Because the equations are decoupled, it is sufficient to show that the space of solutions of the equation
    \[
    \divr_g Q^\pm(\dot{J}, \dot{\sigma}) = \mp f^{-1} \, \scal{\nabla_{J \bullet} \sigma}{Q^\pm(\dot{J}, \dot{\sigma})}
    \]
    has dimension at least $3|\chi(\Sigma)|$. We give the details for the equation concerning $Q^{+}$, the other case being analogous. Consider the differential operator $D_{J}:T_{J}\mathcal{J}(\Sigma) \rightarrow \Lambda^{1}(\Sigma)$ defined by
    \[
    D_{J}(Q)=\divr_g Q + f^{-1} \, \scal{\nabla_{J \bullet} \sigma}{Q} \ .
    \]
    Its principal symbol coincides with that of the divergence operator $\divr_g:T_{J}\mathcal{J}(\Sigma) \rightarrow \Lambda^{1}(\Sigma)$, hence $D_{J}$ is Fredholm and $\mathrm{ind}(D_{J})=\mathrm{ind}(\divr_g)$, where $\mathrm{ind}(\cdot)$ denotes the index of the differential operator, i.e. the difference between the dimension of its kernel and the dimension of its co-kernel. It is well-known (see for instance \cite{tromba2012teichmuller}) that $\dim(\Ker(\divr_g))=3|\chi(\Sigma)|$: the space of traceless, divergence-\hsk free tensors represents the tangent space to the Teichm\"uller space of $\Sigma$. On the other hand, $\divr_g:T_{J}\mathcal{J}(\Sigma) \rightarrow \Lambda^{1}(\Sigma)$ is surjective. Indeed, using the musical isomorphism $\#:\Lambda^{1}(\Sigma) \rightarrow \Gamma(T\Sigma)$ induced by the metric $g = g_J$, this is equivalent to proving that $\divr_g^{\#}: T_{J}\mathcal{J}(\Sigma) \rightarrow \Gamma(T\Sigma)$ is surjective. Let $L: \Gamma(T\Sigma) \rightarrow T_{J}\mathcal{J}(\Sigma)$ denote the Lie derivative operator. Its $L^{2}$-adjoint is $L^{*}(\dot{J}) = - J(\divr_g\dot{J})^{\#}$. In fact,
\begin{align*}
    \scall{LV}{\dot{J}}_{L^{2}}&=\int_{\Sigma}\frac{1}{2}\tr(\dot{J}\Dlie_{V}J)\rho \\
    &=\int_{\Sigma}(\divr_g\dot{J}J)(V)\rho \tag{Equation (\ref{eq:divergence_rel})}\\
    &=\int_{\Sigma}(\divr_g\dot{J})(JV)\rho \tag{Equation (\ref{eq:div_JA})} \\
    &=-\scall{J(\divr_g\dot{J})^{\#}}{V}_{L^{2}} \ .
\end{align*}
Therefore, the operator $L^{*}L: \Gamma(T\Sigma) \rightarrow \Gamma(T\Sigma)$ is self-adjoint. Moreover, if $V \in \Ker(L^{*}L)$ then
\[
    0=\scall{L^{*}L(V)}{V}_{L^{2}}=\|\Dlie_{V}J\|^{2}_{L^{2}} \ , 
\]
which implies that $V=0$ because $(\Sigma, J)$ has no biholomorphisms isotopic to the identity. A standard computation in local coordinates shows that the operator $L^{*}L$ is elliptic, thus by \cite[Theorem 5.1]{Voisin_Hodge} we have an $L^{2}$-orthogonal decomposition
\[
    \Gamma(T\Sigma)=\Ker(L^{*}L)\oplus \Imm(L^{*}L)=\Imm(L^{*}L) \ , 
\]
which shows that every vector field is in the image of $\divr_g^{\#}$.
    
    Hence,
    \[
    \dim(\Ker(D_{J}))\geq \mathrm{ind}(D_{J})=\mathrm{ind}(\divr_g)=3|\chi(\Sigma)| 
    \]
    and this concludes the proof of the assertion.
\end{proof}

\subsubsection{Proof of Lemma \ref{lem:V_transverse_orbit}: transversality to the tangent space to the $\Symp_{0}(\Sigma, \rho)$-orbit}
\begin{lemma} \label{lem:I_of_lie_deriv}
    For every symplectic vector field $X$ on $(\Sigma, \rho)$ and for every $(J,\sigma) \in T^* \mathcal{J}(\Sigma)$, with $\sigma$ that is the real part of a holomorphic quadratic differential on $(\Sigma, J)$, we have $\mathbf{I}(\Dlie_X J, \Dlie_X \sigma) = (- \Dlie_{J X} J, - \Dlie_{J X} \sigma)$.
\end{lemma}

\begin{proof}
    We need to introduce some notation to prove the desired identity. Let $A_V$ be the total derivative of a vector field $V$ tangent to $\Sigma$, i. e. $A_V Y \defin \nabla^g_Y V$, for any tangent vector field $Y$ (as usual $g$ is the Riemannian metric $\rho(\cdot, J \cdot)$). Then we can write
    \begin{align*}
        (\Dlie_V J)Y & = [V, J Y] - J [V, Y] \\
        & = \nabla^g_V (J Y) - \nabla^g_{J Y} V - J \nabla^g_V Y + J \nabla^g_Y V \tag{$\nabla^g$ torsion-\hsk free} \\
        & = J \nabla^g_V Y - A_V(JY) - J \nabla^g_V Y + J A_V Y \tag{$J$ $\nabla^g$-\hsk parallel} \\
        & = (J A_V - A_V J)Y .
    \end{align*}
    In other words, we have that
    \begin{equation} \label{eq:lie_j}
        \Dlie_V J = J A_V - A_V J 
    \end{equation}
    A similar computation shows that
    \begin{equation} \label{eq:lie_sigma}
        \Dlie_V \sigma = \nabla^g_V \sigma + \sigma(A_V \cdot, \cdot) + \sigma(\cdot, A_V \cdot) .
    \end{equation}

    Now we apply the definition of the almost-\hsk complex structure $\mathbf{I}$ to the pair $(\Dlie_X J, \Dlie_X \sigma)$, with $X$ symplectic vector field:
    \[
    \mathbf{I}(\Dlie_X J, \Dlie_X \sigma) = \left( - J (\Dlie_X J), \ - (\Dlie_X \sigma)_0(\cdot, J \cdot) - \scal{\sigma}{\Dlie_X J} \, g \right)
    \]
    Being $J$ $\nabla^g$-\hsk parallel, we have that $A_{J X} = J A_X$. Therefore
    \begin{align*}
        - J (\Dlie_X J) & = - J(J A_X  - A_X J) \tag{rel. \eqref{eq:lie_j} for $V = X$} \\
        & = - (J A_{J X} - A_{J X} J) \\
        & = - \Dlie_{J X} J . \tag{rel. \eqref{eq:lie_j} for $V = J X$}
    \end{align*}
    This shows that the first component of $\mathbf{I}(\Dlie_X J, \Dlie_X \sigma)$ coincides with $- \Dlie_{J X} J$. To study the second component we will need a few additional remarks. First we notice that it is enough to show
    \[
    - (\Dlie_X \sigma)_0(\cdot, J \cdot) = - (\Dlie_{J X} \sigma)_0 ,
    \]
    because the trace part of the second component of a pair $(\dot{J}, \dot{\sigma})$ is uniquely determined by $\dot{J}$ (compare with Lemma \ref{lem:characterization_tangent_space}). Given $V$ a vector field, the endomorphism $A_V$ can be decomposed into a sum
    \[
    A_V = \frac{\tr (A_V)}{2} \1 - \frac{\tr (J A_V)}{2} J + A_V^s ,
    \]
    where the first term is its trace part, the second is the $g$-\hsk skew-\hsk symmetric part, and the third is the traceless and $g$-\hsk symmetric part.
    
    If $X$ is a $\rho$-\hsk symplectic vector field, then the trace part of $A_X$ vanishes. Since $A_{J X} = J A_X$ (again because $J$ is $\nabla^g$-\hsk parallel), the decomposition of $A_{J X}$ is
    \[
        A_{J X} = J A_X = \frac{\tr(J A_X)}{2} \1 + 0 + J A_X^s .
    \]
    In particular, $X$ symplectic implies that the $g$-\hsk skew-\hsk symmetric part of $A_{J X}$ vanishes, and $A_{J X}^s = J A_X^s$.
    Therefore
    \begin{align*}
        (\Dlie_{J X} \sigma)_0 & = \nabla_{J X}^g \sigma + \sigma(A_{J X} \cdot, \cdot ) + \sigma(\cdot, A_{J X} \cdot) - \frac{\tr(g^{-1} \sigma A_{J X} + A_{J X} g^{-1} \sigma)}{2} \, g  \tag{rel. \eqref{eq:lie_sigma}} \\
        & = \nabla_{J X}^g \sigma + \tr(J A_X) \, \sigma + \sigma(J A_X^s \cdot, \cdot) + \sigma(\cdot, J A_X^s \cdot) - \tr(g^{-1} \sigma \, J A_{X}^s) g
    \end{align*}
    Applying Lemma \ref{lem:product_in_TJ} to $g^{-1} \sigma \, J A_X^s$ and $J A_X^s g^{-1} \sigma$ we obtain
    \[
    \sigma(\cdot, J A_X^s \cdot) + \sigma(J A_X^s \cdot, \cdot) = \tr(g^{-1} \sigma \, J A_X^s) \, g ,
    \]
    which, combined with the expression found above, shows that
    \begin{align*}
        (\Dlie_{J X} \sigma)_0 & = \nabla_{J X}^g \sigma + \tr(J A_X) \, \sigma \\
        & = (\nabla_X^g \sigma)(\cdot, J \cdot) + \tr(J A_X) \, \sigma \tag{Lemma \ref{lem:hol_quadr_diff_equivalence} part \textit{iv)}}
    \end{align*}

    On the other hand, by the decomposition of $A_X$, we have
    \begin{align*}
        (\Dlie_X \sigma)_0 & = \nabla^g_X \sigma + \sigma(A_X \cdot, \cdot) + \sigma(\cdot, A_X \cdot) - \frac{\tr(g^{-1} \sigma A_X + A_X g^{-1} \sigma)}{2} \, g \tag{rel \eqref{eq:lie_sigma}} \\
        & = \nabla^g_X \sigma - \frac{\tr(J A_X)}{2} (\sigma(J \cdot, \cdot ) + \sigma(\cdot, J \cdot)) + \sigma(A_X^s \cdot, \cdot) + \sigma(\cdot, A_X^s \cdot) - \tr(g^{-1} \sigma A_X) g \\
        & = \nabla^g_X \sigma - \tr(J A_X) \, \sigma(\cdot, J \cdot) + \sigma(A_X^s \cdot, \cdot) + \sigma(\cdot, A_X^s \cdot) - \tr(g^{-1} \sigma A_X^s) g ,
    \end{align*}
    where in the last step we used that $\tr(g^{-1} \sigma A_X^s) = \tr(g^{-1} \sigma A_X)$. Now, applying Lemma \ref{lem:product_in_TJ} to $g^{-1} \sigma A_X^s$ and $A_X^s g^{-1} \sigma$ we obtain
    \[
    \sigma(\cdot, A_X^s \cdot) + \sigma(A_X^s \cdot, \cdot) = \tr(g^{-1} \sigma A_X^s) \, g ,
    \]
    which reduces the expression above to the equality
    \[
    (\Dlie_X \sigma)_0 = \nabla^g_X \sigma - \tr(J A_X) \, \sigma(\cdot, J \cdot) .
    \]
    The identity $(\Dlie_X \sigma)_0(\cdot, J \cdot ) = (\Dlie_{J X} \sigma)_0$ is now immediate.
\end{proof}

\begin{lemma} \label{lem:linearized_Gauss}
	Let $G$ be the operator $G(J, \sigma) \defin K_h + 1 + \det B$, defined over the space $T^* \mathcal{J}(\Sigma)$ and with values in $\mathscr{C}^\infty(\Sigma)$. Assume that $(J,\sigma)$ satisfies the Gauss-\hsk Codazzi equations and let $U$ be a vector field on $\Sigma$. Then
	\[
	\dd{G}_{(J, \sigma)}(\Dlie_U J, \Dlie_U \sigma) = \frac{1}{2} \Delta_h (f^{-1} \divr_g U) - (1 - \det B) f^{-1} \divr_g U .
	\]
    In particular, if $(\Dlie_U J, \Dlie_U \sigma)$ belongs to the kernel of the differential of $G$, then $U$ is a $\rho$-\hsk symplectic vector field, i. e. $\dd(\iota_U \rho) = 0$.
\end{lemma}

\begin{proof}
	The final goal will be to compute $\dv{t} G(J_t,\sigma_t) |_{t = 0}$, where $(J_t, \sigma_t) = (\psi_t^* J, \psi_t^* \sigma)$. 
	We first determine the Riemannian metric $g_t = \rho(\cdot, J_t \cdot)$ associated to the complex structure $J_t \defin \dd{\psi_t^{-1}} J \dd{\psi_t}$, where $(\psi_t)_t$ represents the flow of $U$.
	\begin{align*}
		g_t & = \rho(\cdot, (\dd{\psi_t^{-1}} J \dd{\psi_t}) \cdot ) = \rho((\dd{\psi_t^{-1}} \dd{\psi_t}) \cdot, (\dd{\psi_t^{-1}} J \dd{\psi_t}) \cdot) \\
		& = (\det(\dd{\psi_t^{-1}}) \circ \psi_t) \ \rho(\dd{\psi_t} \cdot, J \dd{\psi_t} \cdot) = (\det(\dd{\psi_t^{-1}}) \circ \psi_t) \ g(\dd{\psi_t} \cdot, \dd{\psi_t} \cdot) \\
		& = (\det(\dd{\psi_t^{-1}}) \circ \psi_t) \ \psi_t^* g ,
	\end{align*}
	where $g = g_0$ and $\det(\dd{\psi_t^{-1}}) = ((\psi_t^{-1})^* \rho)/\rho$. In particular, $g_t$ is conformal to $\psi_t^* g$ with conformal factor given by $u_t \defin \det(\dd{\psi_t^{-1}}) \circ \psi_t$. We now determine the Riemannian metric $h_t$ associated to the pair $(J_t,\sigma_t)$ as described at the beginning of Section \ref{sec:para_hyperkahler_str_on_MS}:
	\begin{align*}
		h_t & = \left( 1 + \sqrt{1 + \norm{\psi_t^* \sigma}_{g_t}^2} \right) g_t = \left( 1 + \sqrt{1 + \norm{\psi_t^* \sigma}_{g_t}^2} \right) u_t \, \psi_t^* g \\
		& = \frac{1 + \sqrt{1 + \norm{\psi_t^* \sigma}_{g_t}^2}}{1 + \sqrt{1 + \norm{\sigma}_{g}^2 \circ \psi_t}} u_t \, \psi_t^* h = v_t \, \psi_t^* h .
	\end{align*}
	Therefore, the metric $h_t$ also differs from $\psi_t^* h$ by a conformal factor, here denoted by $v_t$ and defined from the relation above (observe that $v_0 \equiv 1$). Using the classical expression for the curvature under conformal change of the metric, we deduce that
	\begin{align*}
		K_{h_t} & = K_{v_t \, \psi_t^* h} \\
		& = v_t^{-1} \left( K_{\psi_t^* h} - \frac{1}{2} \Delta_{\psi_t^* h} \ln v_t \right) \\
		& = v_t^{-1} \left( K_h \circ \psi_t - \frac{1}{2} (\Delta_h \ln(v_t \circ \psi_t^{-1} )) \circ \psi_t \right) .
	\end{align*}
	The last term of the operator $G$ that we need to analyze is the determinant of the endomorphism $B_t$ associated to the pair $(J_t, \sigma_t)$:
	\begin{align*}
		\det B_t & = \det(h_t^{-1} \sigma_t ) \\
		& = v_t^{-2} \det((\psi_t^* h)^{-1} \psi_t^* \sigma) \\
		& = v_t^{-2} (\det B) \circ \psi_t .
	\end{align*}
	We can finally deduce an expression for the term
	\begin{align*}
		K_{h_t} + \det B_t & = v_t^{-1} \left( K_h \circ \psi_t - \frac{1}{2} (\Delta_h \ln(v_t \circ \psi_t^{-1} )) \circ \psi_t + v_t^{-1} (\det B) \circ \psi_t \right) .
	\end{align*}
	
	Combining the relations found above, we can compute the first order variation of operator $G$ along the path $t \mapsto (\psi_t^* J, \psi_t^* \sigma)$, obtaining
	\begin{align*}
		(K_{h_t} + 1 + \det B_t)' & = - \dot{v} ( K_h + \det B ) + U(K_h) - \frac{1}{2} \Delta_h \dot{v} - \dot{v} \det B + U(\det B)  \\
		& = \dot{v} (1 - \det B ) - \frac{1}{2} \Delta_h \dot{v} ,
	\end{align*}
	where, in the last line, we used the fact that $(J,\sigma)$ satisfies $G(J,\sigma) = 0$. The final statement will now follow by computing $\dot{v}$:
	\begin{align*}
		\dot{v} & = \left. \dv{v_t}{t} \right|_{t = 0} = \left. \dv{(v_t \circ \psi_t^{-1})}{t} \right|_{t = 0}   \tag{$v_0 \equiv 1$} \\
		& = \left. \dv{t} \frac{1 + \sqrt{1 + \norm{\sigma_t}_{g_t}^2 \circ \psi_t^{-1} }}{1 + \sqrt{1 + \norm{\sigma}_{g}^2}} u_t \right|_{t = 0} \\
		& = \frac{(\norm{\sigma_t}_{g_t}^2 \circ \psi_t^{-1})'}{2 f (1 + f)} + \dot{u} .
	\end{align*}
    We now observe that
	\begin{align*}
		\norm{\sigma_t}_{g_t}^2 & = \frac{1}{2} \tr(g_t^{-1} \sigma_t \, g_t^{-1} \sigma_t) = \frac{1}{2 u_t^2} \tr((\psi_t^* g)^{-1} \psi_t^* \sigma \, (\psi_t^* g)^{-1} \psi_t^* \sigma) \\
		& = u_t^{-2} (\norm{\sigma}^2_g \circ \psi_t) ,
	\end{align*}
	which implies the following:
	\[
	(\norm{\sigma_t}_{g_t}^2 \circ \psi_t^{-1})' = ((u_t \circ \psi_t^{-1})^{-2})' \norm{\sigma}^2_g = - 2 \dot{u} \, \norm{\sigma}_g^2 = - 2 \dot{u} (f^2 - 1)
	\]
	since $u_0 = \det(\dd{(\id)}) \equiv 1$. A simple computation now shows that $\dot{u} = \dv{u_t}{t} |_{t = 0} = - \divr_g U$. To conclude, we have
	\[
	\dot{v} = \left( - \frac{f^2 - 1}{f(1 + f)} + 1 \right) \dot{u} = - f^{-1} \divr_g U . 
	\]
    
    For the second part of the statement, observe that, being $B$ traceless and $h$-\hsk self-\hsk adjoint, its determinant is a non-\hsk positive function. In particular $(1 - \det B) \geq 1$. Consequently the linear operator
    \[
    T : \lambda \longmapsto - \frac{1}{2} \Delta_h \lambda + (1 - \det B) \lambda
    \]
    is self-\hsk adjoint and positive definite in $L^2(\Sigma, \dd{a_h})$, in particular injective. Therefore, if $(\Dlie_U J, \Dlie_U \sigma)$ lies in the kernel of the differential of $G$, then the function $\lambda = f^{-1} \divr_g U$ is sent to $0$ by the operator $T$, and so $\divr_g U = 0$. By relation \eqref{eq:div_contraction} this is equivalent to say that $U$ is a $\rho$-symplectic vector field.
\end{proof}

\begin{lemma} \label{lem:V_transverse_orbit}
	For every $(J,\sigma) \in \widetilde{\mathcal{MS}_{0}}(\Sigma,\rho)$, we have
	\[
	V_{(J,\sigma)} \cap T_{(J,\sigma)}(\Symp_{0}(\Sigma,\rho) \cdot (J,\sigma))=\{0\} \ . 
	\]
\end{lemma}

\begin{proof}
	Assume $X$ to be a symplectic vector field such that $(\Dlie_X J, \Dlie_X \sigma)$ belongs to $V_{(J,\sigma)}$. By equivariance of Mess homeomorphism, this is equivalent to the fact that $(\Dlie_X J_l, \Dlie_X J_r)$ are exact $1$-forms. In particular we must have that the $h_l$-\hsk divergence of $\Dlie_X J_l$ is an exact $1$-\hsk form. By Proposition \ref{prop:orthogonal_symp}, this implies that $\Dlie_X J_l$ is $\Omega_{J_l}$-\hsk orthogonal to $T_{J_l} \Symp_{0}(\Sigma,\rho) \cdot J_l$. Therefore, for every $\rho$-\hsk symplectic vector field $Y$ we have $\Omega_{J_l}(\Dlie_X J_l, \Dlie_Y J_l) = 0$. Consequently, inside the quotient $\widetilde{\Teich}(\Sigma)$ of $\mathcal{J}(\Sigma)$ by the action of the Hamiltonian group $\Ham(\Sigma, \rho)$, the class $[\Dlie_X J_l]$ is $\widehat{\Omega}_{[J_l]}$-\hsk orthogonal to the $H$-\hsk orbit of $[J_l]$ (compare with Section \ref{subsec:Teich_inf_dim_reduction}, and in particular with Equation \eqref{eq:fakeWP}). As observed by Donaldson \cite{donaldson2003moment} (see also \cite{trautwein2019hyperkahler}), the $H$-\hsk orbits are symplectic submanifolds of $\widetilde{\Teich}(\Sigma)$, and the class of $\Dlie_X J_l$ is tangent to the orbit $H \cdot [J_l]$. By what previously observed, we deduce that the class of $\Dlie_X J_l$ is equal to zero inside $T_{J_l} H \cdot [J_l]$ or, in other words, that $\Dlie_X J_l$ is tangent to the orbit of $J_l$ by the Hamiltonian group inside $\mathcal{J}(\Sigma)$. Since the Lie derivative operator $X \mapsto \Dlie_X J$ is injective, because $(\Sigma, J)$ has no biholomorphisms isotopic to the identity, the vector field $X$ is actually $\rho$-\hsk Hamiltonian, i. e. $\iota_X \rho = \dd{f}$ for some $f \in \mathscr{C}^\infty(\Sigma)$.
    
    On the other hand, if $(\Dlie_X J, \Dlie_X \sigma)$ belongs to $V_{(J,\sigma)}$, the same has to hold for $\mathbf{I}(\Dlie_X J, \Dlie_X \sigma) = - (\Dlie_{J X} J, \Dlie_{J X} \sigma)$ (see Lemma \ref{lem:I_of_lie_deriv}). In particular, the differential of the function $G$ considered in Lemma \ref{lem:linearized_Gauss} applied to $(\Dlie_{J X} J, \Dlie_{J X} \sigma)$ must vanish, by what observed in Lemma \ref{lem:V_linearizes_Gauss_Codazzi}. By the second part of Lemma \ref{lem:linearized_Gauss}, we deduce that $J X$ is $\rho$-\hsk symplectic, i.e. $\dd{(\iota_{J X} \rho)} = 0$. This implies that the $1$-\hsk form $- \dd{f} \circ J = - (\iota_X \rho) \circ J = \iota_{J X} \rho$ is closed, and therefore that the function $f$ is $g$-\hsk harmonic (since $\dd(\dd{f} \circ J) = - \Delta_g f \rho$). Being $\Sigma$ compact, we deduce that $f$ is constant, and therefore that the vector field $X$ is equal to $0$, as desired.
\end{proof}

\subsection{The proof of Proposition \ref{prop:equivalent_def_subspace_V}} \label{subsec:proof_of_prop}

This subsection is dedicated to the proof of Proposition \ref{prop:equivalent_def_subspace_V}, which provides a series of equivalent descriptions of the tangent space to $\mathcal{MS}_0(\Sigma, \rho)$, the model for the deformation space of MGHC AdS structures introduced in Section \ref{subsec:change_coord}. The proof of the result is technically involved, and none of the tools developed here will be used in the rest of the exposition. In particular, the reader who is willing to trust the statement of Proposition \ref{prop:equivalent_def_subspace_V} can skip this part without losing necessary ingredients for the remainder of the paper.

We will first focus on the equivalence of the first three descriptions appearing in the statement of Proposition \ref{prop:equivalent_def_subspace_V}. This part is mainly algebraic and it follows from explicit manipulations of the equations.

\begin{proof}[Proof of $\textit{i)} \Leftrightarrow \textit{ii)} \Leftrightarrow \textit{iii)}$] 
    Recalling that $Q^\pm = Q^\pm(\dot{J}, \dot{\sigma}) \defin f^{-1} g^{-1} \dot{\sigma}_0 \pm \dot{J}$ (these terms formally appeared already in the study of the differential of Mess' map in Proposition \ref{thm:mappaM} (baby version)), one readily checks that $\textit{i)} \Leftrightarrow \textit{ii)}$, since taking the sum and difference of the equations in \eqref{eq:description_V1} one obtains the equations in \eqref{eq:description_V3}, and vice versa. 
    
    Let us now prove $\textit{ii)} \Leftrightarrow \textit{iii)}$. As a preliminary step, we claim that, for every $\dot{J}' \in T_J \mathcal{J}(\Sigma)$ and for every tangent vector field $V$, we have
    \begin{equation} \label{eq:differential_f}
        \dd{f}(\dot{J}'V) = f^{-1} \, \scal{\nabla^g_{(g^{-1} \sigma) V} \, \sigma}{\dot{J}'} = f^{-1} \, \scal{\nabla^g_\bullet \sigma}{\dot{J}'} ((g^{-1} \sigma) V) .
    \end{equation}
    Assuming temporarily this relation, and using that the components  $J_l$ and $J_r$ of the map $\mathcal{M}$ can be expressed as follows (see \eqref{eq:expression_mess_map})
    \[
    J_l = f J + g^{-1} \sigma, \qquad J_r = f J - g^{-1} \sigma ,
    \]
    we obtain the following identities in local coordinates $x_i$:
    \begin{align*}
        (\divr_g Q^+) J J_l V & = \dd{x^i}((\nabla^g_{\partial_i} Q^+)J J_l V) \\
        & =  \dd{x^i}((\nabla^g_{\partial_i} Q^+ J J_l) V) - \dd{x^i}(Q^+ J (\nabla^g_{\partial_i} (f J + g^{-1} \sigma)) V) \\
        & = \divr_g(Q^+ J J_l) V - \left( \partial_i f \, \dd{x^i}(Q^+ J^2 V) + \dd{x^i}(Q^+ J (\nabla^g_{\partial_i} g^{-1} \sigma) V) \right) \\
        & = \divr_g(Q^+ J J_l) V + \dd{f}(Q^+ V) - \dd{x^i}(Q^+ J (\nabla^g_V g^{-1} \sigma) \partial_i) \tag{$g^{-1} \sigma$ Codazzi} \\
        & = \divr_g(Q^+ J J_l)V + \dd{f}(Q^+ V) - 2 \, \scal{\nabla^g_V \sigma}{Q^+ J} 
        \tag{$\nabla^g g^{-1}=0$} \\
        & = \divr_g(Q^+ J J_l)V + \dd{f}(Q^+ V) + 2 \, \scal{(\nabla^g_V \sigma)(\cdot, J \cdot)}{Q^+} \tag{$Q^+ \in T_J \mathcal{J}(\Sigma)$} \\
        & = \divr_g(Q^+ J J_l)V + \dd{f}(Q^+ V) + 2 \, \scal{\nabla^g_{JV} \sigma}{Q^+} \tag{Lemma \ref{lem:hol_quadr_diff_equivalence} part \textit{iv)}} \\
        & = \divr_g(Q^+ J J_l)V + f^{-1}\scal{\nabla^g_{(g^{-1}\sigma)V} \sigma}{Q^+} + 2 \, \scal{\nabla^g_{JV} \sigma}{Q^+} 
        \tag{rel. \eqref{eq:differential_f}}, \\
        & = \divr_g(Q^+ J J_l)V + f^{-1} \scal{\nabla^g_{(g^{-1}\sigma+fJ)V} \sigma}{Q^+} +  \, \scal{\nabla^g_{JV} \sigma}{Q^+},
    \end{align*}
    where in the second to last line we used relation \eqref{eq:differential_f}. This can be rewritten as:
    $$
        (\divr_g Q^+ + f^{-1}\scal{\nabla^g_{J \bullet} \sigma}{Q^+}) \circ J J_l = \divr_g(Q^+ J J_l) + \scal{\nabla^g_{J \bullet} \sigma}{Q^+}.
    $$
    This shows that the first equations in \eqref{eq:description_V2} and \eqref{eq:description_V3} are equivalent. Proceeding in an analogous way for the term $\divr_g Q^-$, we can deduce that the second equations are equivalent too, and therefore conclude that $\textit{ii)} \Leftrightarrow \textit{iii)}$.
    
    It remains to prove relation \eqref{eq:differential_f}. Since $(g^{-1} \sigma)^2 = \norm{\sigma}^2 \, \1$ by Cayley-Hamilton theorem, we have
    \begin{equation} \label{eq:derivative_sigma^2}
        V(\norm{\sigma}^2) \, \1 = (\nabla^g_V g^{-1} \sigma) g^{-1} \sigma + g^{-1} \sigma (\nabla^g_V g^{-1} \sigma) .
    \end{equation}
    Now we observe that
    \begin{align*}
        2 \scal{\nabla^g_{(g^{-1} \sigma) V} \, \sigma}{\dot{J}'} & = \dd{x^i}(\dot{J}' (\nabla^g_{(g^{-1} \sigma)V} g^{-1} \sigma) \partial_i ) 
        \tag{$\nabla^g g^{-1}=0$} \\
        & = \dd{x^i}(\dot{J}' (\nabla^g_{\partial_i} g^{-1} \sigma) (g^{-1} \sigma)V )  \tag{$g^{-1} \sigma$ Codazzi} \\
        & = - \dd{x^i}(\dot{J}' g^{-1} \sigma (\nabla^g_{\partial_i} g^{-1} \sigma)V) + \partial_i(\norm{\sigma}^2) \dd{x^i}(\dot{J}' V) \tag{relation \eqref{eq:derivative_sigma^2}} \\
        & = - \dd{x^i}(\dot{J}' g^{-1} \sigma (\nabla^g_V g^{-1} \sigma)\partial_i) + \dd{(\norm{\sigma}^2)}(\dot{J}' V)  \tag{$g^{-1} \sigma$ Codazzi} \\
        & = - \tr(\dot{J}' g^{-1} \sigma (\nabla^g_V g^{-1} \sigma)) + \dd{(\norm{\sigma}^2)}(\dot{J}' V)  \\
        & = \dd{(\norm{\sigma}^2)}(\dot{J}' V) ,
    \end{align*}
    where, in the last step, we applied Lemma \ref{lem:product_in_TJ} to the triple $\dot{J}'$, $g^{-1} \sigma$, $\nabla^g_V g^{-1} \sigma \in T_J \mathcal{J}(\Sigma)$. Finally, we have
    \[
    \dd{f}(\dot{J}'V) = \frac{1}{2 \sqrt{1 + \norm{\sigma}^2}} \dd{(\norm{\sigma}^2)}\dot{J}' V = f^{-1} \scal{\nabla^g_{(g^{-1} \sigma) V} \, \sigma}{\dot{J}'} ,
    \]
    which concludes the proof of relation \eqref{eq:differential_f}, and hence of the statement.
\end{proof}

The last statement of Proposition \ref{prop:equivalent_def_subspace_V} is slightly more elaborated, because it requires a "conversion" in the linear connection: while the first three characterizations are expressed in terms of the Levi-Civita connection $\nabla^g$ of $g$, the last one involves the Levi-Civita connections $\nabla^l$ and $\nabla^r$ of the Riemannian metrics $h_l$ and $h_r$, respectively. For this reason we will need some additional ingredients, described in Lemmas \ref{lem:derivative_B}, \ref{lem:divergence_h}, \ref{lem:divergence_left_right}. The transition from $\nabla^g$ to $\nabla^{l}$, $\nabla^{r}$ is done passing through the Levi-Civita connection of the Riemannian metric $h$: in Lemma \ref{lem:derivative_B} we express the derivative $\nabla^h_\bullet B$ in terms of $\nabla^g_\bullet \sigma$, and in Lemma \ref{lem:divergence_h} we compute the $h$-\hsk divergence operator in terms of $\divr_g$. With these tools, we will be able to determine the expressions for the divergence with respect to $h_l$ and $h_r$ in terms of $\divr_g$ and the derivative $\nabla^g_\bullet \sigma$, as described in Lemma \ref{lem:divergence_left_right}, which will make the equivalence $\textit{iii)} \Leftrightarrow \textit{iv)}$ simpler to handle. 

\begin{lemma} \label{lem:derivative_B}
	Let $(J,\sigma) \in T^* \mathcal{J}(\Sigma)$, where $\sigma$ is the real part of a $J$-\hsk holomorphic quadratic differential. Let $h$ denote the Riemannian metric $(1 + f) \, g$, with Levi-Civita connection $\nabla^h$, and let $B = h^{-1} \sigma$ be the $h$-\hsk self-\hsk adjoint operator associated to $\sigma$. 
	For every tangent vector field $X$ we have
	\[
	\nabla^h_X B = (1 + f)^{-1} f^{-1} \, \nabla^g_X (g^{-1} \sigma) ,
	\]
	where $f = f(\norm{\sigma}_g) = \sqrt{1 + \norm{\sigma}_g^2}$, and $\nabla^g$ is the Levi-\hsk Civita connection of $g = \rho(\cdot, J \cdot)$.
\end{lemma}

\begin{proof}
    First we observe that, if $\sigma$ is zero, then the relation is obviously satisfied. In what follows we will assume that $\sigma$ is not identically zero.
    
	The tensor $\nabla^g_X(g^{-1} \sigma)$ is a symmetric and traceless endomorphism of the tangent space of $\Sigma$. For every $p \in \Sigma$ outside the set of zeros of $\sigma$ (which is a finite set), the elements $(g^{-1} \sigma)_p$ and $(J g^{-1} \sigma)_p$ form a basis of the space of traceless symmetric endomorphisms of $T_p \Sigma$. In particular, using the scalar product $\scall{\cdot}{\cdot}$ we can represent $\nabla^g_X(g^{-1} \sigma)$ in terms of such basis, obtaining the following expression:
	\begin{align*}
		\nabla^g_X (g^{-1} \sigma) Y & = \frac{1}{\norm{\sigma}_g^2} \left( \scall{g^{-1} \sigma}{\nabla^g_X (g^{-1} \sigma)} \, g^{-1} \sigma + \scall{J g^{-1} \sigma}{\nabla^g_X (g^{-1} \sigma)} \, J g^{-1} \sigma \right) \\
		& = \frac{1}{\norm{\sigma}_g^2} \left( \frac{1}{2} \dd{(\norm{\sigma}_g^2)}(X) \, g^{-1} \sigma + \scall{\sigma}{\nabla^g_{J X} \sigma} \, J g^{-1} \sigma \right) \\
		& = \frac{1}{2 \norm{\sigma}_g^2} \left( \dd{(\norm{\sigma}_g^2)}(X) \, g^{-1} \sigma + \dd{(\norm{\sigma}_g^2)}(J X) \, J g^{-1} \sigma \right) .
	\end{align*}
	From the first to the second line, we are making use of the definitions of the scalar products $\scall{\cdot}{\cdot}$ and Lemma \ref{lem:hol_quadr_diff_equivalence} part \textit{iv)}. By definition of $f$, we have $\norm{\sigma}_g^2 = f^2 - 1$, therefore $\dd{(\norm{\sigma}_g^2)} = 2 f \dd{f}$. Combining this with the chain of equalities above, we obtain that
	\[
	\nabla^g_X (g^{-1} \sigma) Y = (f^2 - 1)^{-1} f \left( \dd{f}(X) \, g^{-1} \sigma + \dd{f}(J X) \, J g^{-1} \sigma \right) .
	\]
	
	The exact same observations made to express $\nabla^g_X (g^{-1} \sigma) Y$ allow us to deduce that
	\[
		(\nabla^h_X B)Y = \frac{1}{2 \norm{B}^2} \left( \dd{(\norm{B}^2)}(X) \, B + \dd{(\norm{B}^2)}(J X) \, J B \right) 
	\]
	(remember that $B$ is Codazzi with respect to $h$ by Lemma \ref{lem:hol_quadr_diff_equivalence} part \textit{i)}).
	Unraveling the definitions of $f$ and $B$, we see that $\norm{B}^2 = (1 + f)^{-1} (f - 1)$, and consequently $\dd{(\norm{B}^2)} = 2 (1 + f)^{-2} \dd{f}$. In particular we obtain that
	\begin{align*}
		(\nabla^h_X B)Y & = (1 + f)^{+ 1 - 2} (f - 1)^{-1} \left( \dd{f}(X) \, B + \dd{f}(J X) \, J B \right) \\
		& = (f^2 - 1)^{-1} (1 + f)^{-1} \left( \dd{f}(X) \, g^{-1} \sigma + \dd{f}(J X) \, J g^{-1} \sigma \right) ,
	\end{align*}
	where in the last step we used that $B = h^{-1} \sigma = (1 + f)^{-1} g^{-1} \sigma$. Outside the zero locus of $\sigma$, the statement now follows by comparing the two expressions found for $\nabla^g_X (g^{-1} \sigma) Y$ and $\nabla^h_X B$, and the identity holds on the whole surface by continuity. 
\end{proof}

\begin{lemma} \label{lem:divergence_h}
	Let $T$ be a smooth section of endomorphisms of $T \Sigma$. Then
	\[
	\divr_h T = \divr_g T + (1 + f)^{-1} \dd{f} \circ T^\textit{s}_0 ,
	\]
	where $T^\textit{s}_0$ is the $g$-\hsk symmetric and traceless part of $T$.
\end{lemma}

\begin{proof}
	By Koszul's formula we have
	\begin{align*}
		2 \, h(\nabla^h_X Y, Z) & = X(h(Y,Z)) + Y(h(X,Z)) - Z(h(X,Y)) + h([X,Y], Z) - h([X,Z],Y) - h([Y,Z],X) \\
		& = \dd{f}(X) \, g(Y,Z) + \dd{f}(Y) 
		\, g(X,Z) - \dd{f}(Z) \, g(X,Y) + 2 h(\nabla^g_X Y, Z),
	\end{align*}
	where in the last step we used the relation $h = (1 + f) \, g$ and the Koszul's formula for $g$. In particular, we deduce that
	\[
	\nabla^h_X Y = \nabla^g_X Y + \frac{1}{2} (1 + f)^{-1} (\dd{f}(X) \, Y + \dd{f}(Y) \, X - g(X,Y) \grd_g f ) ,
	\]
	where $\grd_g f = g^{-1} \dd{f}$ is the $g$-\hsk gradient of $f$. Then we have
	\begin{align*}
		(\divr_h T)X & = \dd{x}^i \left( (\nabla^h_{\partial_i} T) X \right) = \dd{x}^i \left( \nabla^h_{\partial_i} (T X) - T (\nabla^h_{\partial_i} X) \right) \\
		& = \dd{x}^i \left( \nabla^g_{\partial_i} (T X) - T (\nabla^g_{\partial_i} X) \right) + \frac{1}{2} (1 + f)^{-1} \dd{x}^i ( \partial_i f \, T X + \dd{f}(T X) \, \partial_i + \\
		& \qquad \qquad - g(\partial_i , T X) \grd_g f - T(\partial_i f \, X + \dd{f}(X) \, \partial_i - g(\partial_i, X) \grd_g f ) ) \\
		& = (\divr_g T) X + \frac{1}{2} (1 + f)^{-1} (\dd{f}(T X) + 2 \dd{f}(T X) - \dd{f}(T X) - \dd{f}(T X) + \\
		& \qquad \qquad \qquad \qquad \qquad \qquad \qquad \qquad \quad - \dd{f}(X) \tr T + g(T \grd_g f, X) ) \\
		& = (\divr_g T) X + \frac{1}{2} (1 + f)^{-1} (\dd{f}(T X) - \dd{f}(X) \tr T + \dd{f}(T^* X) ) \\
		& = (\divr_g T) X + (1 + f)^{-1} \dd{f} \circ \left( \frac{T + T^*}{2} - \frac{\tr T}{2} \, \1 \right) (X) \\
		& = (\divr_g T) X + (1 + f)^{-1} \dd{f} \circ T^\textit{s}_0 (X) .
	\end{align*}
	From the first to the second line we used the relation found above between the Levi-\hsk Civita connections of $g$ and $h$; in the forth line $T^*$ is denoting the $g$-\hsk adjoint of $T$; in the last line we observed that the endomorphism $(T + T^*)/2$ is the $g$-\hsk symmetric part of $T$, and $(\tr T) / 2 \, \1$ is its full-\hsk trace part.
\end{proof}

Let $A_l$ and $A_r$ be the tensors
\[
A_l \defin \1 - J B, \qquad A_r \defin \1 + J B .
\]
Then the metrics $h_l$ and $h_r$ coincide with $h(A_l \cdot, A_l \cdot)$ and $h(A_r \cdot, A_r \cdot)$, respectively (see for instance the proof of Lemma \ref{lem:V_linearizes_Gauss_Codazzi}). If $\sigma$ is the real part of a $J$-\hsk holomorphic quadratic differential, then the tensors $A_l$ and $A_r$ are $h$-\hsk Codazzi ($B$ is $h$-\hsk Codazzi by Lemma \ref{lem:hol_quadr_diff_equivalence}, and $\1$ and $J$ are $\nabla^h$-\hsk parallel) and $h$-\hsk symmetric. In particular, we can express the Levi-\hsk Civita connection of $h_l$ and $h_r$ respectively as follows:
\begin{equation} \label{eq:levi_civita_left_right}
\nabla^{h_l}_X Y = A_l^{-1} \nabla^h_X (A_l Y) , \qquad \nabla^{h_r}_X Y = A_r^{-1} \nabla^h_X (A_r Y) .	
\end{equation}
These relations can be proved by checking that the connections defined on the right-\hsk hand sides are torsion free (which follows from $A_l$ and $A_r$ being $h$-\hsk Codazzi), and $h_l$- and $h_r$-\hsk symmetric, respectively (which follows from $A_l$ and $A_r$ being $h$-\hsk symmetric and from the description of $h_l$ and $h_r$ given above). See also \cite{krasnov_schlenker_minimal}.

\begin{lemma} \label{lem:divergence_left_right}
	Let $T$ be a smooth section of traceless endomorphisms of $T \Sigma$. Then
	\begin{align*}
		(\divr_{h_l} T)X & = (\divr_g T)X - f^{-1} \scal{\nabla^g_X \sigma}{J T^\textit{s}} \\
		(\divr_{h_r} T)X & = (\divr_g T)X + f^{-1} \scal{\nabla^g_X \sigma}{J T^\textit{s}}
	\end{align*}
	where $T^\textit{s}$ stands for the $g$-\hsk symmetric part of $T$. 
\end{lemma}

\begin{proof}
	Recalling relations \eqref{eq:det_f} and \eqref{eq:inverse_E-JB}, we have that
	\begin{equation} \label{eq:inverse_Alr}
		A_{l,r}^{-1} = \frac{1 + f}{2} (\1 \pm J B) . 
	\end{equation}
	In this expression and the ones that will follow, we consider the sign above in $\pm$ or $\mp$ to be the one appearing for the expression of $A_l$, and the one on the bottom for $A_r$. Being $\1$ and $J$ $\nabla^h$-\hsk parallel, we have $\nabla^h_X A_{l, r} = \mp J \nabla^h_X B$. Applying the expressions \eqref{eq:levi_civita_left_right} for the Levi-\hsk Civita connections of $h_l$ and $h_r$, respectively, we find
	\begin{align*}
		(\divr_{h_{l, r}} T)X & = \dd{x}^i \left( (\nabla^{l, r}_{\partial_i} T) X \right) = \dd{x}^i \left( \nabla^{l,r}_{\partial_i} (T X) - T (\nabla^{l,r}_{\partial_i} X) \right) \\
		& = \dd{x}^i \left(A_{l,r}^{-1} \nabla^h_{\partial_i} (A_{l,r} T X) - T A_{l,r}^{-1} \nabla^h_{\partial_i} (A_{l,r} X) \right) \\
		& = \dd{x}^i \left(A_{l,r}^{-1} (\nabla^h_{\partial_i} A_{l,r}) T X + A_{l,r}^{-1} A_{l,r} (\nabla^h_{\partial_i} T) X  + A_{l,r}^{-1} A_{l,r} T (\nabla^h_{\partial_i} X) + \right. \\
		& \left. - T A_{l,r}^{-1} (\nabla^h_{\partial_i} A_{l,r}) X - T A_{l,r}^{-1} A_{l,r} (\nabla^h_{\partial_i} X) \right) \\
		& = \dd{x}^i( A_{l,r}^{-1} (\nabla^h_{\partial_i} A_{l,r}) T X ) + (\divr_h T)X - \dd{x}^i( T A_{l,r}^{-1} (\nabla^h_{\partial_i} A_{l,r}) X ) \\
		& = (\divr_h T)X + \tr(A_{l,r}^{-1} (\nabla^h_{T X} A_{l,r}) - T A_{l,r}^{-1} (\nabla^h_{X} A_{l,r})) . \tag{$A_{l,r}$ $h_{l, r}$-\hsk Codazzi}
	\end{align*}
	Lemma \ref{lem:divergence_h} allows us to express the first term of the sum in terms of the divergence with respect to $g$, so now we focus on the other two terms. For the first one, we express it as follows:
	\begin{align*}
		\tr(A_{l,r}^{-1} (\nabla^h_{T X} A_{l,r})) & = \frac{1 + f}{2} \tr( (\1 \pm J B) (\nabla^h_{T X} A_{l,r})) \tag{rel. \eqref{eq:inverse_Alr}} \\
		& = \mp \frac{1 + f}{2} \tr((\1 \pm J B) J \nabla^h_{T X} B) \tag{$\nabla^h_X A_{l, r} = \mp J \nabla^h_X B$} \\
		& = - \frac{1 + f}{2} \tr(J B J  \nabla^h_{T X} B) \tag{$J \nabla^h_{T X} B$ traceless} \\
		& = - \frac{1 + f}{2} \tr(B \nabla^h_{T X} B) \tag{$B \in T_J \mathcal{J}(\Sigma)$} \\
		& = - (1 + f)^{-1} \dd{f}(T X) . \tag{$\dd{(\norm{B}^2)} = 2 (1 + f)^{-2} \dd{f}$}
	\end{align*}
	We proceed similarly for the third term:
	\begin{align*}
		\tr(T A_{l,r}^{-1} (\nabla^h_{X} A_{l,r})) & = \frac{1 + f}{2} \tr(T (\1 \pm J B) (\nabla^h_{X} A_{l,r}) ) \tag{rel. \eqref{eq:inverse_Alr}} \\
		& = \mp \frac{1 + f}{2} \tr(T (\1 \pm J B) J \nabla^h_{X} B ) \tag{$\nabla^h_X A_{l, r} = \mp J \nabla^h_X B$} \\
		& = \frac{1}{2} f^{-1} (\mp \tr(T J \nabla^g_X (g^{-1} \sigma)) - \tr(T J B J \nabla^g_X (g^{-1} \sigma) )) \tag{Lemma \ref{lem:derivative_B}} \\
		& = f^{-1} (\pm \scal{\nabla^g_X \sigma}{J T^\textit{s}} - (1 + f)^{-1} \scal{\nabla^g_X \sigma}{T^\textit{a} g^{-1} \sigma}) .
	\end{align*}
	In the last step, we expressed $B$ as $(1 + f)^{-1} g^{-1} \sigma$, and we observed that, if $T = T^a + T^s$, where $T^a$ and $T^s$ denote the $g$-\hsk anti-\hsk symmetric and $g$-\hsk symmetric parts of $T$, respectively, then $T^a$ does not contribute to the term $\tr(T J \nabla^g_X (g^{-1} \sigma))$ since $\nabla^g_X (g^{-1} \sigma)$ is traceless, and similarly $T^s$ does not contribute to the term $\tr(T J B J \nabla^g_X (g^{-1} \sigma) )$ by Lemma \ref{lem:product_in_TJ}, part \textit{ii)} (observe that $T^s$ is traceless since $T$ is, by hypothesis).
	
	Combining the relations obtained above with Lemma \ref{lem:divergence_h} we see that
	\begin{align*}
		(\divr_{h_{l, r}} T)X & = (\divr_h T)X + \tr(A_{l,r}^{-1} (\nabla^h_{T X} A_{l,r}) - A_{l,r}^{-1} T (\nabla^h_{X} A_{l,r})) \\
		& = (\divr_g T) X + (1 + f)^{-1} \dd{f} (T^\textit{s} X) - (1 + f)^{-1} \dd{f}(T X) + \\
		& - f^{-1} (\pm \scal{\nabla^g_X \sigma}{J T^\textit{s}} - (1 + f)^{-1} \scal{\nabla^g_X \sigma}{T^\textit{a} g^{-1} \sigma}) \\
		& = (\divr_g T) X - (1 + f)^{-1} \dd{f} (T^\textit{a} X) + (1 + f)^{-1} f^{-1} \scal{\nabla^g_X \sigma}{T^\textit{a} g^{-1} \sigma} + \\
		& \mp f^{-1} \scal{\nabla^g_X \sigma}{J T^\textit{s}} .
	\end{align*}
	Since the space of $g$-\hsk anti-\hsk symmetric endomorphisms of $T_p \Sigma$ has real dimension $1$, we can write $T^\textit{a} = u \, J$ for some smooth function $u$. Therefore
	\begin{align*}
		\scal{\nabla^g_X \sigma}{T^\textit{a} g^{-1} \sigma} & = u \, \scal{\nabla^g_X \sigma}{J g^{-1} \sigma} \\
		& = u \, \scall{\nabla^g_{J X} \sigma}{\sigma}_g \\
		& = \frac{u}{2} \, \dd{(\norm{\sigma}_g^2)}(J X) \\
		& = u \, f \dd{f}(J X) , \tag{$f^2 - 1 = \norm{\sigma}_g^2$} 
	\end{align*}
	where, from the first to the second line, we used Lemma \ref{lem:hol_quadr_diff_equivalence} part \textit{iv)}. Expressing again $T^\textit{a}$ as $u \, J$ in the relation we found above for $(\divr_{h_{l, r}} T)X$, and combining it with what just shown, we obtain the desired statement.
\end{proof}

We are finally ready to prove the last statement of Proposition \ref{prop:equivalent_def_subspace_V}:

\begin{proof}[Proof of the last statement] We will apply the previous lemma to $T = \dot{J}_l, \dot{J}_r$. To do so, we express the $g$-\hsk symmetric and $g$-\hsk anti-\hsk symmetric parts of $\dot{J}_l, \dot{J}_r$. From the relation described in the proof of Theorem \ref{thm:mappaM} (baby version) for the differential of the Mess homeomorphism $\mathcal{M}$, we see that $\dot{J}_l$ and $\dot{J}_r$ can be expressed in terms of $(\dot{J}, \dot{\sigma})$ as follows:
	\begin{equation} \label{eq:expression_dotJ_left_right}
		\dot{J}_{l,r} = \scal{\sigma}{Q^\pm} J \pm f Q^\pm ,
	\end{equation}
	where $Q^\pm = Q^\pm(\dot{J}, \dot{\sigma})$ is defined as in Proposition \ref{prop:equivalent_def_subspace_V} part \textit{iii)} (we will not need the actual definition of $Q^\pm$, but only the fact that $Q^\pm \in T_J \mathcal{J}(\Sigma)$ and the expression \eqref{eq:description_V3} for the equations of $V_{(J,\sigma)}$). In particular we have that the $g$-\hsk symmetric parts of $\dot{J}_l, \dot{J}_r$ are: 
	\begin{equation}\label{eq:symJlr}
	    (\dot{J}_{l,r})^\textit{s} = \pm f Q^\pm
	\end{equation}
	
	The very last ingredient needed is the following convenient way to express the terms $\dot{J}_l$ and $\dot{J}_r$:
	\begin{equation} \label{eq:relazione_a_caso}
		\dot{J}_{l,r} = \mp Q^\pm J J_{l, r} - \scal{\sigma}{J Q^\pm} \1 ,
	\end{equation}
	whose proof goes as follows:
	\begin{align*}
		\dot{J}_{l,r} & = \scal{\sigma}{Q^\pm} J \pm f Q^\pm \tag{rel. \eqref{eq:expression_dotJ_left_right}} \\
		& = \pm f Q^\pm + J Q^\pm g^{-1} \sigma - \scal{\sigma}{J Q^\pm} \1  \\
		& = \mp Q^\pm J (f J \pm g^{-1} \sigma) - \scal{\sigma}{J Q^\pm} \1 \tag{$J^2 = - \1$ and $Q^\pm \in T_J \mathcal{J}(\Sigma)$} \\
		& = \mp Q^\pm J J_{l,r} - \scal{\sigma}{J Q^\pm} \1 , \tag{rel. \eqref{eq:expression_mess_map}}
	\end{align*}
	where in the second line we applied Lemma \ref{lem:product_in_TJ} to $J Q^\pm, g^{-1} \sigma \in T_J \mathcal{J}(\Sigma)$. Now, applying Lemma \ref{lem:divergence_left_right}, we obtain
	\begin{align*}
		(\divr_{h_{l, r}} \dot{J}_{l,r})X & = (\divr_g \dot{J}_{l,r})X \mp \scal{\nabla^g_X \sigma}{J(\pm Q^\pm)} \tag{rel. \eqref{eq:symJlr}}\\
		& = \mp \divr_g (Q^\pm J J_{l, r})X - \dd{(\scal{\sigma}{J Q^\pm})}X - \scal{\nabla^g_X \sigma}{J Q^\pm} \tag{rel. \eqref{eq:relazione_a_caso}} .
	\end{align*}
Using once again Lemma \ref{lem:hol_quadr_diff_equivalence} part \textit{iv)}, this identity can be rewritten as
$$
\divr_{h_{l, r}} \dot{J}_{l,r}+\dd{(\scal{\sigma}{J Q^\pm})} = \mp \divr_g (Q^\pm J J_{l, r}) - \scal{\nabla^g_{J\bullet} \sigma}{ Q^\pm}.
$$
This finally concludes the proof of the last part of Proposition \ref{prop:characterization_tangent_space2}. By a straightforward computation using relations \eqref{eq:expression_mess_map} and \eqref{eq:expression_dotJ_left_right}, we can see that
\[
    \scal{\sigma}{JQ^{\pm}}=-\frac{\scall{[J_{l},J_{r}]}{\dot{J}_{l}}}{8(1-\scall{J_{l}}{J_{r}})}=\frac{\scall{[J_{l},J_{r}]}{\dot{J}_{r}}}{8(1-\scall{J_{l}}{J_{r}})} \ .
\]
\end{proof}


\section{Geometric interpretations} \label{sec:geom_inter}

In this section we conclude the study of the para-hyperK\"ahler structure on the deformation space $\mathcal{MGH}(\Sigma)$ for $\Sigma$ a closed surface of genus $\geq 2$, giving interpretations in terms of anti-de Sitter geometry to the elements that constitute the para-hyperK\"ahler structure $(\g,\i,\j,\k)$. As a byproduct, we will deduce that the symplectic forms $\omega_\i,\omega_\j,\omega_\k$ are non-degenerate and closed, which will conclude the proof of Theorem \ref{thm:parahyper_structure}. Finally, we study the relation between $\omega_{\i}$, $\omega_{\j}$, $\omega_{\k}$ and Goldman symplectic form $\Omega_{Gol}^{\B}$ on the $\PSL(2,\B)$-character variety.

\subsection{The cotangent bundle parametrization}
Recall from Section \ref{subsec:cotangent} that Krasnov and Schlenker introduced in \cite{krasnov_schlenker_minimal} a way of parametrizing the deformation space $\mathcal{MGH}(\Sigma)\cong\MS(\Sigma)$  via the cotangent bundle $T^{*}\Teich^\conf(\Sigma)$ to the Teichm\"uller space of $\Sigma$. Precisely, they produced a mapping-class group invariant homeomorphism
\[
    \mappa{\mathcal{F}}{\MS(\Sigma)}{T^* \Teich^\conf(\Sigma)}~,
\]
see Theorem \ref{thm:cotangent_parameterization}. 
Using this map, we can identify, up to a multiplicative factor, the natural symplectic structure on $T^* \Teich^\conf(\Sigma)$ with the complex symplectic form $\overline{\omega}_{\i}^\C$, where $\omega_\i^\C = \omega_{\j} + i \omega_{\k}$.

\begin{reptheorem}{thm:cotangent}[genus $\geq 2$]
Let $\Sigma$ be a closed oriented surface of genus $\geq 2$. Then 
$$\mathcal F^*(\mathcal I_{T^*\mathcal{T}(\Sigma)},\Omega^\C_{T^*\mathcal{T}(\Sigma)})=\left(-\i,-\frac{i}{2}\overline\omega_\i^\C\right)~,$$
where $\mathcal I_{T^*\mathcal{T}(\Sigma)}$ denotes the complex structure of $T^*\mathcal{T}(\Sigma)$ and $\Omega^\C_{T^*\mathcal{T}(\Sigma)}$ its complex symplectic form.
\end{reptheorem}

\begin{proof} The proof is an adaptation of the arguments of the genus one version of Theorem \ref{thm:cotangent}, proved in Section \ref{subsec:identificationT2}. A computation identical to Remark \ref{rmk:constant_hqd2} shows that if $(J_t)_t$ is a $1$-\hsk parameter family of complex structures on $\Sigma$, with $J_0 = J$, then the Beltrami differential of the identity map $\mappa{\id}{(\Sigma, J)}{(\Sigma, J_t)}$ is
    \[
    \nu_t = (\1 - J_t J)^{-1} (\1 + J_t J) 
    \]
    and that $\dot{\nu} = \frac{1}{2} \dot{J} J$. Hence given a pair $(\dot J,\dot\sigma)$ in our model of the tangent space $T_{(J,\sigma)}\MS_0(\Sigma)$ (see Proposition \ref{prop:equivalent_def_subspace_V} and Theorem \ref{thm:identification_tangent&V}), we have
    
    \[
    \dd\pi\circ \dd{\mathcal F}_{(J,\sigma)}(\dot{J}, \dot{\sigma}) =  \frac{1}{2} \dot{J} J~.
    \]
    
Now let $g_J$ be the Riemannian metric $\rho(\cdot, J \cdot)$, let $\{e_1, e_2 = J e_1\}$ be a local $g_J$-\hsk orthonormal frame and let $\phi = \sigma - i \, \sigma(\cdot, J \cdot)$ as usual. The same computation as in the genus one case, using the definition of the pairing in $T^*\mathcal{T}^\conf(\Sigma)$ (see \eqref{eq:pairing_Beltrami} and \eqref{eq:pairing2}), shows that
    \begin{align*}
        (\phi \bullet \dot{\nu})(e_1, e_2) & = \frac{1}{2 i} \, (\phi(\dot{\nu}(e_1), e_2) - \phi(\dot{\nu}(e_2), e_1)) \\
        & = - \frac{i}{2} \, \overline{\lambda_{(J,\sigma)}^\C(\dot{J}, \dot{\sigma})} .
    \end{align*}
    Hence we obtain
    \[
   \mathcal F^* \Omega_{T^* \Teich(\Sigma)}^\C =  -\frac{i}{2} (\omega_\mathbf{J} - i \, \omega_\mathbf{K}) = - \frac{i}{2} \overline\omega_\mathbf{I}^\C~,    \]
    which proves one part of the statement. For the pull-back of the complex structure $\mathcal I_{T^*\mathcal{T}(\Sigma)}$, one argues exactly as in the genus one case to conclude that $\mathcal F^* \mathcal I_{T^* \Teich(\Sigma)}=-\i$. 
\end{proof}

As a consequence, we immediately obtain:

\begin{corollary}\label{cor:Iintegrable_higher_genus}
The almost complex structure $\i$ on $\MS_0(\Sigma)$ is integrable; the 2-forms $\omega_\j$ and $\omega_\k$ are symplectic forms. 
\end{corollary}
\begin{proof}
Since $\mathcal I_{T^* \Teich(\Sigma)}$ is an integrable almost-complex structure and $\Omega_{T^* \Teich(\Sigma)}^\C$ is a complex symplectic form, the two statements follow immediately from Theorem \ref{thm:cotangent}. 
\end{proof}

\subsection{The Mess homeomorphism}

In Section \ref{subsec:parameterizations} we explained that, under the identification between $\mathcal{MS}(\Sigma)$ and $\mathcal{MGH}(\Sigma)$, Mess homeomorphism 
\[
    \mappa{\mathcal M^\conf}{\MS(\Sigma)}{\Teich^{\conf}(\Sigma)\times\Teich^{\conf}(\Sigma)}~
\]
is expressed by the formula of Lemma \ref{lem:mess_complex_str}, which is formally the same expression as the map $\mathcal{M}:T^{*}\mathcal{J}(\R^2) \rightarrow \mathcal{J}(\R^2) \times \mathcal{J}(\R^2)$ defined in Section \ref{subsec:Mess_homeo}. This implies that $\mappa{\mathcal M^\conf}{\MS(\Sigma)}{\Teich^{\conf}(\Sigma)\times\Teich^{\conf}(\Sigma)}$ is induced by the map (see Remark \ref{rmk:Mess_inf_dim}) that we introduced in the finite dimensional context. 

Recall also that $\mathcal{T}^\conf(\Sigma)\times \mathcal{T}^\conf(\Sigma)$ is naturally endowed with a para-complex structure 
$\mathcal{P}_{\mathcal{T}^\conf(\Sigma)\times \mathcal{T}^\conf(\Sigma)}$, which is the endomorphism of the cotangent bundle for which the integral submanifolds of the distribution of $1$-eigenspaces are the slices $\mathcal{T}^\conf(\Sigma)\times \{*\}$, and those for the $(-1)$-eigenspaces are the slices $\{*\}\times \mathcal{T}^\conf(\Sigma)$. Plus, it has a para-complex symplectic form
\[
    \Omega^{\mathbb B}_{\mathcal{T}^\conf(\Sigma)\times \mathcal{T}^\conf(\Sigma)}:=\frac{1}{2}(\pi_l^*\Omega_{WP}+\pi_r^*\Omega_{WP})+\frac{\tau}{2} (\pi_l^*\Omega_{WP}-\pi_r^*\Omega_{WP})
\]
where $\Omega_{WP}$ is the Weil-Petersson symplectic form and $\pi_l,\pi_r$ denote the projections on the left and right factor. Here we show the relation of these structures with the para-hyperK\"ahler structure $(\g,\i,\j,\k)$, via Mess' diffeomorphism.

\begin{reptheorem}{thm:mappaM}[genus $\geq 2$]
Let $\Sigma$ be a closed oriented surface of genus $\geq 2$. Then 
$$\mathcal M^*(\mathcal P_{\mathcal{T}(\Sigma)\times \mathcal{T}(\Sigma)},4\Omega^{\mathbb B}_{\mathcal{T}(\Sigma)\times \mathcal{T}(\Sigma)})=(\j,\omega_\j^{\mathbb B})~,$$
where $\mathcal P_{\mathcal{T}(\Sigma)\times \mathcal{T}(\Sigma)}$ denotes the para-complex structure of ${\mathcal{T}(\Sigma)\times \mathcal{T}(\Sigma)}$ and $\Omega^{\mathbb B}_{\mathcal{T}(\Sigma)\times \mathcal{T}(\Sigma)}$ its para-complex symplectic form.
\end{reptheorem}

\begin{proof} Let $W_{(J_{l},J_{r})}$ denote the image of $V_{(J,\sigma)}$ under $\dd\mathcal{M}$. Because $\divr_{h_{l}}\dot{J}_{l}$ and $\divr_{h_{r}}\dot{J}_{r}$ are exact $1$-forms by the last statement in Proposition \ref{prop:equivalent_def_subspace_V}, the vector space $W_{(J_{l},J_{r})}$ is $(\Omega_{J_{l}}\oplus (\pm \Omega_{J_{r}}))$-orthogonal to the tangent space to the orbit of $\Symp_{0}(\Sigma, \rho)$ (see \cite{donaldson2003moment} or Proposition \ref{prop:orthogonal_symp}). Moreover, by Lemma \ref{lem:V_transverse_orbit} and equivariance of the map $\mathcal{M}$, the space $W_{(J_{l},J_{r})}$ is in direct sum with the tangent space to the orbit. Therefore, $(W_{(J_{l},J_{r})},\frac{1}{4}(\Omega_{J_{l}}\oplus (\pm \Omega_{J_{r}})))$ is symplectomorphic to $(T_{[J_{l}]}\mathcal{T}^\conf(\Sigma)\times T_{[J_{r}]}\mathcal{T}^\conf(\Sigma), \pi_{l}^{*}\Omega_{\WP}\pm \pi_{r}^{*}\Omega_{\WP}))$. Finally, the computations in Section \ref{subsec:Mess_homeo} (see the proof of the baby version of Theorem \ref{thm:mappaM}) can be carried out word-by-word in this context and show that
\[
        \frac{1}{2}\mathcal{M}^*((\Omega_{J_{l}} \oplus \Omega_{J_{r}})+\tau(\Omega_{J_{l}} \oplus (-\Omega_{J_{r}})))=\omega_{\i}+\tau\omega_{\k}=\omega_\j^{\mathbb B}~.
\]
The fact that $\mathcal M^*\mathcal P_{\mathcal{T}(\Sigma)\times \mathcal{T}(\Sigma)}=\j$ then follows by the usual argument, as in the conclusion of the baby version of Theorem \ref{thm:mappaM}, at the end of Section \ref{subsec:Mess_homeo}.
\end{proof}

As a consequence, we obtain:

\begin{corollary}\label{cor:omega_ik_nondeg} The almost para-complex structure $\j$ on $\MS_0(\Sigma)$ is integrable; the 2-forms $\omega_\i$ and $\omega_\k$ are symplectic forms.
\end{corollary}
\begin{proof}
This follows immediately from Theorem \ref{thm:mappaM}, the integrability of $\mathcal P_{\mathcal{T}(\Sigma)\times \mathcal{T}(\Sigma)}$  and the closedness and non-degeneracy of $\Omega^{\mathbb B}_{\mathcal{T}(\Sigma)\times \mathcal{T}(\Sigma)}$.
\end{proof}

\begin{corollary}\label{cor:g_nondeg}
The metric $\g$ on $\MS_0(\Sigma)$ is non-degenerate.
\end{corollary}
\begin{proof}
The proof follows by observing (for instance) that $\g(\cdot,\cdot)=\omega_\j(\cdot,\j)$, together with the non-degeneracy of $\omega_\j$ and the invertibility of $\j$. 
\end{proof}

\subsection{The circle action on \texorpdfstring{$\MS(\Sigma)$}{MS(S)}}
We now consider a circle action on $\MS(\Sigma)$, which is simply defined, under the diffeomorphism 
\[
    \mappa{\mathcal{F}}{\MS(\Sigma)}{T^* \Teich^\conf(\Sigma)}~,
\]
by $e^{i\theta}\cdot ([J],q)=([J],e^{i\theta}q)$, for $q$ a holomorphic quadratic differential.  
As in Remark \ref{rmk:rotateB}, we see that the circle action is induced by the following expression in terms of pairs $(h,B)$:
\[
    R_{\theta}(J,\sigma)=(J,\cos(\theta)\sigma+\sin(\theta)\sigma(\cdot, J\cdot)) \ ,
\]
and it is easily checked that $R_{\theta}$ descends to the quotient $\mathcal{MS}(\Sigma)$. As done in the introduction, we denote by $\mathcal{C}_\theta$ the composition $\mathcal{C}\circ R_{\theta}$, for every $\theta \in S^{1}$. We then prove the following results:

\begin{reptheorem}{thm:circle}[genus $\geq 2$]
Let $\Sigma$ be a closed oriented surface of genus $\geq 2$. The circle action on $\mathcal{MGH}(\Sigma)$ is Hamiltonian with respect to $\omega_\i$, and satisfies 
$$R_\theta^*\g=\g\qquad R_\theta^*\omega_{\i}=\omega_{\i}\qquad R_\theta^*\omega_{\i}^\C=e^{-i\theta}\omega_{\i}^\C~.$$
A Hamiltonian is given by the function 
\[
   \mathcal A:\MS(\Sigma)\to\R\qquad \mathcal{A}([J,\sigma])=\int_{\Sigma} \left(1+\sqrt{1+\|\sigma\|_{J}^{2}}\right) \rho \ .
\]
that gives the area of the maximal surface.
\end{reptheorem}

\begin{remark} In \cite{bonsante2013a_cyclic} and \cite{bonsante2015a_cyclic}, the authors studied the landslide flow on $\mathcal{T}^{\conf}(\Sigma)\times \mathcal{T}^{\conf}(\Sigma)$ that in our notations corresponds to the $1$-parameter family of maps $\mathcal{M}\circ R_{\theta/2}\circ \mathcal{M}^{-1}$. They showed that the landslide flow is  Hamiltonian with respect to the symplectic form $\pi_{l}^{*}\Omega_{\WP}+ \pi_{r}^{*}\Omega_{\WP}$. Our Theorem \ref{thm:circle} recovers this result, including it in a more general context.
\end{remark}

\begin{reptheorem}{thm:potentialintro}[genus $\geq 2$]
Let $\Sigma$ be a closed oriented surface of genus $\geq 2$. Then the function $-4\mathcal A$ is a para-K\"ahler potential for  the para-K\"ahler structures $(\g,\j)$ and $(\g,\k)$.
\end{reptheorem}

The proofs are straightforward adaptations of those provided in Section \ref{subsec:circle_action_toy} for the genus one case. We only remark that we modified the natural Hamiltonian that one obtains by integrating the function $H$ over $\Sigma$, namely the function 
$$[J,\sigma]\mapsto\int_\Sigma f(\norm{\sigma}_J)\rho=\int_\Sigma \sqrt{1+\norm{\sigma}^2_J}\rho$$ by adding a constant, so that $\mathcal{A}([J,\sigma])$ can be interpreted as the area of the maximal surface in Anti-de Sitter space corresponding to the point $[J,\sigma]$ under the identification between $\mathcal{MS}(\Sigma)$ and the space of equivariant maximal surfaces in Anti-de Sitter space. Indeed, recalling that the first fundamental form of the maximal surface is given by the metric $h=(1+f(\norm{\sigma}_J))g_J$ where $g_J=\rho(\cdot,J\cdot)$, the area form of $h$ is $$dA_h=(1+f(\norm{\sigma}_J))dA_{g_J}=\left(1+\sqrt{1+\norm{\sigma}^2_J}\right)\rho~.$$

Theorem \ref{thm:mappaM} and Theorem \ref{thm:circle} have other direct consequences. Recalling from Section \ref{subsec:cc_and_circle} the definition of the map 
$$\mathcal C_\theta=\mathcal C\circ R_\theta:{\MS(\Sigma)}\to {\Teich(\Sigma)\times\Teich(\Sigma)}~,$$
 we see that
\begin{equation}\label{eq:pullbackCtheta}
\mathcal C_\theta^* (\mathcal P_{\mathcal{T}(\Sigma)\times \mathcal{T}(\Sigma)},4\Omega^{\mathbb B}_{\mathcal{T}(\Sigma)\times \mathcal{T}(\Sigma)})=(\cos(\theta)\k-\sin(\theta)\j ,\omega_\i-\tau(\cos(\theta)\omega_\j+\sin(\theta)\omega_\k))~.
\end{equation}
As an immediate consequence, we conclude the proofs of Theorems 
\ref{thm:mappaC} and \ref{thm:Htheta}. For the former, it suffices to observe that for $\theta=0$ the parameterization $\mathcal C_{\theta}=\mathcal C\circ R_\theta:{\MS(\Sigma)}\to {\Teich(\Sigma)\times\Teich(\Sigma)}$, given by the induced metric on the two Cauchy surfaces of constant curvature $-2$, is simply $\mathcal C=\mathcal C_{0}$.

\begin{reptheorem}{thm:mappaC}[genus $\geq 2$]
Let $\Sigma$ be a closed oriented surface of genus $\geq 2$. Then
$$\mathcal C^*(\mathcal P_{\mathcal{T}(\Sigma)\times \mathcal{T}(\Sigma)},4\Omega^{\mathbb B}_{\mathcal{T}(\Sigma)\times \mathcal{T}(\Sigma)})=(\k,\omega_\k^{\mathbb{B}})~,$$
where $\mathcal P_{\mathcal{T}(\Sigma)\times \mathcal{T}(\Sigma)}$ denotes the para-complex structure of ${\mathcal{T}(\Sigma)\times \mathcal{T}(\Sigma)}$ and $\Omega^{\mathbb B}_{\mathcal{T}(\Sigma)\times \mathcal{T}(\Sigma)}$ its para-complex symplectic form.
\end{reptheorem}

In particular, we deduce:
\begin{corollary} \label{cor:omega_j_nondeg} 
The almost para-complex structure $\k$ on $\MS_0(\Sigma)$ is integrable.
\end{corollary}

Finally, we have the proof of Theorem \ref{thm:Htheta}, which is expressed purely in terms of Teichm\"uller theory. Namely, recalling that the map 
$$\mathcal H_\theta:T^*\mathcal T(\Sigma)\to \mathcal T(\Sigma)\times \mathcal T(\Sigma)$$
associates to a pair $([J],q)$ the pair $(h_{(J,-e^{i\theta}q)},h_{(J,e^{i\theta}q)})$ of hyperbolic metrics on $\Sigma$, where $h_{(J,q)}$ has the property that the (unique) harmonic map $(\Sigma,J)\to(\Sigma,h)$ isotopic to the identity has Hopf differential equal to $q$. By Lemma \ref{lemma:mapH}, this map is identified to $\mathcal C_\theta^\hyp\circ\mathcal F^{-1}$. Using \eqref{eq:pullbackCtheta} and Theorem \ref{thm:cotangent}, we obtain Theorem \ref{thm:Htheta}:

\begin{reptheorem}{thm:Htheta}[genus $\geq 2$]
Let $\Sigma$ be a closed oriented surface of genus $\geq 2$. Then
$$\Im \mathcal H_\theta^*(2\Omega^{\mathbb B}_{\mathcal{T}(\Sigma)\times \mathcal{T}(\Sigma)})=-\Re (ie^{i\theta}\Omega^\C_{T^*\mathcal T(\Sigma)})~.$$
\end{reptheorem}

\subsection{Para-complex geometry of the \texorpdfstring{$\PSL(2,\B)$}{PSL(2,B)}-character variety} 

Let $\B$ be the algebra of para-complex numbers, i.e $\B=\R\oplus \tau\R$ with $\tau^{2}=1$. In this section we study the para-complex geometry of $\mathcal{MGH}(\Sigma)$ seen as a component of the $\PSL(2,\B)$-character variety. We show that multiplication by $\tau$ on $\B$ induces a para-complex structure on the $\PSL(2,\B)$-character variety that makes Goldman symplectic form $\Omega_{Gol}^{\B}$ para-holomorphic, and that the Goldman form $\Omega_{Gol}^\B$ coincides with the para-complex symplectic form $\omega_\j^{\B}$ up to a multiplicative factor (Corollary \ref{cor:holonomy}). Moreover, we give a formula for Goldman symplectic form based on anti-de Sitter geometry and show that the $\B$-valued Fenchel-Nielsen coordinates defined in \cite{Tambu_FN} are para-holomorphic Darboux coordinates for $\Omega_{Gol}^{\B}$.

\subsubsection{Para-complex structure on the character variety} Let us recall the construction of the isometry group of Anti-de Sitter space in terms of the para-complex numbers, following \cite{dancigerGT}. We denote by $\SL(2,\B)$ the set of $2$-by-$2$ matrices with coefficients in $\B$ and determinant $1$. Any matrix $A\in \SL(2,\B)$ can be written uniquely as $A=A_{+}e^{+}+A_{-}e^{-}$, where $A_{\pm} \in \SL(2,\R)$, $e^{\pm}=\frac{1\pm \tau}{2}$ (see Appendix \ref{appendixB}). The map 
\begin{align*}
    \SL(2,\B) &\rightarrow \SL(2,\R)\times \SL(2,\R)\\
     A &\mapsto (A_{+},A_{-})
\end{align*}
induces an isomorphism between $\PSL(2,\B)$ and $\dPSL$, where by $\PSL(2,\B)$ we mean
$$\PSL(2,\B)=\SL(2,\B)/\{\pm 1,\pm\tau\}~.$$
 We define the $\PSL(2,\B)$-character variety as follows:
 $$\chi(\Sigma, \PSL(2,\B))=\{\rho:\pi_{1}(\Sigma) \rightarrow \PSL(2,\B)\}/\PSL(2,\B)~,$$
 namely the set of conjugacy classes of representations $\rho:\pi_{1}(\Sigma) \rightarrow \PSL(2,\B)$. The aformentioned isomorphism between $\PSL(2,\B)$ and $\dPSL$ identifies $\chi(\Sigma, \PSL(2,\B))$ with $\chi(\Sigma, \PSL(2,\R)) \times \chi(\Sigma, \PSL(2,\R))$ by associating to $\rho$ the pair of representations $\rho_{\pm}:\pi_{1}(\Sigma) \rightarrow \PSL(2,\R)$ defined by the property $\rho(\gamma)=\rho_{+}(\gamma)e^{+}+\rho_{-}(\gamma)e^{-}$ for every $\gamma \in \pi_{1}(\Sigma)$. \\
 
 By the work of Mess (\cite{mess2007lorentz}), the moduli space $\mathcal{MGH}(\Sigma)$ is diffeomorphic, under the holonomy map, to a connected component of $\chi(\Sigma, \PSL(2,\B))$, corresponding to pairs of representations in $\PSL(2,\R)$ that are discrete and faithful, and that induce hyperbolic structures on $\Sigma$ compatible with the fixed orientation of $\Sigma$. Let us denote by $\chidf$ this connected component. By the work of Goldman \cite{Goldman_symplecticnature}, the tangent space of $\chidf$ at $[\rho] \in \chidf$ is isomorphic to the first cohomology group $H^{1}(\Sigma, \mathfrak{sl}_{2}(\B)_{\Ad\rho})$. We recall that elements of $H^{1}(\Sigma, \mathfrak{sl}_{2}(\B)_{\Ad\rho})$ are equivalence classes of closed $1$-forms on $\Sigma$ with values in the flat bundle $\mathfrak{sl}_{2}(\B)_{\Ad\rho}$ defined by 
\[
    \mathfrak{sl}_{2}(\B)_{\Ad\rho}=(\tilde{\Sigma}\times \mathfrak{sl}_{2}(\B))/\sim \ ,
\]
where $(\tilde{x}, v)\sim (\gamma \tilde{x}, \Ad(\rho(\gamma))v)$ for every $\tilde{x} \in \tilde{\Sigma}$, $\gamma \in \pi_{1}(\Sigma)$ and $v \in \mathfrak{sl}_{2}(\B)$. As usual, two $1$-forms are equivalent if their difference is exact. Here, the exterior differential is the $\B$-linear extension of the usual differential for $\mathfrak{sl}_{2}(\R)$-valued forms. \\

We can then endow $\chidf$ with a natural para-complex structure $\mathcal{T}$ that multiplies by $\tau$ an $\mathfrak{sl}_{2}(\B)_{\Ad\rho}$-valued $1$-form. 

\subsubsection{Goldman symplectic form} A general construction by Goldman endows every character variety of a semi-simple Lie group with a symplectic form $\Omega_{Gol}$ (\cite{Goldman_symplecticnature}). In the setting of $\PSL(2,\B)$ this can be obtained as follows. The pairing
\begin{align*}
        B:\mathfrak{sl}_{2}(\B)\otimes \mathfrak{sl}_{2}(\B) &\rightarrow \B \\
            (X,Y) &\mapsto \trace(XY)
\end{align*}
is a non-degenerate $\B$-bi-linear form that is invariant under conjugation. Pre-composing $B$ with the standard cup-product in co-homology, we obtain a bi-linear pairing
\begin{align*}
 H^{1}(\Sigma, \mathfrak{sl}_{2}(\B)_{\Ad\rho})&\times H^{1}(\Sigma, \mathfrak{sl}_{2}(\B)_{\Ad\rho}) \rightarrow \B \\
    ([\sigma \otimes \phi], [\sigma'\otimes \phi']) &\mapsto \int_{\Sigma}B(\phi,\phi')(\sigma\wedge \sigma')
\end{align*}
which is non-degenerate by Poincar\'e duality and skew-symmetric. By general arguments of Goldman (\cite{Goldman_symplecticnature}) and Atiyah-Bott (\cite{AB_Yang-Mills}) the resulting $\B$-valued $2$-form on $\chidf$, which we denote by $\Omega_{Gol}^\B$,  is closed. 

\begin{lemma}\label{lm:parahol_G} The $\B$-valued symplectic form $\Omega_{Gol}^\B$ is para-holomorphic with respect to $\mathcal{T}$. 
\end{lemma}
\begin{proof} Recall that a $\B$-valued $2$-form $\omega$ in a para-complex manifold $(M, \p)$ is para-complex if $\omega(X,\p Y)=\omega(\p X, Y)=\tau\omega(X,Y)$. 
In our setting, for every closed $\mathfrak{sl}_{2}(\B)_{\Ad\rho}$-valued 1-forms $\sigma\otimes \phi$ and $\sigma'\otimes \phi'$, we have
\begin{align*}
    \Omega_{Gol}^\B(\sigma\otimes \phi, \mathcal{T}(\sigma'\otimes \phi'))&=\Omega_{Gol}^\B(\sigma \otimes \phi, \sigma'\otimes \tau \phi')=\int_{\Sigma}B(\phi,\tau\phi')(\sigma\wedge \sigma')\\
    &=\tau\int_{\Sigma}B(\phi,\phi')(\sigma\wedge\sigma')=\tau\Omega_{Gol}^\B(\sigma \otimes \phi, \sigma'\otimes\phi') \ .
\end{align*}
Skew-symmetry and bi-linearity of $\Omega_{Gol}^\B$ imply that $\Omega_{Gol}^\B$ is para-complex. Because $\Omega_{Gol}^\B$ is para-complex and closed, and the exterior differential decomposes as $d=\bar{\partial}_{T}+\partial_{T}$ (see Appendix \ref{appendixB}), we deduce that $\bar{\partial}_{T}\Omega_{Gol}^\B=0$, hence $\Omega_{Gol}^\B$ is para-holomorphic.
\end{proof}

Because $\Omega_{Gol}^\B$ is $\B$-valued and symplectic, its real and imaginary parts are closed $2$-forms and thus define two symplectic structures on $\chidf$. The space  $\chidf$ is identified with the product $\mathcal T^\rep(\Sigma)\times \mathcal T^\rep(\Sigma)$, where $\mathcal T^\rep(\Sigma)$ is meant as the space of discrete and faithful representations with values in $\PSL(2,\R)$, and is therefore endowed with a \emph{real} Goldman form $\Omega_{Gol}^\R$, defined in the analogous way. It is known from the work of Goldman that, if 
$$\mathrm{hol}:\mathcal T^\hyp(\Sigma)\to T^\rep(\Sigma)$$
is the holonomy map for hyperbolic structures on $\Sigma$, then 
\begin{equation}\label{eq:goldmanthm}
\mathrm{hol}^*\Omega_{Gol}^\R=\Omega_{WP}~,
\end{equation} 
where $\Omega_{WP}$ is the Weil-Petersson symplectic form on $\mathcal T^\hyp(\Sigma)$. We now express $\Omega_{Gol}^\B$ in terms of the real Goldman forms on each component $\mathcal T^\rep(\Sigma)$.

\begin{proposition}\label{prop:real_im} Given a closed oriented surface of genus $\geq 2$, we have the following identity on $\chidf\cong \mathcal T^\rep(\Sigma)\times \mathcal T^\rep(\Sigma)$:
\[
    \Omega_{Gol}^\B=\frac{1}{2}(\pi_{l}^{*}\Omega_{Gol}^\R+ \pi_{r}^{*}\Omega_{Gol}^\R)+\frac{\tau}{2}(\pi_{l}^{*}\Omega_{Gol}^\R-\pi_{r}^{*}\Omega_{Gol}^\R).
\]
\end{proposition}
\begin{proof} The isomorphism $\PSL(2,\B)\cong \dPSL$ induces an isomophism of Lie algebras $\mathfrak{sl}_{2}(\B)\cong \mathfrak{sl}_{2}(\R)\times \mathfrak{sl}_{2}(\R)$ given by decomposing $X\in \mathfrak{sl}_{2}(\B)$ into $X=X_{+}e^{+}+X_{-}e^{-}$ with $X_{\pm} \in \mathfrak{sl}_{2}(\R)$. Moreover, the adjoint action of $\rho=\rho_-e^-+\rho_+e^+$ on $\mathfrak{sl}_{2}(\B)$ induces the action of $(\rho_+,\rho_-)$ on $\mathfrak{sl}_{2}(\R)\times \mathfrak{sl}_{2}(\R)$ given by the adjoint action on each factor. As a consequence, an $\mathfrak{sl}_{2}(\B)_{\Ad\rho}$-valued $1$-form $\sigma \otimes \phi$ can be uniquely written as $(\sigma \otimes \phi_{+}e^{+}) + (\sigma \otimes \phi_{-}e^{-})$, and $\sigma \otimes \phi_{\pm}$ are $\mathfrak{sl}_{2}(\R)_{\Ad\rho_{\pm}}$-valued $1$-forms. Finally, we observe that $\sigma \otimes \phi$ is closed (respectively exact) if and only if $\sigma \otimes \phi_{\pm}$ are closed (respectively exact). Therefore, by Lemma \ref{lm:parahol_G},
\begin{align*}
    &\Omega_{Gol}^\B(\sigma\otimes \phi, \sigma'\otimes \phi')=\Omega_{Gol}^\B\left(\sigma \otimes (\phi_{+}e^{+}+\phi_{-}e^{-}), \sigma'\otimes ( \phi'_{+}e^{+}+\phi'_{-}e^{-})\right)\\
    &=\frac{1}{4}\Omega_{Gol}^\B\left(\sigma \otimes( (\phi_{+}+\phi_{-})+\tau(\phi_{+}-\phi_{-})), \sigma'\otimes ((\phi_{+}'+\phi_{-}')+\tau(\phi_{+}'-\phi_{-}'))\right)\\
    &=\frac{1}{4}\left(\Omega_{Gol}^\R\left(\sigma \otimes (\phi_{+}+\phi_{-}), \sigma'\otimes (\phi_{+}'+\phi_{-}')\right)+\Omega_{Gol}^\R\left(\sigma \otimes (\phi_{+}-\phi_{-}), \sigma'\otimes (\phi_{+}'-\phi_{-}')\right)\right)\\&+ \frac{\tau}{4}\left(\Omega_{Gol}^\R\left(\sigma \otimes (\phi_{+}+\phi_{-}), \sigma'\otimes (\phi_{+}'-\phi_{-}')\right)+\Omega_{Gol}^\R\left(\sigma \otimes (\phi_{+}-\phi_{-}), \sigma'\otimes (\phi_{+}'+\phi_{-}')\right)\right)\\
    &=\frac{1}{2}\left(\Omega_{Gol}^\R(\sigma\otimes \phi_{+}, \sigma'\otimes \phi_{+}')+\Omega_{Gol}^\R(\sigma\otimes \phi_{-}, \sigma'\otimes \phi_{-}')\right)\\
    &+\frac{\tau}{2}\left(\Omega_{Gol}^\R(\sigma\otimes \phi_{+}, \sigma'\otimes \phi_{+}')-\Omega_{Gol}^\R(\sigma\otimes \phi_{-}, \sigma'\otimes \phi_{-}')\right)
\end{align*}
and the result follows.
\end{proof}

Now, recall that on $\mathcal{MGH}(\Sigma)$ we have the holonomy map 
$$\mathcal{H}ol:\mathcal{MGH}(\Sigma)\to\chidf~,$$
and moreover Mess homeomorphism 
$$\mathcal M^\hyp:\mathcal{MGH}(\Sigma)\to\mathcal T^\hyp(\Sigma)\times\mathcal T^\hyp(\Sigma)~.$$
The work of Mess, combined with the expression of $\mathcal M$ that we gave in \eqref{eq:left_right_metric} (see Theorem \ref{thm:mess_homeo})  showed that the left and right components of $\mathcal M$ coincide with the holonomy of the MGHC AdS manifold, under the identification $\mathcal T^\hyp(\Sigma)\cong \mathcal T^\rep(\Sigma)$. In other words, we have $\mathcal{H}ol=(\mathrm{hol},\mathrm{hol})\circ \mathcal M$. Combining Proposition \ref{prop:real_im}  with Theorem \ref{thm:mappaM}, and using Goldman's fundamental identity \eqref{eq:goldmanthm}, we obtain that $\mathcal{H}ol^*(\mathcal T, 4\Omega_{Gol}^\B)=\omega_\j^{\mathbb B}$. Since $\Omega_{Gol}^\B$ is para-complex with respect to $\mathcal T$, and $\omega_\j^{\mathbb B}$ is para-complex with respect to $\j$, the usual argument shows that $\mathcal{H}ol^*\mathcal T=\j$. This concludes the proof of the following result.

\begin{repcorollary}{cor:holonomy}
Let $\Sigma$ be a closed oriented surface of genus $\geq 2$. Then
$$\mathcal{H}ol^*(\mathcal T,4\Omega_{Gol}^\B)=(\j,\omega_\j^{\mathbb B})~.$$
\end{repcorollary}

\subsubsection{The para-complex Fenchel-Nielsen coordinates} Goldman symplectic forms on Teichm\"uller space and on the space of quasi-Fuchsian representations are intimately related to hyperbolic geometry because, for instance, the (real or complex) length and twist parameters (\cite{wolpert1983on_the_symplectic},\cite{Platis_complexFN}) provide Darboux coordinates. Here we show that an analogous result holds for $\Omega_{Gol}^{\B}$ using the $\B$-valued Fenchel-Nielsen coordinates introduced in \cite{Tambu_FN} to describe the deformation space of MGHC anti-de Sitter structures. 

We first introduce some facts on Anti-de Sitter geometry and recall the definition of these coordinates. The model of Anti-de Sitter space that we will use is simply the Lie group $\PSL(2,\R)$ endowed with the bi-invariant Lorentzian metric which is induced by $(1/8)$ the Killing form on the Lie algebra, where the factor $(1/8)$ serves to normalize the sectional curvature to be equal to $-1$. The group of orientation-preserving, time-preserving isometries is isomorphic to $\PSL(2,\R)\times\PSL(2,\R)$, acting by left and right multiplication, and the boundary of AdS space identifies to $\R\mathbb P^1\times\R\mathbb P^1$, in such a way that the action of the isometry group extends to the obvious component-wise action of $\PSL(2,\R)\times\PSL(2,\R)$ on $\R\mathbb P^1\times\R\mathbb P^1$.

Now, consider an isometry of the form $(\gamma_+,\gamma_-)\in \PSL(2,\R)\times\PSL(2,\R)$, where both $\gamma_-$ and $\gamma_+$ are loxodromic. Then $(\gamma_+,\gamma_-)$ acts on the boundary $\R\mathbb P^1\times\R\mathbb P^1$ by fixing four points, of which one is attracting and one repelling. It turns out (see for instance the first part of the proof of Lemma \ref{lem:tutto insieme} below) that there exists a spacelike geodesic connecting these two points (a priori the other possibility, that is however excluded by Lemma \ref{lem:tutto insieme}, is that the segment connecting the two points is lightlike and contained in the boundary). This spacelike geodesic is what we will call the principal axis. More formally:

\begin{definition}\label{defi:principal_axis}
Given a pair of loxodromic elements $(\gamma_+,\gamma_-)\in \PSL(2,\R)\times\PSL(2,\R)$, we call principal axis  of $(\gamma_+,\gamma_-)$ the unique spacelike geodesic $\mathrm{Axis}(\gamma_+,\gamma_-)$ of Anti-de Sitter space whose endpoints on $\R\mathbb P^1\times\R\mathbb P^1$ are $(\gamma_+^{rep},\gamma_-^{rep})$ and $(\gamma_+^{att},\gamma_-^{att})$, where $\gamma_\pm^{rep}$ and $\gamma_\pm^{att}$ denote the repelling and attracting fixed points on $\R\mathbb P^1$.
\end{definition}

Now fix a pair of pants decomposition $\mathcal{P}=\{\gamma_{1}, \dots, \gamma_{n}\}$ of $\Sigma$, for $n=(3/2)|\chi(\Sigma)|$. If $\rho:\pi_{1}(\Sigma) \rightarrow \PSL(2,\B)$ is the holonomy of a MGHC Anti-de Sitter manifold, then $\rho(\gamma_{j})$ is loxodromic, i.e. conjugated to a diagonal matrix, for every $\gamma_{j} \in \mathcal{P}$. The $\B$-valued length is defined as
\[
    \ell^{\B}_{\rho}(\gamma_{j})=2\arccosh\left(\frac{\trace(\rho(\gamma_{j}))}{2}\right) \ ,
\]
where the hyperbolic arccosine is computed using its power series expansion. This quantity is related to how a loxodromic isometry acts on anti-de Sitter space. Indeed, 
the real part of $\ell^{\B}_{\rho}(\gamma_{j})$ represents the translation length of $\rho(\gamma_{j})$ on the principal axis under $\rho(\gamma_{j})$, whereas the imaginary part is the rotation angle of $\rho(\gamma_{j})$ acting on the orthogonal of such a geodesic. The $\B$-valued lengths completely determine a $\PSL(2,\B)$-representation of a pair of pants. The $\B$-twist parameter $\tw^{\B}_{\rho}(\gamma_{j})$ indicates how the pair of pants are glued together along the boundary curve $\gamma_{j}$: the real part describes the shear parameter and the imaginary part the bending between the two pairs of pants in anti-de Sitter space.\\
It turns out that $\ell^{\B}_{\rho}(\gamma_{j})$ and $\tw^{\B}_{\rho}(\gamma_{j})$ are related to the classical Fenchel-Nielsen coordinates $\ell_{\bullet}$ and $\tw_{\bullet}$ as follows. Under the isomorphism $\PSL(2,\B) \cong \dPSL$, the holonomy $\rho:\pi_{1}(\Sigma)\rightarrow \PSL(2,\B)$ of a MGHC Anti-de Sitter structure corresponds to a pair of faithful and discrete representations $\rho_{\pm}:\pi_{1}(\Sigma)\rightarrow \PSL(2,\R)$, which are thus the holonomies of hyperbolic metrics on $\Sigma$. We have the following relations:
\[
 \ell^{\B}_{\rho}(\gamma_{j})=\ell_{\rho_{+}}(\gamma_{j})e^{+}+\ell_{\rho_{-}}(\gamma_{j})e^{-}=\left(\frac{ \ell_{\rho_{+}}(\gamma_{j})+\ell_{\rho_{-}}(\gamma_{j})}{2}\right)+\tau\left(\frac{ \ell_{\rho_{+}}(\gamma_{j})-\ell_{\rho_{-}}(\gamma_{j})}{2}\right) \ 
\]
and
\[
    \tw^{\B}_{\rho}(\gamma_{j})=\tw_{\rho_{+}}(\gamma_{j})e^{+}+\tw_{\rho_{-}}(\gamma_{j})e^{-}=\left(\frac{ \tw_{\rho_{+}}(\gamma_{j})+\tw_{\rho_{-}}(\gamma_{j})}{2}\right)+\tau\left(\frac{ \tw_{\rho_{+}}(\gamma_{j})-\tw_{\rho_{-}}(\gamma_{j})}{2}\right) \ .
\]

\begin{reptheorem}{thm:fenchel}
The $\mathbb B$-valued Fenchel-Nielsen coordinates are para-holomorphic for $\mathcal T$, and are Darboux coordinates with respect to the para-complex symplectic form $\Omega_{Gol}^\B$.
\end{reptheorem}
\begin{proof} Let us show the first statement. As a consequence of the definition of $\mathcal{T}$, $\mathcal T$ acts as the identity on the first factor on the tangent space to $\mathcal T^{\mathfrak{rep}}(\Sigma)\times \mathcal T^{\mathfrak{rep}}(\Sigma)$, and as minus the identity on the second factor. Hence for any curve $\gamma_{j} \in \mathcal{P}$, we have
\[
    \mathcal{T}\left(\frac{\partial}{\partial \ell_{\rho_{\pm}}^{j}}\right)=\pm \frac{\partial}{\partial \ell_{\rho_{\pm}}^{j}} \ \ \ \text{and} \ \ \ 
    \mathcal{T}\left(\frac{\partial}{\partial \tw_{\rho_{\pm}}^{j}}\right)=\pm \frac{\partial}{\partial \tw_{\rho_{\pm}}^{j}} \ .
\]
Therefore, 
\[
    d\ell_{\rho}^{\B}\left(\mathcal{T}\left(\frac{\partial}{\partial \ell_{\rho_{\pm}}^{j}}\right)\right)=\frac{\pm 1 + \tau}{2}
\]
and 
\[
    \tau d\ell_{\rho}^{\B}\left(\frac{\partial}{\partial \ell_{\rho_{\pm}}^{j}}\right)=\tau\left(\frac{1\pm \tau}{2}\right)=\frac{\pm 1 + \tau}{2} \ .
\]
On the other hand, it is clear that 
\[
    d\ell_{\rho}^{\B}\left(\mathcal{T}\left(\frac{\partial}{\partial \tw_{\rho_{\pm}}^{j}}\right)\right)=0=\tau d\ell_{\rho}^{\B}\left(\frac{\partial}{\partial \tw_{\rho_{\pm}}^{j}}\right) \ ,
\]
hence the $\B$-lengths are para-holomorphic. A similar computation holds for the $\B$-twist parameters.


For the second statement, let $\ell^{\B,j}_{\rho}$ and $\tw^{\B,j}_{\rho}$ denote the $\B$-length and $\B$-twist parameters along the curve $\gamma_{j} \in \mathcal{P}$. We use an analogous notation for the Fenchel-Nielsen coordinates associated to the representations $\rho_{\pm}$ into $\PSL(2,\R)$ such that $\rho=\rho_{+}e^{+}+\rho_{-}e^{-}$.  By Proposition \ref{prop:real_im}, and the fact that the classical Fenchel-Nielsen coordinates are Darboux for $\Omega_{WP}$, we have
\begin{align*}
    \Re(\Omega_{Gol}^\B)\left( \frac{\partial}{\partial \ell^{\B,i}_{\rho}}, \frac{\partial}{\partial \ell^{\B,j}_{\rho}}\right)=\frac{1}{2}\left(\Omega_{WP}\left( \frac{\partial}{\partial \ell^{i}_{\rho_{+}}}, \frac{\partial}{\partial \ell^{j}_{\rho_{+}}}\right)+\Omega_{WP}\left( \frac{\partial}{\partial \ell^{i}_{\rho_{-}}}, \frac{\partial}{\partial \ell^{j}_{\rho_{-}}}\right)\right)=0,
\end{align*}
and similarly
\begin{align*}
    \Imm(\Omega_{Gol}^\B)\left( \frac{\partial}{\partial \ell^{\B,i}_{\rho}}, \frac{\partial}{\partial \ell^{\B,j}_{\rho}}\right)=\frac{1}{2}\left(\Omega_{WP}\left( \frac{\partial}{\partial \ell^{i}_{\rho_{+}}}, \frac{\partial}{\partial \ell^{j}_{\rho_{+}}}\right)-\Omega_{WP}\left( \frac{\partial}{\partial \ell^{i}_{\rho_{-}}}, \frac{\partial}{\partial \ell^{j}_{\rho_{-}}}\right)\right)=0 \ .
\end{align*}
For the same reason
\[
 \Omega_{Gol}^{\B}\left( \frac{\partial}{\partial \tw^{\B,i}_{\rho}}, \frac{\partial}{\partial \tw^{\B,j}_{\rho}}\right)=0 \ .
\]
On the other hand, 
\begin{align*}
 \Re(\Omega_{Gol}^\B)\left( \frac{\partial}{\partial \ell^{\B,i}_{\rho}}, \frac{\partial}{\partial \tw^{\B,j}_{\rho}}\right)=\frac{1}{2}\left(\Omega_{WP}\left( \frac{\partial}{\partial \ell^{i}_{\rho_{+}}}, \frac{\partial}{\partial \tw^{j}_{\rho_{+}}}\right)+\Omega_{WP}\left( \frac{\partial}{\partial \ell^{i}_{\rho_{-}}}, \frac{\partial}{\partial \tw^{j}_{\rho_{-}}}\right)\right)=\delta_{i,j}
\end{align*}
and
\begin{align*}
 \Imm(\Omega_{Gol}^\B)\left( \frac{\partial}{\partial \tw^{\B,i}_{\rho}}, \frac{\partial}{\partial \ell^{\B,j}_{\rho}}\right)=\frac{1}{2}\left(\Omega_{WP}\left( \frac{\partial}{\partial \tw^{i}_{\rho_{+}}}, \frac{\partial}{\partial \ell^{j}_{\rho_{+}}}\right)-\Omega_{WP}\left( \frac{\partial}{\partial \tw^{i}_{\rho_{-}}}, \frac{\partial}{\partial \ell^{j}_{\rho_{-}}}\right)\right)=0
\end{align*}
We conclude that 
\begin{equation}\label{eq:FenchelDarboux}
 \Omega_{Gol}^{\B}=\sum_{j=1}^{n} d\ell^{\B,j}_{\rho}\wedge d\tw^{\B,j}_{\rho}
\end{equation}
as claimed.
\end{proof}

\subsubsection{The para-complex cosine formula}
Let us now move on to the proof of Theorem \ref{thm:goldmancosine}, namely a $\B$-valued version of the cosine formula for the Weil-Petersson symplectic form. For this purpose, 
it will suffice to focus on the case where there is at least a point of intersection. The fundamental geometric computation is contained in the following lemma.

\begin{lemma}\label{lem:tutto insieme} Let $(\alpha_+,\alpha_-)$ and $(\beta_+,\beta_-)$   be two pairs of loxodromic elements in $\PSL(2,\R)\times\PSL(2,\R)$. Suppose that the actions of $(\alpha_+,\beta_+)$ and $(\alpha_-,\beta_-)$ on $\R\mathbb P^1$ are topologically conjugated (i.e. there exists an orientation-preserving diffeomorphism $f$ of $\R\mathbb P^1$ such that $f\alpha_-f^{-1}=\alpha_+$ and $f\beta_-f^{-1}=\beta_+$) and that the axes of $\alpha_\pm$ and $\beta_\pm$ intersect in $\Hyp^2$ with counterclockwise angle equal to $\varphi_\pm\in [0,\pi)$. Denote $\tilde\alpha=\mathrm{Axis}(\alpha_+,\alpha_-)$ and $\tilde\beta=\mathrm{Axis}(\beta_+,\beta_-)$ the principal axes in Anti-de Sitter space. Then
\begin{enumerate}
    \item There is a unique complete timelike geodesic $\sigma$ that intersects both $\tilde{\alpha}$ and $\tilde{\beta}$ orthogonally. Call the points of intersection $q_{\tilde\alpha}$ and $q_{\tilde\beta}$. Orient $\sigma$ so that its tangent vector is future-directed.
    \item The signed timelike distance along $\sigma$ between $q_{\tilde\alpha}$ and $q_{\tilde\beta}$  equals $(\varphi_+-\varphi_-)/2\in (-\pi/2,\pi/2)$.
    \item The counterclockwise angle of intersection between $\tilde\alpha$ and the parallel transport of $\tilde\beta$ to $q_{\tilde\alpha}$ along the unique orthogonal timelike geodesic $\sigma$ equals $(\varphi_++\varphi_-)/2$.
\end{enumerate}
\end{lemma}

Let us clarify the meaning of the second statement. Given two points $p$ and $q$ in AdS space connected by a timelike geodesic, the timelike distance between $p$ and $q$ is the length of the shortest geodesic segment from $p$ and $q$. Recalling that timelike geodesics are closed and have length $\pi$, there are two such geodesic segments, whose sum equals $\pi$. If the points $p$ and $q$ are not antipodal (i.e. the timelike distance is not $\pi/2$), then we can define the \emph{signed} timelike distance $d(p,q)$, which is just the timelike distance introduced above with positive sign if the realizing geodesic segment is oriented towards the future, and with negative sign if it is oriented towards the past. 
We remark that $d(p,q)=-d(q,p)$.

\begin{proof}
Denote by $\widetilde\alpha_\pm$ the axes of $\alpha_\pm$ in $\Hyp^2$, and analogously $\tilde\beta_\pm$ are the axes of $\beta_\pm$. Applying the action of the group $\PSL(2,\R)\times\PSL(2,\R)$, we can assume that, in the upper half-space model of $\Hyp^2$, the repelling fixed point of $\alpha_\pm$ is $0$, and the attracting fixed point of $\alpha_\pm$ is $\infty$ (hence 
both axes $\tilde\alpha_-$ and $\tilde\alpha_+$ coincide with the geodesic $\ell$ connecting $0$ and $\infty$), and moreover the intersection point of the axes $\tilde\beta_-$ and $\tilde\beta_+$ with $\ell$ is $i$.
Then the set of order two isometries in $\PSL(2,\R)$ having fixed point on $\ell$ is an Anti-de Sitter spacelike geodesic with endpoints $(0,0)$ and $(\infty,\infty)$ in $\R\mathbb P^1\times\R\mathbb P^1$, and is therefore the principal axis $\tilde\alpha$ of $(\alpha_+,\alpha_-)$. Indeed, the Lie group and the Riemannian exponential maps coincide for a bi-invariant metric on a Lie group, and this is a left translate of a one-parameter group of hyperbolic transformations in $\PSL(2,\R)$. To see that the endpoints are precisely $(0,0)$ and $(\infty,\infty)$, one can apply the criterion in \cite[Lemma 3.2.2]{bonsante2020antide}; anyway this follows from the more general observation (\cite[Lemma 3.5.1]{bonsante2020antide}, that we will apply below) that the surface $\mathcal P$ of order two isometries is a totally geodesic spacelike plane in Anti-de Sitter space, whose boundary at infinity is the diagonal in $\R\mathbb P^1\times\R\mathbb P^1$, for which the map $\mathcal P\to\Hyp^2$ sending an order two element to its fixed point is an isometry, equivariant for the action of $\PSL(2,\R)$ (by conjugation on $\mathcal P$, and the obvious action on $\Hyp^2$). \emph{En passant}, we have showed that the principal axis in Definition \ref{defi:principal_axis} is well-defined. 

Now, by our assumption the axis $\tilde\beta_\pm$ is obtained by applying to $\tilde\alpha_\pm$ a rotation $R_{\varphi_\pm}$ of angle $\varphi_\pm$ fixing $i$. Since the actions are topologically conjugated, the attracting and repelling fixed points of $\alpha_\pm$ and $\beta_\pm$ appear with the same cyclic order on $\R\mathbb P^1$. This implies that either the rotations $R_{\varphi_+}$ and $R_{\varphi_-}$ both map attracting (resp. repelling) fixed points of $\alpha_\pm$ to attracting (resp. repelling) fixed points of $\beta_\pm$, or they both map the attracting fixed point to the repelling one and vice versa. This implies that the principal axis $\tilde\beta$ equals $(R_{\varphi_+},R_{\varphi_-})\tilde\alpha$. Now, set
\[
    d=\frac{\varphi_{+}-\varphi_{-}}{2} \ \ \ \text{and} \ \ \ \theta=\frac{\varphi_{+}+\varphi_{-}}{2} \ . 
\]
Observe that $\varphi_\pm\in (0,\pi)$, hence $\theta\in (0,\pi)$ and $d\in (-\pi/2,\pi/2)$. Then we have 
$$(R_{\varphi_+},R_{\varphi_-})=(R_{d},R_{-d})\circ (R_{\theta},R_{\theta})~.$$
By the above considerations on the totally geodesic spacelike plane $\mathcal P$ (in particular the equivariant isometry with $\Hyp^2$), $(R_{\theta},R_{\theta})$ acts as a rotation of angle $\theta$. Moreover, by a general fact that we prove in Proposition \ref{prop:parallel_lie_groups} below for completeness, the parallel transport along the geodesic $t\mapsto R_{2t}$ (which is parameterized by arclength, future-directed, and orthogonal to the totally geodesic plane $\mathcal P$ by a simple symmetry argument) equals the differential of the isometry   $(R_{t},R_{-t})$. In conclusion, we have $q_{\tilde\alpha}=R_\pi$, $q_{\tilde\beta}=(R_{d},R_{-d})q_{\tilde\alpha}$, so their signed timelike distance is $d$, and finally the angle at $q_{\tilde\alpha}$ between the parallel transport of $\tilde\beta$ at $q_{\tilde\alpha}$ and $\tilde{\alpha}$ equals $\theta$. This concludes the proof of the second and third items. 
\end{proof}

\begin{proposition}\label{prop:parallel_lie_groups}
Given a bi-invariant metric on a Lie group $G$ and $X,Y\in\mathfrak g$, the parallel transport of $Y$ along the geodesic $g_t=\exp(tX)$ at the point $g_t$ is given by
$(L_{g_{t/2}})_*(R_{g_{t/2}})_*(Y)$.
\end{proposition}
\begin{proof}
The Levi-Civita connection of a bi-invariant metric on $G$ is given by 
$$\nabla_V W=\frac{1}{2}(D^l_VW+D^r_VW)~,$$
where $D^l$ and $D^r$ are the left and right invariant connections, which are defined in the following way: given a path $\gamma(t)$ such that $\gamma(0)=g$ and $\gamma'(0)=V\in T_gG$, then $D^l_VW$ (resp. $D^r_VW$) is the derivative at time $0$ of the left (resp. right) translated of $W_{\gamma(t)}$ at $g$, namely $(L_{g\gamma(t)^{-1}})_*W_{\gamma(t)}$ (resp. $(R_{g\gamma(t)^{-1}})_*W_{\gamma(t)}$). 

Let us apply this to the geodesic $\gamma(t)=g_t=\exp(tX)$ and the vector field along $\gamma$ defined by $W_{g_t}=(L_{g_{t/2}})_*(R_{g_{t/2}})_*(Y)$. In order to check that $W$ is parallel along $g_t$, it suffices to check it at $t=0$ (namely that $\nabla_X W=0$), because multiplication on the left and on the right by $g_{t/2}$ is an isometry that preserves the vector field $W$ and the geodesic $\gamma$ by definition. We have 
\begin{align*}
    D^l_{X}W&=\left.\frac{d}{dt}\right|_{t=0}(L_{g_t^{-1}})_*(L_{g_{t/2}})_*(R_{g_{t/2}})_*(Y) \\
    &=\left.\frac{d}{dt}\right|_{t=0}(L_{g_{-t/2}})_*(R_{g_{t/2}})_*(Y)=-\mathrm{ad}_XY~.
\end{align*}
A similar computation shows $D^r_{X}W=\mathrm{ad}_XY$. Hence $\nabla_X W=0$, which concludes the claim.
\end{proof}

Motivated by Lemma \ref{lem:tutto insieme} above, we give the following definition. 

\begin{definition}
Given two spacelike geodesics $\tilde\alpha$ and $\tilde\beta$ in AdS space that admit a common orthogonal timelike  geodesic $\sigma$, let $q_{\tilde\alpha}$ and $q_{\tilde\beta}$ the intersection points of $\tilde\alpha$ and $\tilde\beta$ with $\sigma$. Then we define the $\B$-valued angle $d^{\B}(\tilde{\alpha}, \tilde{\beta})$ between $\tilde\alpha$ and $\tilde\beta$ as the para-complex number whose imaginary part equals $d(q_{\tilde\alpha},q_{\tilde\beta})$, and the real part equals the counterclockwise angle between the parallel transport of $\tilde\alpha$ at $q_{\tilde\beta}$ along $\sigma$, and $\tilde\beta$.
\end{definition}

\begin{remark} Let $d^{\B}(\tilde{\alpha}, \tilde{\beta})=\theta+\tau d$ be the $\B$-valued angle introduced above. If we invert the role of $\tilde{\alpha}$ and $\tilde{\beta}$,  clearly the real part of the $\B$-valued angle becomes $\pi-\theta$. The imaginary part instead only changes sign, so
\[
     d^{\B}(\tilde{\beta}, \tilde{\alpha})=(\pi-\theta)-\tau d \ .
\]
Hence 
$\cos(d^{\B}(\tilde{\beta}, \tilde{\alpha}))=-\cos(d^{\B}(\tilde{\alpha}, \tilde{\beta}))$ because
\begin{align*}
    \cos(\theta+\tau d)=\cos\left((\theta+d)e^{+} +(\theta-d)e^{-}\right) =\cos(\theta+d)e^{+}+ \cos(\theta-d)e^{-}
\end{align*}
where the last step is justified by the power series expansion of cosine. This is consistent with the skew-symmetry in the generalized cosine formula of Theorem \ref{thm:goldmancosine}.
\end{remark}

Hence Lemma \ref{lem:tutto insieme} can be restated by the following formula:
\begin{equation}\label{eq:formulone coseni}
   d^{\B}(\tilde{\alpha}, \tilde{\beta})=\frac{\varphi_{+}+\varphi_{-}}{2}+\tau\left(\frac{\varphi_{+}-\varphi_{-}}{2}\right) \ ,
\end{equation}

We are finally ready to prove:

\begin{reptheorem}{thm:goldmancosine} Let $\rho=(\rho_+,\rho_-):\pi_1(\Sigma)\to\PSL(2,\R)\times\PSL(2,\R)$ be the holonomy of a MGHC AdS manifold, and let $\alpha, \beta$ be non-trivial simple closed curves. Then 
\[
 \Omega_{Gol}^\B\left( \frac{\partial}{\partial \tw^{\B, \alpha}_{\rho}}, \frac{\partial}{\partial \tw^{\B, \beta}_{\rho}}\right)=\sum_{p \in \alpha \cap \beta} \cos(d^{\B}(\tilde{\alpha}_{\rho}, \tilde{\beta}_{\rho})) \ ,
\]
where $\tilde\alpha_\rho$ and $\tilde\beta_\rho$ are the principal axes of $\rho(\alpha)$ and $\rho(\beta)$ on AdS space. 
\end{reptheorem}

\begin{proof}
If the closed curves $\alpha, \beta \in \pi_{1}(\Sigma)$ are disjoint, then equality holds, as the LHS vanishes as a consequence of the formula \eqref{eq:FenchelDarboux}, and the RHS vanishes by definition. Let us suppose $\alpha$ and $\beta$ intersect, and let us
evaluate Goldman symplectic form on the twist deformations along $\alpha$ and $\beta$. As before, we write the representation $\rho:\pi_{1}(\Sigma)\rightarrow \PSL(2,\B)$ as $\rho=\rho_{+}e^{+}+\rho_{-}e^{-}$ with $\rho_{\pm}: \pi_{1}(\Sigma) \rightarrow \PSL(2,\R)$. For any $p \in \alpha \cap \beta$, we denote by $\varphi_{\pm}(p) \in (0. \pi)$ the angle between the geodesic representatives of $\alpha$ and $\beta$ for the hyperbolic metrics with holonomies $\rho_{\pm}$ measured from $\alpha$. By the work of Wolpert (\cite{wolpert1983on_the_symplectic}), we have
\begin{align*}
 \Re(\Omega_{Gol}^\B)\left( \frac{\partial}{\partial \tw^{\B, \alpha}_{\rho}}, \frac{\partial}{\partial \tw^{\B, \beta}_{\rho}}\right)&=\frac{1}{2}\left(\Omega_{WP}\left( \frac{\partial}{\partial \tw^{\alpha}_{\rho_{+}}}, \frac{\partial}{\partial \tw^{\beta}_{\rho_{+}}}\right)+\Omega_{WP}\left( \frac{\partial}{\partial \tw^{\alpha}_{\rho_{-}}}, \frac{\partial}{\partial \tw^{\beta}_{\rho_{-}}}\right)\right)\\
 &=\frac{1}{2}\sum_{p \in \alpha \cap \beta}\cos(\varphi_{+}(p))+\cos(\varphi_{-}(p)) 
\end{align*}
and similarly,
\begin{align*}
    \Imm(\Omega_{Gol}^\B)=\left( \frac{\partial}{\partial \tw^{\B, \alpha}_{\rho}}, \frac{\partial}{\partial \tw^{\B, \beta}_{\rho}}\right)&=\frac{1}{2}\left(\Omega_{WP}\left( \frac{\partial}{\partial \tw^{\alpha}_{\rho_{+}}}, \frac{\partial}{\partial \tw^{\beta}_{\rho_{+}}}\right)-\Omega_{WP}\left( \frac{\partial}{\partial \tw^{\alpha}_{\rho_{-}}}, \frac{\partial}{\partial \tw^{\beta}_{\rho_{-}}}\right)\right)\\
    &=\frac{1}{2}\sum_{p \in \alpha \cap \beta}\cos(\varphi_{+}(p))-\cos(\varphi_{-}(p)) \ .
\end{align*}
Therefore, 
\begin{align*}\label{eq:Goldman_gen}
    \Omega_{Gol}^\B\left( \frac{\partial}{\partial \tw^{\B, \alpha}_{\rho}}, \frac{\partial}{\partial \tw^{\B, \beta}_{\rho}}\right)&=\frac{1}{2}\sum_{p \in \alpha \cap \beta}(\cos(\varphi_{+}(p))+\cos(\varphi_{-}(p))+\tau(\cos(\varphi_{+}(p))-\cos(\varphi_{-}(p))) \nonumber \\
    &=\sum_{p \in \alpha \cap \beta} \cos(\varphi_{+}(p))e^{+}+\cos(\varphi_{-}(p))e^{-} \\
    &=\sum_{p \in \alpha \cap \beta} \cos(\varphi_{+}(p)e^{+}+\varphi_{-}(p)e^{-}) \\ 
    &=\sum_{p \in \alpha \cap \beta} \cos\left(\frac{\varphi_{+}+\varphi_{-}}{2}+\tau\frac{\varphi_{+}-\varphi_{-}}{2}\right) \\
    &=\sum_{p \in \alpha \cap \beta} \cos(d^{\B}(\tilde{\alpha}_{\rho}, \tilde{\beta}_{\rho})),
\end{align*}
where the second to last step can be formally justified by considering the definition of cosine as power series, and the last step is obtained by applying Equation \eqref{eq:formulone coseni}. Indeed we can apply Lemma \ref{lem:tutto insieme}, because the representations $\rho_-$ and $\rho_+$ are both in the Teichm\"uller component $\mathcal T^{\mathfrak{rep}}(\Sigma)$, hence their actions on $\R\mathbb P^1$ are topologically conjugated. The proof is then concluded.
\end{proof}


\section{Symplectic reduction}\label{sec:inf_dim_reduction}
In this section we explain the process that led us to the definition of the para-hyperK\"ahler structure on $\mathcal{MS}_{0}(\Sigma,\rho)$ and to the explicit model of its tangent space described in Section \ref{sec:para_hyperkahler_str_on_MS}.
The main tool we use is Donaldson's construction of moment maps on infinite dimensional spaces.

\subsection{The moment maps on \texorpdfstring{$T^{*}\mathcal{J}(\R^{2})$}{T*J(R2)} } \label{subsec:moment_maps_toy}

To apply Donaldson's construction, which we recall in the next section, we first need to provide moment maps for the action of $\PSL(2,\R)$ on $T^{*}\mathcal{J}(\R^{2})$, with respect to the symplectic forms that we introduced in Section \ref{subsec:parahyperkahler_toy}.

Recalling that $\Lsl_2(\R) = \set{X \in \End(\R^2) \mid \tr X = 0}$, we define the maps $\eta_\i,\eta_\j,\eta_\k:T^* \mathcal{J}(\R^2)\to\Lsl_2(\R)^*$:
\begin{align*}
    \eta_\i&(J, \sigma)=f(\norm{\sigma}_J) \tr(J \cdot) \\
    \eta_\j&(J, \sigma) = \scal{\sigma}{[\cdot, J]}_J = - \tr(g_J^{-1} \sigma J \cdot)\\
    \eta_\k&(J, \sigma)=\scal{\sigma}{J [\cdot, J]}_J = \tr(g_J^{-1} \sigma \cdot)
\end{align*}
where $[X,Y] = X Y - Y X \in \Lsl_2(\R)$. Observe that, for every $X \in \Lsl_2(\R)$, the element $[X,J]$ belongs to $T_J \mathcal{J}(\R^2)$.

\begin{theorem}
The maps $\eta_\mathbf{I}$, $\eta_\mathbf{J}$ and $\eta_\mathbf{K}$ are moment maps for the action of $\PSL(2,\R)$ over $T^* \mathcal{J}(\R^2)$ with respect to the symplectic structures $\omega_\mathbf{I}$, $\omega_\mathbf{J}$ and $\omega_\mathbf{K}$ introduced in \eqref{eq:definition_omegaI_toy}, \eqref{eq:definition_omegaJ_toy} and \eqref{eq:definition_omegaK_toy} respectively.
\end{theorem}

\begin{proof}
We start by noticing some properties of our action. If $\mappa{\varphi_A}{T^* \mathcal{J}(\R^2)}{T^* \mathcal{J}(\R^2)}$ denotes the transformation $(J, \sigma) \mapsto A \cdot (J, \sigma)$, then a simple computation shows that, for every $A \in \PSL(2,\R)$ and $X \in \Lsl_2(\R)$,
\begin{equation} \label{eq:infinitesimal_vector1}
V_X (\varphi_A(J, \sigma)) = \dd(\varphi_A)_{(J,\sigma)}(V_{\Ad(A^{-1}) X}(J, \sigma)) ,
\end{equation}
where $V_X (J,\sigma) \defin \dv{t} \exp(tX) \cdot (J, \sigma) |_{t = 0} \in T_{(J,\sigma)} T^* \mathcal{J}(\R^2)$. Actually, this relation is true whenever a Lie group acts by diffeomorphisms on a smooth manifold. For future convenience, we give also an explicit description of $V_X$:
\begin{equation} \label{eq:infinitesimal_vector2}
V_X(J, \sigma) = \left. \dv{t} (\exp(tX) J \exp(-tX), \exp(-tX)^* \sigma) \right|_{t = 0} = ([X,J], - \sigma(X \cdot, \cdot) - \sigma(\cdot, X \cdot)) .
\end{equation}
We also notice that the action of $\PSL(2,\R)$ on $\mathcal{J}(\R^2)$ is by biholomorphisms with respect to the complex structure $\mathcal{I}$ (defined in Section \ref{subsec:linear_almost_complex}), and its extension on $T^* \mathcal{J}(\R^2)$ is natural, in the sense that it preserves the complex Liouville form $\lambda^\C$ of $T^* \mathcal{J}(\R^2)$ (see Section \ref{subsec:liouville_form} for the definition of $\lambda^\C$). This in particular implies that $\Dlie_{V_X} \lambda^\C = 0$ for every $X \in \Lsl_2(\R)$. 

\paragraph*{Property (i) from Definition \ref{def:moment_map}.} 

Using the relation $\norm{\sigma}_J = \norm{A \cdot \sigma}_{A \cdot J}$ (see Lemma \ref{lem:scalar_prod_invariance}) and the invariance of the trace by conjugation, it is straightforward to check that the maps $\eta_\mathbf{I}$, $\eta_\mathbf{J}$ and $\eta_\mathbf{K}$ are $\Ad^*$-\hsk equivariant. 

\paragraph*{Property (ii) from Definition \ref{def:moment_map}.}

From relation (\ref{eq:infinitesimal_vector2}) and the definition of $\lambda^\C$, the maps $\eta_\mathbf{J}$ and $\eta_\mathbf{K}$ satisfy:
\[
(\eta_\mathbf{J} + i \eta_\mathbf{K})_{(J,\sigma)}(X) = (\iota_{V_X} \lambda^\C)_{(J,\sigma)} .
\]
This relation and the observations above are enough to show that $\eta_\mathbf{J}$, $\eta_\mathbf{K}$ are moment maps with respect to the symplectic structures $\omega_\mathbf{J}$, $\omega_\mathbf{K}$, respectively. Indeed we have:
\begin{align*}
\dd(\eta_\mathbf{J}^X + i \eta_\mathbf{K}^X) & = \dd( \iota_{V_X} \lambda^\C) \\
& = \Dlie_{V_X} \lambda^\C - \iota_{V_X} \dd{\lambda^\C} \tag{Cartan's formula} \\
& = \iota_{V_X} \omega^\C \tag{def. of $\omega^\C$ and $\Dlie_{V_X} \lambda^\C = 0$} \\
& = \iota_{V_X} \omega_\mathbf{J} + i \, \iota_{V_X} \omega_\mathbf{K} . \tag{$\omega^\C = \omega_\mathbf{J} + i \omega_\mathbf{K}$}
\end{align*}

It remains to check that, for every $X \in \Lsl_2(\R)$, we have $\iota_{V_X} \omega_\mathbf{I} = \dd{\eta_\mathbf{I}^X}$. To see this, we show that
\begin{equation}\label{eq:derivative_norm_sigma}
    (\norm{\sigma}_J^2)' = 2 \scall{\sigma}{\dot{\sigma}_0}_J .
\end{equation}
Indeed, 
\begin{align*}
(\norm{\sigma}_J^2)' & = \frac{1}{2} \tr((g_J^{-1} \sigma)^2)' = \tr(g_J^{-1} \sigma (g_J^{-1} \sigma)') \\
& = \tr(g_J^{-1} \sigma (- g_J^{-1} \dot{g}_J g_J^{-1} \sigma + g_J^{-1} \dot{\sigma})) \\
& = \tr((g_J^{-1} \sigma)^2 J \dot{J}) + \tr(g_J^{-1} \sigma g_J^{-1} \dot{\sigma}_0) \tag{rel. \eqref{eq:derivative_scal_prod} and $\tr_{g_J} \sigma = 0$} \\
& = \norm{\sigma}_J^2 \tr(J \dot{J}) + \tr(g_J^{-1} \sigma g_J^{-1} \dot{\sigma}_0) \tag{Lemma \ref{lem:product_in_TJ}} \\
& = 2 \scall{\sigma}{\dot{\sigma}_0}_J . \tag{$\tr(J \dot{J}) = 0$}
\end{align*}
Therefore we have:
\[
\dd{\eta_\mathbf{I}^X}(\dot{J}, \dot{\sigma}) = \frac{\scall{\sigma}{\dot{\sigma}_0}_J}{f(\norm{\sigma}_J)} \tr(JX) + f(\norm{\sigma}_J) \tr(\dot{J} X) .
\]

On the other side, we need to determine $(\iota_{V_X} \omega_\mathbf{I})(\dot{J}, \dot{\sigma})$. Let $X^*$ denote the adjoint of $X$ with respect to the metric $g_J$. From the definition of $\omega_\mathbf{I}$ in equation \eqref{eq:definition_omegaI_toy} and the expression \eqref{eq:infinitesimal_vector2}, we see that:
\begin{align} \label{eq:omega_eats_VX}
\begin{split}
(\iota_{V_X} \omega_\mathbf{I})(\dot{J}, \dot{\sigma}) & = - f(\norm{\sigma}_J) \scall{[X, J]}{J \dot{J}}_J - \frac{1}{f(\norm{\sigma}_J)} \scall{(\sigma(X \cdot, \cdot) + \sigma(\cdot, X \cdot))_0}{\dot{\sigma}_0(\cdot, J \cdot)}_J \\
& = - \frac{f(\norm{\sigma}_J)}{2} \tr([X, J] J \dot{J}) - \frac{1}{2 f(\norm{\sigma}_J)} \tr(g_J^{-1} \dot{\sigma}_0 J (g_J^{-1} \sigma X + X^* g_J^{-1} \sigma)_0) .
\end{split}
\end{align}
In order to simplify the last term of the expression, we decompose the endomorphism $X$ as sum of its skew-symmetric part $X^a = - \frac{\tr(JX)}{2} J$ and its symmetric and traceless part $X^s$ (recall that $X$ is traceless, since it lies in $\Lsl_2(\R)$). A simple application of Lemma \ref{lem:product_in_TJ} shows that
\[
(g_J^{-1} \sigma X^s + X^s g_J^{-1} \sigma)_0 = 0 ,
\]
while the term in the skew-symmetric part contributes with
\[
(g_J^{-1} \sigma X^a + (X^a)^* g_J^{-1} \sigma)_0 = (g_J^{-1} \sigma X^a - X^a g_J^{-1} \sigma)_0 = \tr(JX) \, J g_J^{-1}\sigma .
\]
As a result we obtain $(g_J^{-1} \sigma X + X^* g_J^{-1} \sigma)_0 = \tr(JX) \, J g_J^{-1}\sigma$. Making use of this identity in relation \eqref{eq:omega_eats_VX}, we find that
\begin{align*}
(\iota_{V_X} \omega_\mathbf{I})(\dot{J}, \dot{\sigma}) & =  - \frac{f(\norm{\sigma}_J)}{2} \tr([X, J] J \dot{J}) - \frac{\tr(JX)}{2 f(\norm{\sigma}_J)} \tr(g_J^{-1} \dot{\sigma}_0 J^2 g_J^{-1}\sigma) \\
 & = f(\norm{\sigma}_J) \tr(\dot{J} X) + \frac{\scall{\sigma}{\dot{\sigma}_0}_J}{f(\norm{\sigma}_J)} \tr(JX) ,
\end{align*}
which proves the desired equality $(\iota_{V_X} \omega_\mathbf{I})(\dot{J}, \dot{\sigma}) = \dd{\eta_\mathbf{I}^X}(\dot{J}, \dot{\sigma})$.
\end{proof}

\subsection{Donaldson's construction}

Let us now recall briefly the setting of Section \ref{sec:para_hyperkahler_str_on_MS}. We defined $T^* \mathcal{J}(\Sigma)$ as the space of smooth sections of the bundle
\[
P(T^* \mathcal{J}(\R^2)) \longrightarrow \Sigma ,
\]
and an element of $T^* \mathcal{J}(\Sigma)$ identifies with a pair $(J,\sigma)$, in which $J$ is a complex structure on $\Sigma$, and $\sigma$ is a symmetric and $g_J$-\hsk traceless $2$-\hsk tensor. Moreover, a tangent vector $(\dot{J}, \dot{\sigma})$ at $(J,\sigma)$ can be considered as the data of:
\begin{itemize}
    \item a section $\dot{J}$ of $\End(T \Sigma)$ satisfying $\dot{J} J + J \dot{J} = 0$;
    \item a symmetric $2$-\hsk tensor $\dot{\sigma}$ satisfying 
    $\dot{\sigma} = \dot{\sigma}_0 - \scal{\sigma}{J \dot{J}}g$
\end{itemize}
We will often denote by $s=(J,\sigma)$ an element of $T^* \mathcal{J}(\Sigma)$, and by $\dot s$ a tangent vector.

Now, given an $\SL(2,\R)$-invariant symplectic form $\omega$ on $T^* \mathcal{J}(\R^2)$, every vertical space of $P(T^* \mathcal{J}(\R^2))$ inherits a symplectic structure, which we denote by $\hat{\omega}_{s(p)}$. In particular, given $\dot{s}, \dot{s}' \in T_s T^* \mathcal{J}(\Sigma)$ two tangent vectors, we define
\begin{equation}\label{eq:inf_dim_symp_form}
\omega_s(\dot{s}, \dot{s}') \defin \int_\Sigma \hat{\omega}_s(\dot{s}, \dot{s}') \, \rho . 
\end{equation}
This gives a formal symplectic structure on $T^* \mathcal{J}(\Sigma)$, which is preserved by the action of $\Symp_0(\Sigma,\rho)$. Suppose the natural $\SL(2,\R)$-action on $T^* \mathcal{J}(\R^2)$ is Hamiltonian, with moment map
$\eta:T^* \mathcal{J}(\R^2)\to\Lsl_2(\R)^*$. 
Given any section $s\in T^* \mathcal{J}(\Sigma)$, the moment map $\eta$ induces a section $\eta_s$ of the bundle $\End_0(T \Sigma)^*$. Then the action of $\Ham(\Sigma, \rho)<\Symp_0(\Sigma,\rho)$ is Hamiltonian, according to the following result of Donaldson.

\begin{theorem}[Donaldson's map, {\cite[Theorem 9]{donaldson2003moment}}] \label{thm:donaldson_maps}
    Let $\nabla$ be any torsion-free connection on $\Sigma$ satisfying $\nabla \rho=0$. Define the map $\mappa{\mu}{T^* \mathcal{J}(\Sigma)}{\Lambda^2(\Sigma)}$ as follows:
    \[
    \mu(s) = \hat{\omega}(\nabla_\bullet s, \nabla_\bullet s) + \scal{\eta_s}{R^\nabla} - \dd(c(\nabla_\bullet \eta_s )) .
    \]
    Then 
    \begin{enumerate}[i)]
        \item $\mu(s)$ is closed for every $s \in T^* \mathcal{J}(\Sigma)$;
        \item $\mu$ is equivariant with respect to the action of $\Ham(\Sigma,\rho)$;
        \item Given $V$ a vector field in $\Lham(\Sigma,\rho)$, and $\beta_V$ a primitive of $\iota_V \rho$, the differential of the map
        \[
        T^* \mathcal{J}(\Sigma) \ni s \longmapsto \int_\Sigma \beta_V \mu(s)  \in \R
        \]
        equals
        \[
        \omega_s(\dot{s}, \Dlie_V s) = \int_\Sigma \hat{\omega}(\dot{s}, \Dlie_V s) \, \rho .
        \]
    \end{enumerate}
\end{theorem}

Let us clarify the notation of the theorem above. We set $\hat\omega(\nabla_\bullet s, \nabla_\bullet s)$ to be the $2$-\hsk form on $\Sigma$ given by
\[
\hat\omega(\nabla_\bullet s, \nabla_\bullet s)(u,v) \defin \hat\omega(\nabla_u s, \nabla_v s) .
\]
Moreover, we define
\[
c(\nabla_\bullet \eta_s)(v) \defin \sum_j \scal{\nabla_{e_j} \eta_s}{(v \otimes e_j^*)_0} \ ,
\]
where $(e_j)_j$ is a local orthonormal frame and $(e_j^*)_j$ is the associated dual frame in $T^* \Sigma$. In particular $c(\nabla_\bullet \eta_s) \in \Gamma(T^* \Sigma)$ is a $1$-\hsk form on $\Sigma$, because $\nabla_\bullet \eta_s \in \Gamma(T^* \Sigma \otimes \End_0(T \Sigma)^*)$. 

Finally, the curvature tensor $R^\nabla$ of the torsion free connection $\nabla$ is defined as
\[
R^\nabla(U,V) W \defin \nabla_V \nabla_U W - \nabla_U \nabla_V W - \nabla_{[V, U]} W \ ,
\]
for all tangent vector fields $U, V, W$ on $\Sigma$. Since $R^\nabla(U,V) W = - R^\nabla(V, U) W$, the tensor $R^\nabla$ can be considered as a section of $\Lambda^2 (\Sigma) \otimes \End_0(T \Sigma)$. In particular, for every $U$, $V$, we can evaluate the tensor $\eta_s$ on $R^\nabla(U,V) \in \End_0(T \Sigma)$. This determines a $2$-\hsk form on $\Sigma$, that will be denoted by $\scal{\eta_s}{R^\nabla}$.

\begin{remark}\label{rmk:donald}
Donaldson's construction is actually much more general. It can be applied to the context of an $n$-dimensional manifold endowed with a volume form (which in our case is $\Sigma$), and a symplectic manifold $X$, together with a moment map for an action of $\SL(n,\R)$ on $X$ (which in our case is $T^*\mathcal J(\R^2))$. Then one constructs a bundle $P(X)$ similarly to our case, and one obtains a moment map for the action of the group of exact diffeomorphisms (which in dimension two correspond to Hamiltonian diffeomorphisms) on the space of sections of the bundle $P(X)$. The general statement of Donaldson's theorem can be found in \cite[Theorem 9]{donaldson2003moment}, see also \cite{trautwein2019hyperkahler}. We decided  to state the result only in the situation of our paper, since this reduces remarkably the necessary preliminaries.  
\end{remark}

We also recall the following fact, which of course can be stated in much larger generality as explained in Remark \ref{rmk:donald} above:

\begin{lemma}[{\cite[Lemma 13]{donaldson2003moment}}] \label{lem:donaldson_pullback_form}
    There exists a natural closed $2$-\hsk form $\hat{\omega}_{P(T^* \mathcal{J}(\R^2))}$ on $P(T^* \mathcal{J}(\R^2))$ such that, for every section $s \in \Gamma(\Sigma,P(T^* \mathcal{J}(\R^2)))$ we have
    \[
    s^* \hat{\omega}_{P(T^* \mathcal{J}(\R^2))} = \hat{\omega}(\nabla_\bullet s, \nabla_\bullet s) + \scal{\eta_s}{R^\nabla} .
    \]
    In particular, since $T^*\mathcal J(\R^2)$ is contractible, the de Rham cohomology class of $\mu(s)$ in $H^2(\Sigma)$ is independent of the chosen section.
\end{lemma}

\subsection{The moment maps on \texorpdfstring{$T^* \mathcal{J}(\Sigma)$}{T*J(R2)}}

In Section \ref{sec:para_hyperkahler_str_on_MS}, we defined a formal para-hyperK\"ahler structure $(\g, \i, \j, \k)$ on $T^{*}\mathcal{J}(\Sigma)$. Note that the corresponding symplectic forms
\begin{align}
    (\omega_\mathbf{X})_{(J,\sigma)} ((\dot{J}, \dot{\sigma}), (\dot{J}', \dot{\sigma}')) \defin \int_\Sigma \hat\omega_\mathbf{X}((\dot{J}, \dot{\sigma}), (\dot{J}', \dot{\sigma}')) \, \rho ,
\end{align}
for $\mathbf{X}=\i,\j,\k$ (here $\hat \omega_\mathbf{X}$ denotes the symplectic forms induced on each fiber by an area-preserving identification of the tangent space of $\Sigma$ with $\R^2$), can be obtained from the general theory of Donaldson, specifically from Equation (\ref{eq:inf_dim_symp_form}) by integrating fiberwise the three symplectic forms introduced in the toy model. Therefore,
the group $\Symp_{0}(\Sigma, \rho)$ acts on $T^{*}\mathcal{J}(\Sigma)$ preserving all symplectic forms $\omega_{\mathbf{X}}$ and the action of $\Ham(\Sigma, \rho)$ is actually Hamiltonian. In order to describe the moment map $\mu_{\mathbf{X}}$ for the action of $\Ham(\Sigma, \rho)$ with respect to each of the  symplectic form $\omega_{\mathbf{X}}$ it is convenient to introduce the following operator:
\[
\begin{matrix}
r \vcentcolon & \Gamma(T^*_{(1,0)} \Sigma \otimes (T^*_{(0,1)} \Sigma
)^{2 \otimes}) & \longrightarrow & \Gamma(T^*_{(0,1)} \Sigma) \\
& \psi & \longmapsto & \frac{\psi(v, v, \cdot)}{\norm{v}_J^2}
\end{matrix}
\]
for some $v \neq 0$. Since $\psi$ is $\C$-\hsk linear in the first component and anti-\hsk $\C$-\hsk linear in the second one, the $1$-\hsk form $\frac{\psi(v, v, \cdot)}{\norm{v}_J^2}$ does not depend on the choice of $v \neq 0$.

\begin{theorem} \label{thm:donaldson_maps_cotangent_bundle}
Donaldson's maps for the action of $\Ham(\Sigma, \rho)$ over $T^* \mathcal{J}(\Sigma)$ can be expressed as:
    \[
    \mu_\mathbf{I}(J,\sigma) = \frac{\norm{\bar{\partial} \phi}^2 - \norm{\partial \phi}^2}{f(\norm{\sigma})} \, \rho - 2 f(\norm{\sigma}) K_J \, \rho - 2 i \bar{\partial} \partial f(\norm{\sigma}) ,
    \]
    \[
    (\mu_\mathbf{J} + i \mu_\mathbf{K})(J,\sigma) = - 2 i \, \partial r( \partial \bar{\phi}) ,
    \]
    where $\phi$ is the quadratic differential whose real part is equal to $\sigma$, and $\partial = \partial_J$, $\bar{\partial} = \bar{\partial}_J$.
\end{theorem}

\begin{proof}
     In order to determine the expressions for Donaldson's maps $\mu_{\i}, \mu_{\j}, \mu_{\k}$, we apply Theorem \ref{thm:donaldson_maps} starting from the moment maps $\eta_{\i}, \eta_{\j}, \eta_{\k}$ introduced in Section \ref{subsec:moment_maps_toy}. We will focus on each term of Donaldson's maps individually, and then combine the resulting expressions to determine the evaluations of $\mu_{\i}, \mu_{\j}, \mu_{\k}$ at $(J, \sigma) \in T^* \mathcal{J}(\Sigma)$. As a torsion-free connection we will use the Levi-Civita connection of the metric $g_J$, which satisfies $\nabla_\bullet \rho \equiv 0$, since $\rho$ equals the Riemannnian volume form of $g_J$.
    
    \paragraph{The term $\hat{\omega}_\mathbf{I}(\nabla_\bullet s, \nabla_\bullet s)$:} 
    Since $\nabla$ is the Levi-\hsk Civita connection of $g_J$, we have $\nabla_\bullet J \equiv 0$. In particular $\nabla_\bullet s = (0, \nabla_\bullet \sigma)$. Then we have:
    \begin{align*}
        (\hat{\omega}_\mathbf{I})_{(J,\sigma)}((0, \nabla_{e_1} \sigma), (0, \nabla_{e_2} \sigma)) & = \frac{1}{2 f(\norm{\sigma})} \tr(g_J^{-1} (\nabla_{e_1} \sigma) g_J^{-1} (\nabla_{e_2} \sigma)J ) \\
        & = \frac{1}{2 f(\norm{\sigma})} \left( \nabla_1 \sigma_{1 1} \nabla_2 \sigma_{1 2} - \nabla_1 \sigma_{2 2} \nabla_2 \sigma_{2 1} + \nabla_1 \sigma_{1 2} \nabla_2 \sigma_{2 2} - \nabla_1 \sigma_{2 1} \nabla_2 \sigma_{1 1} \right) \\
        & = \frac{1}{f(\norm{\sigma})} \left( \nabla_1 \sigma_{1 1} \nabla_2 \sigma_{1 2} - \nabla_1 \sigma_{1 2} \nabla_2 \sigma_{1 1} \right) ,
    \end{align*}
    where $\nabla_i \sigma_{j k} = (\nabla_{e_i} \sigma)(e_j, e_k)$. In the last step we used the fact that $\nabla_v \sigma$ is symmetric and $g_J$-\hsk traceless for every $v$. The operators $\partial_J$ and $\bar{\partial}_J$ are defined as follows:
    \[
    (\partial_J \phi)(v, \cdot, \cdot) = \frac{1}{2} \left( \nabla_v \phi - i \, \nabla_{J v} \phi \right) , \qquad (\bar{\partial}_J \phi)(v, \cdot, \cdot) = \frac{1}{2} \left( \nabla_v \phi + i \, \nabla_{J v} \phi \right) .
    \]
    A simple but tedious computation shows that
    \[
    \norm{\bar{\partial}_J \phi(e_1, \cdot, \cdot)}^2 - \norm{\partial_J \phi(e_1, \cdot, \cdot)}^2 = \nabla_1 \sigma_{1 1} \nabla_2 \sigma_{1 2} - \nabla_1 \sigma_{1 2} \nabla_2 \sigma_{1 1} .
    \]
    In the end, we get
    \[
    \hat\omega_\mathbf{I}(\nabla_\bullet (J, \sigma), \nabla_\bullet (J, \sigma)) = \frac{\norm{\bar{\partial}_J \phi}^2 - \norm{\partial_J \phi}^2}{f(\norm{\sigma})} \rho
    \]
    where $\norm{\partial_J \phi} \defin \norm{\partial_J \phi (v, \cdot, \cdot)}$, $\norm{\bar{\partial}_J \phi} \defin \norm{\bar{\partial}_J \phi (v, \cdot, \cdot)}$ for some unit vector $v$ (the norm is independent of the choice of such $v$).
    
    \paragraph{The term $\scal{\eta_\mathbf{I}}{R^\nabla}$:} Since we are considering the Levi-\hsk Civita connection of $g_J$, the tensor $R^\nabla$ coincides with the Riemann tensor of $g_J$. Then we have $R_J =K_J \, J \otimes \rho$, where $K_J$ denotes the Gaussian curvature of $g_J$. In particular $\scal{\eta_{\mathbf{I}}(J,\sigma)}{R_J} = -2 f(\norm{\sigma}) K_J \, \rho$.
    
    \paragraph{The term $\dd(c(\nabla_\bullet \eta_\mathbf{I}))$:} First we observe that
    \[
    (\nabla_v {\eta_{\i}}_{(J,\sigma)})(X) = \dd(f(\norm{\sigma}))(v) \tr(J X) ,
    \]
    where $X \in \Lsl(T_\cdot \Sigma, \rho) = \End_0(T_\cdot \Sigma)$. Then we have
    \begin{align*}
        c(\nabla_\bullet {\eta_{\i}}_{(J,\sigma)})(w) & \defin \scal{\nabla_{e_1} {\eta_{\i}}_{(J,\sigma)}}{(w \otimes e_1^*)_0} + \scal{\nabla_{e_2} {\eta_{\i}}_{(J,\sigma)}}{(w \otimes e_2^*)_0} \\
        & = \dd(f(\norm{\sigma}))(e_1) \tr(J (w \otimes e_1^*)_0) + \dd(f(\norm{\sigma}))(e_2) \tr(J (w \otimes e_2^*)_0) \\
        & = \dd(f(\norm{\sigma}))(e_1) \, e_1^*(J w) + \dd(f(\norm{\sigma}))(e_2) \, e_2^*(J w) \\
        & = (\dd(f(\norm{\sigma})) \circ J)(w) .
    \end{align*}
    Therefore $c(\nabla_\bullet {\eta_{\i}}_{(J,\sigma)}) = \dd(f(\norm{\sigma})) \circ J$. It is immediate to check that, for every function $\varphi \in \mathscr{C}^\infty(\Sigma)$, we have $\dd( \dd{\varphi} \circ J) = - \Delta_{g_J} \varphi \, \rho = - 2 i \partial_J \bar{\partial}_J \varphi = 2 i \bar{\partial}_J \partial_J \varphi$.
    
    The combination of the terms found above provides the desired description of the Donaldson's map $\mu_{\i}$. We now focus on the addends appearing in the expression of $(\mu_\mathbf{J} + i \mu_\mathbf{K})(J,\sigma)$. 
    
    \paragraph{The term $(\hat\omega_\mathbf{J} + i \, \hat\omega_\mathbf{K})(\nabla_\bullet s, \nabla_\bullet s)$:} Since $\nabla_\bullet J = 0$, we have
    \[
    \hat\omega^\C(\nabla_\bullet (J, \sigma), \nabla_\bullet (J, \sigma)) = (\hat\omega_\mathbf{J} + i \hat\omega_\mathbf{K})( (0, \nabla_\bullet \sigma), (0,\nabla_\bullet \sigma)) = 0
    \]
    
    \paragraph{The term $\scal{\eta_\mathbf{J} + i \, \eta_\mathbf{K}}{R^\nabla}$:}
    Observe that $\scal{(\eta_\mathbf{J} + i \, \eta_\mathbf{K})_{(J,\sigma)}}{K_J \, J \otimes \rho} = 0$ because $[J,J] = 0$.
    
    \paragraph{The term $\dd{(c(\nabla_\bullet \eta_\mathbf{J} + i \, \eta_\mathbf{K}))}$:} From the definition of the moment maps $\eta_\mathbf{J}$ and $\eta_\mathbf{K}$, we observe the following:
    \begin{align*}
        \scal{(\eta_\mathbf{J} + i \, \eta_\mathbf{K})_{(J,\sigma)}}{X} & = - \tr(g_J^{-1} \sigma J X) + i \tr(g_J^{-1} \sigma X) \\
        & = i \tr(g_J^{-1} (\sigma + i \, \sigma(\cdot, J \cdot)) X) \\
        & = i \tr(g_J^{-1} \bar{\phi} X) , \tag{rel. \eqref{eq:quadratic_diff}}
    \end{align*}
    where we are denoting by $\phi$ the quadratic differential whose real part is equal to $\sigma$. Since $\nabla$ is the Levi-\hsk Civita connection of $g_J$, we have $\nabla_\bullet g_J \equiv 0$ and, as already mentioned, $\nabla_\bullet J \equiv 0$. In particular we deduce that:
    \begin{align*}
        c(\nabla_\bullet (\eta_\mathbf{J} + i \, \eta_\mathbf{K}))(v) & = \sum_j \scal{\nabla_{e_j} (\eta_\mathbf{J} + i \, \eta_\mathbf{K})}{(v \otimes e_j^*)_0} \\
        & = i \, \sum_j \tr(g_J^{-1} (\nabla_{e_j} \bar{\phi})(v \otimes e_j^*)_0) \\
        & = i \, \sum_j \tr(g_J^{-1} (\nabla_{e_j} \bar{\phi})(v \otimes e_j^*)) \\
        & = i \, \sum_j (\nabla_{e_j} \bar{\phi})(v, e_j) \\
        & = i \, (\divr_{g_J} g_J^{-1} \bar{\phi})(v) .
    \end{align*}
    
    On the other hand, we see that:
    \begin{align*}
        r(\partial_J \bar{\phi})(v) & = (\partial_J \bar{\phi})(e_1, e_1, v) \\
        & = \frac{1}{2} ((\nabla_{e_1} \bar{\phi})(e_1, v) - i \, (\nabla_{J e_1} \bar{\phi})(e_1, v) ) \\
        & = \frac{1}{2} ((\nabla_{e_1} \bar{\phi})(e_1, v) + (\nabla_{J e_1} \bar{\phi})(J e_1, v) ) \\
        & = \frac{1}{2} (\divr_{g_J} g_J^{-1} \bar{\phi})(v) .
    \end{align*}
    Therefore $c(\nabla_\bullet (\eta_\mathbf{J} + i \, \eta_\mathbf{K})) = 2 i \, r(\partial_J \bar{\phi}) \in \Gamma(T^*_{(0,1)} \Sigma)$. Since $\dd = \partial_J + \bar{\partial}_J$ and $\bar{\partial}_J(r(\partial_J \bar{\phi}))=0$, we obtain
    \[
    \dd(c(\nabla_\bullet (\eta_\mathbf{J} + i \, \eta_\mathbf{K}))) = 2 i \, \partial_J r(\partial_J \bar{\phi}) ,
    \]
    which, combined with Theorem \ref{thm:donaldson_maps}, proves the expression for Donaldson's map $\mu_\mathbf{J} + i \, \mu_\mathbf{K}$.
\end{proof}

\begin{proposition} \label{prop:moment_maps_infinite_dim}
    Let $\rho$ be a fixed volume form on $\Sigma$, and let $C \defin - \frac{4 \pi \chi(\Sigma)}{\Area(\Sigma,\rho)}$.
    
    \begin{itemize}
        \item The map 
        \[
    \begin{matrix}
    \tilde{\mu}_\mathbf{I} \vcentcolon & T^* \mathcal{J}(\Sigma) & \longrightarrow & \Lham(\Sigma, \rho)^* \\
    & (J, \sigma) & \longmapsto & \mu_{\i}(J, \sigma) - C \, \rho
    \end{matrix}
    \]
    is a moment map for the action of $\Ham(\Sigma, \rho)$ over the space of smooth sections $T^* \mathcal{J}(\Sigma)$, endowed with the symplectic form $\omega_\mathbf{I}$.
    \item The map
    \[
    \begin{matrix}
    \tilde{\mu}_\mathbf{J} + i \, \tilde{\mu}_\mathbf{K} \vcentcolon & T^* \mathcal{J}(\Sigma) & \longrightarrow & \Lsymp(\Sigma, \rho)^* \otimes \C \\
    & (J, \sigma) & \longmapsto & [- 2 i \, r(\partial_J \bar{\phi})] 
    \end{matrix}
    \]
    is a moment map for the action of $\Symp_0(\Sigma, \rho)$ over the space of smooth sections $T^* \mathcal{J}(\Sigma)$, endowed with the symplectic form $\omega_\mathbf{J} + i \, \omega_\mathbf{K}$ (where $\phi$ denotes the quadratic differential whose real part is equal to $\sigma$).
    \end{itemize}
\end{proposition}

\begin{proof}
    By Lemma \ref{lem:donaldson_pullback_form}, the integral of the $2$-\hsk form $\mu_\mathbf{I}(J,\sigma)$ does not depend on the choice of the section $(J,\sigma) \in T^* \mathcal{J}(\Sigma)$. If $\sigma \equiv 0$, then by Gauss-\hsk Bonnet theorem we have
    \[
    \int_\Sigma \mu_\mathbf{I}(J, 0) = - 4 \pi \chi(\Sigma) .
    \]
    Therefore, the integral of the $2$-\hsk form $\mu_\mathbf{I}(J, \sigma) - C \, \rho$ vanishes. By Poincar\'e's duality, a $2$-\hsk form on $\Sigma$ is exact if and only if its integral over $\Sigma$ vanishes. This in particular proves that $\tilde{\mu}_\mathbf{I}$ takes values in $B^2(\Sigma)$, which is contained in the dual of the Lie algebra of $\Ham(\Sigma,\rho)$, as observed in Section \ref{subsec:Teich_inf_dim_reduction}. Following Definition \ref{def:moment_map}, the analogous properties i) and ii) for $\tilde{\mu}_\mathbf{I}$ are guaranteed by Theorem \ref{thm:donaldson_maps}, since the term $C \, \rho$ is independent of the section $(J,\sigma) \in T^* \mathcal{J}(\Sigma)$.
    
    Our moment map $\mu_\mathbf{J} + i \, \mu_\mathbf{K}$ is exactly the same as the map $\mu_2 + i \, \mu_3$ appearing in the original process of symplectic reduction developed by \cite{donaldson2003moment} to describe the hyperK\"ahler structure on the space of almost-\hsk Fuchsian manifolds. In particular, the argument in \cite[Section 3.1]{donaldson2003moment} to $\mu_2 + i \, \mu_3$ applies verbatim to our context, showing that $\mu_\mathbf{J} + i \, \mu_\mathbf{K}$ can be promoted to a moment map $\tilde{\mu}_\mathbf{J} + i \, \tilde{\mu}_\mathbf{K}$ for the action of $\Symp_0(\Sigma, \rho)$ on $T^* \mathcal{J}(\Sigma)$. For a more detailed exposition of this phenomenon, we refer to \cite[Theorem 4.10]{trautwein2019hyperkahler}.
\end{proof}

\subsection{The symplectic quotient} We are interested in the symplectic quotient
\[
    \tilde{\mu}_{\i}^{-1}(0)\cap \tilde{\mu}_{\j}^{-1}(0)\cap \tilde{\mu}_{\k}^{-1}(0) / \Symp_{0}(\Sigma, \rho) \ .
\]
The aim of this section is to show that this quotient can be identified with $\mathcal{MS}(\Sigma)$ so that the para-hyperK\"ahler structure on $\mathcal{MS}(\Sigma)$ defined in Section \ref{sec:para_hyperkahler_str_on_MS} is inherited from the infinite dimensional space $T^{*}\mathcal{J}(\Sigma)$. Although the arguments are inspired by similar constructions in hyperK\"ahler geometry (\cite{donaldson2003moment}, \cite{trautwein2019hyperkahler}, \cite{hodgethesis}), substantial difficulties arise when dealing with a pseudo-Riemannian metric, which we now explain.

In the more common setting of infinite dimensional hyperK\"ahler quotients, one starts with the data of an infinite dimensional manifold $M$ endowed with three complex structures $\i,\j$ and $\k$ satisfying the relations of quaternionic numbers and a Riemannian metric $\g$ compatible with each of the complex structures. This defines three symplectic forms $\omega_{\mathbf{X}}=\g(\cdot, \mathbf{X})$ for $\mathbf{X}=\i,\j,\k$ on $M$. Assume now that a group $G$ acts on $M$ by Hamiltonian symplectomorphisms with respect to each of the symplectic forms with moment map $\zeta_{\mathbf{X}}$ and let $N$ be the submanifold 
\[
    N=\zeta_{\i}^{-1}(0)\cap \zeta_{\j}^{-1}(0) \cap \zeta_{\k}^{-1}(0) \ . 
\]
The properties of the moment maps and the Riemannian metric $\g$ allow to orthogonally decompose the tangent space of $M$ at every point $p \in N$ as
\[
    T_{p}M=T_{p}N\oplus \i(T_{p}(G \cdot p)) \oplus \j(T_{p}(G\cdot p)) \oplus \k(T_{p}(G\cdot p))
\]
and the $\g$-orthogonal of $T_{p}(G\cdot p)$ inside $T_{p}N$ furnishes a model for the tangent space of the quotient $N/G$ that is invariant under $\i$,$\j$ and $\k$. This is sufficient to conclude that the quotient inherits a hyperK\"ahler structure  (\cite{Hitchin_selfdual}). 

In our pseudo-Riemannian setting, however, the absence of a Hilbert space structure on the tangent space prevents us from obtaining a similar orthogonal decomposition using only the properties of the moment maps. However, the hyperK\"ahler construction suggests that the tangent space at $(J,\sigma)$ to the zero locus of the three moment maps is the largest subspace $V_{(J,\sigma)} \subset T_{(J,\sigma)}T^{*}\mathcal{J}(\Sigma)$ that is $\g$-orthogonal to $T_{(J,\sigma)}(\Symp_{0}(\Sigma, \rho)\cdot (J,\sigma))$ and invariant under $\i$, $\j$ and $\k$, inside the kernel of the differential of the three moment maps. These are the properties that led us to the equations defining $V_{(J,\sigma)}$ in Section \ref{sec:para_hyperkahler_str_on_MS}. In what follows we explain how to derive the equations in Proposition \ref{prop:equivalent_def_subspace_V} starting from these geometric conditions.

\subsubsection{Identifying $\widetilde{\mathcal{MS}}_0(\Sigma,\rho)$ as the zero locus of the moment maps}
Let us start by characterizing the pairs $(J,\sigma)$ that lie in the zero locus of the three moment maps.

\begin{lemma}\label{lm:reduction_hqd}
    A pair $(J,\sigma) \in T^{*}\mathcal{J}(\Sigma)$ satisfies $(\tilde{\mu}_\mathbf{J} + i \, \tilde{\mu}_\mathbf{K}) (J,\sigma)=0$ if and only if $\sigma$ is the real part of a holomorphic quadratic differential $\phi$ on the Riemann surface $(\Sigma,J)$.
\end{lemma}
\begin{proof} Let us consider the subalgebra of $\Lsymp(\Sigma, \rho)$ defined by
\[
    \mathfrak{h}_{J}=\{ V \in \Gamma(T\Sigma) \ | \ d(\iota_{V}\rho)=d(\iota_{JV}\rho)=0 \} \ .
\]
Because $(\tilde{\mu}_\mathbf{J} + i \, \tilde{\mu}_\mathbf{K}) (J,\sigma)=0$, by definition of moment map, we have
\[
    0=-2i\int_{\Sigma} r(\partial_{J}\bar{\phi}) \wedge \iota_{V}\rho
\]
for every $V \in \mathfrak{h}_{J}$. Note that the form $r(\partial_{J}\bar{\phi})$ is anti-holomorphic, being it in the zero locus of $\mu_\mathbf{J} + i \, \mu_\mathbf{K}$. Since $\{ \iota_{V}\rho \ | \ V \in \mathfrak{h}_{J}\}$ parameterizes harmonic $1$-forms on $(\Sigma, J)$, by Poincar\'e duality the real and imaginary parts of $r(\partial_{J}\bar{\phi})$ are exact. Hence, $r(\partial_{J}\bar{\phi})$ is exact and identically zero because the only anti-holomorphic functions on $(\Sigma, J)$ are constant. This shows that $\partial_{J}\bar{\phi}=0$, thus $\phi$ is a holomorphic quadratic differential on $(\Sigma, J)$.
\end{proof}

\begin{lemma}\label{lm:equivalent_muI}
Let $(J,\sigma) \in T^{*}\mathcal{J}(\Sigma)$. If $\sigma$ is the real part of a holomorphic quadratic differential, then
    \[
    \tilde{\mu}_{\mathbf{I}}(J,\sigma) = - \frac{4 K_h}{1 + \det B} \, \rho -C\rho,
    \]
    where
    \[
    h \defin \left(1 + f(\norm{\sigma}) \right) \, g_J , \qquad B \defin h^{-1} \sigma \, \qquad C \defin -\frac{4\pi\chi(\Sigma)}{\Area(\Sigma, \rho)}.
    \]
\end{lemma}
\begin{proof}
    If $\sigma$ is the zero quadratic differential, then the statement is immediate. From now on, we will assume $\sigma$ to be non-zero. We make use of the computations developed by Trautwein in \cite{trautwein2019hyperkahler}. In particular, we will need the following relations: if $\lambda$ denotes the function $\norm{\sigma}^2$, then, outside the zeros of $\sigma$,
    \[
    - \frac{i}{2 \lambda} \bar{\partial} \lambda \wedge \partial \lambda = \norm{\partial \phi}^2 \, \rho , \qquad \bar{\partial} \partial (\ln \lambda) = 2 i \, K_J \, \rho ,
    \]
    where $\partial = \partial_J$ and $\bar{\partial} = \bar{\partial}_J$. In particular we can write $\tilde{\mu}_\mathbf{I}$ as follows:
    \[
    \tilde{\mu}_{\mathbf{I}}(J,\sigma) = i \left( \frac{\bar{\partial} \lambda \wedge \partial \lambda}{2 \lambda \sqrt{1 + \lambda}} + \sqrt{1 + \lambda} \ \bar{\partial} \partial (\ln \lambda) - 2 \bar{\partial} \partial (\sqrt{1 + \lambda}) \right) - C\rho
    \]
    
    From here, the strategy is the same as the one in \cite[Proposition 4.5.16]{trautwein_thesis}. In particular, we can find the following sequence of identities:
    \begin{align*}
        \tilde{\mu}_{\mathbf{I}}(J,\sigma) & = i \left( \frac{\bar{\partial} \lambda \wedge \partial \lambda}{2 \lambda \sqrt{1 + \lambda}} + \sqrt{1 + \lambda} \ \bar{\partial} \partial (\ln \lambda) - 2 \bar{\partial} \partial (\sqrt{1 + \lambda}) \right) -C\rho \\
        & = i \left( \bar{\partial} ( \sqrt{1 + \lambda}) \wedge \partial(\ln \lambda) + \sqrt{1 + \lambda} \ \bar{\partial} \partial (\ln \lambda) - 2 \bar{\partial} \partial (\sqrt{1 + \lambda}) \right) -C\rho \\
        & = i \bar{\partial} \left( \sqrt{1 + \lambda} \ \partial (\ln \lambda) - 2 \partial (\sqrt{1 + \lambda}) \right) -C\rho \\
        & = i \bar{\partial} \left( - 2 \partial \left(\ln(1 + \sqrt{1 + \lambda})\right) + \partial(\ln \lambda) \right) -C\rho \\
        & = \left( \Delta_{g_J} \ln(1 + \sqrt{1 + \norm{\sigma}^2}) - 2 K_J -C \right) \, \rho , \\
        & = \Big( \Delta_{g_J} \ln(1 + f(\norm{\sigma})) - 2 K_J -C\Big) \, \rho ,
    \end{align*}
    where, in the second to last step, we used again the identity $\bar{\partial} \partial (\ln \lambda) = 2 i \, K_J \, \rho$, and the fact that $- 2 i \bar{\partial} \partial \varphi = \Delta_{g_J} \varphi \, \rho$.
    
    From the formula of the curvature under conformal change, we obtain that the Gaussian curvature $K_h$ of the metric $h$ satisfies:
    \[
    K_h = \frac{1}{1 + f(\norm{\sigma})} \left( K_J - \frac{1}{2} \Delta_{g_J} \ln \left(1 + f(\norm{\sigma}) \right) \right)
    \]
    A computation exactly as in Lemma \ref{lem:det_B} shows the identity \[
    \frac{2}{1 + f(\norm{\sigma})} = 1 + \det B.
    \]
    If we combine this relation with the expression for the Gaussian curvature of $h$ and the description we found above for $\tilde{\mu}_\mathbf{I}$, we see that
    \[
    \tilde{\mu}_{\mathbf{I}}(J,\sigma)  = - \left(\frac{4 K_h}{1 + \det B}+C\right) \, \rho .
    \]
    This proves the desired relation outside of the zero locus of $\sigma$, which is a finite set. The statement then follows by continuity of the expression.
    
\end{proof}

Combining Lemma \ref{lm:reduction_hqd} and Lemma \ref{lm:equivalent_muI}, we can then identify the zero locus of the three moment maps with $\widetilde{\mathcal{MS}}_{0}(\Sigma, \rho)$. Recall that the symplectic form $\rho$ that we fixed in Section \ref{subsec:change_coord} has area $-\pi\chi(\Sigma)$.

\begin{corollary}\label{cor:zero moment maps} Let $(J,\sigma) \in T^{*}\mathcal{J}(\Sigma)$. Then $\tilde{\mu}_{\i}(J,\sigma)=\tilde{\mu}_{\j}(J,\sigma)=\tilde{\mu}_{\k}(J,\sigma)=0$ if and only if $(J,\sigma) \in \widetilde{\mathcal{MS}}_{0}(\Sigma, \rho)$.
\end{corollary}
\begin{proof} Recall from Section \ref{subsec:change_coord} that $\widetilde{\mathcal{MS}}_{0}(\Sigma, \rho)$ is the space of pairs $(J, \sigma)$ such that $h=(1 + f(\norm{\sigma}))g_J$ and $B=h^{-1}\sigma$ satisfy the Gauss-Codazzi equations. By Lemma \ref{lm:reduction_hqd}, we know that $\sigma$ is the real part of a holomorphic quadratic differential, hence $B$ is Codazzi by Lemma \ref{lem:hol_quadr_diff_equivalence}. Now, from our choice of $\rho$, we see that
\[
    C=-\frac{4\pi\chi(\Sigma)}{\Area(\Sigma, \rho)}=4
\]
so by Lemma \ref{lm:equivalent_muI} we have that $\tilde{\mu}_{\i}(J,\sigma)=0$ if and only if $K_{h}=-1-\det(B)$, which is exactly the Gauss equation for space-like surfaces in Anti-de Sitter space. 
\end{proof}

\subsubsection{The differential of the map $\mu_{\i}$} One subtlety we have to take care of concerns the fact that the action of $\Symp_{0}(\Sigma, \rho)$ is not Hamiltonian with respect to $\omega_{\i}$. However, we can show by an explicit computation that (up to a sign) property ii) in Definition \ref{def:moment_map}
\[
    \scal{\dd\tilde\mu_{\i}(\dot{J},\dot{\sigma})}{V}_{\Lsymp}=\omega_{\i}((\dot{J}, \dot{\sigma}), (\Dlie_{V}J,\Dlie_{V}\sigma))
\]
holds for every $(\dot{J}, \dot{\sigma}) \in T_{(J,\sigma)}T^{*}\mathcal{J}(\Sigma)$ and for every $V \in \Lsymp(\Sigma, \rho)$, not only for $V \in \Lham(\Sigma, \rho)$. 

\begin{proposition}\label{prop:variation_muI} For every $(J,\sigma) \in T^{*}\mathcal{J}(\Sigma)$ such that $\sigma$ is the real part of a holomorphic quadratic differential, and for every tangent vector $(\dot{J}, \dot{\sigma}) \in T_{(J,\sigma)}T^{*}\mathcal{J}(\Sigma)$ we have
\[
    \dd\tilde\mu_{\i}(\dot{J},\dot{\sigma})=-\dd(f\divr_{g_J}\dot{J}+\dd\dot{f}\circ J + \dd f\circ \dot{J}-f^{-1}\beta) \ ,
\]
where $f=f(\|\sigma\|_{J})=\sqrt{1+\|\sigma\|_{J}^{2}}$ and $\beta$ is the $1$-form on $\Sigma$ defined by $\beta(V):=\scall{\dot{\sigma}_{0}}{(\nabla_{V}\sigma)(\cdot, J\cdot)}$.
\end{proposition}
\begin{proof} The main bulge of the proof is to show that
\begin{equation}\label{eq:mainpart_dmuI}
    \left(\frac{\norm{\bar{\partial} \phi}^2 - \norm{\partial \phi}^2}{f(\norm{\sigma})}\rho\right)'=\dd(f^{-1}\beta)+2\dot{f}K_{J}\rho+\divr_{g_J}\dot{J}\wedge \dd f ,
\end{equation}
where we are taking the derivative of the expression with respect to $(J,\sigma)$ in the direction $(\dot{J}, \dot{\sigma})$. Assuming this relation for now, we see how to conclude. The variations of the other two terms appearing in the expression for $\tilde\mu_{\i}$ (see Theorem \ref{thm:donaldson_maps_cotangent_bundle}) are simpler to handle. We have that
\begin{align*}
    (-2fK_{J}\rho)'& = -2\dot{f}K_{J}\rho-2f\dd K_{J}(\dot{J})\rho \\
    & = -2\dot{f}K_{J}\rho-f\dd(\divr_{g_J}\dot{J}) ,
\end{align*}
where in the last step we used the expression for the derivative of the curvature described in Remark \ref{rmk:derivative_curvature}, and
\[
(\Delta_{g_J}f)'\rho=-\dd((\dd f\circ J)')=-\dd(\dd\dot{f}\circ J+\dd f \circ \dot{J}) .
\]
By combining these relations, we find
\begin{align*}
    \dd\tilde\mu_{\i}(\dot{J},\dot{\sigma})&=\dd(f^{-1}\beta)+2\dot{f}K_{J}\rho+\divr_{g_J}\dot{J}\wedge \dd f-2\dot{f}K_{J}\rho-f\dd(\divr_{g_J}\dot{J})-\dd(\dd\dot{f}\circ J+\dd f \circ \dot{J}) \\
    &=-\dd(f\divr_{g_J}\dot{J}+\dd f \circ \dot{J} + \dd \dot{f}\circ J-f^{-1}\beta) ,
\end{align*}
which proves the desired formula. 

Let us now focus on \eqref{eq:mainpart_dmuI}. As derived in the proof of Theorem \ref{thm:donaldson_maps_cotangent_bundle}, we have that
\begin{equation} \label{eq:schifo}
        \frac{\norm{\bar{\partial} \phi}^2 - \norm{\partial \phi}^2}{f(\norm{\sigma})}=\hat\omega_{\i}((0, \nabla_{e_{1}}\sigma),(0,\nabla_{e_{2}}\sigma)) ,
\end{equation}
for any choice of a local frame $e_1, e_2$ satisfying $\rho(e_1, e_2)=1$. We will determine relation \eqref{eq:mainpart_dmuI} by computing the derivative
\begin{align*}
    \left(\frac{\norm{\bar{\partial} \phi}^2 - \norm{\partial \phi}^2}{f(\norm{\sigma})}\right)' & = \left(\hat\omega_{\i}((0, \nabla_{e_{1}}\sigma),(0,\nabla_{e_{2}}\sigma))\right)' \\
    & = \left(f^{-1}\scall{\nabla_{e_{1}}\sigma}{(\nabla_{e_{2}}\sigma)(\cdot, J\cdot)}\right)' \\
    & = \left( \frac{1}{2} f^{-1} \tr(g_J^{-1} (\nabla_{e_{1}}\sigma) g_J^{-1} (\nabla_{e_{2}}\sigma)J)) \right)'
\end{align*}
In order to simplify the development of this first order variation, we need a few preliminary observations. Since equation \eqref{eq:schifo} holds for any local frame $e_1, e_2$ with $\rho$-volume equal to $1$, we can assume $e_1, e_2$ to be a $g_J$-orthonormal frame and to not change under the variation $\dot{J}$. Moreover, in light of Lemma \ref{lem:product_in_TJ} part \textit{ii)} and relation \eqref{eq:derivative_scal_prod}, the traces of the terms involving the derivatives of $g_J^{-1}$ and of $J$ vanish. This allows us to reduce the number of addends to study. In particular we have that
\begin{align} \label{eq:schifo2}
\begin{split}
    \left(\frac{\norm{\bar{\partial} \phi}^2 - \norm{\partial \phi}^2}{f(\norm{\sigma})}\right)'&=-f^{-2}\dot{f}\scall{\nabla_{e_{1}}\sigma}{(\nabla_{e_{2}}\sigma)(\cdot, J\cdot)}+f^{-1}\scall{(\nabla_{e_{1}}\sigma)'}{(\nabla_{e_{2}}\sigma)(\cdot, J\cdot)} + \\
    & \qquad \qquad \qquad \qquad \qquad +f^{-1}\scall{\nabla_{e_{1}}\sigma}{(\nabla_{e_{2}}\sigma)'(\cdot, J\cdot)} .
\end{split}
\end{align}

Now, using Lemma \ref{lem:variation_levicivita}, we can obtain an expression for $(\nabla_{e_{i}}\sigma)'$. In fact, observing that
\begin{equation}\label{eq:nabla_sigma_dot}
    (\nabla_{V}\sigma)'=\nabla_{V} \dot{\sigma}+\dot{\nabla}_{V}\sigma \ ,
\end{equation}
we compute
\begin{align*}
\begin{split}
    (\dot{\nabla}_{V}\sigma)(X,Y)&=-\sigma(\dot{\nabla}_{V}X,Y)-\sigma(X,\dot{\nabla}_{V}Y) \\
    &=\frac{1}{2}\left((\divr_{g_J}\dot{J})V(\sigma(JX,Y)+\sigma(X,JY))+\sigma(J(\nabla_{V}\dot{J})X,Y)+\sigma(X,J(\nabla_{V}\dot{J})Y)\right) \\
    &=(\divr_{g_J}\dot{J})V \, \sigma(JX,Y)+\scal{\sigma}{J(\nabla_{V}\dot{J})}\, g_{J}(X,Y)  \ , 
\end{split}
\end{align*}
where, in the last line, we applied Lemma \ref{lem:product_in_TJ} to the products $g^{-1}_J \sigma J (\nabla_V \dot{J})$ and $J (\nabla_V \dot{J}) g^{-1}_J \sigma$. Moreover we have
\begin{align*}
    \nabla_{V}\dot{\sigma}&=\nabla_{V}(\dot{\sigma}_{0}-\scal{\sigma}{J\dot{J}} \,g_{J})=\nabla_{V}\dot{\sigma}_{0}-V(\scal{\sigma}{J\dot{J}}) \, g_{J} \nonumber\\
    &=\nabla_{V}\dot{\sigma}_{0}-(\scal{\nabla_{V}\sigma}{J\dot{J}}+\scal{\sigma}{J\nabla_{V}\dot{J}}) \, g_{J} \ .
\end{align*}
Combining the previous two equations with (\ref{eq:nabla_sigma_dot}), we get
\begin{align}\label{eq:final_dotnablasigma}
    \begin{split}
        (\nabla_{V}\sigma)'&=\nabla_{V}\dot{\sigma}_{0}-(\scal{\nabla_{V}\sigma}{J\dot{J}}+\scal{\sigma}{J\nabla_{V}\dot{J}}) g_{J}+(\divr_{g_J}\dot{J})V \, \sigma(\cdot, J\cdot)+\scal{\sigma}{J\nabla_{V}\dot{J}} \, g_{J} \\
        &= \nabla_{V}\dot{\sigma}_{0} - \scal{\nabla_{V}\sigma}{J\dot{J}} \, g_J + (\divr_{g_J}\dot{J})V \, \sigma(\cdot, J\cdot) \ .
    \end{split}
\end{align}

In light of these observations, we can combine relations \eqref{eq:final_dotnablasigma} and \eqref{eq:schifo2} to deduce that
\begin{align}\label{eq:step2}
    \begin{split}
    \left(\frac{\norm{\bar{\partial} \phi}^2 - \norm{\partial \phi}^2}{f(\norm{\sigma})}\right)'
    &=-f^{-1}\dot{f} \, \hat\omega_{\i}((0, \nabla_{e_{1}}\sigma),(0,\nabla_{e_{2}}\sigma))+f^{-1}\scall{\nabla_{e_{1}}\dot{\sigma}_{0}}{(\nabla_{e_{2}}\sigma)(\cdot, J \cdot)} + \\
    & \ \ \ +f^{-1}\scall{\nabla_{e_{1}}\sigma}{(\nabla_{e_{2}}\dot{\sigma}_{0})(\cdot, J\cdot)} + f^{-1}(\divr_{g_J}\dot{J})(e_{1}) \, \scall{\sigma(\cdot, J\cdot)}{(\nabla_{e_{2}}\sigma)(\cdot,J\cdot)}\\
    & \ \ \ +f^{-1}(\divr_{g_J}\dot{J})(e_{2}) \, \scall{\nabla_{e_{1}}\sigma}{\sigma(\cdot,J^{2}\cdot)} \\
    &=-f^{-1}\dot{f} \, \hat\omega_{\i}((0, \nabla_{e_{1}}\sigma),(0,\nabla_{e_{2}}\sigma))+\hat\omega_{\i}((0, \nabla_{e_{1}}\dot{\sigma}_{0}),(0,\nabla_{e_{2}}\sigma))\\
    &\ \ \ +\hat\omega_{\i}((0, \nabla_{e_{1}}\sigma),(0,\nabla_{e_{2}}\dot{\sigma}_{0})) + f^{-1}(\divr_{g_J}\dot{J})(e_{1}) \, e_{2}(\|\sigma\|^{2}/2) + \\
    & \ \ \ - f^{-1} (\divr_{g_J}\dot{J})(e_{2})
    \, e_{1}(\|\sigma\|^{2}/2)\\
    &=-f^{-1}\dot{f} \, \hat\omega_{\i}((0, \nabla_{e_{1}}\sigma),(0,\nabla_{e_{2}}\sigma))+\hat\omega_{\i}((0, \nabla_{e_{1}}\dot{\sigma}_{0}),(0,\nabla_{e_{2}}\sigma))\\
    &\ \ \ +\hat\omega_{\i}((0, \nabla_{e_{1}}\sigma),(0,\nabla_{e_{2}}\dot{\sigma}_{0}))+(\divr_{g_J}\dot{J}\wedge \dd f)(e_{1},e_{2})
    \end{split}
\end{align}
where in the first step we used the fact that the trace part in the expression for $(\nabla_{V}\sigma)'$ does not give a contribution.
Comparing relation \eqref{eq:mainpart_dmuI} with \eqref{eq:step2}, the proof is complete if we show that
\begin{enumerate}[a)]
    \item $(\dd f\wedge \beta)(e_{1},e_{2})=f\dot{f} \, \hat\omega_{\i}((0, \nabla_{e_{1}}\sigma),(0, \nabla_{e_{2}}\sigma))$ 
    \item $\dd \beta(e_{1},e_{2})= f \hat\omega_{\i}((0, \nabla_{e_{1}}\dot{\sigma}_{0}),(0,\nabla_{e_{2}}\sigma)) +\hat\omega_{\i}((0, \nabla_{e_{1}}\sigma),(0,\nabla_{e_{2}}\dot{\sigma}_{0})))-2K_{J}\scall{\dot{\sigma}_{0}}{\sigma}.$
\end{enumerate}
\paragraph{Proof of relation a)} First we observe that, if $\sigma$ is zero, then the relation is obviously satisfied. In what follows we will assume that $\sigma$ is not identically zero. By definition of $\beta$ and of wedge product
\begin{align*}
    \dd f\wedge \beta(e_{1},e_{2})&=e_{1}(f)\scall{\dot{\sigma}_{0}}{(\nabla_{e_{2}}\sigma)(\cdot, J\cdot)}-e_{2}(f)\scall{\dot{\sigma}_{0}}{(\nabla_{e_{1}}\sigma)(\cdot, J\cdot)} \\
    &=f^{-1}\left(\scall{\sigma}{\nabla_{e_{1}}\sigma}\scall{\dot{\sigma}_{0}}{(\nabla_{e_{2}}\sigma)(\cdot, J\cdot)}-\scall{\sigma}{\nabla_{e_{2}}\sigma}\scall{\dot{\sigma}_{0}}{(\nabla_{e_{1}}\sigma)(\cdot, J\cdot)} \right) \ .
\end{align*}

	As observed in the proof of Lemma \ref{lem:derivative_B}, for every $p \in \Sigma$ outside the set of zeros of $\sigma$, the elements $(g^{-1} \sigma)_p$ and $(J g^{-1} \sigma)_p$ form a basis of the space of traceless symmetric endomorphisms of $T_p \Sigma$. In particular, using the scalar product $\scall{\cdot}{\cdot}$ we can represent $\dot\sigma_0$ in terms of such basis, obtaining the expression $\dot{\sigma}_{0}=\frac{1}{\|\sigma\|^{2}}\left(\scall{\dot{\sigma}_{0}}{\sigma}\sigma+\scall{\dot{\sigma}_{0}}{\sigma(\cdot, J\cdot)}\sigma(\cdot, J\cdot)\right)$. Replacing this identity in the previous equation, we get
\begin{align*}
    \dd f\wedge \beta(e_{1},e_{2})
    &=f^{-1}\left(\scall{\sigma}{\nabla_{e_{1}}\sigma}\scall{\dot{\sigma}_{0}}{(\nabla_{e_{2}}\sigma)(\cdot, J\cdot)}-\scall{\sigma}{\nabla_{e_{2}}\sigma}\scall{\dot{\sigma}_{0}}{(\nabla_{e_{1}}\sigma)(\cdot, J\cdot)} \right) \\
    &= \frac{1}{\|\sigma\|^{2} f}  \left( \scall{\sigma}{\nabla_{e_{1}}\sigma} (\scall{\dot{\sigma}_{0}}{\sigma}\scall{\sigma}{(\nabla_{e_{2}}\sigma)(\cdot, J\cdot)}+\scall{\dot{\sigma}_{0}}{\sigma(\cdot, J\cdot)}\scall{\sigma(\cdot, J\cdot)}{(\nabla_{e_{2}}\sigma)(\cdot, J\cdot)}) + \right. \\
    & \ \ \ \left. - \scall{\sigma}{\nabla_{e_{2}} \sigma} (\scall{\dot{\sigma}_{0}}{\sigma}\scall{\sigma}{(\nabla_{e_{1}}\sigma)(\cdot, J\cdot)}+\scall{\dot{\sigma}_{0}}{\sigma(\cdot, J\cdot)}\scall{\sigma(\cdot, J\cdot)}{(\nabla_{e_{1}}\sigma)(\cdot, J\cdot)}) \right) \\
    &= \frac{1}{\|\sigma\|^{2} f}  \left( \scall{\sigma}{\nabla_{e_{1}}\sigma} (\scall{\dot{\sigma}_{0}}{\sigma}\scall{\sigma}{(\nabla_{e_{2}}\sigma)(\cdot, J\cdot)}+\scall{\dot{\sigma}_{0}}{\sigma(\cdot, J\cdot)}\scall{\sigma}{\nabla_{e_{2}}\sigma}) + \right. \\
    & \ \ \ \left. - \scall{\sigma}{\nabla_{e_{2}} \sigma} (\scall{\dot{\sigma}_{0}}{\sigma}\scall{\sigma}{(\nabla_{e_{1}}\sigma)(\cdot, J\cdot)}+\scall{\dot{\sigma}_{0}}{\sigma(\cdot, J\cdot)}\scall{\sigma}{\nabla_{e_{1}}\sigma}) \right) \\
    &= \frac{\scall{\dot{\sigma}_{0}}{\sigma}}{\|\sigma\|^{2}f}  \left( \scall{\sigma}{\nabla_{e_{1}}\sigma} \scall{\sigma}{(\nabla_{e_{2}}\sigma)(\cdot, J\cdot)} - \scall{\sigma}{\nabla_{e_{2}} \sigma} \scall{\sigma}{(\nabla_{e_{1}}\sigma)(\cdot, J\cdot)} \right) \\
     &= \frac{\scall{\dot{\sigma}_{0}}{\sigma}}{\|\sigma\|^{2} f}  \scall{\scall{\sigma}{\nabla_{e_{1}}\sigma} \sigma + \scall{\sigma(\cdot, J\cdot)}{\nabla_{e_{1}}\sigma} \sigma(\cdot, J \cdot)}{(\nabla_{e_{2}}\sigma)(\cdot, J\cdot)}
\end{align*}
    At this point, we use once again the fact that $g^{-1} \sigma$ and $J g^{-1} \sigma$ form a basis of the space of traceless symmetric endomorphisms of $T \Sigma$ (outside the zeros of $\sigma$) to express $\nabla_{e_{1}}\sigma$ as $\frac{1}{\|\sigma\|^{2}}\left(\scall{\nabla_{e_{1}}\sigma}{\sigma}\sigma+\scall{\nabla_{e_{1}}\sigma}{\sigma(\cdot, J\cdot)}\sigma(\cdot, J\cdot)\right)$. Combining this observation with the computations above, we deduce that
\begin{align*}
    \dd f\wedge \beta(e_{1},e_{2}) & = f^{-1}\scall{\dot{\sigma}_0}{\sigma}\scall{\nabla_{e_{1}}\sigma}{(\nabla_{e_{2}}\sigma)(\cdot, J\cdot)} \\
    &=\dot{f}\scall{\nabla_{e_{1}}\sigma}{(\nabla_{e_{2}}\sigma)(\cdot, J\cdot)} \\ 
    &=f\dot{f}\hat\omega_{\i}((0, \nabla_{e_{1}}\sigma),(0, \nabla_{e_{2}}\sigma)) \ ,
\end{align*}
which finally proves the identity a) outside the zeros of $\sigma$. In fact, since both terms of the equality extends continuously at the zeros of $\sigma$ (which form a finite set), we conclude that the identity $\dd f\wedge \beta = f\dot{f}\hat\omega_{\i}((0, \nabla_\bullet \sigma),(0, \nabla_\bullet \sigma))$ holds everywhere on $\Sigma$.

\paragraph{Proof of relation b)} By definition of exterior differential we have
\begin{align*}
    \dd \beta(e_{1},e_{2})&=e_{1}\left(\scall{\dot{\sigma}_{0}}{(\nabla_{e_{2}}\sigma)(\cdot,J\cdot)}\right)-e_{2}\left(\scall{\dot{\sigma}_{0}}{(\nabla_{e_{1}}\sigma)(\cdot,J\cdot)}\right)-\scall{\dot{\sigma}_{0}}{(\nabla_{[e_{1},e_{2}]}\sigma)(\cdot, J\cdot)} \\
    &=\scall{\nabla_{e_{1}}\dot{\sigma}_{0}}{(\nabla_{e_{2}}\sigma)(\cdot,J\cdot)}-\scall{\nabla_{e_{2}}\dot{\sigma}_{0}}{(\nabla_{e_{1}}\sigma)(\cdot,J\cdot)}+\scall{\dot{\sigma}_{0}}{R(e_{1},e_{2})\sigma(\cdot, J\cdot)}\\
    &=f\left(\hat\omega_{\i}((0,\nabla_{e_{1}}\dot{\sigma}_{0}),(0,\nabla_{e_{2}}\sigma))-\hat\omega_{\i}((0,\nabla_{e_{2}}\dot{\sigma}_{0}),(0,\nabla_{e_{1}}\sigma))\right)+\scall{\dot{\sigma}_{0}}{R(e_{1},e_{2})\sigma(\cdot, J\cdot)}
\end{align*}
where $R(e_{1},e_{2})\sigma=\nabla_{e_{1}}\nabla_{e_{2}}\sigma-\nabla_{e_{2}}\nabla_{e_{1}}\sigma-\nabla_{[e_{1},e_{2}]}\sigma$. The same proof that relates the Riemann curvature tensor with the Gaussian curvature adapts to the operator $R(e_{1},e_{2})\sigma$ and shows that
\[
    R(e_{1},e_{2})\sigma=2K_{J}\sigma(\cdot, J \cdot) \ .
\]
Relation b) follows.
\end{proof}

\begin{remark} \label{rmk:pippo}
In what follows, we will fix a primitive of the two form $d\tilde{\mu}_{\i}$ found in Proposition \ref{prop:variation_muI}, and define the linear map $F_{(J,\sigma)}: T_{(J,\sigma)} T^* \mathcal{J}(\Sigma) \rightarrow \Lambda^{1}(\Sigma)/B^{1}(\Sigma) \subset \Lsymp(\Sigma, \rho)^{*}$ so that $F_{(J,\sigma)}(\dot J,\dot \sigma)$ equals this primitive (modulo exact 1-forms). Note that $\Ker (F_{(J,\sigma)})$ is a subspace of $T_{(J,\sigma)}\tilde{mu}_\i^{-1}(0)$, and a priori the inclusion might be strict. By an abuse of notation, for the rest of the paper we will simply denote $(d\tilde{\mu}_{\i})_{(J,\sigma)}=F_{(J,\sigma)}$.

\end{remark}

We can finally show the connection between $\omega_{\i}$ and $\dd \tilde\mu_{\i}$:
\begin{proposition}\label{prop:fakemoment} Let $(J,\sigma) \in \widetilde{\mathcal{MS}}_{0}(\Sigma, \rho)$. For every $(\dot{J},\dot{\sigma}) \in T_{(J,\sigma)}T^{*}\mathcal{J}(\Sigma)$ and for every $V \in \Lsymp(\Sigma, \rho)$, we have
\[
    \omega_{\i}((\Dlie_{V}J,\Dlie_{V}\sigma),(\dot{J},\dot{\sigma}))=-\scal{\dd\tilde\mu_{\i}(\dot{J}, \dot{\sigma})}{V}_{\Lsymp} \ .
\]
\end{proposition}
\begin{proof} By definition of $\omega_{\i}$, we have
\begin{equation}\label{eq:omega_i_proof}
    \omega_{\i}((\Dlie_{V}J,\Dlie_{V}\sigma),(\dot{J},\dot{\sigma}))=\int_{\Sigma}\left(-f\scall{\Dlie_{V}J}{J\dot{J}}+f^{-1}\scall{(\Dlie_{V}\sigma)_{0}}{\dot{\sigma}_{0}(\cdot, J\cdot)}\right)\rho \ .
\end{equation}
In order to simplify the second term inside the integral, we make use of the following relation the proof of which will be postponed to the end:
\begin{align}\label{eq:scal_Lie_div}
    \scall{\dot{\sigma}_{0}(\cdot, J\cdot)}{(\Dlie_{V}\sigma)_{0}}=\scall{\dot{\sigma}_{0}(\cdot, J\cdot)}{\nabla_{V}\sigma}-\divr_{g_J}(JV)\scall{\sigma}{ \dot{\sigma}_{0}}
\end{align}
Using this fact and recalling that $\dot{f} = f^{-1} \scall{\sigma}{\dot{\sigma}_0}$ (by relation \eqref{eq:derivative_norm_sigma}), we have
\begin{align*}
\begin{split}
    f^{-1}\scall{(\Dlie_{V}\sigma)_{0}}{\dot{\sigma}_{0}(\cdot, J\cdot)}&=f^{-1}\scall{\dot{\sigma}_{0}(\cdot, J\cdot)}{\nabla_{V}\sigma}-f^{-1}\divr_{g_J}(JV)\scall{\sigma}{ \dot{\sigma}_{0}} \\
    &=f^{-1}\scall{\dot{\sigma}_{0}(\cdot, J\cdot)}{\nabla_{V}\sigma}-\divr_{g_J}(f^{-1}\scall{\sigma}{ \dot{\sigma}_{0}}JV)+\dd(f^{-1}\scall{\sigma}{\dot{\sigma}_{0}})JV \\
    &=f^{-1}\scall{\dot{\sigma}_{0}(\cdot, J\cdot)}{\nabla_{V}\sigma}-\divr_{g_J}(f^{-1}\scall{\sigma}{ \dot{\sigma}_{0}}JV)+\dd \dot{f}(JV) 
\end{split}
\end{align*}
and integrating over $\Sigma$ we find
\begin{align}\label{eq:PartII}
\begin{split}
    \int_{\Sigma}\left(f^{-1}\scall{(\Dlie_{V}\sigma)_{0}}{\dot{\sigma}_{0}(\cdot, J\cdot)}\right)\rho &=\int_{\Sigma}\left(-f^{-1}\scall{\dot{\sigma}_{0}}{(\nabla_{V}\sigma)(\cdot,J\cdot)}+\dd \dot{f}(JV)\right)\rho \\
    &= \int_{\Sigma} (\dd \dot{f}(JV)-f^{-1}\beta(V)) \rho \ , 
\end{split}
\end{align}
Using Equation (\ref{eq:divergence_rel}), we can rewrite the first term in 
(\ref{eq:omega_i_proof}) as 
\begin{align*}\label{eq:PartI}
\begin{split}
    \int_{\Sigma}-f\scall{\Dlie_{V}J}{J\dot{J}}\rho 
    &=\int_{\Sigma} \left(f(\divr_{g_J}\dot{J})V-f\divr_{g_J}(\dot{J}V)\right)\rho \\
    &=\int_{\Sigma} \left(f(\divr_{g_J}\dot{J})V+\dd f(\dot{J}V)-\divr_{g_J}(f\dot{J}V)\right)\rho \\
    &=\int_{\Sigma} \left(f(\divr_{g_J}\dot{J})V+\dd f(\dot{J}V)\right)\rho \ .
\end{split}
\end{align*}
Combining this last relation with \eqref{eq:omega_i_proof} and  \eqref{eq:PartII}, we find
\begin{align*}
    \omega_{\i}((\Dlie_{V}J,\Dlie_{V}\sigma),(\dot{J},\dot{\sigma}))&=\int_{\Sigma} \left( f(\divr_{g_J}\dot{J})V+ \dd f \circ \dot{J}V-f^{-1}\beta(V)+\dd \dot{f}(JV)\right)\rho \\
    &=\int_{\Sigma} \iota_{V}\left(f\divr_{g_J}\dot{J}+\dd f\circ \dot{J}+\dd \dot{f}\circ J-f^{-1}\beta\right)\rho \\
    &=\int_{\Sigma} (f\divr_{g_J}\dot{J}+\dd f\circ \dot{J}+\dd \dot{f}\circ J-f^{-1}\beta)\wedge \iota_{V}\rho \tag{by \eqref{eq:iota}}\\
    &=-\scal{\dd \tilde{\mu}_{\i}(\dot{J}, \dot{\sigma})}{V}_{\Lsymp} \ .
\end{align*}
We are left to prove Equation (\ref{eq:scal_Lie_div}). Let $A_{V}$ denote the endomorphism of $T\Sigma$ given by $A_{V}(X)=\nabla_{X}V$. It is easy to verify using the definition of Lie derivative that
\[
    \Dlie_{V}\sigma=\nabla_{V}\sigma+A_{V}^{t}\sigma+\sigma A_{V} \ .
\]
Because $V \in \Lsymp(\Sigma, \rho)$, we know that $\tr(A_{V})=0$. In particular, $A_{V} \in \mathrm{Span}(J,\dot{J},J\dot{J})$ and we can write
\begin{align*}
    &A_{V}=-\frac{1}{2}\tr(JA_{V})J+\frac{1}{2\tr(\dot{J}^{2})}\left(\tr(\dot{J}A_{V})\dot{J}+\tr(J\dot{J}A_{V})J\dot{J}\right) \\
    &A_{V}^{*}=\frac{1}{2}\tr(JA_{V})J+\frac{1}{2\tr(\dot{J}^{2})}\left(\tr(\dot{J}A_{V})\dot{J}+\tr(J\dot{J}A_{V})J\dot{J}\right) \\
    &\dot{J}A_{V}^{*}+A_{V}\dot{J}=-\tr(JA_{V})J\dot{J}+\frac{1}{2}\tr(\dot{J}A_{V})\1 
\end{align*}
where $A_{V}^{*}$ denotes the adjoint of $A_{V}$ with respect to $g_{J}$. Then for every $\dot{J}\in T_{J}\mathcal{J}(\Sigma)$, we have
\begin{align*}
    \scall{\dot{J}}{(\Dlie_{V}\sigma)_{0}} 
    &=\frac{1}{2}\tr(\dot{J}g_{J}^{-1} \Dlie_{V}\sigma) \\
    &=\frac{1}{2}\left(\tr(\dot{J}g_{J}^{-1}\nabla_{V}\sigma)+\tr(\dot{J}g_{J}^{-1}A_{V}^{t}\sigma)+\tr(\dot{J}g_{J}^{-1}\sigma A_{V})\right) \\
    &=\frac{1}{2}\left(\tr(\dot{J}g_{J}^{-1}\nabla_{V}\sigma)+\tr(\dot{J}A_{V}^{*}g_{J}^{-1}\sigma)+\tr(A_{V}\dot{J}g_{J}^{-1}\sigma)\right)\\
    &=\scall{\dot{J}}{\nabla_{V}\sigma}+\scal{\sigma}{\dot{J}A_{V}^{*}+A_{V}\dot{J}}\\
    &=\scall{\dot{J}}{\nabla_{V}\sigma}-\tr(JA_{V})\scal{\sigma}{J\dot{J}}\\
    &=\scall{\dot{J}}{\nabla_{V}\sigma}-\divr_{g_J}(JV)\scal{\sigma}{J\dot{J}} \ 
\end{align*}
and Equation (\ref{eq:scal_Lie_div}) follows by taking $\dot{J}=\dot{\sigma}_{0}(\cdot, J\cdot)$.
\end{proof}

\begin{remark}\label{rmk:mu_hat} Note that, because $\tilde{\mu}_{\i}$ and $\mu_{\i}$ differ by a constant, Proposition \ref{prop:fakemoment} holds for $\mu_{\i}$ as well.     
\end{remark}

\subsubsection{The differential of the map $\tilde{\mu}_{\j} + i \, \tilde{\mu}_{\k}$}
With similar techniques, we can compute the differential of the other two moment maps:
\begin{proposition} \label{prop:variation_moment_maps}
	For every $(J, \sigma) \in T^* \mathcal{J}(\Sigma)$ and for every $(\dot{J}, \dot{\sigma}) \in T_{(J, \sigma)} T^* \mathcal{J}(\Sigma)$ we have:
	\[
	\dd(\tilde{\mu}_\mathbf{J} + i \, \tilde{\mu}_{\mathbf{K}})(\dot{J}, \dot{\sigma}) = - i \, \left( \divr_{g_J}(g^{-1} \dot{\bar{\phi}}_0 )+ \scal{\nabla_{J \bullet} \bar{\phi}}{\dot{J}} \right) - \,  \scal{\bar{\phi}}{\nabla_{\bullet} \dot{J} + i \, \nabla_{J \bullet} \dot{J} }  \in \Lsymp(\Sigma,\rho)^* \otimes \C ,
	\]
	where $\phi$ is the quadratic differential whose real part is equal to $\sigma$, $\dot{\bar{\phi}}_0$ denotes the $g_J$-\hsk traceless part of $\dot{\bar{\phi}}$. 
\end{proposition} 
\begin{proof} To simplify the notation, we set $\divr (T)=\divr_{g_{J}}(g^{-1}T)$ for any symmetric $2$-tensor $T$. In the proof of Theorem \ref{thm:donaldson_maps_cotangent_bundle}, we showed that $r(\partial_J \bar{\phi}) = \frac{1}{2} (\divr \bar{\phi})$. In particular we have $(\tilde{\mu}_\mathbf{J} + i \, \tilde{\mu}_\mathbf{K})(J,\sigma) = [- i \divr \bar{\phi}]$, where $\phi$ is the quadratic differential whose real part coincides with $\sigma$. Therefore, in order to compute the differential of the moment map $\tilde{\mu}_\mathbf{J} + i \, \tilde{\mu}_\mathbf{K}$, we need to understand the variation of the quantity $\divr \bar{\phi}$ along $(\dot{J}, \dot{\sigma}) \in T_{(J, \sigma)} T^* \mathcal{J}(\Sigma)$.
	
	Let $\mathcal{C}^i_j S$ denote the contraction of the tensor $S$ in its $i$-\hsk th covariant and $j$-\hsk th contravariant entries. Then the $g_J$-\hsk divergence of a symmetric $2$-\hsk tensor $T$ can be expressed as 
	\[
	\mathcal{C}_1^1 \mathcal{C}_1^1 (g^* \otimes \nabla_\bullet T) ,
	\]
	where $g^*$ denotes the metric induced by $g = g_J$ on $T^* \Sigma$. As seen in relation \eqref{eq:derivative_scal_prod}, the first order variation of $g = \rho(\cdot, J \cdot)$ along $\dot{J}$ can be expressed as $\dot{g} = - g(\cdot, J \dot{J} \cdot)$. It is simple to check that the corresponding variation of $g^*$ satisfies $\dot{g}^* = g^*(\cdot , (J \dot{J})^t \cdot)$, where $(J \dot{J})^t$ is the transpose of $J \dot{J}$.  In particular, for every symmetric $2$-\hsk tensor $T$ we have
	\begin{align*}
	(\divr T)' & = \mathcal{C}_1^1 \mathcal{C}_1^1 (\dot{g}^* \otimes \nabla_\bullet T + g^* \otimes \dot{\nabla}_\bullet T + g^* \otimes \nabla_\bullet \dot{T} ) \\
	& = \underbrace{\sum_i (\nabla_{e_i} T)(J \dot{J} e_i, \cdot)}_{\text{term $1$}} + \underbrace{\sum_i (\dot{\nabla}_{e_i} T)(e_i, \cdot)}_{\text{term $2$}} + \underbrace{\vphantom{\sum_i} \divr \dot{T}}_{\text{term $3$}} ,
	\end{align*}
	where $(e_i)_i$ is a local $g$-\hsk orthonormal frame. In order to compute the differential of $\tilde{\mu}_\mathbf{J} + i \, \tilde{\mu}_{\mathbf{K}}$, we will study each term of this expression for $T = \bar{\phi}$.
	
	\paragraph*{Term 1} For every $\dot{J} \in T_J \mathcal{J}(\Sigma)$, for every $J$-\hsk quadratic differential $\phi$ and for every tangent vector field $V$ on $\Sigma$, we have
	\[
	(\nabla_X \bar{\phi})(J \dot{J} \cdot, \cdot) = \scal{\nabla_X \bar{\phi}}{J \dot{J}} \ g  + \scal{\nabla_X \bar{\phi}}{\dot{J}} \ g(\cdot, J \cdot) .
	\]
	This relation is a simple application of Lemma \ref{lem:product_in_TJ}, where the endomorphisms in $T_J \mathcal{J}(\Sigma)$ are the real (or imaginary) part of $g^{-1} \nabla_X \bar{\phi}$, and $J \dot{J}$ (we are implicitly extending the bilinear pairing $\scal{\cdot}{\cdot}$ to the complexified bundles by requiring the $\C$-\hsk linearity in its arguments, i. e. $i \, \scal{\cdot}{\cdot} = \scal{i \,\cdot}{\cdot} = \scal{\cdot}{i \,\cdot}$). From this expression we deduce that
	\begin{align*}
	\sum_i (\nabla_{e_i} \bar{\phi})(J \dot{J} e_i, V) & = \sum_i (\scal{\nabla_{e_i} \bar{\phi}}{J \dot{J}} \ g(e_i, V)  + \scal{\nabla_{e_i} \bar{\phi}}{\dot{J}} \ g(e_i, J V) ) \\
	& = \scal{\nabla_V \bar{\phi}}{J \dot{J}} + \scal{\nabla_{J V} \bar{\phi}}{\dot{J}} \\
	& = - i \, \scal{\nabla_V \bar{\phi}}{\dot{J}} + \scal{\nabla_{J V} \bar{\phi}}{\dot{J}} ,
	\end{align*}
	where in the last step we used the fact that $\nabla_V \bar{\phi}$ is $J$-\hsk antilinear in its arguments.
	
	\paragraph*{Term 2}
	
	Applying Lemma \ref{lem:variation_levicivita}, we see that
	\begin{align*}
	(\dot{\nabla}_V \bar{\phi})(X,Y) & = - \bar{\phi}(\dot{\nabla}_V X, Y) - \bar{\phi}(X, \dot{\nabla}_V Y) \\
	& \begin{multlined}
	= \frac{1}{2} \left( (\divr \dot{J})V (\bar{\phi}(J X, Y) + \bar{\phi}(X, J Y)) + \bar{\phi}(J (\nabla_V \dot{J})X, Y) + \right. \\
	\left. + \bar{\phi}(X, J (\nabla_V \dot{J})Y) \right) 
	\end{multlined} \\
	& = - i \, (\divr \dot{J})V \ \bar{\phi}(X, Y) - \frac{i}{2} \left( \bar{\phi}((\nabla_V \dot{J})X, Y) + \bar{\phi}(X, (\nabla_V \dot{J})Y) \right) .
	\end{align*}
	In the last line, we used the fact that $\bar{\phi}$ is $J$-\hsk antilinear. Applying Lemma \ref{lem:product_in_TJ} to the real (and imaginary, separately) part of $\bar{\phi}$ and to $\nabla_V \dot{J}$, we deduce the following relation:
	\[
	\bar{\phi}((\nabla_V \dot{J}) \cdot, \cdot) + \bar{\phi}( \cdot, (\nabla_V \dot{J}) \cdot) = 2 \scal{\bar{\phi}}{\nabla_V \dot{J}} \ g ,
	\]
	This, combined with the previous computation leads us to the following expression:
	\begin{equation} \label{eq:derivata_covariante_derivata}
	(\dot{\nabla}_V \bar{\phi})(X,Y) = - i \, (\divr \dot{J})V \ \bar{\phi}(X, Y) - i \, \scal{\bar{\phi}}{\nabla_V \dot{J}} \ g(X,Y) .
	\end{equation}
	Moreover, the following equality holds:
	\begin{equation} \label{eq:relazione_tecnica_div}
	\sum_i (\divr \dot{J}) e_i \ \bar{\phi}(e_i, V) = \scal{\bar{\phi}}{\nabla_V \dot{J} + i \, \nabla_{JV} \dot{J} } .
	\end{equation}
	We will temporarily assume this fact, the proof is postponed to the end of the current argument. We can now express the second term of our initial expression as follows:
	\begin{align*}
	\sum_i (\dot{\nabla}_{e_i} \bar{\phi})(e_i, V) & = - i \sum_i (\divr \dot{J}) e_i \ \bar{\phi}(e_i, V) - i  \sum_i \scal{\bar{\phi}}{\nabla_{e_i} \dot{J}} \ g(e_i, V) \tag{relation \eqref{eq:derivata_covariante_derivata}} \\
	& = - i \, \scal{\bar{\phi}}{\nabla_V \dot{J} + i \, \nabla_{JV} \dot{J} } - i \sum_i \scal{\bar{\phi}}{\nabla_{e_i} \dot{J}} \ g(e_i, V)  \tag{relation \eqref{eq:relazione_tecnica_div}} \\
	& = - i \, \scal{\bar{\phi}}{\nabla_V \dot{J} + i \, \nabla_{JV} \dot{J} } - i \, \scal{\bar{\phi}}{\nabla_V \dot{J}} .
	\end{align*}
	
	\paragraph*{Term 3}
	
	Following the same argument of the proof of Lemma \ref{lem:characterization_tangent_space}, we see that the first order variation of a quadratic differential $\phi$ is of the form $\dot{\phi} = \dot{\phi}_0 - \scal{\phi}{J \dot{J}} \ g$, where $\dot{\phi}_0$ denotes the $g$-\hsk traceless part of $\dot{\phi}$. In particular, we deduce that
	\begin{align*}
	(\divr \dot{\bar{\phi}})V & =  \divr(\dot{\bar{\phi}}_0 - \scal{\bar{\phi}}{J \dot{J}} \ g) V \\
	& = (\divr \dot{\bar{\phi}}_0)V - \dd(\scal{\bar{\phi}}{J \dot{J}})V \\
	& = (\divr \dot{\bar{\phi}}_0)V + i \, \dd(\scal{\bar{\phi}}{\dot{J}})V \tag{$\bar{\phi}$ $J$-\hsk antilinear} \\
	& = (\divr \dot{\bar{\phi}}_0)V + i \, \scal{\nabla_V \bar{\phi}}{\dot{J}} + i \, \scal{\bar{\phi}}{\nabla_V \dot{J}} .
	\end{align*}
	
	\vspace{0.5cm}
	
	Finally, we combine the expressions of the three terms involved in the derivative $(\divr \bar{\phi})'$ that we developed above, obtaining:
	\begin{align*}
	\dd(\tilde{\mu}_\mathbf{J} + i \, \tilde{\mu}_{\mathbf{K}})(\dot{J}, \dot{\sigma}) & = - i \, (\divr \bar{\phi})' \\
	& = - i \left( \sum_i (\nabla_{e_i} \bar{\phi})(J \dot{J} e_i, \cdot) + \sum_i (\dot{\nabla}_{e_i} \bar{\phi})(e_i, \cdot) +\divr \dot{\bar{\phi}} \right) \\
	& \begin{multlined}
	= - i \left( - i \, \scal{\nabla_\bullet \bar{\phi}}{\dot{J}} + \scal{\nabla_{J \bullet} \bar{\phi}}{\dot{J}} - i \, \scal{\bar{\phi}}{\nabla_\bullet \dot{J} + i \, \nabla_{J \bullet} \dot{J} } + \right. \\
	\left. - i \, \scal{\bar{\phi}}{\nabla_\bullet \dot{J}} + \divr \dot{\bar{\phi}}_0 + i \, \scal{\nabla_\bullet \bar{\phi}}{\dot{J}} + i \, \scal{\bar{\phi}}{\nabla_\bullet \dot{J}} \right)
	\end{multlined} \\
	& = - i \left( \divr \dot{\bar{\phi}}_0 + \scal{\nabla_{J \bullet} \bar{\phi}}{\dot{J}} - i \, \scal{\bar{\phi}}{\nabla_\bullet \dot{J} + i \, \nabla_{J \bullet} \dot{J} } \right) ,
	\end{align*}
	which proves our statement.
 
    Finally, we provide a proof of Equation \eqref{eq:relazione_tecnica_div}. 
    Since $\dot{J}$ is symmetric with respect to $g$, the same is true for $\nabla_X \dot{J}$ for any tangent vector field $X$. In particular, if $(\nabla_i \dot{J})_{j k}$ denotes $g((\nabla_{e_i} \dot{J})e_j,e_k)$, we must have $(\nabla_i \dot{J})_{j k} = (\nabla_i \dot{J})_{k j}$ for every $i, j, k$. Unraveling the definition of the divergence of $\dot{J}$, we see that
	\begin{align*}
	    (\divr \dot{J})(e_1) \, e_1 + (\divr \dot{J})(e_2) \, e_2 & = ((\nabla_1 \dot{J})_{11} + (\nabla_2 \dot{J})_{12}) \, e_1 + ((\nabla_{1} \dot{J})_{21} + (\nabla_{2} \dot{J})_{22}) \, e_2 \\
	    & = ((\nabla_1 \dot{J})_{11} + (\nabla_2 \dot{J})_{21}) \, e_1 + ((\nabla_{1} \dot{J})_{12} + (\nabla_{2} \dot{J})_{22}) \, e_2 \\
	    & = (\nabla_{e_1} \dot{J})e_1 + (\nabla_{e_2} \dot{J})e_2 .
	\end{align*}
	Hence we have
	\begin{align*}
	    \sum_i (\divr \dot{J}) e_i \ \bar{\phi}(e_i, V) & = \bar{\phi}((\nabla_{e_1} \dot{J})e_1 + (\nabla_{e_2} \dot{J})e_2, V) \\
	    & = V^1 \, \bar{\phi}((\nabla_{e_1} \dot{J})e_1 + (\nabla_{e_2} \dot{J})e_2, e_1) + V^2 \, \bar{\phi}((\nabla_{e_1} \dot{J})e_1 + (\nabla_{e_2} \dot{J})e_2, e_2) ,
	\end{align*}
	where $V^i = g(V, e_i)$. We will now make multiple use of the following elementary properties:
	\begin{itemize}
	    \item $e_2 = J e_1$, $V^1 = g(JV, e_2) = (JV)^2$, $V^2 =  - g(JV, e_1) = - (JV)^1$;
	    \item $(\nabla_{e_i} \dot{J})$ is $g$-symmetric and traceless. In particular, it anticommutes with the complex structure $J$;
	    \item the tensor $\bar{\phi}$ is anti-bilinear in its entries, therefore $\bar{\phi}(J \cdot, \cdot) = \bar{\phi}(\cdot, J \cdot) = - i \bar{\phi}(\cdot, \cdot)$.
	\end{itemize}
	From the previous expression we deduce
	\begin{align} \label{eq:rel_tecnica_steps}
	\begin{split}
	    \sum_i (\divr \dot{J}) e_i \ \bar{\phi}(e_i, V) & = V^1 \bar{\phi}((\nabla_{e_1} \dot{J})e_1, e_1) + (JV)^2 \bar{\phi}((\nabla_{e_2} \dot{J}) J e_1, e_1) + \\
	    & \ \ \ - (JV)^1 \bar{\phi}((\nabla_{e_1} \dot{J})e_1, J e_1) + V^2 \bar{\phi}((\nabla_{e_2} \dot{J})J e_1, J e_1) \\
	    & = V^1 \bar{\phi}((\nabla_{e_1} \dot{J})e_1, e_1) + i V^2 \bar{\phi}(J(\nabla_{e_2} \dot{J})e_1, e_1) + \\
	    & \ \ \ - (JV)^2 \bar{\phi}(J(\nabla_{e_2} \dot{J})e_1, e_1) + i (JV)^1 \bar{\phi}((\nabla_{e_1} \dot{J})e_1, e_1) \\
	    & = V^1 \bar{\phi}((\nabla_{e_1} \dot{J})e_1, e_1) + V^2 \bar{\phi}((\nabla_{e_2} \dot{J})e_1, e_1) + \\
	    & \ \ \ + i (JV)^2 \bar{\phi}((\nabla_{e_2} \dot{J})e_1, e_1) + i (JV)^1 \bar{\phi}((\nabla_{e_1} \dot{J})e_1, e_1) \\
	    & = \bar{\phi}((\nabla_{V} \dot{J})e_1, e_1) + i \, \bar{\phi}((\nabla_{JV} \dot{J})e_1, e_1)
	\end{split}
	\end{align}
	In order to derive the desired relation, we are left to show that \[
	\bar{\phi}((\nabla_{V} \dot{J})e_1, e_1) + i \, \bar{\phi}((\nabla_{JV} \dot{J})e_1, e_1) = \scal{\bar{\phi}}{\nabla_V \dot{J} + i \nabla_{JV} \dot{J}}
	\]
    This equality can be deduced from the properties of $\bar{\phi}$ and $\nabla_X \dot{J}$ previously mentioned. Indeed we have
	\begin{align*}
	    \scal{\bar{\phi}}{\nabla_V \dot{J}} & = \frac{1}{2} \tr(g^{-1} \bar{\phi} \nabla_V \dot{J}) \\
	    & = \frac{1}{2} \left( \bar{\phi}((\nabla_V \dot{J})e_1, e_1) + \bar{\phi}((\nabla_V \dot{J}) J e_1, J e_1) \right) \\
	    & = \frac{1}{2} \left( \bar{\phi}((\nabla_V \dot{J})e_1, e_1) - \bar{\phi}(J(\nabla_V \dot{J}) e_1, J e_1) \right) \\
	    & = \bar{\phi}((\nabla_V \dot{J})e_1, e_1) .
	\end{align*} 
	By replacing the role of $V$ with $JV$, we obtain also that $\scal{\bar{\phi}}{\nabla_{JV} \dot{J}} = \bar{\phi}((\nabla_{J V} \dot{J})e_1, e_1)$. This cocludes the proof of relation \eqref{eq:relazione_tecnica_div}.
\end{proof}

\subsubsection{Model for the tangent space to $\mathcal{MS}(\Sigma)$} 
We finally come to the main statement of this section.

\begin{reptheorem}{thm:char_VJSigma}
For every $(J,\sigma)\in \widetilde{\mathcal{MS}}_0(\Sigma,\rho)$, 
$V_{(J,\sigma)}$ is the largest subspace of $T_{(J,\sigma)}\widetilde{\mathcal{MS}}_0(\Sigma,\rho)$ that is:
\begin{itemize}
    \item invariant under $\i$, $\j$ and $\k$;
    \item $\g$-orthogonal to $T_{(J,\sigma)}(\Symp_{0}(\Sigma, \rho)\cdot (J,\sigma))$
\end{itemize}
\end{reptheorem}

The proof of Theorem \ref{thm:char_VJSigma} is completed in the rest of this section, by means of three lemmas that simplify the statement in several steps.
\begin{lemma}\label{lm:invariance} Let $(J,\sigma) \in \widetilde{\mathcal{MS}}_{0}(\Sigma, \rho)$ and $(\dot{J},\dot{\sigma}) \in T_{(J,\sigma)}T^{*}\mathcal{J}(\Sigma)$. The following conditions are equivalent:
\begin{enumerate}[a)]
    \item $(\dot{J},\dot{\sigma}), \i(\dot{J},\dot{\sigma}), \j(\dot{J},\dot{\sigma})$ and $\k(\dot{J},\dot{\sigma})$ belong to $\Ker(\dd \tilde{\mu}_{\i})\cap\Ker(\dd \tilde{\mu}_{\j})\cap \Ker(\dd\tilde{\mu}_{\k})$;
    \item $(\dot{J},\dot{\sigma})$ and $\mathbf{J}(\dot{J},\dot{\sigma})$ belong to $\Ker(\dd \tilde{\mu}_{\mathbf{J}})\cap \Ker(\dd\tilde{\mu}_{\k})$. 
\end{enumerate}
\end{lemma}
\begin{proof} We only need to prove that b) implies a). By Proposition \ref{prop:fakemoment} (plus Remark \ref{rmk:mu_hat}) and the quaternionic relations between $\i$, $\j$ and $\k$, we find 
\begin{align*}
    \scal{\dd \tilde{\mu}_{\i}(\j(\dot{J},\dot{\sigma}))}{V}_{\Lsymp}&=-\omega_{\i}(\j(\dot{J},\dot{\sigma}),(\Dlie_{V}J,\Dlie_{V}\sigma))
    =\g(\i\j(\dot{J},\dot{\sigma}), (\Dlie_{V}J,\Dlie_{V}\sigma)) \\
    &=\g(\k(\dot{J},\dot{\sigma}), (\Dlie_{V}J,\Dlie_{V}\sigma))
    =-\omega_{\k}((\dot{J},\dot{\sigma}),(\Dlie_{V}J,\Dlie_{V}\sigma))\\
    &=\scal{\dd \tilde{\mu}_{\k}(\dot{J},\dot{\sigma})}{V}_{\Lsymp} \ .
\end{align*}
Therefore, $(\dot{J},\dot{\sigma}) \in \Ker(\dd \tilde{\mu}_{\k})$ if and only if $\j(\dot{J},\dot{\sigma}) \in \Ker(\dd \tilde{\mu}_{\i})$. With a similar computation, one can also show that $(\dot{J},\dot{\sigma}) \in \Ker(\dd \tilde{\mu}_{\k})$ if and only if $\i(\dot{J},\dot{\sigma}) \in \Ker(\dd \tilde{\mu}_{\j})$. 
It follows that if we start from a pair $(\dot{J},\dot{\sigma})$ satisfying b), then 
\begin{itemize}
    \item $(\dot{J},\dot{\sigma}) \in \Ker(\dd \tilde{\mu}_{\k}) \Rightarrow \i(\dot{J},\dot{\sigma}) \in \Ker(\dd \tilde{\mu}_{\j})$ and $\j (\dot{J},\dot{\sigma}) \in \Ker(\dd \tilde{\mu}_{\i})$;
    \item $(\dot{J},\dot{\sigma}) \in \Ker(\dd \tilde{\mu}_{\j}) \Rightarrow \i(\dot{J},\dot{\sigma}) \in \Ker(\dd \tilde{\mu}_{\k})$ and $\k (\dot{J},\dot{\sigma}) \in \Ker(\dd \tilde{\mu}_{\i})$;
    \item $\j(\dot{J},\dot{\sigma}) \in \Ker(\dd \tilde{\mu}_{\k}) \Rightarrow \k(\dot{J},\dot{\sigma}) \in \Ker(\dd \tilde{\mu}_{\j})$ and $(\dot{J},\dot{\sigma}) \in \Ker(\dd \tilde{\mu}_{\i})$;
    \item $\j(\dot{J},\dot{\sigma}) \in \Ker(\dd \tilde{\mu}_{\j}) \Rightarrow \k(\dot{J},\dot{\sigma}) \in \Ker(\dd \tilde{\mu}_{\k})$ and $\i (\dot{J},\dot{\sigma}) \in \Ker(\dd \tilde{\mu}_{\i})$
\end{itemize}
which implies a). 
\end{proof}

\begin{lemma}\label{lm:orthogonality}Let $(J,\sigma) \in \widetilde{\mathcal{MS}}_{0}(\Sigma, \rho)$ and $(\dot{J},\dot{\sigma}) \in T_{(J,\sigma)}T^{*}\mathcal{J}(\Sigma)$. The following conditions are equivalent:
\begin{enumerate}[1)]
    \item $(\dot{J},\dot{\sigma})$ is $\g$-orthogonal to $T_{(J,\sigma)}(\Symp_{0}(\Sigma, \rho)\cdot (J,\sigma))$;
    \item for some $\mathbf{L}\in \{\i,\j,\k\}$, we have $\mathbf{L}(\dot{J},\dot{\sigma}) \in \Ker(\dd \tilde{\mu}_{\mathbf{L}})$;
    \item for all $\mathbf{L}\in \{\i,\j,\k\}$, we have $\mathbf{L}(\dot{J},\dot{\sigma}) \in \Ker(\dd \tilde{\mu}_{\mathbf{L}})$.
\end{enumerate} 
\end{lemma}      
\begin{proof} We first show that 1) is equivalent to 2) for $L=\j$, the other cases being analogous. The properties of the moment maps imply
\begin{align*}
        \mathbf{g}((\Dlie_V J,\Dlie_{V}\sigma), (\dot{J}, \dot{\sigma})) & =  \mathbf{g}((\Dlie_V J,\Dlie_{V}\sigma), \mathbf{J}^2(\dot{J}, \dot{\sigma})) \tag{$\mathbf{J}^2 = \1$} \\
        & = \omega_\mathbf{J}((\Dlie_V J,\Dlie_{V}\sigma), \mathbf{J} (\dot{J},\dot{\sigma})) \tag{$\omega_\mathbf{J} = \mathbf{g}(\cdot, \mathbf{J} \cdot)$} \\
        & = -\scal{\dd{\tilde{\mu}_\mathbf{J}}(\mathbf{J}(\dot{J}, \dot{\sigma}))}{V}_{\Lsymp} \ . 
\end{align*}
so $(\dot{J},\dot{\sigma})$ is $\g$-orthogonal to $(\Dlie_V J,\Dlie_{V}\sigma)$ for every $V\in \Lsymp$ if and only if $\j(\dot{J},\dot{\sigma}) \in \Ker(\dd \tilde{\mu}_{\mathbf{J}})$. 
Note that the above relation proves that 1) and 3) are equivalent as well.
\end{proof}

\noindent Combining Lemma \ref{lm:invariance} and Lemma \ref{lm:orthogonality}, we see that the subspace $V_{(J,\sigma)} \subset T_{(J,\sigma)}T^{*}\mathcal{J}(\Sigma)$ we are interested in can be described as
\[
    V_{(J,\sigma)}=\{(\dot{J},\dot{\sigma}) \in T_{(J,\sigma)}T^{*}\mathcal{J}(\Sigma) \ | \ (\dot{J},\dot{\sigma}), \mathbf{J}(\dot{J},\dot{\sigma}) \in \Ker(\dd \tilde{\mu}_{\mathbf{J}})\cap \Ker(\dd\tilde{\mu}_{\k})\}
\]
Finally, the following result allows to describe $V_{(J,\sigma)}$ as the solution of the system of partial differential equations appearing in Proposition \ref{prop:equivalent_def_subspace_V} part $ii)$.
\begin{lemma} \label{lem:equations_tangent}
    Let $(J, \sigma) \in T^* \mathcal{J}(\Sigma)$ be such that $\sigma$ is the real part of a holomorphic quadratic differential $\phi$ on $(\Sigma, J)$, and let $(\dot{J}, \dot{\sigma}) \in T_{(J,\sigma)} T^* \mathcal{J}(\Sigma)$. Then $(\dot{J}, \dot{\sigma})$ and $\mathbf{J}(\dot{J}, \dot{\sigma})$ belong to $\ker \dd{\tilde{\mu}_\mathbf{J}} \cap \ker \dd{\tilde{\mu}_\mathbf{K}}$ if and only if
    \begin{equation}\label{eq:system}
    \begin{cases}
    \divr_{g} (g^{-1} \dot{\sigma}_0 + J \, \dot{J} \, g^{-1} \sigma) = \scal{\nabla^g_{J \bullet} \sigma}{\dot{J}} , \\
				\divr_{g} (f \, \dot{J} + f^{-1} \, J \, g^{-1} \dot{\sigma}_0 \, g^{-1} \sigma) = \scal{\nabla^g_{J \bullet} \sigma}{f^{-1}g^{-1} \dot{\sigma}_0} .
	\end{cases}
    \end{equation}
\end{lemma}
\begin{proof}
    Applying Proposition \ref{prop:variation_moment_maps} and dividing $\dd{(\tilde{\mu}_\mathbf{J} + i \, \tilde{\mu}_\mathbf{K})}$ into real and imaginary part, respectively, we see that the following relations hold:
    \begin{align*}
    \dd{\tilde{\mu}_\mathbf{J}}(\dot{J}, \dot{\sigma}) & = [\divr_{g}(g^{-1}\dot{\sigma}_0(\cdot, J \cdot)) + \scal{\nabla_{J \bullet} \sigma}{J \dot{J}} - \scal{\sigma}{\nabla_\bullet \dot{J} - J \nabla_{J \bullet} \dot{J}}] \\
    & = [J^*(\divr_{g} g^{-1}\dot{\sigma}_0 + \scal{\nabla_\bullet \sigma}{J \dot{J}} + \scal{\sigma}{J (\nabla_\bullet \dot{J} - J \nabla_{J \bullet} \dot{J})})] \in \Lambda^1(\Sigma) / B^1(\Sigma), \\
    & \\
    \dd{\tilde{\mu}_\mathbf{K}}(\dot{J}, \dot{\sigma}) & = - [ \divr_g g^{-1}\dot{\sigma}_0 + \scal{\nabla_{J \bullet} \sigma}{\dot{J}} + \scal{\sigma}{J (\nabla_\bullet \dot{J} - J \nabla_{J \bullet} \dot{J})} ] \in \Lambda^1(\Sigma) / B^1(\Sigma) .
    \end{align*}
    
    Since $\phi$ is a holomorphic quadratic differential, it satisfies $\partial_J \bar{\phi} \equiv 0$, which is equivalent to require $\nabla_{J \bullet} \sigma = (\nabla_\bullet \sigma)(\cdot, J \cdot)$. In light of this relation, if we set $\alpha$ to be the $1$-\hsk form
    \[
    \alpha \defin \divr_{g}g^{-1} \dot{\sigma}_0 + \scal{\nabla_\bullet \sigma}{J \dot{J}} + \scal{\sigma}{J(\nabla_\bullet \dot{J} - J \nabla_{J \bullet} \dot{J})} ,
    \]
    then we can express $\dd{\tilde{\mu}_\mathbf{J}}(\dot{J}, \dot{\sigma})$ and $\dd{\tilde{\mu}_\mathbf{K}}(\dot{J}, \dot{\sigma})$ as follows:
    \[
    \dd{\tilde{\mu}_\mathbf{J}}(\dot{J}, \dot{\sigma}) = [\alpha \circ J] , \ \dd{\tilde{\mu}_\mathbf{K}}(\dot{J}, \dot{\sigma}) = [- \alpha] \in \Lambda^1(\Sigma) / B^1(\Sigma) \subseteq \Lsymp(\Sigma, \rho)^* .
    \]
    
    Assume now that $(\dot{J}, \dot{\sigma})$ satisfies $\dd{\tilde{\mu}_\mathbf{J}} (\dot{J}, \dot{\sigma}) = \dd{\tilde{\mu}_\mathbf{K}} (\dot{J}, \dot{\sigma}) = [0] \in \Lambda^1(\Sigma) / B^1(\Sigma)$ or, equivalently, that the forms $\alpha$ and $\alpha \circ J$ are exact. In particular there exists a smooth function over $\Sigma$ such that $\alpha = \dd f$. Since $\alpha \circ J$ is exact, we also have
    \[
    - (\Delta_{g_J} f) \, \rho = \dd(\dd f \circ J) = \dd(\alpha \circ J) = 0 .
    \]
    In other words, the function $f$ has to be harmonic with respect to $g_J$ and therefore constant, since $\Sigma$ is compact without boundary. This proves in particular that the $1$-\hsk form $\alpha$ vanishes identically if and only if $(\dot{J}, \dot{\sigma})$ belongs to $\ker \dd{\tilde{\mu}_\mathbf{J}} \cap \ker \dd{\tilde{\mu}_\mathbf{K}}$. The form $\alpha$ can be expressed as follows:
    \begin{align*}
    \alpha & = \divr_g (g^{-1}\dot{\sigma}_0) + \scal{\nabla_\bullet \sigma}{J \dot{J}} + \scal{\sigma}{J(\nabla_\bullet \dot{J} - J \nabla_{J \bullet} \dot{J})} \\
    & = \divr_g(g^{-1} \dot{\sigma}_0 + \scal{\sigma}{J \dot{J}} \, \1 + \scal{\sigma}{\dot{J}} J) - \scal{\nabla_{J \bullet} \sigma}{\dot{J}} \\
    & = \divr_g(g^{-1} \dot{\sigma}_0 + J \dot{J} g^{-1} \sigma) - \scal{\nabla_{J \bullet} \sigma}{\dot{J}} . \tag{Lemma \ref{lem:product_in_TJ}}
    \end{align*}
    
    If we apply the same argument to $\mathbf{J}(\dot{J}, \dot{\sigma})$ (see Section \ref{subsec:parahyperkahler_toy} for the definition of $\mathbf{J}$), we obtain that
    \begin{align}
    (\dot{J}, \dot{\sigma}) \in \ker \dd{\tilde{\mu}_\mathbf{J}} \cap \ker \dd{\tilde{\mu}_\mathbf{K}} & \Leftrightarrow \divr_g(g^{-1} \dot{\sigma}_0 + J \dot{J} g^{-1} \sigma) = \scal{\nabla_{J \bullet} \sigma}{\dot{J}} , \label{eq:char1} \\
    \mathbf{J}(\dot{J}, \dot{\sigma}) \in \ker \dd{\tilde{\mu}_\mathbf{J}} \cap \ker \dd{\tilde{\mu}_\mathbf{K}} & \Leftrightarrow \divr_g(f \dot{J} + f^{-1} J g^{-1} \dot{\sigma}_0 g^{-1} \sigma) = \scal{\nabla_{J \bullet} \sigma}{f^{-1}g^{-1} \dot{\sigma}_0} \label{eq:char2} ,
    \end{align}
    where $f = f(\norm{\sigma}) = \sqrt{1 + \norm{\sigma}^2}$, thus obtaining the desired statement. 
\end{proof}

\begin{proof}[Proof of Theorem \ref{thm:char_VJSigma}] 
By Corollary \ref{cor:zero moment maps}, the zero locus of the moment maps $\tilde{\mu}_{\i}$, $\tilde{\mu}_{\j}$ and $\tilde{\mu}_{\k}$ coincides precisely with $\widetilde{\mathcal{MS}}_0(\Sigma,\rho)$. Observe, however, that $T_{(J,\sigma)}\widetilde{\mathcal{MS}}_0(\Sigma,\rho)$ is larger than  $\Ker(\dd \tilde{\mu}_{\i})\cap\Ker(\dd \tilde{\mu}_{\j})\cap \Ker(\dd\tilde{\mu}_{\k})$ by Remark \ref{rmk:pippo}. Nonetheless, the largest subspace $W$ of $T_{(J,\sigma)}\widetilde{\mathcal{MS}}_0(\Sigma,\rho)$ that is $\g$-orthogonal to $T_{(J,\sigma)}(\Symp_{0}(\Sigma, \rho)\cdot(J,\sigma))$ and invariant under $\i$, $\j$ and $\k$ is contained in $\Ker(\dd \tilde{\mu}_{\i})\cap\Ker(\dd \tilde{\mu}_{\j})\cap \Ker(\dd\tilde{\mu}_{\k})$. Indeed, if $(\dot{J}, \dot{\sigma})$ is in $W$, then the same is true for $\i(\dot{J}, \dot{\sigma})$. Since $\i(\dot{J}, \dot{\sigma})$ is $\g$-orthogonal to the tangent of the orbit by $\Symp_0(\Sigma,\rho)$, Lemma \ref{lm:orthogonality} implies that $\i^2 (\dot{J},\dot{\sigma}) = - (\dot{J}, \dot{\sigma})$ lies in $\Ker (\dd \tilde{\mu}_\i)$. Being $(\dot{J}, \dot{\sigma})$ arbitrary, we deduce that $W$ is contained in $\Ker(\dd \tilde{\mu}_{\i})\cap\Ker(\dd \tilde{\mu}_{\j})\cap \Ker(\dd\tilde{\mu}_{\k})$.

Now, by Lemma \ref{lm:invariance}, Lemma \ref{lm:orthogonality}, and Lemma \ref{lem:equations_tangent} the subspace $W$ is described by the Equation \eqref{eq:system}. Taking the sum and the difference of the two equations, and using that $J_l = f \, J + g_J^{-1} \sigma$, $J_r = f \, J - g_J^{-1} \sigma$, it is straightforward to verify that \eqref{eq:system} is equivalent to \eqref{eq:description_V2} of Proposition \ref{prop:equivalent_def_subspace_V}.
\end{proof}


\appendix

\section{Para-complex geometry} \label{appendixB}

In this appendix we introduce the algebra $\B$ of para-complex numbers, para-complex coordinates, and $\B$-valued symplectic forms. These will allow us to give an equivalent characterization of para-K\"ahler potential (Lemma \ref{lemma:deldelbar}) and prove a criterion to show that a manifold is para-hyperK\"ahler (Lemma \ref{lm:integrability}) . \\

\paragraph{Para-complex numbers} Let $\B$ be the $\R$-algebra generated by $\tau$ with $\tau^{2}=1$. Elements of $\B$ are called para-complex numbers. Notice that $\B$ is a two-dimensional real vector space and we will often denote $\B=\R \oplus \tau\R$. Borrowing terminology from the complex numbers, we talk about real and imaginary part of an element of $\B$ and we define a conjugation 
\[
    \overline{a+\tau b}=a-\tau b \ .
\]
This induces a norm on $\B$ by taking
\[
    |a+\tau b|^{2}=(a+\tau b)\overline{(a+\tau b)}=a^{2}-b^{2} \ .
\]
Note that elements of $\B$ may have negative norm. In fact, the bi-linear extension of this norm defines a Minkowski inner product on $\R^{2}$ with orthonormal basis $\{1,\tau\}$. It is also convenient to work with the basis of idempotents
\[
  e^{+}=\frac{1+\tau}{2} \ \ \ \text{and} \ \ \  e^{-}=\frac{1-\tau}{2} \ , 
\]
because the map 
\[
    ae^{+}+be^{-} \mapsto (a,b)
\]
gives an isomorphism of $\R$-algebras between $\B$ and $\R\oplus \R$, where in the latter operations are carried out component by component. \\

\paragraph{Para-complex structures} Let $\p:V \rightarrow V$ be a para-complex structure on a real vector space $V$. We denote by $V^{+}$ and $V^{-}$ the eigenspaces of $\p$ relative to the eigenvalues $+1$ and $-1$ respectively. 

\begin{remark} The multiplication by $\tau$ on $\B$ is a para-complex structure on $\B$ with eigenspaces $V^{+}=\R e^{+}$ and $V^{-}=\R e^{-}$. 
\end{remark}

Given a para-complex structure $\p$ on $V$, we can define the para-complexification of $V$ as the  $\B$-module $V^{\B}=V\otimes_{\R}\B$. We can then extend the para-complex structure $\p$ to $V^{\B}$ by $\B$-linearity and define
\begin{align*}
    V^{1,0}&=\{v \in V^{\B} \ | \ \p v=\tau v\}=\{v+\tau\p v \ | \ v \in V\} \\
    V^{0,1}&=\{v \in V^{\B} \ | \ \p v=-\tau v\}=\{v-\tau\p v \ | \ v \in V\}
\end{align*}
so that $V^{\B}=V^{1,0}\oplus V^{0,1}$.\\

A para-complex structure $\p$ induces a para-complex structure $\p^{*}$ on the dual space $V^{*}$ by requiring $(\p^{*}\alpha)(v)=\alpha(\p v)$ for any $\alpha \in V^{*}$ and any $v \in V$. As before, we can decompose $V^{*\B}=V_{1,0}\oplus V_{0,1}$ where
\begin{align*}
    V_{1,0}&=\{\alpha \in V^{*\B} \ | \ \p^{*} \alpha=\tau \alpha\}=\{\alpha+\tau\p^{*
    }\alpha \ | \ \alpha \in V^{*}\} \\
    V_{0,1}&=\{\alpha \in V^{*\B} \ | \ \p^{*} \alpha=-\tau \alpha\}=\{\alpha-\tau\p^{*} \alpha \ | \ \alpha \in V^{*}\} \ .
\end{align*}
More in general, we have a decomposition of $\B$-valued $n$-forms into types
\[
    \bigwedge^{n}V^{*\B}=\bigoplus_{p+q=n}\bigwedge^{p,q}V^{*\B} \ , 
\]
where $\bigwedge^{p,q}V^{*\B}$ denotes the vector space spanned by $\alpha \wedge \beta$, with $\alpha \in \bigwedge^{p}V_{1,0}$ and $\beta \in \bigwedge^{q}V_{0,1}$. \\

\paragraph{Para-complex coordinates} Recall that an almost para-complex structure $\p$ on $M$ is a bundle endomorphism $\p:TM\rightarrow TM$ such that $\p^{2}=\1$ and the eigenspaces $T^{\pm}M$ of $\p$ relative to the eigenvalues $\pm1$ have the same dimension.

\begin{definition}Let $(M, \p)$ and $(N,\p')$ be two almost para-complex manifolds. A smooth map $f:M\rightarrow N$ is para-holomorphic if $\p'\circ df=df\circ \p$. In particular, a function $f:(M,\p)\rightarrow \B$ is para-holomorphic if for every vector field $V\in \Gamma(TM)$ we have $df(\p V)=\tau df(V)$. 
\end{definition}

An almost para-complex structure $\p$ on $M$ is said to be integrable if the eigen-distributions $T^{\pm}M$ of $\p$ are involutive. This is equivalent to the existence of local charts $\phi_{\alpha}:U_{\alpha}\subset M \rightarrow \B^{n}$ such that change of coordinates are para-holomorphic functions. If $z^{i}$ is a $\B$-valued coordinate, its real and imaginary part 
\[
    \Re(z^{i})=x^{i}=\frac{1}{2}(z^{i}+\bar{z}^{i}) \ \ \ \ \  \Imm(z^{i})=y^{i}=\frac{1}{2\tau}(z^{i}-\bar{z}^{i})
\]
define usual $\R$-valued coordinates on $M$. \\

The para-complex tangent bundle of a para-complex manifold $(M,\p)$ is $T^{\B}M=TM\otimes_{\R}\B$. Extending the endomorphism $\p$ on $T^{\B}M$ by $\B$-linearity, we have a decomposition $T^{\B}M=T^{1,0}M\oplus T^{0,1}M$ where
\begin{align*}
    T^{1,0}M&=\{X \in T^{\B}M \ | \ \p X=\tau X\}=\{X+\tau\p X \ | \ X \in TM\} \\
    T^{0,1}M&=\{X \in T^{\B}M \ | \ \p X=-\tau X\}=\{X-\tau\p X \ | \ X \in TM\} \ .
\end{align*}
In local $\B$-valued coordinates $z^{i}=x^{i}+\tau y^{i}$, the vector spaces $T^{1,0}M$ and $T^{0,1}M$ are spanned by the vector fields
\[
    \frac{\partial}{\partial z^{i}}=\frac{1}{2}\left( \frac{\partial}{\partial x^{i}}+\tau \frac{\partial}{\partial y^{i}}\right) \ \ \ \ \ \frac{\partial}{\partial \bar{z}^{i}}=\frac{1}{2}\left( \frac{\partial}{\partial x^{i}}-\tau \frac{\partial}{\partial y^{i}}\right) \ .
\]
Similarly, the induced para-complex structure $\p^{*}$ on $T^{*}M$ induces a decomposition $T^{*\B}M=\Lambda^{1,0}(M)\oplus \Lambda^{0,1}(M)$ where 
\begin{align*}
    \Lambda^{1,0}(M)&=\{\alpha \in T^{*\B}M \ | \ \p^{*} \alpha=\tau \alpha\}=\{\alpha+\tau\p^{*} \alpha \ | \ \alpha \in T^{*}M\} \\
    \Lambda^{0,1}(M)&=\{\alpha \in T^{*\B}M \ | \ \p^{*} \alpha=-\tau \alpha\}=\{\alpha-\tau\p^{*} \alpha \ | \ \alpha \in T^{*}M\} \ .
\end{align*}
In local $\B$-valued coordinates $z^{i}=x^{i}+\tau y^{i}$, the vector spaces $\Lambda^{1,0}(M)$ and $\Lambda^{0,1}(M)$ are spanned by the $\B$-valued forms
\[
    dz^{i}=dx^{i}+\tau dy^{i}  \ \ \ \text{and} \ \ \ d\bar{z}^{i}=dx^{i}-\tau dy^{i} \ .
\]
More in general, we have a decomposition of $\B$-valued $n$-forms on $M$ into types 
\[
    \Lambda^{n}(M)=\bigoplus_{p+q=n}\Lambda^{p,q}(M) \ .
\]
This induces a splitting of the $\B$-linear exterior differential $d:\Lambda^{n}(M)\rightarrow \Lambda^{n+1}(M)$ as $d=\partial_{\p} +\bar{\partial}_{\p}$ with 
\[
\begin{gathered}
    \partial_{\p}: \Lambda^{p,q}(M) \rightarrow \Lambda^{p+1,q}(M) \ \ \ \text{and} \ \ \ 
   \bar{\partial}_{\p}: \Lambda^{p,q}(M) \rightarrow \Lambda^{p,q+1}(M) \ .
\end{gathered}
\]

\paragraph{Para-K\"ahler potential} Recall that a para-K\"ahler structure is the data of an almost para-complex structure $\p$ and a (non-degenerate) pseudo-Riemannian metric $\g$ on $M$ such that
\begin{itemize}
    \item $\g(\p X, \p Y)=-\g(X,Y) \ \ \ \ \ \text{for every $X,Y \in \Gamma(TM)$} \ ,$
    \item $\p$ is parallel for the Levi-Civita connection of $\g$.
\end{itemize}
This second condition is equivalent (\cite{homo_parakahler}) to the simultaneous integrability of the eigendistributions of $\p$ and the closedness of the $2$-form $\omega_{\p}=\g(\cdot, \p \cdot)$.
It is easy to see that the eigendistributions of $\p$ are isotropic for $\g$, thus $\g$ has signature $(n,n)$.

\begin{definition} We say that $f:M \rightarrow \R$ is a para-K\"ahler potential of the closed $2$-form $\omega_{\p}$ if $(\tau/2)\bar{\partial}_{\p}\partial_{\p}f=\omega_{\p}$.
\end{definition}
\begin{lemma}\label{lemma:deldelbar}
For every smooth function $f:M \rightarrow \R$ we have $2\tau\bar{\partial}_{\p}\partial_{\p}f=d(df\circ \p)$.
\end{lemma}
\begin{proof} In local para-holomorphic coordinates $z^{i}=x^{i}+\tau y^{i}$ we have
\[
    df\circ \p \left(\frac{\partial}{\partial x^{i}}\right)=df\left(\frac{\partial}{\partial y^{i}}\right)=\frac{\partial f}{\partial y^{i}} \ \ \ \text{and} \ \ \ 
    df\circ \p \left(\frac{\partial}{\partial y^{i}}\right)=df\left(\frac{\partial}{\partial x^{i}}\right)=\frac{\partial f}{\partial x^{i}}
\]
thus 
\[
    df\circ \p=\sum_{i=1}^{n}\frac{\partial f}{\partial y^{i}}dx^{i}+\frac{\partial f}{\partial x^{i}}dy^{i}
\]
and 
\[
    d(df\circ \p)=\sum_{i=1}^{n}\left(\frac{\partial^{2} f}{\partial x^{i}}-\frac{\partial^{2} f}{\partial y^{i}}\right)dx^{i}\wedge dy^{i}\ .
\]
On the other hand
\begin{align*}
    2\tau \bar{\partial}_{\p}\partial_{\p} f&=\tau \bar{\partial}_{\p}\sum_{i=1}^{n}\left(\frac{\partial f}{\partial x^{i}}+\tau\frac{\partial f}{\partial y^{i}}\right)(dx^{i}+\tau dy^{i})\\
    &=\frac{1}{2}\tau\sum_{i=1}^{n}\left(\frac{\partial^{2} f}{\partial x^{i}}-\frac{\partial^{2} f}{\partial y^{i}}\right)(dx^{i}-\tau dy^{i})\wedge(dx^{i}+\tau dy^{i})\\
    &=\sum_{i=1}^{n}\left(\frac{\partial^{2} f}{\partial x^{i}}-\frac{\partial^{2} f}{\partial y^{i}}\right)dx^{i}\wedge dy^{i}\ .
\end{align*}
\end{proof}

\paragraph{Criterion for a para-hyperK\"ahler structure}

Recall that a para-hyperK\"ahler structure on $M$ is a quadruple $(\g,\i,\j,\k)$ where $(\g,\k)$ and $(\g,\j)$ are para-K\"ahler structures on $M$ and $\i=\k\j$ is an almost complex structure on $M$ that is compatible with $\g$ in the sense that
\[
    \g(\i X,\i Y)=\g(X,Y) \ \ \ \ \text{for all $X,Y\in \Gamma(TM)$} \ 
\] 
and that is parallel for the Levi-Civita connection of $\g$.

\begin{lemma}\label{lm:integrability} Let $\j,\k$ be almost para-complex structures and let $\i=\k\j$ be an almost complex structure on $M$. Assume that there is pseudo-Riemannian metric $\g$ compatible with $\i,\j$ and $\k$. If the $2$-forms $\omega_{\Ll}=\g(\cdot,\Ll\cdot)$ are closed for every $\Ll=\i,\j,\k$, then $(M,\i,\j,\k,\g)$ is para-hyperK\"ahler.
\end{lemma}
\begin{proof} We first extend the differential forms $\omega_{\Ll}$ to $\B$-valued differential forms by $\B$-linearity. Because
\[
    \omega_{\i}(X,Y)=\g(X,\i Y)=-\g(X,\j\k Y)=\g(\j X, \k Y)=\omega_{\k}(\j X, Y)
\]
for any section $X$ of $T^{\B}M$, we have $\j X=\tau X$ if and only if $\iota_X\omega_{\i}=\tau\iota_{X}\omega_{\k}$. Suppose that $\j X=\tau X$ and $\j Y=\tau Y$, then
\begin{align*}
 \iota_{[X,Y]}\omega_{\i}&=\mathcal{L}_{X}(\iota_{Y}\omega_{\i})-\iota_Y\mathcal{L}_{X}\omega_{\i} \\
 &=\mathcal{L}_{X}(\tau\iota_{Y}\omega_{\k})-\iota_Y\mathcal{L}_{X}\omega_{\i} \\
 &=\mathcal{L}_{X}(\tau\iota_{Y}\omega_{\k})-\iota_Yd(\iota_{X}\omega_{\i})  \tag{Cartan's formula with $\omega_{\i}$ closed} \\
 &=\mathcal{L}_{X}(\tau\iota_{Y}\omega_{\k})-\iota_Yd(\tau\iota_{X}\omega_{\k}) \\
 &=\mathcal{L}_{X}(\tau\iota_{Y}\omega_{\k})-\tau\iota_Y\mathcal{L}_{X}\omega_{\k} \tag{Cartan's formula with $\omega_{\k}$ closed}\\
 &=\tau \iota_{[X,Y]}\omega_{\k} \ .
\end{align*}
Hence $\j[X,Y]=\tau[X,Y]$ and and the eigendistributions for $\j$ are involutive. The integrability of $\j$ and the closedness of $\omega_{\j}$ imply that $\j$ is parallel for the Levi-Civita connection of $\g$. Repeating a similar computation for $\i$ and $\k$ we obtain the result. 
\end{proof}

\emergencystretch=1em

\bibliographystyle{alpha}
\bibliography{biblio}

\end{document}